\documentclass[10pt]{amsbook}

\usepackage{amscd,verbatim,amssymb,amsmath,stmaryrd,mathrsfs}
\usepackage[all]{xy}
\usepackage[colorlinks,linkcolor=blue,citecolor=blue,urlcolor=red]{hyperref}
\usepackage{bm}
\usepackage[utf8x]{inputenc}
\usepackage{ulem} 
\usepackage{graphicx}
\usepackage{url}
\usepackage{bbm}
\usepackage[multiple]{footmisc}
\usepackage{wasysym} 
\usepackage{imakeidx} 
\makeindex
\makeindex[name=not,title=Index of notation]

\usepackage{newunicodechar}
\newunicodechar{∞}{\ensuremath{\infty}}
\newcommand{\uS}{\ensuremath{\underline{\Sigma}}}

\newcommand{\PP}{\mathbb{P}}

\newcommand{\NN}{\mathbb{N}}

\renewcommand{\AA}{\mathbb{A}}

\newcommand{\C}{\mathbb{C}}
\newcommand{\CC}{\mathbb{C}}
\newcommand{\FF}{\mathbb{F}}
\newcommand{\GG}{\mathbb{G}}
\newcommand{\Q}{\mathbb{Q}}
\newcommand{\QQ}{\mathbb{Q}}

\newcommand{\Z}{\mathbb{Z}}
\newcommand{\ZZ}{\mathbb{Z}}
\newcommand{\sA}{\mathcal{A}}
\newcommand{\sB}{\mathcal{B}}
\newcommand{\sC}{\mathcal{C}}
\newcommand{\sD}{\mathcal{D}}

\newcommand{\sI}{\mathcal{I}}
\newcommand{\sJ}{\mathcal{J}}
\newcommand{\sL}{\mathcal{L}}
\newcommand{\sO}{\mathcal{O}}
\newcommand{\OO}{\mathcal{O}}
\newcommand{\sP}{\mathcal{P}}
\newcommand{\sQ}{\mathcal{Q}}
\newcommand{\sR}{\mathcal{R}}
\newcommand{\sS}{\mathcal{S}}
\newcommand{\sT}{\mathcal{T}}
\newcommand{\sU}{\mathcal{U}}
\newcommand{\sV}{\mathcal{V}}
\newcommand{\sW}{\mathcal{W}}
\newcommand{\sX}{\mathcal{X}}
\newcommand{\sY}{\mathcal{Y}}
\newcommand{\sZ}{\mathcal{Z}}

\newcommand{\bZ}{\mathbb{Z}}

\newcommand{\m}{\mathfrak{m}}

\newcommand{\p}{\mathfrak{p}}

\newcommand{\q}{\mathfrak{q}}

\newcommand{\Cor}{\operatorname{\mathbf{Cor}}}
\newcommand{\SmCor}{\operatorname{\mathbf{SmCor}}}

\newcommand{\hA}{\ol{A}}
\newcommand{\hB}{\ol{B}}
\newcommand{\hC}{\ol{C}}

\newcommand{\hM}{\ol{M}}
\newcommand{\hN}{\ol{N}}

\newcommand{\hP}{\ol{P}}
\newcommand{\hQ}{\ol{Q}}
\newcommand{\hR}{\ol{R}}
\newcommand{\hS}{\ol{S}}
\newcommand{\hT}{\ol{T}}
\newcommand{\hU}{\ol{U}}
\newcommand{\hV}{\ol{V}}
\newcommand{\hW}{\ol{W}}
\newcommand{\hX}{\ol{X}}
\newcommand{\hY}{\ol{Y}}
\newcommand{\hZ}{\ol{Z}}
\newcommand{\iA}{A^\circ}

\newcommand{\iM}{M^\circ}

\newcommand{\iP}{P^\circ}
\newcommand{\iQ}{Q^\circ}
\newcommand{\iR}{R^\circ}
\newcommand{\iS}{S^\circ}
\newcommand{\iT}{T^\circ}
\newcommand{\iU}{U^\circ}
\newcommand{\iV}{V^\circ}
\newcommand{\iW}{W^\circ}
\newcommand{\iX}{X^\circ}
\newcommand{\iY}{Y^\circ}
\newcommand{\iZ}{Z^\circ}
\newcommand{\mA}{A^\infty}
\newcommand{\mB}{B^\infty}
\newcommand{\mC}{C^\infty}

\newcommand{\mM}{M^\infty}

\newcommand{\mP}{P^\infty}
\newcommand{\mQ}{Q^\infty}

\newcommand{\mS}{S^\infty}
\newcommand{\mT}{T^\infty}
\newcommand{\mU}{U^\infty}

\newcommand{\mW}{W^\infty}
\newcommand{\mX}{X^\infty}
\newcommand{\mY}{Y^\infty}
\newcommand{\mZ}{Z^\infty}

\newcommand{\ul}[1]{{\underline{#1}}}

\newcommand{\Set}{{\operatorname{\mathbf{Set}}}}
\newcommand{\Pro}{{\operatorname{\mathbf{Pro}}}}

\newcommand{\PSh}{{\operatorname{\mathbf{PSh}}}}
\newcommand{\Shv}{{\operatorname{\mathbf{Shv}}}}

\newcommand{\ulMH}{\operatorname{\mathbf{\underline{M}H}}}
\newcommand{\ulMDM}{\operatorname{\mathbf{\underline{M}DM}}}
\newcommand{\ulMDA}{\operatorname{\mathbf{\underline{M}DA}}}

\newcommand{\Hom}{\operatorname{Hom}}

\newcommand{\Pic}{\operatorname{Pic}}

\newcommand{\Frac}{\operatorname{Frac}}

\newcommand{\Spec}{\operatorname{Spec}}

\newcommand{\Sm}{\operatorname{\mathbf{Sm}}}

\newcommand{\Sch}{\operatorname{\mathbf{Sch}}}
\newcommand{\Fppf}{{\operatorname{\mathbf{Fppf}}}}
\newcommand{\Qf}{{\operatorname{\mathbf{Qf}}}}
\newcommand{\Et}{{\operatorname{\mathbf{Et}}}}
\newcommand{\Op}{{\operatorname{\mathbf{Op}}}}

\newcommand{\tr}{{\operatorname{tr}}}

\newcommand{\eff}{{\operatorname{eff}}}

\renewcommand{\min}{{\operatorname{min}}}
\newcommand{\fin}{{\operatorname{fin}}}
\newcommand{\qcqs}{{\operatorname{qcqs}}}
\newcommand{\qcs}{{\operatorname{qc.sep}}}

\newcommand{\amb}{{\operatorname{amb}}}

\renewcommand{\o}{{\operatorname{o}}}
\newcommand{\op}{{\operatorname{op}}}

\newcommand{\red}{{\operatorname{red}}}

\newcommand{\Zar}{{\operatorname{Zar}}}
\newcommand{\Nis}{{\operatorname{Nis}}}
\newcommand{\et}{{\operatorname{\acute{e}t}}}
\newcommand{\qfh}{{\operatorname{qfh}}}
\newcommand{\fppf}{{\operatorname{fppf}}}

\newcommand{\id}{{\operatorname{Id}}}

\newcommand{\pr}{{\operatorname{pr}}}

\renewcommand{\lim}{\operatornamewithlimits{\varprojlim}}
\newcommand{\amblim}{\operatornamewithlimits{{\varprojlim^p}}}
\newcommand{\colim}{\operatornamewithlimits{\varinjlim}}

\newcommand{\ol}{\overline}

\renewcommand{\epsilon}{\varepsilon}

\newcommand{\length}{{\operatorname{length}}}

\newcommand{\Bl}{{\mathbf{Bl}}}

\newcommand{\ulP}{\operatorname{\mathrm{\underline{P}}}}
\newcommand{\ulM}{\operatorname{\mathrm{\underline{M}}}}
\newcommand{\ulPtau}{\operatorname{\mathrm{\tau}}}
\newcommand{\ulPZar}{\operatorname{\mathrm{Zar}}}
\newcommand{\ulPNis}{\operatorname{\mathrm{Nis}}}
\newcommand{\ulPet}{\operatorname{\mathrm{\acute{e}t}}}
\newcommand{\ulPfppf}{\operatorname{\mathrm{fppf}}}

\newcommand{\ulPrqfh}{\operatorname{\mathrm{rqfh}}}
\newcommand{\ulPfin}{\operatorname{\mathrm{fin}}}
\newcommand{\ulMtau}{\operatorname{\mathrm{\underline{M}\tau}}}
\newcommand{\ulMZar}{\operatorname{\mathrm{\underline{M}Zar}}}
\newcommand{\ulMNis}{\operatorname{\mathrm{\underline{M}Nis}}}
\newcommand{\ulMet}{\operatorname{\mathrm{\underline{M}\acute{e}t}}}
\newcommand{\ulMfppf}{\operatorname{\mathrm{\underline{M}fppf}}}
\newcommand{\ulMfpqc}{\operatorname{\mathrm{\underline{M}fpqc}}}

\newcommand{\ulMrqfh}{\operatorname{\mathrm{\underline{M}rqfh}}}

\newcommand{\rqfh}{\operatorname{\mathrm{rqfh}}}
\newcommand{\finite}{\operatorname{\mathrm{fin}}}

\newcommand{\Fun}{\operatorname{\mathrm{Fun}}}

\newcommand{\mc}{\ZZ_{\tr}}
\newcommand{\mcq}{\QQ_{\tr}}
\newcommand{\ZZtr}{\ZZ_{\tr}}
\newcommand{\QQtr}{\QQ_{\tr}}

\newcommand{\bcube}{{\ol{\square}}}

\newcommand{\SCH}{\operatorname{\mathbf{SCH}}}
\newcommand{\ulPSm}{\operatorname{\mathbf{\underline{P}Sm}}}
\newcommand{\ulPQf}{\operatorname{\mathbf{\underline{P}Qf}}}
\newcommand{\ulPFppf}{\operatorname{\mathbf{\underline{P}Fppf}}}
\newcommand{\ulPEt}{\operatorname{\mathbf{\underline{P}Et}}}
\newcommand{\ulPOp}{\operatorname{\mathbf{\underline{P}Op}}}
\newcommand{\ulPSch}{\operatorname{\mathbf{\underline{P}Sch}}}
\newcommand{\ulPSCH}{\operatorname{\mathbf{\underline{P}SCH}}}
\newcommand{\ulMSm}{\operatorname{\mathbf{\underline{M}Sm}}}
\newcommand{\ulMQf}{\operatorname{\mathbf{\underline{M}Qf}}}
\newcommand{\ulMFppf}{\operatorname{\mathbf{\underline{M}Fppf}}}
\newcommand{\ulMEt}{\operatorname{\mathbf{\underline{M}Et}}}
\newcommand{\ulMOp}{\operatorname{\mathbf{\underline{M}Op}}}
\newcommand{\ulMSch}{\operatorname{\mathbf{\underline{M}Sch}}}
\newcommand{\ulMSCH}{\operatorname{\mathbf{\underline{M}SCH}}}

\newcommand{\ulMAGP}{\operatorname{\mathbf{\underline{M}AGP}}}

\newcommand{\ulMCor}{\operatorname{\mathbf{\underline{M}Cor}}}
\newcommand{\ulMSmCor}{\operatorname{\mathbf{\underline{M}SmCor}}}

\newcommand{\ambtimes}[1][]{\stackrel{p}{\times_{#1}}}

\newcommand{\Comp}{\operatorname{\mathbf{Comp}}}

\newcommand{\KMSY}{\operatorname{\mathbf{KMSY}}}

\def\bZ{\mathbb{Z}}
\def\Ztr{\bZ_\tr}

\newcounter{spec}
{\end{list}}%

\newcommand{\aab}{abstract admissible blowup{}}
\newcommand{\abb}{abstract admissible blowup{}}
\newcommand{\amm}{ambient minimal morphism}
\renewcommand{\mp}{modulus pair}

\newtheorem{lemma}{Lemma}[chapter]
\newtheorem{lemm}[lemma]{Lemma}
\newtheorem{thm}[lemma]{Theorem}
\newtheorem{theo}[lemma]{Theorem}

\newtheorem*{thm*}{Theorem}

\newtheorem{prop}[lemma]{Proposition}

\newtheorem{cor}[lemma]{Corollary}
\newtheorem{coro}[lemma]{Corollary}

\theoremstyle{definition}
\newtheorem{defi}[lemma]{Definition}
\newtheorem{defn}[lemma]{Definition}
\newtheorem{constr}[lemma]{Construction}

\newtheorem{axio}[lemma]{Axiom}
\newtheorem{property}[lemma]{Property}
\theoremstyle{remark}
\newtheorem{warn}[lemma]{Warning}
\newtheorem{ques}[lemma]{Question}

\newtheorem{rema}[lemma]{Remark}
\newtheorem{remark}[lemma]{Remark}

\newtheorem{exam}[lemma]{Example}

\newtheorem{claim}[lemma]{Claim}

\numberwithin{equation}{chapter}

\def\Comp{\Comp^{\fin}}

\def\ulMSm{\operatorname{\mathbf{\ul{M}Sm}}}

\def\Comp{\operatorname{\mathbf{Comp}}}

\begin{document}

\title[]
{Modulus sheaves with transfers}
\author[S. Kelly]{Shane Kelly}
\address{
Tokyo Institute of Technology, 
2-12-1 Ookayama, Meguro-ku, 
Tokyo 152-8551, Japan}
\email{shanekelly@math.titech.ac.jp}
\author[H. Miyazaki]{Hiroyasu Miyazaki}
\address{RIKEN iTHEMS, Wako, Saitama 351-0198, Japan}
\email{hiroyasu.miyazaki@riken.jp}
\date{\today}

\begin{abstract}
We generalise Kahn, Miyazaki, Saito, Yamazaki's theory of modulus pairs to pairs $(\hX, \mX)$ consisting of a qcqs scheme $\hX$ equipped with an effective Cartier divisor $\mX$ representing a ramification bound. 

We develop theories of sheaves on such pairs for modulus versions of the Zariski, Nisnevich, étale, fppf, and qfh-topologies. We extend the Suslin-Voevodsky theory of correspondances to modulus pairs, under the assumption that the interior $\iX = \hX \setminus \mX$ is Noetherian. %

The resulting point of view highlights connections to (Raynaud-style) rigid geometry, and potentially provides a setting where wild ramification can be compared with irregular singularities.

This framework leads to a homotopy theory of modulus pairs $\ulMH(\sS)$ and a theory of motives with modulus $\ulMDM^\eff(\sS)$ over a general base $\sS = (\hS, \mS)$. For example, the case where $\hS$ is the spectrum of a rank one valuation ring (of mixed or equal characteristic) equipped with a choice $\mS$ of pseudo-uniformiser is allowed.
%
\end{abstract}


\maketitle
\setcounter{tocdepth}{2}
\tableofcontents
%
%

\section{Introduction}

\textbf{Background.} Singular cohomology, $\ell$-adic cohomology, and crystalline cohomology all have very similar behaviour, suggesting that they are shadows of some underlying theory. This observation lead Grothendieck to introduce the category of pure motives in the 60's. An object of this category is simply a symbol 
\[ (X, p, n) \]
where $X$ is a smooth projective variety, $p$ a projector,\footnote{More precisely, $p$ is an element of $CH_{\dim X}(X \times X)$ which satisfies $p \circ p = p$, where $\circ$ is a certain map $CH_{\dim X}(X \times Y) \times CH_{\dim Y}(Y \times Z) \to CH_{\dim X}(X \times Z)$ generalising composition.} and $n$ an integer. This symbol represents the piece of the cohomology of $X$ cut out by $p$ and twisted by $n$, \cite{Sch94}.\footnote{The interpretation of ``twisted'' depends on the cohomology theory in question. For $\ell$-adic cohomology, it means tensoring with the Galois module $\varprojlim_n \mu_{\ell^n}$. For de Rham cohomology, it means tensoring with the Hodge structure $2 \pi i \QQ$.} Anything proven about pure motives can then be converted into a statement about the relevant cohomology theory by applying a realisation functor.

A shortcoming of Grothendieck's theory of pure motives is that it is only valid for smooth projective varieties. In the 90's, a theory of mixed motives was developed by Voevodsky and others, valid for not necessarily smooth, not necessarily projective varieties. From the modern point of view, a Voevodsky motive is also represented by a symbol---a diagram of smooth varieties\footnote{Readers familiar with Voevodsky's work may complain that a Voevodsky motive is a complex of presheaves with transfers, not a diagram, but this is the same thing: a presheaf $F: SmCor \to D(Ab)$ with values in the $\infty$-category of complexes of abelian groups correspond to the diagram $\int F \to SmCor$ indexed by the slice $\infty$-category $(SmCor \downarrow F)$.}---which represents a recipe to construct a group out of pieces of cohomology groups of various smooth varieties using morphisms of smooth varieties, and various transfer maps. Again, anything proven about Voevodsky motives can then be converted into a statement about the relevant cohomology theory by applying a realisation functor.

A shortcoming of Voevodsky's theory is that it cannot see ``additive'' information (in the sense of $\GG_a$ vs{.} $\GG_m$). For example, Voevodsky motives cannot see the unipotent part of the Picard group of a cusp, or $H^\ast(-, \OO_-)$ in general characteristic, and cannot see wild ramification or $H^\ast_{et}(-, \ZZ / p^n)$ in characteristic $p$.

An ever present tool in the study of cohomology in general, and ramification in particular, is filtration. An important example would be Deligne's weight filtration on the $\CC$-linear cohomology of a smooth $\CC$-variety, which measures the behaviour of cohomology classes ``at infinity''. Another important example is the Brylinski-Kato filtration on $\QQ_p / \ZZ_p$-linear étale cohomology of a smooth $\FF_p$-variety. This also measures behaviour of cohomology classes ``at infinity''. A motivating example in our story is the generalised Jacobian $J(\hX,\mX)$ associated to a curve $\hX$ with a divisor $\mX$, which was introduced by Rosenlicht \cite{Rosenlicht} and plays a fundamental role in Rosenlicht-Serre's  class field theory for curves \cite{SerreCFT}: as one considers various $\mX$ with the same support, one gets a system of extensions $0 \to U(\hX,\mX) \to J(\hX,\mX) \to J(\hX) \to 0$ of the Jacobian of the total space by various (mostly) unipotent algebraic groups.

In recent years, a theory of motives ``with modulus'' was introduced in a series of works \cite{kmsy1}, \cite{kmsy2} \cite{kmsy3} by Kahn-Miyazaki-Saito-Yamazaki (based on  \cite{nistopmod} by the second author and \cite{kami} by Kahn-Miyazaki), in the case that the base is the spectrum of a field. 
Similar to the way Grothendieck's theory uses symbols $(X, p, n)$ to represent pieces of cohomology, KMSY's theory uses symbols $(\hX, \mX)$ to represent a filtered piece of the cohomology of $\iX := \hX \setminus \mX$. In their work, $\hX$ is a normal variety, $\mX$ is an effective Cartier divisor, $\iX$ is asked to be smooth, and $(\hX, \mX)$ represents the cohomology of $\iX$ whose ramification at infinity is bounded by $\mX$.

\textbf{The current work.} The main purpose of this paper is to establish a foundation to construct the theory of modulus pairs \emph{over an arbitrary base}, which generalises the theory over the field. As motivation, let us point out two targets our theory can be applied to:
\begin{enumerate}
 \item The construction of a homotopy category of modulus pairs $\ulMH(\sS)$ over a general qcqs basic pair $\sS$, satisfying a localisation theory à la Morel-Voevodsky.

 \item An Ayoub-style category of étale motives $\ulMDA^\eff(\sS)$ (again over a general qcqs base pair $\sS$) which recovers KMSY's category with rational coefficients over a field. Notably, this gives a construction of $\ulMDM^\eff(k)$ that does not explicitly use transfers, and therefore makes realisation functors much easier to define.
\end{enumerate}

\textbf{Future directions.} As we just mentioned, the material in this work can be used to develop modulus versions of the Morel-Voevodsky unstable homotopy category and Ayoub's $\mathbf{DA}^\eff$. One great advantage of not having transfers explicitly built into the definitions is that realisation functors become easier to define. In work in progress we can define a de Rham realisation functor which recovers R{\"u}lling-Saito's presheaf $\tilde{\Omega}^q$ of K\"ahler differentials with modulus \cite[\S 6]{RS18} which have smooth total space and modulus with sncd support. The definition leads naturally to a definition of Hochschild homology for modulus pairs, with an evident HKR isomorphism. There also seems to be a theory of noncommutative motives with modulus coming into view.

We mention three more future directions which are more speculative. Firstly, there are strong analogies between wild ramification (in positive characteristic) and irregular singularities (in characteristic zero). Having a good theory of motives which can see both of these non-$\AA^1$-invariant phenomenon \emph{and} exists in mixed characteristic provides a setting where one might hope to directly compare these two phenomenon.

Secondly, there is a history of cycle theoretic / motivic interpretations of various abelianised fundamental groups of schemes. One state of the art result is Kerz-Saito's isomorphism for smooth varieties $\iX$ over a finite field $\FF_q$ (with $q$ odd), \cite[Thm.III]{KS16},
\[ \varprojlim_{(\hX, \mX)} C(\hX, \mX)^0 \cong \pi_1^{ab}(\iX)^0.  \]
Here, the limit is over proper $\hX$ equipped with an effective Cartier divisor $\mX$ such that $\iX = \hX \setminus \mX$, the group $C(\hX, \mX)$ is a certain quotient of zero cycles on $\iX := \hX \setminus \mX$ sensitive to the multiplicity of $\mX$, the group $\pi_1^{ab}(\iX)$ is the abelianised étale fundamental group of $\iX$, and the $(-)^0$ means the contributions from $\overline{\FF}_q / \FF_q$ are ignored. Notice that wild ramification information is visible. 
Having access to a \emph{non-abelian} motivic homotopy category suggests that one might be able to git a motivic description of the entire (i.e., not abelianised) fundamental group.

Thirdly, from Artin's point of view, moduli spaces and stacks are best handled when arbitrary rings (or qcqs schemes) are allowed. In Chapter~\ref{chap:relCyc} we develop the modlus version of the Suslin-Voevodsky paper on relative cycles, \cite{SVRelCyc}. We only develop the parts we need to define motives with modulus over a general base, but there seems to be no serious obstacle to obtaining modulus versions of everything in \cite{SVRelCyc}. There is certainly another paper to be written on this topic, leading to a modulus theory of Hilbert schemes and Chow varieties.

\textbf{Definitions.} To save the reader the pain of chasing definitions, let's give an executive summary of the construction of $\ulMDA^\eff(\sS)$ (the construction of $\ulMH(\sS)$ is analogous).

A \emph{modulus pair} is a pair $(\hX, \mX)$ where $\hX$ is \emph{any} scheme, and $\mX$ an effective Cartier divisor. An \emph{ambient morphism} $(\hX, \mX) \to (\hY, \mY)$ of modulus pairs is a morphism $f: \hX \to \hY$ of the underlying schemes, such that $\mX \geq f^*\mY$. Modulus pairs together with ambient morphisms form an obvious category $\ulPSCH$, and we obtain the \emph{category $\ulMSCH$ of modulus pairs} by formally inverting the class $\Sigma$ of \emph{abstract admissible blowups}: ambient morphisms $(\hX, \mX) \to (\hY, \mY)$ such that $f: \hX \to \hY$ is proper, $\mX = f^*\mY$, and $\iX := \hX \setminus \mX \to \iY := \hY \setminus \mY$ is an isomorphism. So 
\[ \ulMSCH := \ulPSCH[\Sigma^{-1}]. \]
Given a modulus pair $\sS$, we write $\ulMSm_{\sS} \subseteq \ulMSCH / \sS$ for the full subcategory of the slice category $\ulMSCH / \sS$, whose objects are isomorphic to ambient morphism $f: (\hX, \mX) \to (\hS, \mS)$ such that $f: \hX {\to} \hS$ is of finite type, and $\iX {\to} \iS$ is smooth. The \emph{$\ulMet$-topology} on $\ulMSm_\sS$ is the topology generated by families $\{f_i: (\hU_i, \mU_i) \to (\hX_i, \mX_i)\}_{i \in I}$ of ambient morphisms such that $\{\hU_i \to \hX_i\}_{i \in I}$ is an étale covering in the classical sense, and $\hU_i = f_i^*\hX$ for each $i$. Finally, to a modulus pair $\sX = (\hX, \mX)$ we associate 
\[ \bcube_\sX = (\PP^1 {\times} \hX, \quad \PP^1 {\times}\mX {+} \hX {\times} \{\infty\}), \]
and define
\[ \ulMDA^\eff(\sS) := \frac{D(\Shv_{\ulMet}(\ulMSm_\sS, \QQ))}{\langle \bcube_\sX \to \sX \rangle } \]
as the Verdier quotient of the derived category of $\ulMet$-sheaves of $\QQ$-vector spaces on $\ulMSm_\sS$ by the subcategory generated by cones of the canonical morphisms $\bcube_\sX \to \sX$ as $\sX$ ranges over objects of $\ulMSm_\sS$. Notice that the functor $\ulMSm_\sS \to \Sm_{\iS}; \sX \mapsto \iX$ induces a canonical functor
\[ \ulMDA^\eff(\sS) \to \mathbf{DA}^\eff(\iS) := \frac{D(\Shv_{\et}(\Sm_{\iS}, \QQ))}{\langle \AA^1_X \to X \rangle }. \]
This functor has a right adjoint which is fully faithful and induces a semi-orthogonal decomposition
\[ \ulMDA^\eff(\sS) = \langle \mathcal{U}, \mathbf{DA}^\eff(\iS) \rangle. \] 
We expect $\mathcal{U}$ is the subcategory generated by ``unipotent'' objects in some appropriate to-be-discovered sense.

\textbf{Connection to other works.} The KMSY theory of motives with modulus has a close relationship with Kahn-Saito-Yamazaki's \emph{reciprocity sheaves} \cite{KSY16}, a generalisation of Ivorra-R\"ulling's \emph{reciprocity functors} \cite{IR16}.
Reciprocity sheaves include all of the following notions: Voevodsky's $\AA^1$-homotopy invariant sheaf, any commutative algebraic group, the sheaf of K\"ahler differentials and the de Rham-Witt sheaf. 
Note that all of the above examples (except the first one) are not $\AA^1$-homotopy invariant, and therefore cannot be captured by Voevodsky motives.

The theory of motives with modulus already has several applications.
For example, in \cite{RS18}, R\"ulling-Saito developed the theory of \emph{motivic conductors}, which gives a universal constructions of various conductors which capture information of ramification of coverings  and irregularity of rank 1 integrable connections. 
On the other hand, R\"ulling-Sugiyama-Yamazaki made in \cite{RSY19} a detailed study of the tensor structure of modulus sheaves, and applied it to the computation of \emph{zero cycles with modulus}. 
Cycles with modulus have been studied by many authors.
Bloch-Esnault introduced \emph{additive higher Chow group} for $0$-dimensional varieties in \cite{BE03} to give a cycle-theoretic description of the sheaf of K\"ahler differential. 
Jinhyun Park extended their construction for general varieties in \cite{Park08}. Park, Krishna and Levine studied many properties of the additive higher Chow groups in \cite{KL08}, \cite{KP} etc.
On the other hand, Kerz-Saito introduced \emph{Chow group of zero cycles with modulus} in \cite{KS16} to establish higher dimensional class field theory over finite fields. 
The above two groups of ``cycles with modulus'' are simultaneously generalised to \emph{higher Chow groups with modulus} by Binda-Saito in \cite{BS19}.
It is expected that higher Chow group with modulus should have a close relation with motives with modulus, as Bloch's higher Chow group with modulus described the Hom-group of Voevodsky motives. 

We expect there is also a connection to the theory of log-motives developed by Binda, Park, {\O}stv{\ae}r, \cite{BPO}.

\textbf{Outline.} Now let us give an outline of this work.

Chapter~\ref{chap:modPairs} can be seen as the beginning of an EGA for modulus pairs. The first main achievement is the existence of a categorical fibre product of modulus pairs. This is much more than a curiosity; for example it unlocks functoriality for modulus sites. It also encodes a modulus version of ``strict transform'' which we use to prove a Raynaud-Gruson-esque ``finite-ification'' result, Thm.~\ref{thm:finHeavy}. Chap.~\ref{chap:modPairs} Sec.~\ref{sec:proModPairs} can be seen as a modulus version of the section ``Limites projectives de préschémas'', \cite[\S 8]{EGAIV3}.

Chapter~\ref{chap:coverings} develops modulus versions of the Zariski, Nisnevich, étale, fppf, and qfh topologies. They are the default choices in the sense that they make the canonical functors $\ulPSCH \to \ulMSCH$ and $\ulPSCH \to \SCH; \sX \mapsto \hX$  continuous. As we already mentioned, Thm.~\ref{thm:finHeavy} is the ``heaviest'' technical result, using quasi-compactness of the Riemann-Zariski space, and special properties of valuation rings, cf.Lem.~\ref{lemm:valRingProperFlatGenFin}. 

Chapter~\ref{chap:locProMod} is one half of the modulus version of the first author's paper with Gabber, \cite{GK15}. We identify the local rings for the various modulus topologies. Basically, they are local rings for the non-modulus topologies with a valuation ring glued to the closed point, cf. semi-valuation rings, \cite[Rem.2.1.2]{Tem11}. One way to understand this result would be to think of the small modulus sites\footnote{Cf. $\ulMQf_\sS$, $\ulMFppf_\sS$, $\ulMEt_\sS$, $\ulMOp_\sS $ in Def.\ref{defi:mtop}.} as Nisnevich (resp. étale, fppf, qfh) versions of relative Riemann-Zariski spaces, \cite{Tem11}. The proofs in this chapter may look longish, but there is nothing particularly surprising in them, except maybe for the fact that its possible to give a usable description of the $\ulMfppf$-local modulus pairs even though there is no concrete description of $\fppf$-local schemes. The chapter finishes with the modulus version of the result that a finite scheme over a hensel local ring is a finite product of hensel local rings, Lem.~\ref{lemm:finiteulMNislocal}. This lemma is a basic ingredient in Voevodsky's proof that sheafification preserves transfers. 

In Chapter~\ref{sec:sites} we define various sites, and do the other half of the modulus version of \cite{GK15}. That is, we use local modulus pairs to describe conservative families of fibre functors for our modulus sites. This was surprisingly fiddly, due to the fact that the modulus version of the qfh-topology uses finite type morphisms instead of finite presentation morphisms, so working with pro-objects becomes more delicate.

In Chap.~\ref{sec:sites}, Sec.~\ref{sec:extending} we apply our theory of points to show that given a modulus pair $\sX = (\hX, \mX)$, every non-modulus covering of the interior $\iX$ can be extended to a modulus covering of $\sX$. As consequence, we know that the various functors of the form  $\Shv(\Sch_{{\iS}}) \to \Shv(\ulMSch_{\sS})$ are exact.

In Chap.~\ref{sec:sites}, Sec.~\ref{sec:hypercompleteness} we prove hypercompleteness of the modulus Zariski and modulus Nisnevich sites. Due to the fact that the known cd-structures are not bounded, we cannot use Voevodsky's proof. Instead, we observe that the small modulus Zariski (resp. Nisnevich) sites are the filtered limits of the non-modulus ones over all admissible blowups. Then we use Clausen-Mathew's result \cite[Cor.3.11]{2019arXiv190506611C} that filtered limits preserve homotopy dimension. Hypercompleteness is a consequence of homotopy dimension being locally finite.

Chapter~\ref{chap:relCyc} is the modulus version of the Suslin-Voevodsky paper on relative cycles, \cite{SVRelCyc}. We only develop the parts we need to define motives with modulus over a general base, but there seems to be no serious obstacle to obtaining modulus versions of everything in \cite{SVRelCyc}. There is certainly another paper to be written on this topic, leading to a modulus theory of Hilbert schemes and Chow varieties.

Chapter~\ref{chap:shTr} develops the modulus theory of sheaves with transfers. In this chapter and the previous chapter we move from qcqs generality to qc separated total space, Noetherian interior. This is because there are at least two divergent theories of cycles on non-Noetherian schemes, exhibited by the following phenomenon: on an infinite profinite set $S = \lim {S_{λ}}$, the free abelian group of points $\ZZ S$ is not the filtered colimit $\varinjlim \ZZ S_{λ}$ of the free abelian groups $\ZZ S_{λ}$ of the finite quotients.

In Chapter~\ref{chap:box} we begin looking towards motives. There is an alternative non-categorical product of modulus pairs that has traditionally been used to set up the theory. We call this product the ``box product'' and use the notation $\boxtimes$ to highlight its similarity to the exterior product of sheaves. One of the main obstacles in developing the theory of motives with modulus along Voevodsky's lines is that the diagonal $\bcube \to \bcube \boxtimes \bcube$ is not a well-defined morphism of modulus pairs. This problem disappears if we use the categorical fibre product, but the box product is still useful. In Morel-Voevodsky's proof of localisation, a certain contraction is performed which only works with the box product. In Chapter~\ref{chap:box} we show that, up to admissible blowup and Zariski localisation, the functors $\bcube \boxtimes - $ and $\bcube \times - $ are isomorphic. So they can be used interchangeably when developping the theory of motives with modulus (or homotopy theory of modulus pairs).

\subsection*{Acknowledgements.} 

{\small
The first author is supported by the JSPS KAKENHI Grant (19K14498). The second author is supported by RIKEN Special Postdoctoral Researchers (SPDR) Program, by RIKEN Interdisciplinary Theoretical and Mathematical Sciences Program (iTHEMS), and by JSPS KAKENHI Grants (19K23413) and (21K13783).}


\chapter{Categories of modulus pairs} \label{chap:modPairs}

\section{The category ${\protect \ulPSCH}$}

\newcommand{\newterm}[1]{\emph{#1}\marginpar{\fbox{#1}}\index{#1}}
\newcommand{\newsymb}[2]{$#2$\marginpar{\fbox{$#2$}}\index[not]{#1@$#2$}} %

\begin{defn} \label{defi:modulusPair}
A %
\emph{modulus pair}\marginpar{\fbox{modulus pair}}\index{modulus pair} %
is a pair 
\[ \sX = (\hX,\mX) \]
 with ${\hX}$ a scheme and ${\mX}$ an effective Cartier divisor on ${\hX}$. %
For a modulus pair $\sX$, we write 
\[ {\iX} := {\hX} - {\mX}. \]
We call %
${\iX}$%
\marginpar{\fbox{${\iX}$}}%
\index[not]{X@${\iX}$} %
the %
\emph{interior}%
\marginpar{\fbox{interior}}%
\index{modulus pair!interior} %
of $\sX$.
We call 
${\hX}$%
\marginpar{\fbox{${\hX}$}}%
\index[not]{X@${\hX}$} %
the %
\emph{total space}%
\marginpar{\fbox{total space}}%
\index{modulus pair!total space} %
and %
${\mX}$%
\marginpar{\fbox{${\mX}$}}%
\index[not]{X@${\mX}$} %
the %
\emph{modulus}.%
\marginpar{\fbox{modulus}}%
\index{modulus pair!modulus} %
\end{defn}

\begin{rema}
The symbol $\sX$ is to be thought of as a formal ``quotient'' of $\iX$ (or the ``motif'' or ``homotopy type'' or \dots of $\iX$) where ramification larger than $\mX$ is killed.
\end{rema}

\begin{exam}
A key case to keep in mind is when $\hX$ is the spectrum of a rank one valuation ring $R$, and $\mX$ is generated by a pseudo-uniformiser $π$. 
\end{exam}

\begin{defn}\label{def:ambientmorph} \ 
\begin{enumerate}
\item Define the ambient category of modulus pairs %
$\ulPSCH$%
\marginpar{\fbox{$\ulPSCH$}}%
\index[not]{MSCHamb@$\ulPSCH$} %
as follows. 
The objects of $\ulPSCH$ are all modulus pairs. 
For two modulus pairs $\sX = ({\hX},{\mX})$ and $\sY = ({\hY},{\mY})$, a morphism $f : \sX \to \sY$ is a morphism of schemes $\ol{f} : {\hX} \to {\hY}$ which satisfies the following \emph{modulus condition}:
\begin{itemize}
\item the closed immersion ${\mY} \times_{{\hY}} {\hX} \to {\hX}$ factors through ${\mX}$.
\end{itemize} 

Obviously, the modulus condition is preserved by the composition of morphisms of schemes.
The morphisms in $\ulPSCH$ are called %
\emph{ambient morphisms}.%
\marginpar{\fbox{ambient morphism}}%
\index{morphism of modulus pairs!ambient} %

\item For a modulus pair $\sS$, we define %
$\ulPSCH_{\sS}$%
\marginpar{\fbox{$\ulPSCH_{\sS}$}}%
\index[not]{MSCHambS@$\ulPSCH_{\sS}$} %
as the category whose objects are morphisms $\sX \to \sS$ in $\ulPSCH$ and whose morphisms are commutative triangles
\[\xymatrix{
\sX \ar[rr] \ar[dr] && \sY \ar[dl] \\
&\sS.&
}\]

For a commutative ring $R$, we write %
$\ulPSCH_R$%
\marginpar{\fbox{$\ulPSCH_R$}}%
\index[not]{MSCHambR@$\ulPSCH_R$}. %
Since $(\Spec (\Z) , \emptyset)$ is a terminal object in $\ulPSCH$ (by Corollary \ref{lem:coincide} below), we have %
$\ulPSCH_\ZZ$%
\marginpar{\fbox{$\ulPSCH_\ZZ$}}%
\index[not]{MSCHambZ@$\ulPSCH_\ZZ$}.

\item \label{item:morphprop} Let $(P)$ be a property of morphisms of schemes. We say that ambient morphism $\sX \to \sS$ of modulus pairs has the property $(P)$ if the underlying morphism $\hX \to \hS$ does. For example, $\sX \to \sS$ is said to be separated (resp. quasi-compact, quasi-separeted, of finite type, of finite presentation) if $\hX \to \hS$ is.
\end{enumerate} 
\end{defn}

\begin{rema}\label{rem:mapint}
Note that any ambient morphism $f : \sX {\to} \sY$ induces a morphism %
$f^\o : {\iX} {\to} {\iY}$%
\marginpar{\fbox{$f^\o$}}%
\index[not]{fcirc@$f^\o$}.
\end{rema}

\begin{defi}[{\cite[Tag 01WV]{stacks-project}}]
Given a morphism of schemes $f: Y \to X$ and an effective Cartier divisor $D$ on $X$, we say the \emph{pullback of $D$ is defined} if $Y \times_X D$ is again an effective Cartier divisor. In this situation we will write $f^*D$ or $D|_Y$ for $Y \times_X D$.
\end{defi}

\begin{prop} \label{prop:pullbackOK}
Let $f : \sX \to \sY$ be a morphism in $\ulPSCH$.
Then the pullback $\ol{f}^\ast {\mY}$ of the effective Cartier divisor ${\mY}$ by $\ol{f} : {\hX} \to {\hY}$ is defined. 
\end{prop}

\begin{proof}
Since the question is local on ${\hX}$, we are reduced to the case that ${\hX}=\Spec (R)$, ${\hY}=\Spec (S)$ are affine and that ${\mX} = \Spec (R/rR)$ and ${\mY} = \Spec (S/sS)$ are principal.
Let $\varphi : S \to R$ be the homomorphism corresponding to $\ol{f}$.
It suffices to prove that $\varphi (r)$ is a nonzerodivisor.
The assumption that ${\mY} \times_{{\hY}} {\hX} \subset {\mX}$ is equivalent to that $sS \subset \varphi (r)S$, i.e., that there exists $s' \in S$ such that $s=\varphi (r)s'$.
Therefore, since $s$ is a nonzerodivisor by assumption, so is $\varphi (r)$. 
This finishes the proof.
\end{proof}

\begin{defn} \ 
An ambient morphism $f : \sX \to \sY$ is %
 \emph{minimal}\marginpar{\fbox{minimal morphism}}\index{morphism of modulus pairs!minimal} %
 if the natural morphism ${\mY} \times_{{\hY}} {\hX} \to {\mX}$ is an isomorphism. 
\end{defn}

In order to make locally principal closed subschemes into effective Cartier divisors, we will frequently use the following notions.

\begin{defi}[{\cite[Tags 01RA, 01R5]{stacks-project}, \cite[Def.11.10.2]{EGAIV3}}] \label{defi:SchDom}
Let $f: Y \to X$ be a morphism of schemes. The \emph{scheme theoretic image} of $f$ is the closed subscheme $\underline{\Spec}(\OO_X / \sI)$ corresponding to the kernel $\sI = ker(\OO_X \to f_*\OO_Y)$. The morphism $f$ is called \emph{scheme theoretically dominant} or \emph{schematically dominant} if $\OO_X \to f_*\OO_Y$ is injective (or equivalently, if the scheme theoretical image is $X$). A scheme theoretically dominant open immersion is said to be \emph{scheme theoretically dense}. The scheme theoretical image of a subscheme is called the \emph{scheme theoretical closure}.
\end{defi}

\begin{prop}\label{prop:always-dense}
For any modulus pair $\sX = ({\hX},{\mX})$, the interior ${\iX} = {\hX}-{\mX}$ of $\sX$ is scheme theoretically dense in ${\hX}$, i.e., for any open subscheme $U \subset {\hX}$, the scheme theoretic image of $U \cap {\iX} \to U$ is equal to $U$. %
In particular, ${\iX}$ is topologically dense in ${\hX}$.
\end{prop}

\begin{proof}
We prove the first assertion.
Clearly, we may assume that $U = {\hX}$.
Let $Z$ be a closed subscheme of ${\hX}$ such that ${\iX} \to {\hX}$ factors through $Z$.
It suffices to show that $Z={\hX}$.
Since the problem is local on ${\hX}$, we may assume that ${\hX}=\Spec (R)$ is affine and ${\mX} = \Spec (R/rR)$ for some nonzerodivisor $r \in R$.
Note that we have ${\iX} = {\hX}-{\mX} = \Spec (R[1/r])$.
Let $I \subset R$ be the ideal which defines $Z$, i.e., $Z = \Spec (R/I)$. 
Then the assumption implies that there exists a natural map of $R$-algebras $R/I \to R[1/r]$, i.e., that the natural homomorphism $R \to R[1/r]$ sends $I$ to zero.
This means that for any $x \in I$, there exists $n \in \Z_{>0}$ such that $r^n x=0$, which shows $x=0$ since $r$ is a nonzerodivisor.
Therefore, we have $I=0$ and hence $Z={\hX}$.

The second assertion follows from the first assertion since ${\iX} \to {\hX}$ is affine and hence quasi-compact.%
\footnote{If an open immersion is quasi-compact and scheme theoretically dense, then it is topologically dense.}
This finishes the proof.
\end{proof}

\begin{lemma}\label{lem:coincide}
Let $f,g: X \rightrightarrows Y$ be two morphisms of schemes and $D$ an effective Cartier divisor on $X$.
Assume that $f$ and $g$ coincide on $U:=X - D$.
Then we have $f=g$.
\end{lemma}

\begin{proof}
Since the problem is local on $X$, we may assume that $X=\Spec (R)$ and $Y = \Spec (S)$ are affine, and that $D = \Spec (R/rR)$ for some nonzerodivisor $r \in R$.
Let $\varphi , \psi : S \to R$ be the ring maps corresponding to $f,g$, and let $\xi : R \to R[1/r]$ be the localisation map, which is injective since $r$ is a nonzerodivisor.
The assumption implies $\xi \circ \varphi = \xi \circ \psi$, and hence $\varphi = \psi$. This finishes the proof.
\end{proof}

\begin{cor}\label{cor:coincide}
Let $f,g : \sX \rightrightarrows \sY$ be two morphisms in $\ulPSCH$.
If $f^\o = g^\o$, then $f=g$. \qed
\end{cor}

\begin{cor}\label{cor:twoforgetfuncs}
For any modulus pair $\sS$, there are two canonical faithful ``forgetful'' functors 
\begin{align*}
\ulPSCH_{\sS} &\to \SCH_{\iS} ; \quad \sX \to \iX ,\\
\ulPSCH_{\sS} &\to \SCH_{\hS} ; \quad \sX \to \hX .
\end{align*}
\end{cor}

\begin{proof}
The first functor obviously exists, and the latter does by Remark \ref{rem:mapint}.
The faithfulness follows from Cor.~\ref{cor:coincide}.
\end{proof}

The following result shows that the modulus condition can be checked fpqc-locally on the ambient space.

\begin{lemma}\label{lem:moduluscond-fpqc}
Let $X$ be a scheme and $Y, Z$ closed subschemes of $X$. 
Let $\{U_i \to X\}_i$ be a fpqc-covering of $X$, and assume that $Y \times_X U_i \to U_i$ factors through $Z \times_X U_i$.
Then $Y \to X$ factors through $Z$. 
\end{lemma}

\begin{proof}
This is a consequence of fpqc-descent for quais-coherent sheaves. 
\end{proof}

We also prepare the following lemma for later use.

\begin{lemma}\label{lem:qc-interior}
For any modulus pair $\sX = (\hX, \mX)$, the induced open immersion $\iX = \hX \setminus \mX \to \hX$ is a quasi-compact morphism.
\end{lemma}

\begin{remark}
Note that it is not always the case that an open immersion is quasi-compact. 
The above lemma helps us avoid such a pathological situation. 
In fact, an open immersion is quasi-compact if and only if it's closed complement is the support of a coherent sheaf of ideals.
\end{remark}

\begin{proof}
Since quasi-compactness is a Zariski-local property, we may assume that $\hX = \Spec (A)$ is affine and $\mX = \Spec (A/aA)$ for some nonzerodivisor $a \in A$. Then $\iX = \Spec (A_f)$ is also affine, hence quasi-compact. 
\end{proof}

\subsubsection{Abstract admissible blow-ups}

The category $\ulPSCH$ does not admit finite limits. However, we gain categorical finite limits if we invert blowups inside the modulus. 

\begin{rema}
This construction follows the one in \cite{kmsy1}. The idea of localising a class of admissible blowup also appears in early (2015) work of Binda, \cite{Bin20}.
\end{rema}

\begin{defn} \label{defn:Sigma}
Define 
$\ul{\Sigma}$%
\marginpar{\fbox{$\ul{\Sigma}$}}%
\index[not]{Sigma@$\ul{\Sigma}$} %
as the class of morphisms of $\ulPSCH$ consisting of those minimal morphisms $s : \sX' \to \sX$ with $\ol{s} : {\hX}' \to {\hX}$ a proper morphism such that $\ol{s}^{-1} ({\iX}) \xrightarrow{\sim} {\iX}$. Morphisms in $\ul{\Sigma}$ are called %
\emph{abstract admisible blowups}.%
\marginpar{\fbox{abstract admisible blowups}}%
\index{morphism of modulus pairs!abstract admisible blowups} %
\end{defn}

\begin{lemma}\label{lem:pb-dense}
Let $\sX = ({\hX},{\mX})$ be a modulus pair, and let $f : {\hX}' \to {\hX}$ be a morphism of schemes 
such that the scheme theoretic image of the open immersion $j:f^{-1} ({\iX} ) \to {\hX}'$ is equal to ${\hX}'$.
Then, $f^{-1} ({\iX} )$ is scheme theoretically dense in ${\hX}'$. 
Moreover, ${\mX} \times_{{\hX}} {\hX}'$ is an effective Cartier divisor on ${\hX}'$.
\end{lemma}

\begin{proof}
We prove the first assertion. 
Since ${\iX} \to {\hX}$ is affine as the complement of the effective Cartier divisor ${\mX}$, so is $j$ as its base change by $f$.
In particular, $j$ is quasi-compact, and hence $f^{-1}({\iX})$ is scheme theoretically dense in ${\hX}'$.

We prove the second assertion.
Since the problem is local on ${\hX}'$, we may assume that ${\hX}=\Spec (R)$ and ${\hX}'=\Spec (R')$ are affine, and that ${\mX} = \Spec (R/rR)$ for some nonzerodivisor $r \in R$.
Then ${\iX} = \Spec (R[1/r])$ and $f^{-1} ({\iX}) = \Spec (R[1/r] \otimes_R R')$.
Since $f^{-1} ({\iX})$ is scheme theoretically dense in ${\hX}'$ by assumption, the map $R' \to R[1/r] \otimes_R R'$ is injective. 
This implies that the image of $r$ under $R \to R'$ is a nonzerodivisor, finishing the proof.
\end{proof}

\begin{lemma}\label{lem:pb-bu}
Let $X$ be a scheme, and let $Z$, $D$ be closed subschemes of $X$ such that $D$ is an effective Cartier divisor. 
Let $\pi : X' \to X$ be the blow up of $X$ along $Z$.
Then $D \times_X X'$ is an effective Cartier divisor on $X'$. 
\end{lemma}

\begin{proof}
Since the problem is local on $X$, we may assume that $X=\Spec (R)$ is affine, $D=\Spec (R/rR)$ and $Z=\Spec (R/I)$ for some nonzerodivisor $r \in R$ and an ideal $I \subset R$.
Then the blow up $X'$ is covered by the spectra of ``affine blow up algebras'': $X'=\cup_{a \in I} \Spec (R[\frac{I}{a}])$.
Here, any element of $R[\frac{I}{a}]$ is represented by as $\frac{x}{a^n}$ for some $n \geq 0$ and $x \in I^n \subset \oplus_{i \geq 0} I^i$.
Two elements $\frac{x}{a^n}$ and $\frac{y}{a^m}$ coincide if and only if there exists $k \geq 0$ such that $a^k (a^m x - a^n y)=0$ in $R$.
Fix $a \in I$ and let $\varphi : R \to R[\frac{a}{I}]$ be the natural ring map.
It suffices to show that $\varphi (r)$ is a nonzerodivisor.
Let $\frac{x}{a^n} \in R[\frac{I}{a}]$ be an element and assume $\varphi (r) \cdot \frac{x}{a^n} =0$.
Noting that $\varphi (r) = r \in R = I^0 \subset \oplus_{i \geq 0} I^i$, we have $\varphi (r) \cdot \frac{x}{a^n} = \frac{rx}{a^n}$, where we regard $rx$ as an element of $\oplus_{i \geq 0} I^i$ of degree $n$.
Then $\frac{rx}{a^n} = 0 = \frac{0}{1}$ implies that there exists $k \geq 0$ such that $a^k (rx \cdot 1 - 0 \cdot a^n) = r a^{k} x = 0$. 
Since $r$ is a nonzerodivisor, we have $a^k x = a^k (x \cdot 1 - 0 \cdot a^n) = 0$.
This in turn show that $\frac{x}{a^n} = 0$, and hence that $\varphi (r)$ is a nonzerodivisor, as desired.
\end{proof}

\begin{lemma}\label{lem:pb-sch-dense}
(1) Let $f : \sY \to \sX$ and $g : \sZ \to \sX$ be two morphisms in $\ulPSCH$, and set $U:={\iY} \times_{{\iX}} {\iZ}$.
Let ${\hW}_0$ be the scheme theoretic image of the open immersion $U \to {\hY} \times_{{\hX}} {\hZ}$.
Consider the commutative diagram
\[\xymatrix{
{\hW}_0 \ar[r]^{\ol{i}} \ar[d]_{\ol{h}}& {\hY} \ar[d]^{\ol{f}}\\
{\hZ} \ar[r]^{\ol{g}}& {\hX}.
}\]

Then $U=\ol{i}^{-1}({\iY}) \cap \ol{h}^{-1}({\iZ})$.

(2) Consider the open immersion  $j : U \to {\hW}_0$. Then $U$ is scheme theoretically dense in ${\hW}_0$.
\end{lemma}

\begin{proof}
(1): Obvious.

(2): Since ${\iY} \to {\hY}$ and ${\iZ} \to {\hZ}$ are affine morphisms as the complements of effective Cartier divisors. 
Therefore, $\ol{i}^{-1}({\iY}) \to {\hW}_0$ and $\ol{h}^{-1}({\iZ}) \to {\hW}_0$ are also affine, and hence so is $j$ by (1).
In particular, $j$ is quasi-compact. 
Since the scheme theoretic image of $j$ is equal to ${\hW}_0$ by construction, we conclude that $U$ is scheme theoretically dense in ${\hW}_0$, as desiered.
\end{proof}

\begin{prop}\label{prop:cal-frac}
The pair $(\ulPSCH , \ul{\Sigma})$ enjoys a calculus of right fractions, i.e., the following assertions hold:
\begin{enumerate}
\item The identities of $\ulPSCH$ are in $\ul{\Sigma}$.
\item $\ul{\Sigma}$ is closed under composition.
\item Any diagram $\sX' \xrightarrow{s} \sX \xleftarrow{f} \sY$ with $s \in \ul{\Sigma}$ can be completed to a commutative square in $\ulPSCH$
\[\xymatrix{
\sY'\ar[r]^{f'}\ar[d]_t&\sX' \ar[d]^s\\
\sY\ar[r]^f&\sX
}\]
such that $t \in \ul{\Sigma}$.
\item If $f,g \in \ulPSCH (\sX,\sY)$ and $s : \sY \to \sZ$ is a morphism in $\ul{\Sigma}$ such that $sf=sg$, then there exists a morphism $t : \sX' \to \sX$ in $\ul{\Sigma}$ such that $ft=gt$.
\end{enumerate}
\end{prop}

\begin{proof}
(1) and (2) are obvious.
We prove (3). Let ${\hY}'$ be the scheme theoretic image of the open immersion $j:(\sX')^\o \times_{{\iX}} {\iY} \to {\hX}' \times_{{\hX}} {\hY}$, and let $\ol{t}$ and $\ol{f}'$ be the composite maps ${\hY}' \to {\hX}' \times _{{\hX}} {\hY} \to {\hY}$ and ${\hY}' \to {\hX}' \times _{{\hX}} {\hY} \to {\hX}'$, respectively.

We claim that $\ol{t}^{-1} ({\iY}) = (\sX')^\o \times_{{\iX}} {\iY}$.
Indeed, noting ${\iX} \times_{{\hX}} {\hX}' \simeq {\iX}$, we have
\[\ol{t}^{-1} ({\iY}) = {\iY} \times_{{\hX}} {\hX}' \simeq {\iY} \times_{{\iX}} {\iX}  \times_{{\hX}} {\hX}' \simeq {\iY} \simeq  (\sX')^\o \times_{{\iX}} {\iY} ,\]
which shows the claim.

Combined with the claim, Lemma \ref{lem:pb-dense} shows that ${\mY} \times_{{\hY}} {\hY}'$ is an effective Cartier divisor on ${\hY}'$.
Set $(Y')^\infty := {\mY} \times_{{\hY}} {\hY}'$ and $\sY' := ({\hY}',(Y')^\infty)$.
Then $\ol{t}$ defines a minimal morphism $t : \sY' \to \sY$.
Since $\ol{t}$ is proper and $\ol{t}^{-1} ({\iY}) = (\sX')^\o \times_{{\iX}} {\iY} \simeq {\iY}$ by construction, we have $t \in \ul{\Sigma}$.
Moreover, the map $\ol{f}'$ defines an ambient morphism $f' : \sY' \to \sX'$ since 
\begin{align*}
(X')^\infty \times_{{\hX}'} {\hY}' &= ({\mX} \times_{{\hX}}{\hX}' ) \times_{{\hX}'} {\hY}' \\
&\simeq ({\mX} \times_{{\hX}} {\hY}) \times_{{\hY}} {\hY}' \\
&\subset {\mY} \times_{{\hY}} {\hY}' \\
&= (Y')^\infty ,
\end{align*}
where the first equality follows from the minimality of $s$, and the inclusion follows from the modulus condition for $f$.
This finishes the proof of (3).

Finally, we check (4). The assumption shows that $f^\o = g^\o$. Therefore, we have $f=g$ by Corollary \ref{cor:coincide}, and we can take $t=\id_{\sX}$.
This finishes the proof.
\end{proof}

\section{The category $\protect\ulMSCH$}

\begin{defn}
Let $\ulPSCH \to \ulPSCH [\ul{\Sigma}^{-1}]$ be the localisation functor. 
We write 
\[
\ulMSCH := \ulPSCH [\ul{\Sigma}^{-1}].
\]

We say that a morphism $f$ in $\ulMSCH$ is \emph{ambient} if it comes from a morphism in $\ulPSCH$ (cf. Def.~\ref{def:ambientmorph}).
\end{defn}

By Proposition \ref{prop:cal-frac}, any morphism $f : \sX \to \sY$ in $\ulMSCH$ is represented by a diagram in $\ulPSCH$ of the form $\sX \xleftarrow{s} \sX' \xrightarrow{g} \sY$ such that $s \in \ul{\Sigma}$.
Note that the diagram induces a morphism $F_{s,g} : {\iX} \xleftarrow{\sim} (\sX')^\o \to {\iY}$.

\begin{defn}
For a modulus pair $\sS$, we define 
$\ulMSCH_{\sS}$
\marginpar{\fbox{$\ulMSCH_{\sS}$}}%
\index[not]{$\ulMSCH_{\sS}$} %
 as a category whose objects are morphisms $\sX \to \sS$ in $\ulMSCH$ and whose morphisms are commutative triangles 
\[\xymatrix{
\sX \ar[rr] \ar[dr] && \sY \ar[dl] \\
&\sS&
}\]
in $\ulMSCH$.
Since $(\Spec (\Z),\emptyset)$ is a terminal object in $\ulMSCH$ by Lemma \ref{lem:coincidence2} (2) below, we have $\ulMSCH = \ulMSCH_\Z$.
\end{defn}

The following lemma ensures that there is no ambiguity in the definition of $\ulMSCH_{\sS}$.

\begin{lemm} \label{lemm:invertSliceGeneral}
Suppose that $\sC$ is a category and $\Sigma$ a class of morphisms satisfying a calculus of fractions. Given an object $S \in \sC$, consider the subclass $\Sigma_S$ of $\Sigma$ consisting of those morphisms in $\sC_S$. Then
\[ \bigl ( \sC_S \bigr ) [\ul{\Sigma}_S^{-1}] \stackrel{\sim}{\to} \bigr (\sC[\ul{\Sigma}^{-1}]\bigl)_S \] 
is an equivalence of categories. 
\end{lemm}

\begin{proof}
Essential surjectivity is clear since objects in $\sC[\Sigma^{-1}]_S$ can be represented by roofs $X \stackrel{\in\Sigma}{\leftarrow} X' \to S$, so up to isomorphism, every object is in the image of $\sC_S$. 

For faithfulness, note that two morphisms $X \leftarrow X' \to Y$ and $X \leftarrow X'' \to Y$ in $\sC_S[\ul{\Sigma}_S^{-1}]$ represent the same morphism if and only if there is a commutative diagram of $S$-objects of the form
\[ \xymatrix{
& X' \ar[dl]_{\in\Sigma} \ar[dr] & \\
X & \ar[l]^{\in\Sigma} \ar[d]_{\in\Sigma} \ar[r] \ar[u]^{\in\Sigma} X''' & Y \\
& X'' \ar[ul]^{\in\Sigma} \ar[ur] & 
} \]
The same is true in $\sC[\ul{\Sigma}^{-1}]$ without the condition that the diagram be in the slice category $\sC_S$.

For fullness, consider a morphism of $\bigr (\sC[\ul{\Sigma}^{-1}]\bigl)_S$ between two objects $X \to S$, $Y \to S$ in the image of $\sC_S$. As a morphism of $\sC[\ul{\Sigma}^{-1}]$ it can be represented by a roof $X \stackrel{\in\Sigma}{\leftarrow} X' \to Y$. Since $\hom_{\sC[\ul{\Sigma}^{-1}]}(X, S) = \varinjlim_{W{\to}X \in \Sigma}(W, S)$,  a necessary and sufficient condition that this roof represent a morphism of the slice category $\bigr (\sC[\ul{\Sigma}^{-1}]\bigl)_S$ is that the there exists some $X'' \to X'$ in $\Sigma$ such that the two compositions
\[ X'' {\to} X' {\to} X {\to} S \qquad \textrm{ and } \qquad X'' {\to} X' {\to} Y {\to} S \]
agree in $\sC$. In other words, there exists a roof $X \stackrel{\in\Sigma}{\leftarrow} X'' \to Y$ representing our original morphism, which is contained in the slice category $\sC_S$. Hence, our original morphism of $\sC[\Sigma^{-1}]_S$ comes from $\sC_S[\Sigma_S^{-1}]$.
\end{proof}

\begin{coro} \label{coro:invertSliceMSCH}
The functor 
\[ \biggl ( \ulPSCH_\sS \biggr ) [\ul{\Sigma}_\sS^{-1}] \stackrel{\sim}{\to} \biggr (\ulPSCH[\ul{\Sigma}^{-1}]\biggl)_\sS = \ulMSCH_\sS \] 
is an equivalence of categories.
\end{coro}

\begin{lemma}\label{lem:coincidence2}\ 
\begin{enumerate} 
 \item Let $\sX \stackrel{s}{\leftarrow} \stackrel{f}{\to} \sY$ represent a morphism in $\ulMSCH$, so $s$ is an abstract admissible blowup. The morphism $F_{s,g} : {\iX} \to {\iY}$ is independent of the choice of $(s,g)$.
So, we write $f^\o := F_{s,g}$.

 \item Let $f_1,f_2 : \sX \rightrightarrows \sY$ be two morphisms in $\ulMSCH$. 
If $f_1^\o = f_1^\o$, then $f_1=f_2$.
\end{enumerate}
\end{lemma}

\begin{proof}
We prove (1).
Let $\sX \xleftarrow{t} \sX'' \xrightarrow{h} \sY$ be another diagram which represents $f$.
Then there exists a commutative diagram in $\ulPSCH$
\begin{equation}\label{diag1.1}\begin{gathered}\xymatrix{
& \ar[ld]_s \sX' \ar[rd]^g &\\
\sX & \ar[l]_c \sZ \ar[u]_a \ar[d]^b & \sY ,\\
& \ar[lu]^t \sX'' \ar[ru]_h &
}\end{gathered}\end{equation}
such that $c \in \ul{\Sigma}$.
Taking the interiors, we obtain a commutative diagram of schemes
\begin{equation}\label{eq1.2}\begin{gathered}\xymatrix{
& \ar[ld]_{s^\o} (\sX')^\o \ar[rd]^{g^\o} &\\
{\iX} & \ar[l]_{c^\o} {\iZ} \ar[u]_{a^\o} \ar[d]^{b^\o} & {\iY} ,\\
& \ar[lu]^{t^\o} (\sX'')^\o \ar[ru]_{h^\o} &
}\end{gathered}\end{equation}
where all the arrows appearing in the left two triangles are isomorphisms. 
This shows $F_{s,g} = F_{t,h}$, finishing the proof of (1).

We prove (2). Suppose that $f_1$ and $f_2$ are represented by diagrams $\sX \xleftarrow{s} \sX' \xrightarrow{g} \sY$ and $\sX \xleftarrow{t} \sX'' \xrightarrow{h} \sY$ such that $s,t \in \ul{\Sigma}$.
Then, by the Proposition \ref{prop:cal-frac} (3), we can find a diagram of the form \eqref{diag1.1} such that the left two triangles are commutative and such that $a \in \ul{\Sigma}$.
Then $c = sa \in \ul{\Sigma}$.

We claim $ga=hb$.
Indeed, the assumption that $f^\o = g^\o$ shows that this diagram \eqref{eq1.2}, which is obtained by taking interiors, is commutative. 
In particular, $(ga)^o = g^\o a^\o = h^\o b^\o = (hb)^\o$. 
Therefore, Corollary \ref{cor:coincide} proves the claim.

Thus, the diagram \eqref{diag1.1} is commutative, and hence the diagrams $\sX \xleftarrow{s} \sX' \xrightarrow{g} \sY$ and $\sX \xleftarrow{t} \sX'' \xrightarrow{h} \sY$ are equivalent, i.e., $f_1=f_2$, as desired.
\end{proof}

\begin{cor}
For any modulus pair $\sS$, there exists a faithful forgetful functor
\[
\ulMSCH_{\sS} \to \SCH_{\iS}; \sX \to \iX.
\]
\end{cor}

\begin{proof}
This follows from Cor.~\ref{cor:twoforgetfuncs} and Lemma \ref{lem:coincidence2}.
\end{proof}

\begin{lemma}\label{lem:key}
Let $X$ be a scheme, and let $D_1,D_2$ be two effective Cartier divisors on $X$.
Assume that $E:=D_1 \times_X D_2$ is an effective Cartier divisor on $X$.
Then the following assertions hold.
\begin{enumerate}
\item The Cartier divisors $D'_1 := D_1 - E$ and $D'_2 := D_2 - E$ are effective, and $D'_1 \cap D'_2 = \emptyset$.
\item Consider the effective Cartier divisor $D_3 := D'_1 + D'_2 +E = D_1 + D_2 - E$.
Then, for any effective Cartier divisor $D$ on $X$ such that $D_1 \subset D$ and $D_2 \subset D$, we have $D_3 \subset D$.
\end{enumerate}
\end{lemma}

\begin{proof}

(1) is proven in \cite[Lemma 1.10.1]{kmsy1}, but we provide a proof here for the convenience of the reader.
Since the problem is local on $X$, we may assume that $X=\Spec (R)$ is affine, and that $D_i = \Spec (R/d_i R)$ and $E=\Spec (R/eR)$ for some nonzerodivisors $d_i \in R$, $i=1,2$ and $e \in R$.
Then the equality $D_1 \times_X D_2 = E$ implies $d_1R + d_2R  = eR$, i.e., 
\begin{equation}\label{eq1.5}
\frac{d_1}{e} R + \frac{d_1}{e} R = R.
\end{equation}
This means $D'_1 \times_X D'_2 = \Spec (R/1 \cdot R) = \emptyset$, which proves (1).

We prove (2). We make the same assumption as in the proof of (1).
Then $D_3 = \Spec (R/d_3)$, where $d_3 = \frac{d_1 d_2}{e}$.
By shrinking $X$ again if necessary, we may also assume that $D=\Spec (R/dR)$ for some nonzerodivisor $d \in R$.
The equality \eqref{eq1.5} shows that $X=\Spec (R)$ is covered by Zariski affine opens $\Spec (R[\frac{e}{d_1}])$ and $\Spec (R[\frac{e}{d_2}])$.
Therefore, by the sheaf condition for the structure sheaf $\mathcal{O}_X$, 
the following sequence of $R$-modules
\[
0 \to R \xrightarrow{a} \frac{e}{d_1} R \oplus \frac{e}{d_1} R \xrightarrow{b} \frac{e}{d_1 d_2} R
\]
is exact, where $a(x)=(x,x)$ and $b(x,y)=x-y$.
In particular, $R = \frac{e}{d_1} R \cap \frac{e}{d_1} R$, where the intersection is taken in the total quotient ring of $R$.
Now, the assumptions $D_1 \subset D$ and $D_2 \subset D$ imply $\frac{d}{d_1},\frac{d}{d_2} \in R$. 
Therefore, we have
$
R[\frac{e}{d_1}] \ni \frac{e}{d_1} \cdot \frac{d}{d_2} = \frac{de}{d_1d_2} = \frac{d}{d_1} \cdot \frac{e}{d_2} \in R[\frac{e}{d_2}],
$
and hence $\frac{d}{d_3}=\frac{de}{d_1d_2} \in R$, which implies $D_3 \subset D$. 
This finishes the proof.
\end{proof}

\section{The ambient product}

In this subsection, we introduce an auxiliary product structure of modulus pairs, which we call \emph{the ambient product}. 
In general, this does not represent a fiber product on $\ulPSCH_{\sS}$ (unless one morphism is minimal, cf.~Prop.\ref{prop:minFibPro}), but it is often convenient in the following discussions. 

\begin{constr}[Ambient product]\label{cons:admBlowup}
Let $\sT \to \sS, \sX \to \sS$ be two ambient morphisms of modulus pairs. %
Let ${\hY}_0$ be the scheme theoretic image of the open immersion ${\iT} \times_{{\iS}} {\iX} \to {\hT} \times_{{\hS}} {\hX}$. %
Then we have a commutative diagram
\[\xymatrix{
{\hY}_0 \ar[r]^{\ol{g}} \ar[d]_{\ol{q}} & {\hX} \ar[d]^{\ol{p}} \\
{\hT} \ar[r]^{\ol{f}} & {\hS}.
}\]

By Lemma \ref{lem:pb-sch-dense}, we have ${\iT} \times_{{\iS}} {\iX} = \ol{q}^{-1} ({\iT}) \cap \ol{g}^{-1} ({\iX})$, and ${\iT} \times_{{\iS}} {\iX}$ is scheme theoretically dense in ${\hY}_0$.
In particular, $\ol{q}^{-1} ({\iT})$ and $\ol{g}^{-1} ({\iX})$ are also scheme theoretically dense in ${\hY}_0$.
Lemma \ref{lem:pb-dense} shows that the two subschemes $\ol{q}^\ast {\mT} := {\mT} \times_{{\hT}} {\hY}_0$ and $\ol{g}^\ast {\mX} := {\mX} \times_{{\hX}} {\hY}_0$ are effective Cartier divisors on ${\hY}_0$.
Set 
\begin{equation}\label{eq:1.4}
F:=(\ol{q}^\ast {\mT} ) \times_{{\hY}_0} (\ol{g}^\ast {\mX} ),
\end{equation}
and consider the blow up of ${\hY}_0$ along $F$:
\[
{\hY} := \Bl_F ({\hY}_0) \xrightarrow{\pi} {\hY}_0.
\]
By Lemma \ref{lem:pb-bu}, the pullbacks of effective Cartier divisors $D_T := \pi^\ast \ol{g}^\ast {\mT}$ and $D_X := \pi^\ast \ol{g}^\ast {\mX}$ are defined.
Moreover, 
\begin{equation}\label{eq1.6}
D_T \times_{{\hY}} D_X = E
\end{equation}
holds, where $E=\pi^{-1} (F)$ denotes the exceptional divisor. 
Indeed, 
\begin{align*}
D_T \times_{{\hY}} D_X &= (\ol{q}^\ast {\mT} \times_{{\hY}_0} {\hY}) \times_{{\hY}} (\ol{g}^\ast {\mX} \times_{{\hY}_0} {\hY}) \\
&= (\ol{q}^\ast {\mT} \times_{{\hY}_0} \ol{g}^\ast {\mX} ) \times_{{\hY}_0} {\hY} \\
&= F \times_{{\hY}_0} {\hY} \\
&= E.
\end{align*}
Therefore, by Lemma \ref{lem:key}, we have $|D_T - E| \cap |D_X - E| = \emptyset$.
Define an effective Cartier divisor ${\mY} \subset {\hY}$ by 
\[
{\mY} := D_T + D_X - E,
\]
and set 
\[
\sY := ({\hY},{\mY} ).
\]
Note that by construction, we have 
\begin{equation}\label{eq:1.3}
{\iY} \simeq {\iT} \times_{{\iS}} {\iX}.
\end{equation}
Since $\pi^\ast \ol{q}^\ast {\mT} = D_T \leq {\mY}$ and $\pi^\ast \ol{g}^\ast {\mX} = D_X \leq {\mY}$, we obtain a commutative square in $\ulPSCH$:
\[\xymatrix{
\sY \ar[r]^{g'} \ar[d]_{q'} & \sX \ar[d]^p \\
\sT \ar[r]^f & \sS.
}\]
\end{constr}

\begin{defi} \label{defi:admBlowup}
We will write 
\[ \sY \ambtimes[\sX] \sZ \]
\marginpar{\fbox{$\sY \ambtimes[\sX] \sZ $}}%
\index[not]{$\sY \ambtimes[\sX] \sZ $} %
for the $\sW$ constructed above in Construction~\ref{cons:admBlowup}, and call it the 
\emph{ambient product}
\marginpar{\fbox{ambient product}}%
\index{product of modulus pairs!ambient product} %
 of $\sY$ and $\sZ$ over $\sX$.
\end{defi}

\subsubsection{Some elementary properties of ambient product}

We begin with a special case. 

\begin{lemm} \label{lemm:minProduct}
Let $\sX \to \sS$ be an ambient morphism and $\sT \to \sS$ be a minimal ambient morphism. 
Then 
\[ \sT \ambtimes[\sS] \sX = (\hY, \mX|_{\ol{Y}}). \]
where $\hY$ is the scheme theoretic image of $\iT \times_{\iS} \iX \to \hT \times_{\hS} \hX$.

If moreover, $\hT \to \hS$ is flat, then 
\[ \sT \ambtimes[\sS] \sX = (\hT \times_{\hS} \hX, \mX|_{\hT \times_{\hS} \hX}). \]
\end{lemm}

\begin{proof}
We use the notation of Construction~\ref{cons:admBlowup}. We have $\mT = \mS|_{\hT}$, so since 
$\mX \geq \mS|_{\hX}$ it follows that 
$\mX|_{{\hY}_0} \geq \mS|_{{\hY}_0} = \mT|_{{\hY}_0}$, i.e., 
$\ol{q}^\ast {\mT} \subseteq \ol{g}^\ast {\mX}$, so 
$F := (\ol{q}^\ast {\mT} ) \times_{{\hY}_0} (\ol{g}^\ast {\mX} ) = \ol{q}^\ast {\mT}$. 
Since this is already an effective Cartier divisor, ${\hY} := \Bl_F ({\hY}_0) \xrightarrow{\pi} {\hY}_0$ is an isomorphism, hence the total space in the statement. Continuing with the calculation, we see $E = D_T \times_{{\hY}} D_X = D_T$, so $\mY = D_T + D_X - E = D_X = \mX|_{\hY}$. 

If $\hT \to \hS$ is flat, the pullback $\hT \times_{\hS} \iX \to \hT \times_{\hS} \hX$ of the scheme theoretically dense morphism $\iX \to \hX$ is also scheme theoretically dense. Since $\hX \to \hS$ is admissible, $\iX$ factors through $\iS$ so $\hT \times_{\hS} \iX = (\hT \times_{\hS} \iS) \times_{\iS} \iX$ and since $\sT \to \sS$ is minimal, $\iT = \hT \times_{\hS} \iS$ so $\hT \times_{\hS} \iX = \iT \times_{\iS} \iX$. That is, the pullback of $\iX \to \hX$ along $\hT \to \hS$ is $\iT \times_{\iS} \iX \to \hT \times_{\hS} \hX$, so this is scheme theoretically dense.
\end{proof}

\begin{prop} \label{prop:minFibPro}
Let $\sX \to \sS$ be an ambient morphism and $\sT \to \sS$ be a minimal ambient morphism. 
Then $\sT \ambtimes_\sS \sX$ is the categorical fibre product in $\ulPSCH$.
\end{prop}

\begin{rema}
We will often use the symbol $\times$ instead of $\times^p$ to denote the ambient product when the assumption in Prop.~\ref{prop:minFibPro} is satisfied. 
\end{rema}

\begin{proof}
Let $\sY := \sT \ambtimes_\sS \sX$ and let $\sW$ be any other modulus pair. By definition we have
\begin{align*}
\hom_{\ulPSCH}(\sW, \sY) &\subseteq \hom_{\SCH}(\hW, \hY) \\
&\subseteq \hom_{\SCH}(\hW, \hT \times_{\hS} \hX) \\
&= 
\hom_{\SCH}(\hW, \hT) \times_{\hom_{\SCH}(\hW, \hS)} \hom_{\SCH}(\hW, \hX) 
\end{align*}
and
\begin{align*}
&
\hom_{\ulPSCH}(\sW, \sT) \times_{\hom_{\ulPSCH}(\sW, \sS)} \hom_{\ulPSCH}(\sW, \sX) 
\\&\subseteq 
\hom_{\SCH}(\hW, \hT) \times_{\hom_{\SCH}(\hW, \hS)} \hom_{\SCH}(\hW, \hX) 
\end{align*}
so it suffices to show that these two subsets agree. Moreover, due to the canonical commutative square as on the left
\begin{equation} \label{coro:minFibPro:eq}
 \xymatrix{
\sY \ar[r] \ar[d] & \sX \ar[d] 
&&
\sW \ar[r] \ar[d] & \sX \ar[d] \\
\sT \ar[r] & \sS 
&&
\sT \ar[r] & \sS
} 
\end{equation}
there is a canonical inclusion 
\[
\hom_{\ulPSCH}(\sW, \sY) 
\subseteq 
\hom_{\ulPSCH}(\sW, \sT) \times_{\hom_{\ulPSCH}(\sW, \sS)} \hom_{\ulPSCH}(\sW, \sX) 
\]
We wish to show that this inclusion is surjective. Suppose that we are given an element on the right or equivalently, a commutative square as on the right of \eqref{coro:minFibPro:eq}. This square induces a canonical morphism $\hW \to \hT \times_{\hS} \hX$. We claim that $\hW \to \hT \times_{\hS} \hX$ factors through the closed subscheme $\hY \subseteq \hT \times_{\hS} \hX$. Indeed, the sheaf of ideals defining $\hY$ is the kernel $\sI$ of $\OO_{\hT \times_{\hS} \hX} \to i_*\OO_{\iT \times_{\iS} \iX}$ (where $i$ is the obvious inclusion) so it suffices to show the elements of $\sI$ are sent to zero on $\hW$. But they are sent to zero on $\iW$ by commutivity of 
\[ \xymatrix{
\iW \ar[r] \ar[d] & \iT \times_{\iS} \iX \ar[d] \\
\hW \ar[r] & \hT \times_{\hS} \hX
} \]
and since $\iW \to \hW$ is schematically dominant, this suffices to deduce that $\sI$ is sent to zero on $\hW$. So we have produced a (unique) morphism
\[ \hW \to \hY \]
compatible with the two squares \eqref{coro:minFibPro:eq}. Now we use minimality of $\sT \to \sS$, or rather, minimality of $\sY \to \sX$, Lem.~\ref{lemm:minProduct}, to see that this morphism we have produced is admissible.
\end{proof}

\begin{cor}\label{cor:ambprod-func2}
If $\sT \to \sS$ is a minimal ambient morphism, then the ambient product defines a functor
\[
\sT \ambtimes_{\sS} -  : \ulPSCH_\sS \to \ulPSCH_\sS.
\]
\end{cor}

\begin{proof}
This follows directly from Prop.~\ref{prop:minFibPro} since the categorical fibre product is unique up to unique isomorphism.
\end{proof}

\begin{coro} \label{coro:ambtimesambtimes}
Let $\sX \to \sS$ and $\sS_2 \to \sS_1 \to \sS$ be ambient morphisms, and suppose one of:
\begin{enumerate}
 \item $\sX \to \sS$ is minimal, or
 \item $\sS_2 \to \sS_1 \to \sS$ are both minimal.
\end{enumerate}
is satisfied. Then there is an isomorphism of modulus pairs
\[ \sS_2 \ambtimes[\sS_1] (\sS_1 \ambtimes[\sS] \sX) \cong \sS_2 \ambtimes[\sS] \sX. \]
\end{coro}

\begin{proof}
This follows directly from Lem.~\ref{lemm:minProduct} and Prop.~\ref{prop:minFibPro} since the categorical fibre product is unique up to unique isomorphism.
\end{proof}

\begin{prop}
The construction $\sT \ambtimes[\sS] \sX$ of Def.\ref{defi:admBlowup} is functorial in $\sS$ for all ambient morphisms. 
Moreover, for any ambient morphism $\sS_1 \to \sS_0$, the induced morphism $\sT \ambtimes[\sS_1] \sX \to \sT \ambtimes[\sS_0] \sX$ is minimal. 
\end{prop}

\begin{proof}
Let $\sT \to \sS_1$, $\sX \to \sS_1$ and $\sS \to \sS_0$ be three ambient morphisms. Define $\sY_1 = \sT \ambtimes_{\sS_1} \sX$ and $\sY_0 := \sT \ambtimes[\sS_0] \sX$. The construction of $\sY_ε$ proceeds by first taking the scheme theoretical closure $\hY_{0ε}$ of $\iT \times_{\iS_ε} \iX$ in $\hT \times_{\hS_ε} \hX$. In particular, the canonical commutative diagram
\[ \xymatrix@R=12pt{
\iT \times_{\iS_1} \iX \ar[d] \ar[r] &  \iT \times_{\iS_0} \iX \ar[d] \\
\hY_{00} \ar@{-->}[r] \ar[d] & \hY_{01} \ar[d] \\
\hT \times_{\hS_1} \hX  \ar[r] &  \hT \times_{\hS_0} \hX 
} \]
induces the canonical dashed morphism. Next, $\hY_ε$ is defined to be the blowup of $\hY_{0\epsilon}$ along $\mT|_{\hY_{0ε}} \times_{\hY_{0ε}} \mX|_{\hY_{0ε}}$. In particular, $\hY_1 \to \hY_{01}$ is the strict transform of $\hY_0 \to \hY_{00}$. So we obtain a canonical morphism $\hY_1 \to \hY_0$. Finally, by definition,
\[ \mY_ε := \mT|_{\hY_ε} + \mX|_{\hY_ε} - E_ε \]
where $E_ε = (\mT|_{\hY_ε}) \times_{\hY_ε} (\mX|_{\hY_ε})$ is the exceptional divisor. As 
\[ E_1 = E_0|_{\hY_1} \]
we find that
\begin{align*}
\mY_1 
&= \mT|_{\hY_1} + \mX|_{\hY_1} - E_1 \\
&= \mT|_{\hY_1} + \mX|_{\hY_1} - E_0|_{\hY_1} \\
&= \biggl ( \mT|_{\hY_0} + \mX|_{\hY_0} - E_0 \biggr ) |_{\hY_1} \\
&= \mY_0|_{\hY_1},
\end{align*}
which shows that $\hY_1 \to \hY_0$ defines an ambient minimal morphism $\sY_1 \to \sY_0$.
\end{proof}

\begin{lemm} \label{lemm:ambtimesProper}
Let $\sT \to \sS \leftarrow \sX$ be two ambient morphisms, and $\sY := \sT \ambtimes[\sS] \sX$.
\begin{enumerate}
 \item \label{lemm:lemm:ambtimesProper:Int} In general, $\iY = \iT \times_{\iS} \iX$.
 \item \label{lemm:lemm:ambtimesProper:ft} If $\hX \to \hS$ is finite type, then $\hY \to \hT$ is finite type.
 \item \label{lemm:lemm:ambtimesProper:3a} If $\hX \to \hS$ is universally closed, then $\hY \to \hT$ is universally closed.
 \item \label{lemm:lemm:ambtimesProper:3} If $\hX \to \hS$ is proper, then $\hY \to \hT$ is proper.
 \item \label{lemm:lemm:ambtimesProper:4} If $\hX \to \hS$ is finite (resp. closed, resp. affine) and $\sT \to \sS$ is minimal, then $\hY \to \hT$ is finite (resp. closed).
 \item \label{lemm:lemm:ambtimesProper:5} If $\hX \to \hS$ is finite (resp. closed, resp. affine) and minimal, then $\hY \to \hT$ is also finite (resp. closed, resp. affine) and minimal.
 \item \label{lemm:lemm:ambtimesProper:6} If $\{\sU_i \to \sS\}_{i \in I}$ is a family of ambient morphisms such that $\amalg \hU_i \to \hS$ is surjective and universally closed, then for any $\sX \to \sS$ the family %
 $\{\sX \ambtimes_\sS \sU_i \to \sX\}_{i \in I}$ is also jointly surjective on total spaces. 
\end{enumerate}
\end{lemm}

\begin{proof}
We use notation from Construction~\ref{cons:admBlowup}. Case \eqref{lemm:lemm:ambtimesProper:Int} is Eq.~\ref{eq:1.3}. For case \eqref{lemm:lemm:ambtimesProper:ft}, \eqref{lemm:lemm:ambtimesProper:3a}, and \eqref{lemm:lemm:ambtimesProper:3}, $\hY \to \hT$ is the composition $\Bl_F(\ol{W}_0) \to \ol{W}_0 \subseteq \hT \times_{\hS} \hX \to \hT$ of a blowup, and a closed immersion, and the pullback of a proper morphism. Case \eqref{lemm:lemm:ambtimesProper:4} and \eqref{lemm:lemm:ambtimesProper:5} is similar except without the first blowup, cf. Lem.~\ref{lemm:minProduct}. Finally, for case \eqref{lemm:lemm:ambtimesProper:6}, we learn from case \eqref{lemm:lemm:ambtimesProper:3a} that on total spaces, each $\sX \ambtimes_\sS \sU_i \to \sX$ is closed, so it suffices to know that the family is jointly surjective on interiors. This follows from case \eqref{lemm:lemm:ambtimesProper:Int}.
\end{proof}

\begin{lemm} \label{lemm:ambtimesAab}
Let $\sT \to \sS$ be an ambient morphism and $\sX \to \sS$ an abstract admissible blowup. Then $\sY := \sX \ambtimes[\sS] \sT \to \sT$ is also an abstract admissible blowup.
\end{lemm}

\begin{proof}
By Lem.~\ref{lemm:minProduct}, we see that $\sY \to \sT$ is minimal. As the composition $\hY \to \hT \times_{\hS} \hX \to \hT$ of two proper morphisms, $\hY \to \hT$ is proper. By Lem.~\ref{lemm:ambtimesProper}, as $\iT \to \iS$ is an isomorphism, it is surjective, so $\hY \to \hX$ is surjective. Finally, by construction, cf. Lem.~\ref{lemm:ambtimesProper} \eqref{lemm:lemm:ambtimesProper:Int} , $\iY \to \iX$ is an isomorphism.
\end{proof}

\section{Fibre products in ${\protect \ulMSCH}$}

The goal of this subsection is the following result, which is a refinement of \cite[Corollary 1.10.8]{kmsy1}.

\begin{thm}\label{thm:pullback}
Let $\sS$ be a modulus pair, and $\sX \to \sS \leftarrow \sT$ any two ambient morphisms. Then 
\begin{equation}\label{eq1.7statement}
\begin{gathered}
\xymatrix{
\sT \ambtimes[\sS] \sX \ar[r] \ar[d] & \sX \ar[d] \\
\sT \ar[r] & \sS.
}\end{gathered}
\end{equation}
is a categorical pullback square in $\ulMSCH_{\sS}$. In particular, all finite limits exist in $\ulMSCH_{\sS}$.
\end{thm}\marginpar{\fbox{fiber product}}\index{product of modulus pairs!fiber product} 

\begin{proof}
By definition, $\sS_0$ is a terminal object in $\ulMSCH_{\sS_0}$, so the existence of all finite limits follows from existence of all fibre products.  
Since fibre products are defined up to unique isomorphism, it suffices to prove the first claim in the statement, namely that any two ambient morphisms have a categorical fibre product in $\ulMSCH_{\sS_0}$. 

We will freely use the notation from Construction~\ref{cons:admBlowup}, in particular, $\sY = \sT \ambtimes[\sS] \sX$. 
Suppose that we are given two morphisms $a,b$ which makes the following diagram commute:
\begin{equation}\label{eq1.7b}\begin{gathered}\xymatrix{
\sA \ar@{.>}[rd]^c \ar@/^10pt/[rrd]^a \ar@/_10pt/[ddr]_b \\
&\sY \ar[r]^{i'} \ar[d]_{h'} & \sX \ar[d]^f \\
&\sT \ar[r]^g & \sS.
}\end{gathered}\end{equation}

It suffices to show that there exists a unique morphism $c: \sA \to \sY$ in $\ulMSCH_{\sS_0}$ which makes the diagram commute.
The uniqueness of $c$ follows from the uniqueness of $c^\o$ (which is a consequence of \eqref{eq:1.3}) and from Lemma \ref{lem:coincidence2} (2).

We prove the existence of $c$.
By calculus of fractions, we may assume that $a$ and $b$ are ambient without loss of generality.
Then we obtain a commutative diagram of schemes
\[\xymatrix{
{\hA} \ar@/^10pt/[rrd]^{\ol{a}} \ar@/_10pt/[ddr]_{\ol{b}} \\
&{\hY}_0 \ar[r]^{\ol{i}} \ar[d]_{\ol{h}} & {\hX} \ar[d]^{\ol{f}} \\
&{\hT} \ar[r]^{\ol{g}} & {\hS}.
}\]

Recall that ${\hY}_0$ is the scheme theoretic image of ${\iX} \times_{{\iS}} {\iT} \to {\hX} \times_{{\hS}} {\hT}$. 
We claim that the induced morphism $\ol{c}_0 : {\hA} \to {\hX} \times_{{\hS}} {\hT}$ factors through ${\hY}_0$.
Indeed, since $a$ and $b$ are morphisms in $\ulPSCH_{\sS_0}$, we have ${\iA} \to {\iX} \times_{{\iS}} {\iT} = {\iY} \subset {\hY}_0$.
Therefore, ${\iA} \subset d^{-1} ({\hY}_0) \subset {\hA}$.
Since ${\iA}$ is scheme theoretically dense in ${\hA}$ by Lemma \ref{prop:always-dense}, we have $\ol{c}_0^{-1} ({\hY}_0) = {\hA}$, and hence the claim.
Thus, $\ol{c}_0$ induces a morphism $\ol{c}_1 : {\hA} \to {\hY}_0$:
\[\xymatrix{
{\hA} \ar[rd]^{\ol{c}_1} \ar@/^10pt/[rrd]^{\ol{a}} \ar@/_10pt/[ddr]_{\ol{b}} \\
&{\hY}_0 \ar[r]^{\ol{i}} \ar[d]_{\ol{h}} & {\hX} \ar[d]^{\ol{f}} \\
&{\hT} \ar[r]^{\ol{g}} & {\hS}.
}\]

Consider the blow up of ${\hA}$ along $F_A :=\ol{c}_1^{-1} (F)$ (recall \eqref{eq:1.4}):
\[
{\hB}:= \Bl_{F_A} ({\hA}) \xrightarrow{\pi_A} {\hA},
\]
and set 
\[
{\mB} := {\mA} \times_{{\hA}} {\hB}, \ \ \sB := ({\hB},{\mB}),
\]
where ${\mA} \times_{{\hA}} {\hB}$ is an effective Cartier divisor by Lemma \ref{lem:pb-bu} and by the inclusions
\[
F_A = \ol{c}_1^{-1} (F) \subset  \ol{c}_1^{-1} \ol{i}^{-1} {\mY} = \ol{a}^{-1} ({\mY}) \subset {\mA} .
\]

Then $\pi_A$ induces a minimal morphism \[p :\sB \to \sA,\] and $p \in \ul{\Sigma}$ by construction. 
Therefore, $p$ is an isomorphism in $\ulMSCH_{\sS_0}$.

Moreover, by the universal property of blow up, there exists a morphism $\ol{c}_2 : {\hB} \to {\hY}$ which makes the diagram 
\begin{equation}\label{eq:1.28}\begin{gathered}\xymatrix{
{\hB} \ar[r]^{\ol{c}_2} \ar[d]_{\ol{p}=\pi_A} & {\hY} \ar[d]^{\pi} \\
{\hA} \ar[r]^{\ol{c}_1} & {\hY}_0
}\end{gathered}\end{equation}
commute.

We claim $\ol{c}_2^{-1} {\mW} \subset {\mB}$.
Indeed, by the commutativity of the above diagrams, we have
\[
\ol{c}_2^{-1} D_Y 
=\ol{c}_2^{-1} (\pi^{-1} \ol{i}^{-1} {\mY}) 
=\pi_A^{-1} (\ol{c}_1^{-1} \ol{i}^{-1} {\mY}) 
=\pi_A^{-1} (\ol{a}^{-1} {\mY} ),
\]
and 
\[
\ol{c}_2^{-1} D_Z 
=\ol{c}_2^{-1} (\pi^{-1} \ol{h}^{-1} {\mZ}) 
=\pi_A^{-1} (\ol{c}_1^{-1} \ol{h}^{-1} {\mZ}) 
=\pi_A^{-1} (\ol{b}^{-1} {\mZ}),
\]
and since $\ol{a}^{-1} {\mY} \subset {\mA}$ and $\ol{b}^{-1} {\mZ} \subset {\mA}$ by assumption, we have 
\[
\ol{c}_2^{-1} D_Y , \ol{c}_2^{-1} D_Z \subset \pi_A^{-1} {\mA} = {\mB} .
\]
Therefore, noting that \eqref{eq1.6} implies 
\[
(\ol{c}_2^{-1} D_Y) \times_{{\hB}} (\ol{c}_2^{-1} D_Y) = (D_Y \times_{{\hY}} D_Z) \times_{{\hY}} {\hB} = \ol{c}_2^{-1} E,
\]
we obtain by Lemma \ref{lem:key} (2) that
\[
\ol{c}_2^{-1} {\mW} = \ol{c}_2^{-1} (D_Y + D_Z - E) \subset {\mB} ,
\]
which proves the claim.

The above claim let $\ol{c}_2$ define an ambient morphism $c_2 : \sB \to \sY$.
Thus, we obtain a morphism $c : \sA \to \sY$ in $\ulMSCH_{\sS_0}$ by the diagram
\[
c : \sA \xleftarrow{p} \sB \xrightarrow{c_2} \sY ,
\]
and it makes \eqref{eq:1.28} commute. 
This finishes the proof.
\end{proof}

\section{Coproducts}

\begin{lemm} \label{prop:finitePCoprod}
The category $\ulPSCH$ admit coproducts. They are preserved and detected by the forgetful functor $\ulPSCH \to \SCH$; $\sX \mapsto \hX$ from Corollary \ref{cor:twoforgetfuncs}. 
\end{lemm}

\begin{proof}
The first part follows from the definitions. Explicitly, we have a canonical commutative square
\[ \xymatrix{
\hom_{\ulPSCH}((\coprod \hX_i, \coprod \mX_i), \sZ) 
\ar[d] \ar@{}[r]|\subseteq & 
\hom_{\SCH}(\coprod \hX_i, \hZ) \ar[d]^\cong \\
\Pi \hom_{\ulPSCH}(\sX_i, \sZ)  \ar@{}[r]|\subseteq & \Pi \hom_{\SCH}(\hX_i, \sZ) 
} \]
and one checks easily that a morphism on the upper left satisfies the modulus condition if and only if it satisfies the modulus condition on the lower left.
\end{proof}

\begin{prop} \label{prop:finiteCoprod}
The category $\ulMSCH$ admits finite coproducts. They are preserved and detected by the functor $\ulMSCH \to \SCH$; $\sX \mapsto \iX$. More explicitly: 
\begin{enumerate}
 \item For any two modulus pairs $\sX$ and $\sY$ the modulus pair $(\hX \sqcup \hY, \mX \sqcup \mY)$ satisfies the property of categorical coproduct in $\ulMSCH$.
 \item For any modulus pair $\sX$ such that $\iX = \iX_0 \sqcup \iX_1$, there exist modulus pairs $\sX_0$ and $\sX_1$ with interiors $\iX_0$ and $\iX_1$ respectively, and an abstract admissible blowup $(\hX_0 \sqcup \hX_1, \mX_0 \sqcup \mX_1) \to \sX$.
\end{enumerate}
\end{prop}

\begin{proof}
The first part follows from the statement for $\ulPSCH$, Lem.~\ref{prop:finitePCoprod}, since the category of abstract admissible blowups of $(\hX \sqcup \hY, \mX \sqcup \mY)$ is the product of the categories of abstract admissible blowups of $\sX$ and $\sY$.

We prove the second part. 
Let $\ol{\iota_i} : \hX_i \to \hX$ be the schematic closure of $\iX_i$ in $\hX$ for $i=0,1$.
Then the closed subschemes $\mX \times_{\hX} \hX_i$ is an effective Cartier divisor for each $i=0.1$. Indeed, by Lemma \ref{lem:pb-dense}, it suffices to show that the schematic image of the open immersion $\ol{\iota}_i^{-1} (\iX) \to \hX_i$ equals $\hX_i$, which is obvious by $\iX_i \subset \iota_i^{-1} (\iX)$.
Therefore, we obtain modulus pairs $\sX_i := (\hX_i , \mX_i)$, where $\mX_i := \mX \times_{\hX} \hX_i$ for all $i$.
Then, for each $i$, the map $\ol{\iota_i}$ induces a minimal morphism $\iota_i : \sX_i \to \sX$. 
Let $p : \sX_0 \sqcup \sX_1 \to \sX$ be the induced morphism. 

We will show that $p$ is an abstract admissible blowup, i.e., that $p$ is minimal, $\ol{p}$ is proper, and the induced morphism $\ol{p}^{-1}(\iX) \to \iX$ is an isomorphism. 
The minimality is immediate by construction.
Since $\ol{p} : \hX_0 \sqcup \hX_1 \to \hX$ is induced by the closed immersions $\ol{\iota}_i$, it is proper (to show this fact, we use the commutativity of coproduct and fiber product of schemes). 
Finally, we prove an isomorphism $\ol{p}^{-1}(\iX) \cong \iX$.
Recall the fact that the construction of schematic image of a quasi-compact morphism commutes with the restriction to the open subscheme.
Then, since the open immersion $j : \iX \to \hX$ is quasi-compact by Lemma \ref{lem:qc-interior},  the following two objects are the same: \begin{enumerate}
\item the schematic image of $\iX_i = j^{-1}(\iX_i)$ in $\iX$,
\item the pullback along $j$ of the schematic image of $\iX_i \to \hX$.
\end{enumerate}
But (1) is nothing but $\iX_i$ since it is already closed in $\iX = \iX_0 \sqcup \iX_1$, and 
(2) equals $\hX_i \times_{\hX} \iX$ by definition of $\hX_i$.
Therefore, we have $\iX \times_{\hX} \hX_i = \iX_i$ for all $i$, which implies $\ol{p}^{-1}(\iX) \cong \iX$. 
This finishes the proof. 
\end{proof}


\section{Pro-modulus pairs} \label{sec:proModPairs}

Throughout this section, we fix a base-modulus pair $\sS$, and often assume that $\sS$ is qcqs, that is, $\hS$ is quasi-compact and quasi-separated.
We show that, for such $\sS$ and nice pro-objects $(\sX_{λ})$ in $\ulPSCH_\sS$, we have
\[ 
\varinjlim \hom_{\ulMSCH_\sS}(\sX_λ, \sY) 
\cong 
\hom_{\ulMSCH_\sS}(\amblim \sX, \sY)
\]
if $\hY \to \hS$ is finite type and separated, cf.,Prop.~\ref{prop:limXYMsch}. This is used in Section~\ref{sec:locProMod} to connect pro-modulus pairs that are local for a topology, and modulus pairs that are local for a topology. We also use it in Thm.~\ref{thm:finHeavy} in conjunction with the relative Riemann-Zariski space.

Under the assumption that the total space $\hS$ is noetherian, we can get a fully faithful embedding 
\[ \ulPSCH_\sS^{\qcqs} \subseteq \Pro(\ulPSch_\sS), \]
where $\ulPSCH_{\sS}^{\qcqs}$ denotes the full subcategory of $\ulPSCH_{\sS}$ consisting of qcqs $\sS$-schemes (cf. Def.~\ref{def:ambientmorph}~\eqref{item:morphprop}).

\subsection{The category ${{\protect\ulMSch_\sS}}$}

\begin{defi}\ 
\begin{enumerate}
 \item An ambient morphism $\sX \to \sY$ is called \emph{separated}\marginpar{\fbox{separated}}\index{morphism of modulus pairs!ambient morphism!separated} %
 (resp. \emph{finite type}\marginpar{\fbox{finite type}}\index{morphism of modulus pairs!ambient morphism!finite type})
 if the morphism ${\hX} \to {\hY}$ is separated (resp. finite type). 
 
 \item A morphism $\sX \to \sY$ in $\ulMSCH$ is called 
 \emph{separated}\marginpar{\fbox{separated}}\index{morphism of modulus pairs!separated} %
 (resp. \emph{finite type}\marginpar{\fbox{finite type}}\index{morphism of modulus pairs!finite type})
  if it is of the form $f \circ s^{-1}$ with $f$ a separated (resp. finite type) ambient morphism and $s$ an abstract admissible blowup. 

 \item $\ulMSch_\sS$
  \marginpar{\fbox{$\ulMSch_\sS$}}%
  \index[not]{$\ulMSch_\sS$} %
  denotes the full subcategory of $\ulMSCH_\sS$ whose objects are isomorphic to $\sS$-modulus pairs of finite type.
\end{enumerate}
\end{defi}

\begin{prop}
Let $\sS$ be a modulus pair, and let $\ul{\Sigma}_{f.t.}$ be the class of abstract admissible blowups $\sX \to \sY$ fitting into commutative triangles %
$\underset{\sY}{\stackrel{\sX}{\downarrow}} {}^{\searrow}_{\nearrow} {\sS}$ in $\ulPSch_{\sS}$. Then there is a canonical equivalence of categories
\[ 
\ulMSch_\sS \cong \ulPSch_\sS[\ul{\Sigma}_{f.t.}^{-1}]
\]
where $\ulPSch_\sS$ 
is the full category of $\ulPSCH_{\sS}$ consisting of $\sS$-schemes of finite type. 
\end{prop}

\begin{proof}
First notice that since elements of $\ul{\Sigma}_{f.t.}$ are isomorphisms in $\ulMSch_\sS$, every object can be represented by one whose structural morphism is ambient. Then we calculate directly that in all cases, the hom's are calculated as colimits over commutative diagrams of the form
\[ \xymatrix{
\sX \ar[dr] & \ar[l]_{\in \ul{\Sigma}} \ar[d] \ar[r] \sX' & \sY \ar[dl] \\
&\sS
} \]
where transition morphisms between two diagrams are of the form
\[ \xymatrix{
& \sX'' \ar[d] \ar[dl]_{\in \ul{\Sigma}} \ar[dr] & \\
\sX \ar[dr] & \ar[l]_{\in \ul{\Sigma}} \ar[d] \ar[r] \sX' & \sY \ar[dl] \\
&\sS,
} \]
which finishes the proof.
\end{proof}

\subsection{Pro-modulus pairs over a qcqs base pair}

We begin with a scheme-theoretic result that we do not know a reference for. Recall that we provided the definition of schematically dominant in Def.~\ref{defi:SchDom}. The result Prop.~\ref{prop:SCHLIMIT} below is well-known when $Y \to S$ is \emph{of finite presentation}, without assuming that the transition morphisms of $(X_λ)_{λ \in \Lambda}$ are schematically dominant. Our observation is that if we strength the conditions on $(X_λ)_{λ \in \Lambda}$ we can weaken the conditions on $Y$.

\begin{prop} \label{prop:SCHLIMIT}
Let $S$ be a scheme, $(X_λ)_{λ \in \Lambda}$ a filtered system of $S$-schemes with affine \emph{schematically dominant} transition morphisms, $Y \to S$ a morphism locally of finite type, and suppose all $S, X_λ, Y$ are qcqs. Then the canonical morphism
\[ 
\varinjlim \hom_{\SCH_S}(X_λ, Y) 
\to 
\hom_{\SCH_S}(\lim X_λ, Y)
\]
is an isomorphism.
\end{prop}

\begin{proof}
By \cite[Prop.8.13.1]{EGAIV3} the morphism is injective, so it suffices to prove surjectivity. We can write $Y$ as a filtered limit $Y = \lim Y_\mu$ of $S$-schemes of finite presentation, with all $Y \to Y_\mu$ and $Y_\mu \to Y_{\mu'}$ closed immersions, \cite[Tag 09ZQ]{stacks-project}. Then by \cite[Prop.8.13.2]{EGAIV3}, the canonical morphism 
\[ \hom_S(X, Y) \to \lim \varinjlim \hom_S(X_λ, Y_\mu) \]
is a bijection, where $X = \lim X_λ$. Let $(f_\mu: X_{λ_\mu} \to Y_\mu)$ be any representative of an $S$-morphism $f: X \to Y$ that we would like to lift to the colimit in the statement. In particular, for each $λ$ we have a commutative square 
\[ \xymatrix@!=6pt{
X \ar[r]^f \ar[d] & Y \ar[d] \\
X_{λ_\mu} \ar[r]_{f_\mu} & Y_\mu.
} \]
Now since $Y \to Y_\mu$ is a closed immersion and $X \to X_{λ_\mu}$ is schematically dominant, we find a unique diagonal morphism $X_{λ_\mu} \to Y$ making both new triangles commute. Hence, $f: X \to Y$ factors through some $X \to X_{λ_\mu} \to Y$.
\end{proof}

\begin{warn}
The above does \emph{not} imply that
\[ \lim: 
\left \{ 
\begin{array}{c}
\textrm{pro-objects of } \Sch_S \textrm{ with} \\
\textrm{affine, schematically dominant} \\
\textrm{transition morphisms} 
\end{array}
\right\} 
\to \SCH_S \]
is fully faithful (where the source is the full subcategory of $\Pro(\Sch_S)$ described). Indeed, let $\NN \sqcup \{∞\}$ be the one point compactification of the discrete space $\NN$, and consider the system $(U_{λ})_λ$ of open subspaces containing $∞$. Then $\lim U_λ = \{∞\}$, but as a pro-object $(U_λ)_λ$ is not isomorphic to the constant pro-object $\{∞\}$. Now note that everything we just said can be transferred to $\Sch_\ZZ$ by the canonical fully faithful functor sending a profinite set $\lim_i S_i$ to the affine scheme $\colim_i \prod_{s \in S_i} \mathbb{C}$.
\end{warn}

\begin{lemm}[Some filtered limits exist] \label{lemm:limitExists}
Let $\sS$ be a modulus pair, let $(\sX_λ)_{λ \in \Lambda}$ be a filtered system in $\ulPSCH_\sS$ such that $\hX := \lim \hX_λ$ exists in the category of schemes, the system of ind-closed subschemes $(\mX_λ|_{\hX})$ stabilises to some $\mX$, and this colimit $\mX$ is an effective Cartier divisor. Then each $(\hX, \mX) \to (\hX_λ, \mX_λ)$ is an admissible morphism of ambient modulus pairs and for any $\sY \in \ulPSCH_\sS$ the canonical morphism
\[ 
\hom_{\ulPSCH_\sS}(\sY, (\hX, \mX)) 
\to 
\lim 
\hom_{\ulPSCH_\sS}(\sY, \sX_λ) 
\] 
is an isomorphism. In this case we will write 
\[ \amblim \sX_λ := (\lim \hX_λ, \lim \mX_λ|_{\lim \hX_λ}) \]
to emphasize that the limit property holds in $\ulPSCH_\sS$. 
\end{lemm}

\begin{proof} \label{lemm:limitExistsExists}
All claims follow from the hypotheses and the definitions.
\end{proof}

\begin{lemm}[Affine constant modulus filtered limits exist]
In the notation of Lem.~\ref{lemm:limitExists}, the hypotheses of that lemma are satisfied if there exists $λ_0$ such that the morphisms $\sX_λ \to \sX_{λ'}$ are minimal, and the morphisms $\hX_λ \to \hX_{λ'}$ are %
affine morphisms for all $λ' \geq λ_0$. %
In this case $\amblim \sX_λ \to \sX_λ$ is minimal for every $λ \geq λ_0$.
\end{lemm}

\begin{proof}
The limit $\lim \hX_λ$ exists in the category of schemes by affineness of the transition morphisms. For the claim that $\mX_λ|_{\hX}$ is an effective Cartier divisor, it suffices to consider the case that $\hX_λ$ is affine and $\mX_λ$ is globally principal. Now the claim follows from the fact that for any filtered colimit of rings $A = \varinjlim A_λ$, and $f_λ \in A_λ$ the image $f \in A$ of $f_λ$ is a zero divisor if $f_λ|_{A_\mu}$ is a nonzero divisor for all $\mu \geq λ$.
\end{proof}

\begin{lemm} \label{lemm:pullbacklimamb}
Let $\sS_{α}$ be a modulus pair, $(\sS_λ)_λ$ a filtered system in $\ulPSCH_\sS$ with affine minimal transition morphisms, $\sX_α \in \ulPSCH_{\sS_α}^\qcqs$, and set $\sX_λ := \sS_λ \ambtimes_{\sS_α} \sX_α$. Then the system $(\sX_λ)_λ$ also has affine minimal transition morphisms, and there is a canonical isomorphism
\[
\amblim \sX_λ \cong (\amblim \sS_λ) \ambtimes_{\sS_α} \sX_α
\]
in $\ulPSCH_\sS$.
\end{lemm}

\begin{proof}
The transition morphisms of $(\sX_λ)_λ$ are minimal by Lem.~\ref{lemm:minProduct} and affine by Lem.~\ref{lemm:ambtimesProper}\eqref{lemm:lemm:ambtimesProper:5}. The isomorphism follows directly from Lem.~\ref{lemm:limitExists}, Prop.~\ref{prop:minFibPro}, and uniqueness of categorical limits.
\end{proof}

\begin{prop}[Approximating finite type morphisms] \label{prop:cannotLimitBlowups}
Suppose that $\sX$, $\sY$ are qcqs modulus pairs and $\hX \to \hY$ is a morphism of finite type. Then ${\hX} \to {\hY}$ is a filtered limit 
\[ {\hX} = \lim {\hX}_λ \to {\hY} \]
of ${\hY}$-schemes such that:
\begin{enumerate}
 \item ${\hX}_λ \to {\hY}$ is \emph{of finite presentation} for all $λ$, and all transition morphisms ${\hX}_λ \to {\hX}_\mu$ and the morphisms ${\hX} \to {\hX}_λ$ are closed immersions.

 \item \label{prop:cannotLimitBlowups:Modulus} Each $\hX_λ$ is equipped with a locally principal closed subscheme $\mX_λ$ such that $\mX = \hX \times_{\hX_λ} \mX_λ$ and $\mX_\mu = \hX_\mu \times_{\hX_λ} \mX_λ$ for all $\mu \geq λ$. 
 
 \item If $\mX \supseteq \mY|_{\hX}$ (resp, $\mX = \mY|_{\hX}$) then we can assume $\mX_λ \supseteq \mY|_{\hX}$ (resp, $\mX_λ = \mY|_{\hX_\lambda}$) for all $λ$.

 \item If $\hX \to \hY$ is proper we can assume all $\hX_λ \to \hY$ are proper.
 
 \item \label{prop:cannotLimitBlowups:centreIso}If $\iX \to \iY$ is an isomorphism, we can assume all $\iX_λ := \hX_λ \setminus \mX_λ \to \iY$ are isomorphims.
\end{enumerate}
\end{prop}

\begin{proof}
\begin{enumerate}
 \item This is none-other-than \cite[Tag 09ZQ]{stacks-project}.

 \item Effective cartier divisors are of finite presentation so $\mX \to \hX$ descends to some $\mX_λ \to \hX_λ$ of finite presentation, \cite[Thm.8.8.2(ii)]{EGAIV3}, which we can assume is a closed immersion, \cite[Thm.8.10.5(iv)]{EGAIV3}. Since $\hX_λ$ is quasi-compact, up to refining $λ$ we can assume that this of finite presentation closed immersion is locally principal, since it's limit is locally principal. Then replace the system $(\hX_\mu)$ with the subsystem of $\mu \geq λ$ for this $λ$, and equip all other $\hX_\mu$ with the pullback of $\mX_λ$.

 \item The $\supseteq$ case follows from \cite[Thm.8.8.2(i)]{EGAIV3} and the equality from \cite[Thm.8.10.5(i)]{EGAIV3}.

 \item The proper case is \cite[Tag~09ZR]{stacks-project}.
 
 \item The morphism $X^\circ \to Y^\circ$ is of finite presentation because it is an isomorphism, and each $X^\circ_λ \to Y^\circ$ is finite type (since it is of finite presentation), so the closed immersion $X^\circ \to X^\circ_λ$ is of finite presentation \cite[Tag 00F4(4)]{stacks-project} and therefore it's defined by a coherent sheave of ideals, \cite[Tag 01TV]{stacks-project}. Since $X^\circ_λ$ are all quasi-compact, it follows that there is some $λ$ for which $X^\circ_{λ'} = X^\circ$ for all $λ' \geq λ$.
\end{enumerate}
\end{proof}

\begin{rema}
In the situation of Prop.~\ref{prop:cannotLimitBlowups}, if \eqref{prop:cannotLimitBlowups:centreIso} is satisfied, then we have ${\mY}|_{{\hX}_λ}$ is an effective Cartier divisor if and only if ${\hX} = {\hX}_λ$. 

Indeed, since $\mX = \mX_λ|_{\hX}$ is an effective Cartier divisor, we have $\OO_{{\hX}} \subseteq \OO_{X^\circ}$ and it follows that 
\[ \ker(\OO_{{\hX}_λ} {\to} \OO_{{\hX}}) = \ker(\OO_{{\hX}_λ} {\to} \OO_{X^\circ_λ}) \]
Hence, the surjection $\OO_{{\hX}_λ} {\to} \OO_{{\hX}}$ is an isomorphism if and only if $\OO_{{\hX}_λ} {\to} \OO_{X^\circ_λ}$ is a monomorphism, i.e., if and only if $\mX_λ$ is an effective Cartier divisor.
\end{rema}

\begin{prop}\label{prop:qcqs-limcolim}
Let $\sS$ be a modulus pair, $\sY \in \ulPSch_\sS$, let $(\sX_λ)_{λ \in \Lambda} \in \Pro(\ulPSCH_\sS)$ be a filtered system with affine, minimal transition morphisms, and suppose all $\sS, \sX_λ, \sY$ are qcqs. 
Assume that one of the following conditions is satisfied:
\begin{enumerate}
 \item The transition morphisms of $(\sX_λ)_{λ \in \Lambda}$ are schematically dominant.
 \item $\hY \to \hS$ is of finite presentation.
\end{enumerate}
Then
\[ 
\varinjlim \hom_{\ulPSCH_\sS}(\sX_λ, \sY) 
\to 
\hom_{\ulPSCH_\sS}(\amblim \sX_λ, \sY)
\]
is an isomorphism, where $\amblim \sX_λ$ is the categorical limit in $\ulPSCH_\sS$, cf.Lem.~\ref{lemm:limitExists}.
\end{prop}

\begin{proof}
Set $\sX = \amblim_\lambda \sX_\lambda$.
Since $\hom_{\ulPSCH}(\sX_λ, \sY) \subseteq \hom_{\SCH}(\hX_λ, \hY)$ (and similar for $\sX$), and filtered colimits preserve monomorphisms, we have 
\[\xymatrix{
\varinjlim \hom_{\ulPSCH_\sS}(\sX_λ, \sY) \ar[d] \ar@{}[r]|\subseteq 
&
\varinjlim \hom_{\SCH_{\hS}}(\hX_λ, \hY) \ar[d]^\cong 
\\
\hom_{\ulPSCH_\sS}(\sX, \sY)  \ar@{}[r]|\subseteq 
&
\hom_{\SCH_{\hS}}(\hX, \hY)
} \]
Here the isomorphism on the right comes from Prop.~\ref{prop:SCHLIMIT} in the first case and \cite[Thm.8.13.1]{EGAIV3} in the second case. So it suffices to prove that any admissible morphism $\hX \to \hY$ can be descended to an admissible morphism $\hX_λ \to \hY$ for some $λ$. Suppose that $\hX_λ \to \hY$ is a morphism such that $\mX_λ|_{\hX} \geq \mY|_{\hX}$. Let $\mZ$ be an effective Cartier divisor of $\hX$ such that $\mX_λ|_{\hX} = \mY|_{\hX} + \mZ$. Since $\mZ$ is a closed immersion of finite presentation, up to refining $λ$, we can find some $\mZ_λ \to \hX_λ$ which is also a closed immersion of finite presentation, \cite[Thm.8.8.2(ii), Thm.8.10.5(iv)]{EGAIV3}. Refining $λ$ further, we can assume that $\mZ_λ$ is locally principal. Now consider the locally principal ideal $(\sI_{\mY}|_{\hX_λ}) \cdot \sI_{\mZ_λ}$. On $\hX$ we have $((\sI_{\mY}|_{\hX_λ}) \cdot \sI_{\mZ_λ})|_{\hX} = \sI_{\mX_λ}|_{\hX}$.\footnote{That is, if $a, b, c$ are locally generators for $\sI_{\mX_λ}$, $(\sI_{\mY}|_{\hX_λ}$, $\sI_{\mZ_λ}$, respectively, then $ua = bc$ for some unit $u$.} This implies that up to changing $λ$ we can assume $(\sI_{\mY}|_{\hX_λ}) \cdot \sI_{\mZ_λ} = \sI_{\mX_λ}$ on $\hX_λ$, since this can be checked locally on affines, and all our schemes are quasi-compact. As local generators for $\mX_λ$ are nonzero divisors, this latter equality implies that local generators for $\mZ_λ$ must also be nonzero divisors. In other words, $\mZ_λ$ is an effective Cartier divisor. Since $\mX_λ = \mY|_{\hX_λ} + \mZ_λ$ by construction, $\hX_λ \to \hY$ is admissible.
\end{proof}

\begin{coro} \label{coro:descendMorphisms}
Suppose that $\sS$ is a qcqs modulus pair and %
$(\sX_λ)_{λ}$, $(\sY_\mu)_{\mu}$ are filtered systems in $\ulPSch_\sS$ with affine minimal transition morphisms. %
Assume that one of the following conditions is satisfied:
\begin{enumerate}
 \item The transition morphisms of $(\sX_λ)_{λ \in \Lambda}$ are schematically dominant.
 \item The morphisms $\hY_μ \to \hS$ are all of finite presentation.\footnote{This can be weakened to: $\iY_μ \to \iS$ are all flat of finite presentation and $\hom_{\Pro(\ulPSch_\sS)}(\sX, -)$ sends \aab{}s to isomorphisms, cf.Prop.\ref{prop:allIsoToAmbEtale}.}
\end{enumerate}
Then
\[ 
\hom_{\ulPSCH_\sS}(\amblim \sX_λ, \amblim \sY_\mu) 
= 
\lim_\mu \varinjlim_λ \hom_{\ulPSch_\sS}(\sX_λ, \sY_\mu). \]
where $\amblim$ indicates the categorical limit in $\ulPSch_\sS$, cf.Lem.~\ref{lemm:limitExists}, Lem.~\ref{lemm:limitExistsExists}.
\end{coro}

\begin{coro} \label{coro:embedAmb}
Let $\sS$ be a qcqs modulus pair. Then the category $\ulPSCH_\sS^{\qcs}$ of qc sparated ambient $\sS$-modulus pairs embeds fully faithfully into the category of separated ambient pro-$\sS$-modulus pairs of finite type:
\[ \ulPSCH_\sS^{\qcs} \subseteq \Pro(\ulPSch_\sS). \]
It is identified as the full subcategory of pro-objects which can be represented by systems whose transition morphisms are affine, minimal, and schematically dominant. 
\end{coro}

\begin{proof}
Let $\sP$ be the subcategory of $\Pro(\ulPSch_\sS)$ described in the statement. %
By Lem.~\ref{lemm:limitExists} the categorical limit in $\ulPSCH$ of pro-systems $(\sX_\lambda)_\lambda$, $(\sY_\mu)_\mu$ exist and by Cor.~\ref{coro:descendMorphisms} we have
\[ \hom_{\ulPSCH}(\amblim \sX_λ, \amblim \sY_\mu) 
= \hom_{\Pro(\ulPSch_\sS)}((\sX_λ), (\sY_\mu)).
\]

So it suffices to show that every object of $\ulPSCH_\sS^{\qcs}$ is of the form $\amblim \sX_λ$ for a pro-object of $\ulPSch_\sS)$ with affine, minimal, schematically dominant transition morphisms. Let $\sX \in \ulPSCH_\sS^{\qcs}$. Since $\hX$ and $\hS$ are both qcqs, by Temkin's version of Thomason's approximation theorem \cite[Thm.1.1.2]{Tem11}, $\hX \to \hS$ factors as an affine morphism $\hX \to \hX_0$ and a separated morphism of finite presentation $\hX_0 \to \hS$. By \cite[Cor.6.9.9]{EGAI}, the $\OO_{\hX_0}$-algebra $\OO_{\hX}$ is the filtered union of its sub-$\OO_{\hX_0}$-algebras of finite type. In particular, replacing $\hX_0$ if necessary, we can assume that the effective Cartier divisor $\mX$ is pulled back from some locally principal closed subscheme $\mX_0 \subseteq \hX_0$ which satisfies $\mX_{λ_0} \supseteq \hX_{λ_0} \times_{\hS} \mS$.%
\footnote{An effective Cartier divisor can be described by an open affine cover $(Spec A_i)$ and nonzero divisors $f_i \in A_i$ such that on the intersections there is a unit $u_{ij}$ such that $f_i = u_{ij}f_j$. Taking a finite open cover, all $f_i, u_{ij}$ are eventually contained in the same sub-algebra. For the divisibility claim, the ring version is: if $B$ is an $A$-algebra and $A$, $B$ are equipped with elements $f$, $g$, such that (the image of) $f$ divides $g$ then there is a finite type sub-$A$-algebra of $B$ where (the image of) $f$ divides $g$. This also extends directly to the non-affine setting since our schemes are quasi-compact.} %
Now replace $\hX_0$ with the scheme theoretic image of $\hX \to \hX_0$ (we still have an affine morphism $\hX \to \hX_0$ because $\mX$ is an effective Cartier divisor), Prop.~\ref{prop:always-dense}. Then $\mX$ becomes an effective Cartier divisor and we get an admissible morphism $\sX_0 := (\hX_0,\mX_0) \to \sS$. Now apply \cite[Cor.6.9.9]{EGAI} again to get a system $(\hX_λ)_{λ}$ of finite type $\hS$-schemes with schematically dominant affine transition morphisms whose limit is $\hX$. Since the transition morphisms are schematically dominant, each $\mX_0|_{\hX_λ}$ remains an effective Cartier divisor, and we get the desired pro-object of $\ulPSch_\sS$.
\end{proof}

\begin{prop} \label{prop:limlimlim}
Let $\sS_α$ be a qcqs modulus pair, and $(\sS_λ)_λ$ a filtered system in $\ulPSCH_\sS$ with affine minimal transition morphisms. Consider $\sX_α \in \ulPSCH_\sS^\qcqs$ and $\sY_α \in \ulPSch_\sS$, and set $\sX_λ := \sS_λ \ambtimes_{\sS_α} \sX_α$ and $\sY_λ := \sS_λ \ambtimes_{\sS_α} \sY_α$. Suppose that either:
\begin{enumerate}
 \item $\hY_α \to \hS_α$ is of finite presentation, or
 \item all transition morphisms $\sX_μ \to \sX_λ$ are schematically dominant.
\end{enumerate}
Then the canonical morphism
\[ \colim \hom_{\ulPSCH_{\sS_λ}}(\sX_λ, \sY_λ) 
\to 
\hom_{\ulPSCH_{\amblim \sS_λ}}(\amblim \sX_λ, \amblim \sY_λ) \]
is an isomorphism. 
\end{prop}

\begin{proof}
Set $\sX := \amblim \sX_λ$, $\sY := \amblim \sY_λ$, and $\sS := \amblim \sS_λ$. We have
\begin{align}
\hom_{\ulPSCH_{\sS}}(\sX, \sY)
&\cong
\hom_{\ulPSCH_{\sS_α}}(\sX, \sY_α) \label{prop:limlimlim:A}
\\&\cong
\colim_λ \hom_{\ulPSCH_{\sS_α}}(\sX_λ, \sY_α) \label{prop:limlimlim:B}
\\&\cong
\colim_λ \hom_{\ulPSCH_{\sS_λ}}(\sX_λ, \sY_λ) \label{prop:limlimlim:C}
\end{align}
where \eqref{prop:limlimlim:A} is because $\sY \cong \sS \ambtimes_{\sS_α} \sY_α$, Lem.~\ref{lemm:pullbacklimamb}, the isomorphism \eqref{prop:limlimlim:B} is by Prop.~\ref{prop:qcqs-limcolim} and the isomorphism \eqref{prop:limlimlim:C} comes from $\sY_λ := \sS_λ \ambtimes_{\sS_α} \sY_α$ for all $λ$ since $\ambtimes$ is the categorical fibre product in $\ulPSCH$, Prop.~\ref{prop:minFibPro}.
\end{proof}

\begin{prop}[{Abstract admissible blowups descend through limits}] \label{prop:descendAdmBlowup}
Let $\sS$ be a qcqs modulus pair, and suppose that $(\sX_λ)_λ$ is a pro-object in $\ulPSCH^{\qcqs}$ with affine minimal transition morphisms. %
Suppose that $\sY \to \sX := \amblim \sX_λ$, is an \abb{}.
Then there exists some $λ$ and $\sY_λ \to \sX_λ$ which is also an \abb{} and an isomorphism %
$\sY \cong \sX \ambtimes_{\sX_λ} \sY_λ$.
\end{prop}

\begin{proof}
Factor $\hY \to \hX$ as a closed immersion $\hY \to \hW$ and a %
proper surjective morphism %
 $\hW \to \hX$ of finite presentation such that $\hW \to \hX$ is an isomorphism over $\iX$, Prop.~\ref{prop:cannotLimitBlowups}. By \cite[Thm.8.8.2(ii), Thm.8.10.5(i), (vi), (xii)]{EGAIV3}, %
for some $λ$ there is a proper morphism %
of finite presentation $\hW_λ \to \hX_λ$ such that $\iW_λ \to \iX_λ$ is an isomorphism %
and $\hW = \hX \times_{\hX_λ} \hW_λ$, where $\iW_λ := \hW_λ \times_{\hX_λ} \iX_λ$. %
Define $\hY_λ$ to be the scheme theoretic closure of $\iW_λ$ in $\hW_λ$, and $\mY_λ := \mX_λ|_{\hY_λ}$. %
So we have the following diagram where the vertical morphisms on the right (resp. left) are isomorphisms over $\iX_λ$ (resp. $\iX$). %
\[ \xymatrix{
& \hY \ar[d]_-{(\ast)} && \\
& \hX \times_{\hX_λ} \hY_λ \ar[d] \ar[r] & \hY_λ \ar[d] \\
& \hW \ar[d] \ar[r] & \hW_λ \ar[d] \\
& \hX \ar[r] & \hX_λ
} \]
Then $\sY_λ \to \sX_λ$ satisfies the requirements in the statement. %
Indeed, $\hY_λ \to \hX_λ$ is the composition of a closed immersion and a proper morphism so it is proper. %
We have already noted that $\iY_λ = \iX_λ$. 
It remains to show that $\sY = \sX \ambtimes_{\sX_λ} \sY_λ$. Since the total space of $\sX \ambtimes_{\sX_λ} \sY_λ$ is %
the scheme theoretic closure of $\iX \times_{\iX_λ} \iY_λ$ in $\hX \times_{\hX_λ} \hY_λ$, Lem.~\ref{lemm:minProduct}, %
it suffices to show that $\hY$ is this same closure. %
Certainly, $\hY$ is the scheme theoretic closure of $\iY$ by Prop.~\ref{prop:always-dense}, but then since the vertical morphisms are isomorphisms on the interiors, it follows that $\iX \times_{\iX_λ} \iY_λ = \iY$.
\end{proof}

Next we consider the non-ambient version of Prop.~\ref{prop:qcqs-limcolim}.

\begin{prop} \label{prop:limXYMsch}
Let $\sS$ be a modulus pair, let $\sY \in \ulPSch_\sS$, let $(\sX_λ)_{λ \in \Lambda}$ be a filtered system in $\ulPSCH_\sS$ with affine, minimal transition morphisms, %
and suppose all $\sS, \sX_λ, \sY$ are qcqs. If one of:
\begin{enumerate}
 \item The transition morphisms are schematically dominant.
 \item $\hY \to \hS$ is finite presentation.
\end{enumerate}
is satisfied, then
\[ 
\varinjlim \hom_{\ulMSCH_\sS}(\sX_λ, \sY) 
\to 
\hom_{\ulMSCH_\sS}(\amblim \sX_λ, \sY)
\]
is an isomorphism, where $\amblim \sX_λ$ is the categorical limit in $\ulPSCH_\sS$, cf.Lem.~\ref{lemm:limitExists}. 
\end{prop}

\begin{proof}
Set $\sX := \amblim_\lambda \sX_\lambda$.
Since $\hom_{\ulPSCH}(\sX_λ, \sY) \subseteq \hom_{\SCH}(\iX_λ, \iY)$ (and similar for $\sX$), and filtered colimits preserve monomorphisms, we have 
\[\xymatrix{
\varinjlim_λ \varinjlim_{\sW_λ \in \ul{\Sigma}/\sX_λ} \hom_{\ulPSCH_\sS}(\sW_λ, \sY) \ar[d] \ar@{}[r]|-\subseteq 
&
\varinjlim \hom_{\SCH_{\iS}}(\iX_λ, \iY) \ar[d]^\cong 
\\
\varinjlim_{\sW \in \ul{\Sigma}/\sX} \hom_{\ulPSCH_\sS}(\sW, \sY)  \ar@{}[r]|-\subseteq 
&
  \hom_{\SCH_{\iS}}(\iX, \iY)
} \]
Here the isomorphism on the right comes from Prop.~\ref{prop:SCHLIMIT} in the first case and \cite[Thm.8.13.1]{EGAIV3} in the second case. %
So it suffices to prove that a hat of admissible morphisms $\sX \leftarrow \sW \to \sY$ with $\sX \leftarrow \sW$ an abstract admissible blowup can be completed to a commutative diagram of admissible morphisms
\[ \xymatrix{
\sW \ar[d] \ar[r] & \sW_λ \ar[d] \ar[r] & \sY \\
\sX \ar[r] & \sX_λ
} \]
in $\ulPSCH_\sS$ for some $λ$ with $\hX_λ \leftarrow \hW_λ$ an admissible blowup.

The \abb{} descends to some $\sX_α \leftarrow \sW_α$ by Prop.~\ref{prop:descendAdmBlowup}. Defining $\sW_λ := \sX_λ \ambtimes_{\sX_α} \sW_α$ produces a filtered system of \abb{}s, cf.~Lem.\ref{lemm:ambtimesAab}, of which $\sW$ is the $\amblim$, cf.~Lem.\ref{lemm:pullbacklimamb}. Now the $\sS$-morphism $\sW \to \sY$ factors through some $\sS$-morphism $\sW_λ \to \sY$ for some $λ \geq α$ by Prop.~\ref{prop:qcqs-limcolim}.
\end{proof}

We finish this subsection by showing that it is never possible to approximate a non-of finite presentation abstract admissible blowup by ones of finite presentation.

\begin{prop}
Let $\sX \to \sY$ be an abstract admissible blowup of qcqs modulus pairs. Suppose that there is a filtered system $(\sX_λ \to \sY)$ of asbtract admissible blowups with transition morphisms ${\hX}_λ \to {\hX}_\mu$ affine, and each ${\hX}_λ \to {\hY}$ of finite presentation, such that
\[ {\hX} = \lim {\hX}_λ. \]
Then ${\hX} \to {\hY}$ is of finite presentation.
\end{prop}

\begin{proof}
Write ${\hX}$ as in Prop.~\ref{prop:cannotLimitBlowups} (or rather, as in \cite[Tag 09ZR]{stacks-project}) as a filtered limit ${\hX} = \lim_{λ' \in \Lambda'} {\hX}'_{λ'}$ of proper ${\hY}$-schemes of finite presentation such that each ${\hX} \to {\hX}_λ'$ is a closed immersion. Since $\lim_{λ' \in \Lambda'} {\hX}'_{λ'} = \lim_{λ \in \Lambda} {\hX}_λ$, for each $λ' \in \Lambda'$, there is some $λ \in \Lambda$ and a factorisation ${\hX} \to {\hX}_λ \to {\hX}'_{λ'}$. Since ${\hX} \to {\hX}'_{λ'}$ is a closed immersion and all three are ${\hY}$-separated, it follows that ${\hX} \to {\hX}_λ$ is also a closed immersion. On the other hand, (as sheaves of $\OO_{{\hX}_λ}$-algebras, for example), $\OO_{X^\circ_λ} \cong \OO_{X^\circ}$ and $\OO_{{\hX}_λ} \subseteq \OO_{X^\circ_λ}$, so $\OO_{{\hX}_λ} \to \OO_{{\hX}}$ is monomorphism. Since it is also a surjective, we discover that
\[ {\hX} \cong {\hX}_λ. \]
\end{proof}

\begin{ques}
Are all abstract admissible blowups between qcqs schemes of finite presentation?
\end{ques}


\chapter{Coverings} \label{chap:coverings}

In this section, we will introduce several topologies on categories of modulus pairs. The Nisnevich topology is used in the construction of motives with modulus. 

\section{Disclaimer on conventions of coverings}

In SGA, Suslin--Voevodsky's works, and all the first author's previous works ``covering'' means a morphism (or family of morphisms) in the site which becomes an epimorphism (or jointly surjective family) in the topos. However, when working with multiple sites equipped with highly related topologies (e.g., the Zariski topology on $X_{Zar}$, $\Sm_X$, $\Sch_X$, $\SCH_X^{\qcqs}$, $\SCH_X, \dots$) it makes more sense to follow the Stacks Project and define coverings as specific concretely defined families of morphisms independent of any site / topos. Then these families are used to define various sites / topoi. 

In practice, the reader will probably notice no difference to whichever convention they are used to, since 
for all coverings and sites we consider in this work, a morphism (of family of morphisms) of the site becomes surjective in the topos if and only if it is refinable by a covering.

\section{Covering families of schemes}

\begin{defi} \label{defi:coveringsSch}
Let $\sU = \{f_i: U_i \to X\}_{i \in I}$ be a \emph{jointly surjective}\marginpar{\fbox{jointly surjective}}\index{jointly surjective} family of morphisms of schemes.
\begin{enumerate}
 \item $\sU$ is a \emph{Zariski covering}\marginpar{\fbox{Zariski covering}}\index{covering!Zariski covering} or $\Zar$-covering if each $f_i$ is an open immersion of finite presentation.
 \item $\sU$ is a \emph{Nisnevich covering}\marginpar{\fbox{Nisnevich covering}}\index{covering!Nisnevich covering} or $\Nis$-covering if each $f_i$ is {\'e}tale and of finite presentation, and there exists a finite sequence $\varnothing = Z_0 \to Z_1 \to \dots Z_n = X$ of closed immersions of finite presentation such that for each $j$ there is an $i_j$ such that $Z_j \times_X U_{i_j} \to Z_j$ admits a section. 
 \item $\sU$ is an \emph{{\'e}tale covering}\marginpar{\fbox{{\'e}tale covering}}\index{covering!{\'e}tale covering} %
 or $\et$-covering %
 if each $f_i$ is {\'e}tale and of finite presentation.
 \item $\sU$ is an \emph{$\fppf$-covering}\marginpar{\fbox{$\fppf$-covering}}\index{covering!$\fppf$-covering} if each $f_i$ is flat and of finite presentation.
 \item $\sU$ is a \emph{closed covering}\marginpar{\fbox{closed covering}}\index{covering!closed covering} if each $f_i$ is a closed immersion.
 \item $\sU$ is a \emph{finite covering}\marginpar{\fbox{finite covering}}\index{covering!finite covering} %
 or $\finite$-covering %
 if each $f_i$ is a finite morphism.
 \item $\sU$ is an \emph{$\rqfh$-covering}\marginpar{\fbox{$\rqfh$-covering}}\index{covering!$\rqfh$-covering} if it can be written as a finite composition of Zariski coverings and finite coverings. That is, if $\sU$ is of the form %
 $$\{U_{i_0 i_1 \dots i_n} \to \dots \to U_{i_0 i_1} \to U_{i_0} \to X\}_{i_0 \in I_0, i_1 \in I_{i_0}, i_2 \in I_{i_1}, \dots, i_n \in I_{i_{n-1}}}$$ %
 where each $\{U_{i_0 i_1 \dots i_j} \to U_{i_0 i_1 \dots i_{j-1}}\}_{i_j \in I_{i_{j-1}}}$ is either a Zariski covering or a finite covering.
\end{enumerate}
\end{defi}

\begin{rema} \label{rema:qfhvsrqfh}
On Noetherian schemes, the $\rqfh$-topology is just the qfh-topology, \cite[Lem.3.4.2]{Voe96}, \cite[Thm.8.4]{Ryd10}, but in general the $\rqfh$-topology is finer. For example, the reduction of $\Spec(k[t_1, t_2, t_3, \dots] / \langle t_1^2, t_2^2, t_3^2, \dots \rangle)$ is an $\rqfh$-covering, but not a $\qfh$-covering. 

We choose ``r'' for ``reduced''. There is no direct connection between the $\rqfh$-topology and the rh-topology,\footnote{i.e., the topology generated by the Zariski topology and completely decomposed proper morphisms.} whose ``r'' probably stands for Riemann-Zariski. 
\end{rema}

\begin{rema} \label{rema:fppfQfh}
Despite appearances to the contrary, every $\fppf$-covering can be refined by a $\rqfh$-covering, \cite[\S IV.6.3]{SGA3}, \cite[Cor.17.16.2, Thm.18.5.11(c)]{EGAIV4}. So the $\rqfh$-topology (on $\Sch_\ZZ$ for example) is finer than the $\fppf$-topology.
\end{rema}

\section{Descending flat coverings through limits}

\begin{prop} \label{prop:descendFlatCoverings}
Let $(X_{λ})_{λ}$ be a pro-object in $\SCH$ with affine transition morphisms such that $X_\lambda$ is quasi-compact and quasi-separated for some $\lambda$.\footnote{Since any affine morphism is quasi-compact and separated, this condition implies that $X_{\lambda'}$ is qcqs for all $\lambda' \geq \lambda$.}
Set $X := \varprojlim X_{λ}$, and 
let $τ$ be one of $\Zar$, $\Nis$, $\et$, or $\fppf$. %
Then, for any $τ$-covering family $\sU = \{U_i \to X\}_{i \in I}$ inside $\Sch_{X}$, there exists some $λ$ and a $τ$-covering family $\{U_{jλ} \to X_λ\}_{j = 1}^n$ such that $U_{i_j} = X \times_{X_λ} U_{jλ}$ for some $i_j = 1, \dots, n$.
\end{prop}

\begin{proof}
Descending the $U_i$ individually is by 
\cite[Thm.8.8.2(ii)]{EGAIV3}, noting that each $U_i \to X$ is of finite presentation by definition. 
We may assume each $U_{λ_i} \to X_λ$ is an open immersion (resp. étale, flat of finite presentation) since $X$ has the limit topology (resp. by \cite[Tag 07RP]{stacks-project}, by \cite[Tag 04AI]{stacks-project}). Since $X$ is quasi-compact (as the limit of quasi-compact spaces), there is a finite subfamily such that $\amalg_{j = 1}^n U_{i_j} \to X$ is surjective, then refining, we can assume that all $λ_{i_j}$ are the same. Then surjectivity comes from \cite[Tag 07RR]{stacks-project}. Finally, for the Nisnevich case we want to descend a splitting stratification. This comes from \cite[Thm.8.8.2(i,ii), Thm.8.10.5(iv, vi)]{EGAIV3}.
\end{proof}

\begin{rema} \label{rema:qfhNoGood}
In general, we cannot descend closed, (and therefore finite or $\rqfh$-) coverings through limits as above, even if the transition morphisms are schematically dominant. For example, let $A_n := k[x_1, \dots, x_n] / \langle x_ix_j : i \neq j\rangle$, so $\colim A_n = k[x_1, x_2, \dots] / \langle x_ix_j : i \neq j\rangle$. Then the irreducible components of $\Spec(\colim A_n)$ correspond to the ideals $I_n = \langle x_i : i \neq n\rangle$, and they do not descend to any $\Spec(A_n)$. 

Or for a simpler example, take any non-finite profinite set $\lim S_i$ (with any scheme structure, e.g. $\Spec(\colim \prod_{S_i} k)$).

Consequently, $\rqfh$-sheafification does not commute in general with left Kan extension.

If we restrict to Noetherian interior we are ok though, cf.Prop.~\ref{prop:descendPQfhCoverings} below.
\end{rema}

\section{Covering families of modulus pairs} \label{sec:modPairCov}

\begin{defi} \label{defi:coveringsMSchamb}
Let $\tau$ be any of the classes of coverings from Def.~\ref{defi:coveringsSch}. %

We say that a family of minimal ambient morphisms $\{\sU_i \to \sX\}_{i \in I}$ in $\ulPSCH$ is a 
\emph{$\ulPtau$-covering}%
\marginpar{\fbox{$\ulPtau$-covering}}%
\index{covering@$\ulPtau$-covering} %
if $\{\hU_i \to \hX\}_{i \in I}$ is a $\tau$-covering. 

The class of \emph{$\ul{\textrm{M}}\tau$-coverings}%
\marginpar{\fbox{$\ul{\textrm{M}}\tau$-covering}}%
\index{covering@$\ul{\textrm{M}}\tau$-covering} %
in $\ulMSCH$ is the smallest class of families of morphisms containing the images of $\ulPtau$-coverings, and closed under isomorphism and composition.
\end{defi}

\begin{exam}\label{exa:covering}
So for example, if $\sY \to \sX$ is an abstract blowup and $\{\hV_i \to \hY\}_{i \in I}$ is a Nisnevich covering, then 
\[ \{(\hV_i, \mY|_{\hV_i}) \to \sY\}_{i \in I} \]
is a 
\emph{$\ulPNis$-covering}%
\marginpar{\fbox{$\ulPNis$-covering}}%
\index{covering@$\ulPNis$-covering} %
in $\ulPSCH$ and
\[ \{(\hV_i, \mY|_{\hV_i}) \to \hX\}_{i \in I} \]
is a 
\emph{$\ulMNis$-covering}%
\marginpar{\fbox{$\ulMNis$-covering}}%
\index{covering@$\ulMNis$-covering} %
in $\ulMSCH$.
\end{exam}

\begin{rema}
We will see later on, Cor.~\ref{coro:normalForm}, that up to isomorphism and refinement, for $\tau = \Zar$, $\Nis$, $\et$, $\fppf$, $\rqfh$, all $\ul{\textrm{M}}\tau$-coverings of qcqs modulus pairs are (the image of) a composition of a $\ulPtau$-covering $\{\sU_i \to \sY\}$ and an abstract admissible blowup $\sY \to \sX$ as in Example \ref{exa:covering}.
\end{rema}

\begin{prop} \label{prop:coveringAmbtimes}
Suppose that $\sX$ is a modulus pair, $\tau$ is any of $\Zar$, $\Nis$, $\et$, $\fppf$, $\finite$, $\rqfh$, and $\{\sU_i \to \sX\}_{i \in I}$ is a $\ulPtau$-covering. Then for any ambient morphism $\sY \to \sX$, the ambient product $\sY \ambtimes_{\sX} \sU_i$ represents the fibre product in $\ulPSCH$, and the family $\{\sY \ambtimes[\sX] \sU_i \to \sY\}_{i \in I}$ is again a $\ulPtau$-covering.
\end{prop}

\begin{proof}
The case $\tau = \Zar, \Nis, \et, \fppf$ follows from Lemma \ref{lemm:minProduct} for flat morphisms. 
It remains to show the cases $\tau = \fin, \rqfh$. 
By the functoriality of $\sY \ambtimes_{\sX} -$ (which holds by the minimality of $\sY \to \sX$; see Corollary \ref{cor:ambprod-func2}), the case $\tau = \rqfh$ follows from the cases $\tau = \Zar, \fin$. Therefore, we are reduced showing the case $\tau = \fin$.

Let $\{\sU_i \to \sX \}_i$ be a $\ulP\fin$-covering, and set $\sV_i := \sY \ambtimes_{\sX} \sU_i$.
By the minimality of $\sY \to \sX$ and by Lemma \ref{lemm:minProduct}, the ambient space $\sV_i$ is the schematic closure of $\iY \times_{\iX} \iU_i$ inside $\hY \times_{\hX} \hU_i$ for each $i$. Then, since $\hY \times_{\hX} \hU_i \to \hY$ are finite as pullback of finite morphisms, the induced morphisms $\hV_i \to \hY$ between closed subschemes are also finite. This shows $\tau=\fin$ case, finishing the proof.
\end{proof}

\begin{lemm} \label{lemm:pqrfhIsFinZar}
Any $\ulPrqfh$-covering can be written as a composition of $\ulPZar$-coverings and $\ulPfin$-coverings.
\end{lemm}

\begin{proof}
Recall that we defined $\rqfh$-coverings as compositions of $\finite$-coverings and $\Zar$-coverings. The problem that might occur is that pullback of effective Cartier divisors along finite morphisms are not guaranteed to be effective Cartier divisors any more. So we have have some $\sY \to \sX$ such that we can write $\hY \to \hX$ as a sequence $\hY = \hU_n \to \dots \to \hU_1 \to \hU_0 = \hX$ of finite coverings and Zariski coverings, by $\mX|_{\hU_i}$ is not an effective Cartier divisor for some $i$. However, in this case, we replace each $\hU_i$ with the scheme theoretic closure of $\hU_i \times_{\hX} \iX$ in $\hU_i$. 

Since $\mY = \mX|_{\hY}$, this does not alter $\mY$. Finally one checks that this operation preserves finite coverings and Zariski coverings, as we have already seen in the proof of Proposition \ref{prop:coveringAmbtimes}.
\end{proof}

\begin{lemm} \label{lemm:schRqfhGivesmpRqfh}
Let $\sP$ be a modulus pair, and $\{\hU_i \to \hP\}_{i \in I}$ a $\rqfh$-covering. Then there exists a $\ulPrqfh$-covering $\{\sV_i \to \sP\}_{i \in I}$ such that each $\hV_i$ is a closed subscheme of $\hU_i$.
\end{lemm}

\begin{proof}
Define $\hV_i$ to be the scheme theoretical closure of $\hU_i \times_{\hP} \iP$ in $\hU_i$. Then each $\mP|_{\hU_i}$ is an effective Cartier divisor, and we get a family of minimal morphisms $\{(\hV_i, \mP|_{\hV_i}) \to \sP\}_{i \in I}$. To finish, we must show that $\{\hV_i \to \hP\}_{i \in I}$ is still an $\rqfh$-covering. In light of the definition of $\rqfh$-coverings as being compositions of Zariski and finite coverings, and the fact that pullback preserves interiors since all morphisms are minimal, it suffices to consider the two cases separately that $\{\hU_i \to \hP\}_{i \in I}$ has a Zariski covering, or a finite covering. If it was a Zarisk covering, then all morphisms are flat, so $\mP|_{\hU_i}$ is already an effective Cartier divisor, so $\hV_i = \hU_i$. If it was a finite covering, then $\{\hV_i \to \hP\}$ is still a family of finite morphisms, and since we didn't change the interiors, and $\iP \to \hP$ is schematically dense, the new family is still jointly surjective. 
\end{proof}

\begin{prop} \label{prop:descendPQfhCoverings}
Let $(\sP_λ)_λ$ be a pro-object of $\ulPSCH^{\qcqs}$ with affine minimal transition morphisms, set $\sP := \amblim \sP_λ$ and suppose that $\iP$ is Noetherian. 
Suppose 
$\{\sW_i \to \sP\}_{i \in I}$ 
is an $\ulPrqfh$-covering. Then there exists a finite subset $J \subseteq I$, a $λ$, and a $\rqfh$-covering family $\{\sW_{jλ} \to \sP_λ\}_{j \in J}$ such that $\sW_j \cong \sP \ambtimes_{\sP_λ} \sW_{jλ}$ for each $j$.
\end{prop}

\begin{proof}
By definition, $\rqfh$-coverings are finite compositions of Zariski coverings and finite coverings, so it suffices to treat the two cases separately. The Zariski case is Prop.~\ref{prop:descendFlatCoverings}. 

So suppose that $\{\sW_i \to \sP\}_{i \in I}$ is a finite covering. %
Since $\iP$ is Noetherian, there is a finite subfamily which is still jointly surjective on the interior, and since $\iP \to \hP$ is topologically dense, by going up for finite morphisms, there is in fact a finite subfamily $\{\sW_j \to \hP\}_{j \in J}$ which is still a finite covering (i.e., jointly surjective on total spaces).

Now we proceed as in the proof of Prop.~\ref{prop:descendAdmBlowup}. By \cite[Prop.6.9.16(i)]{EGAI}, since $\hP$ is qcqs, we can write $\OO_{\hW_j}$ as a quotient $\OO_{\hW_j} = \sA / \sI$ of an $\OO_{\hP}$-algebra of finite presentation. %
By \cite[Cor.6.9.15]{EGAI}, $\sI$ is the filtered colimit of the finite type ideals it contains. %
Since $\iP$ is Noetherian, $\sI|_{\iP}$ is finite type, so there is in fact a finite type ideal $\sJ \subseteq \sI$ such that $\sJ|_{\iP} = \sI|_{\iP}$. Thus, we find a factorisation $\hW_j \to \hY_j \to \hP$ as a closed immersion followed by a(n affine) morphism of finite presentation such that $\hW_j \to \hY_j$ is an isomorphism over $\iP$. %
Moreover, we can take $\hY_i \to \hP$ to be finite by \cite[Thm.~8.10.5~(x)]{EGAIV3}.
Now by \cite[Thm.8.8.2(ii), Thm.8.10.5(x)]{EGAIV3}, we can descend $\hY_j$ to some morphism $\hY_{jλ} \to \hP_λ$, which is finite. Define $\hW_{jλ} \subseteq \hY_{jλ}$ to be the scheme theoretic closure of $\iP_λ \times_{\hP_λ} \hY_{jλ}$ in $\hY_{jλ}$. Then $\mP_λ|_{\hW_{jλ}}$ is an effective Cartier divisor by construction, so we obtain a minimal morphism $\sW_{jλ} \to \sP_λ$ of modulus pairs which is finite on the total spaces. Since the total space of $\sP \ambtimes_{\sP_λ} \sW_{jλ}$ is the closure of 
\[ \iP \times_{\iP_λ} \iW_{jλ} = \iP \times_{\iP_λ} \iY_{jλ} = \iY_{i} = \iW_j \]
in
\[ \hP \times_{\hP_λ} \hW_{jλ} \underset{imm.}{\overset{cl.}{\subseteq}} \hP \times_{\hP_λ} \hY_{jλ} = \hY, \]
it follows that the total space of $\sP \ambtimes_{\sP_λ} \sW_{jλ}$ is canonically identified with $\hW_j$. 

Now do this for each $j$, and choose some $λ$ so that the descended morphisms are all defined on the same $λ$. Then by \cite[Thm.8.10.5(vi)]{EGAIV3}, up to further refining $λ$ we can assume that $\{\sW_{jλ} \to \sP_λ\}_{j\in J}$ is jointly surjective on the interiors. But since $\iP_λ \to \hP_λ$ is topologically dense, by the going up theorem for finite morphisms it is jointly surjective on total spaces too.
\end{proof}


\section{${\protect\ulMtau}$-coverings of normal form} \label{sec:heaviness}

In this subsection we move towards the result that for the topologies $τ$ that we are interested in, any $\ulMtau$-covering can be represented as a composition of a $τ$-covering on total spaces, and a single abstract admissible blowup, Cor.~\ref{coro:normalForm}.

We also prove Lem.~\ref{lemm:valRingProperFlatGenFin} which, in addition to being used to show the $\rqfh$-case of the result just mentioned above, will be used to show that the modulus version of ``every finite algebra over a hensel local ring is a finite product of hensel rings'' is true, Lem.\ref{lemm:finiteulMNislocal}.

\begin{thm} \label{thm:etaleHeavy}
Let $\sX$ be a qcqs \mp, $\sY \to \sX$ an \amm, $\sY' \to \sY$ an \aab, and suppose that $\hY \to \hX$ is $(\ast)$ where 
$(\ast)$ is one of:
\begin{enumerate}
 \item \label{thm:etaleHeavy:op} a quasi-compact open immersion,
 \item \label{thm:etaleHeavy:etm} \'etale,
 \item \label{thm:etaleHeavy:fppfm} flat of finite presentation,
 \item \label{thm:etaleHeavy:Zar} a quasi-compact Zariski covering,
 \item \label{thm:etaleHeavy:Nis} a Nisnevich covering,
 \item \label{thm:etaleHeavy:etc} an \'etale covering,
 \item \label{thm:etaleHeavy:fppfc} an fppf covering.
\end{enumerate}
Then there exists an \aab{} $\sX' \to \sX$, 
an \amm{} $\sV \to \sX'$ such that $\hV \to \hX'$ is $(\ast)$, and a commutative diagram as on the left, inducing a commutative diagram with isomorphisms as indicated on the right.
\[ \xymatrix@!=3pt{
\sV \ar@{-->}[dr]_{(\ast)} \ar@{-->}[rrrr]^{adm.b.u} &&&& \sY' \ar[dl]^{adm.b.u} &&
\iV \ar@{-->}[dr]_{(\ast)} \ar@{=}[rrrr] &&&& Y'^\circ \ar@{=}[dl] \\
& \sX' \ar@{-->}[dr]_{adm.b.u} && \sY \ar[dl]^{(\ast)} & &&
& X'^\circ \ar@{=}[dr] && \iY \ar[dl]^{(\ast)} & \\
&& \sX &&&& 
&& \iX.
} \]
Moreover, $\sV \to \sY'$ is also an {\aab}s.
In particular, if $\iX, \iY, Y'^\circ$ are smooth over some lower base modulus pair, then so are $Y'^\circ$ and $\iV$. 
\end{thm}

\begin{proof} 
Cf. \cite[Prop.5.9]{SV00}. %
Since $\hY$ is qcqs, the proper morphism $\hY' \to \hY$ factors as a closed immersion $\hY' \to \hY_\lambda$ and a proper morphism of finite presentation $\hY_\lambda \to \hY$ such that both morphisms are isomorphisms over $\iY$ (but $\mY|_{\hY_\lambda}$ is probably not Cartier), Prop.~\ref{prop:cannotLimitBlowups}. %
Now that $\hY_\lambda \to \hX$ is finite presentation, by Raynaud-Gruson flatification, \cite[Thm.5.2.2]{RG71} or \cite[Tag 0815]{stacks-project}, there exists an admissible blowup $\sX' \to \sX$ such that the strict transform $\hY_\lambda^{st} \to {\hX'}$ of $\hY_\lambda \to \hX$ is flat and of finite presentation. So we have the following diagram of strict transforms. 
\[ \xymatrix{
{\hY'}^{st} \ar[r] \ar[d]^{(I)} 
& \hY' \ar[d]^{cl.imm} \\
\hY_\lambda^{st} \ar[r] \ar[d]^{(II)} \ar@/_18pt/[dd]_{flat, f.p.} 
& \hY_\lambda \ar[d]^{prop., f.p.} \\
\hY^{st} \ar[r] \ar[d]^-{(\ast)} 
& \hY \ar[d]^-{(\ast)} \\
\ol{X}' \ar[r]_{adm.b.u.} & \ol{X}
} \]
We claim that (I) is an isomorphism. Indeed, the curved morphism is flat, so $\mX|_{\hY_\lambda^{st}}$ is an effective Cartier divisor.%
\footnote{Pulling back along a blowup inside the divisor remains Cartier: It suffices to check on charts so consider some $A \to A[\tfrac{I}{g}] = \colim(\dots \stackrel{g}{\to} I^n \stackrel{g}{\to}  I^{n+1} \stackrel{g}{\to} \to \dots)$. Now we want to see that multiplication by $f$ is injective on $A[\tfrac{I}{g}]$ assuming it is injective on $A$. This is clear since filtered colimits preserve monomorphisms.

Pulling back along flat morphisms preserves Cartier divisors: Cartier means $A \to A; a \mapsto af$ is injective. So if $A \to B$ is flat, then $B \to B; b \mapsto bf$ is also injective.}
 Note $\mX|_{{\hY'}^{st}} = {Y'}^\infty|_{{\hY'}^{st}}$ is also an effective Cartier divisor, and (I) is an isomorphism over $\iX$. So in the following square,
\[ \xymatrix{
\iX \times_{\hX} {\hY'}^{st} \ar[r] \ar[d] & {\hY'}^{st} \ar[d]^{(I)} \\
\iX \times_{\hX} \hY_\lambda^{st} \ar[r] & \hY^{st}_\lambda 
} \]
both horizontal morphisms are schematically dominant, and have isomorphic sources, so (I) is a schematically dominant closed immersion, cf.~\cite[Tag 080D]{stacks-project}, that is, (I) is an isomorphism, so we can choose $\sV := \sY_\lambda^{st} = \sY_\lambda'^{st}$.

This finishes the case~\eqref{thm:etaleHeavy:fppfm}. To get to \eqref{thm:etaleHeavy:fppfc} it remains to show that the curved morphism is surjective. In the case \eqref{thm:etaleHeavy:fppfc}, the two starred morphisms are surjective (the bottom square is a pullback square), so it suffices to show that ${\hY'}^{st} \to \hY^{st}$ is surjective. But it's an admissible blowup, so it's surjective.

In the cases %
\eqref{thm:etaleHeavy:op}, 
\eqref{thm:etaleHeavy:etm}, 
\eqref{thm:etaleHeavy:Zar}, 
\eqref{thm:etaleHeavy:Nis}, and
\eqref{thm:etaleHeavy:etc}, 
the left $(\ast)$ is \'etale, so it's diagonal $\hY^{st} \to \hY^{st} \times_{\hX'} \hY^{st}$ is open, %
so the forgetful functor $\sO_{\hY^{st}}$-mod $\to \sO_{\hX'}$-mod \emph{detects} flatness, cf.\cite[Pf.of Cor.2.15]{HK18}. %
So $\hY_\lambda^{st} \to \hY^{st}$ is flat. It is an isomorphism over a dense subscheme, so it is of relative dimension zero. It is also proper. But quasi-finite proper morphisms are finite \cite[Tag 02LS]{stacks-project}, so it is finite flat. Again, it is an isomorphism over a dense subscheme, so it is finite flat of rank one. In other words, (II) is also an isomorphism. 
\end{proof}

\begin{thm} \label{thm:finHeavy}
Let $\sX$ be a qcqs modulus pair, $\sV \to \sX$ a finite surjective ambient morphism, and $\sW \to \sV$ an {\aab}. Then there exists a finite surjective ambient morphism $\sZ \to \sY$ and an admissible blowup $\sY \to \sX$ that fit into a commutative diagram
\[ \xymatrix{
& \sW \ar[d]^{adm.b.u.} \\
\sZ \ar[d]_{fin.surj.} \ar[ur]^{adm.b.u.}   & \sV \ar[d]^{fin.surj.} \\
\sY \ar[r]_{adm.b.u.} & \sX
} \]
where $\sZ \to \sW$ is also an {\aab}.
\end{thm}

\begin{rema}
Before starting a rigorous proof, we give a heuristic idea behind it.
For every ``small'' $\ulMZar$-point $\sP$ of $\sX$ (which we will introduce in \S \ref{chap:locProMod}), the ``strict transform'' $\sP \ambtimes[{\sX}] \sW \to \sP$ of $\hW \to \hX$ is finite surjective. From this we deduce that there is some element $\sP_\lambda$ in the pro-system defining $\sP$ where the strict transform is already finite, and by quasi-compactness of the relative Riemann-Zariski space, we can glue these $\sP_\lambda$ together to get the desired $\hY$. Now lets fill in the details of this strategy. Note that  alternatively we could avoid references to the Riemann-Zariski and \cite{Tem11} by using Prop.~\ref{prop:conservativeLocalPairs}.
\end{rema}

\begin{proof}
First we recall some facts about the relative Riemann-Zariski space $RZ_{\iX}\hX$. By definition, it is the inverse limit 
\[ RZ_{\iX}\hX = \varprojlim_{\sX_\lambda \to \sX} \hX_\lambda \]
in the category of locally ringed spaces over all {\aab}s $\sX_\lambda \to \sX$, \cite[\S2.1]{Tem11}. The points of $RZ_{\iX}\hX$ are in bijection with pairs $(x, R)$ where $x \in \iX$ and $R \subseteq k(x)$ is a valuation ring of $k(x)$ centred on $\hX$ such that $\Spec(R) \times_{\hX} \iX = x$, \cite[Cor.3.2.5]{Tem11}. The local ring of $RZ_{\iX}\hX$ at $(x, R)$ is $R \times_{k(x)} \OO_{\iX, x}$, and it is obtained as the filtered colimit 
\begin{equation} \label{equa:RZlocalRings}
 R \times_{k(x)} (\OO_{\iX, x}) = \varinjlim_{\Spec(R) \to \hU_{\lambda, \mu} \to \hX_\lambda \to \hX} \Gamma(\hU_{\lambda, \mu}, \OO_{\hU_{\lambda, \mu}})
\end{equation}
over factorisations with $\sX_\lambda \to \sX$ an admissible blowup, and $\hU_{\lambda, \mu} \to \hX_\lambda$ an open immersion.

Now we study ``strict transforms''. 
Set $\hP = \Spec(R \times_{k(x)} (\OO_{\iX, x}))$ and $\sP = (\hP, \mX|_{\hP})$ for a point $(x, R)$ of the relative Riemann Zariski space. Recall that for any ambient morphism $\sY \to \sX$, the total space of $\sY \ambtimes[{\sX}] \sW$ is the scheme theoretic closure of $\iY \times_{\iX} \iW$ in $\hY \times_{\hX} \hW$, equipped with the pullback of $\mY$, Lem.~\ref{lemm:minProduct}. Hence, for $\sW \to \sX$ as in the statement, $\sY \ambtimes[{\sX}] \sW \to \sY$ is always finite over $\iY$, and proper over $\hY$. In the case $\sY = \sP$, it is in fact finite on all of $\hP$, see Lem.~\ref{lemm:ulMZarambtimesFin} below. This implies that there is $\lambda, \mu$ as in Eq.~\ref{equa:RZlocalRings} with $\hU_{\lambda', \mu'} \ambtimes[{\sX}] \sW \to \hU_{\lambda', \mu'}$ finite for all $\lambda', \mu' \geq \lambda, \mu$, since $\ambtimes$ commutes with nice filtered limits, Lem.~\ref{lemm:pullbacklimamb}, and a filtered limit of schemes is affine if and only if the terms all become affine at some point in the system, \cite[01Z6]{stacks-project}. Choosing one $(\lambda_{(x,R)}, \mu_{(x,R)})$ for each $(x, R)$ we obtain an open covering of $RZ_{\iX}\hX$. Since $RZ_{\iX}\hX$ is quasi-compact, \cite[Prop.3.1.5]{Tem11}, this covering contains a finite subcovering $\{\hU_{\lambda_1, \mu_1}, \dots, \hU_{\lambda_n, \mu_n} \}$. Choosing a $\lambda$ bigger than all $\lambda_1, \dots, \lambda_n$ and pulling back the $\hU_{\lambda_i, \mu_i}$ we obtain an admissible blowup $\sX_\lambda \to \sX$, and an open covering $\{\hU_{\lambda, \mu_1}, \dots, \hU_{\lambda, \mu_n}\}$ of $\hX_\lambda$ such that each 
$\sU_{\lambda, \mu_i} \ambtimes[{\sX}] \sW \to \sU_{\lambda, \mu_i}$
is finite. This implies that 
$\sX_{\lambda} \ambtimes[{\sX}] \sW \to \sX_\lambda$ 
is finite, Lem.~\ref{lemm:minProduct}, Cor.~\ref{coro:ambtimesambtimes}. Setting 
$\sZ = \sX_{\lambda} \ambtimes[{\sX}] \sW$ and $\sY = \sX_\lambda$ we have found the desired commutative diagram depicted in the statement.
\end{proof}

\begin{lemm} \label{lemm:valRingProperFlatGenFin}
Suppose that $R$ is a valuation ring and $f: Y \to \Spec(R)$ is a flat, proper, generically finite morphism. Then $f$ is finite.
\end{lemm}

\begin{rema}
If we add ``finite presentation'' to the hypotheses on $f$, then we can just apply \cite[Tag 0D4J]{stacks-project}. However, we don't want to add this hypothesis, so we provide a proof.
\end{rema}

\begin{proof}
We will show that $f$ is finite by showing that it is quasi-finite.
By \cite[Tag 02LS]{stacks-project}, $f$ is quasi-finite iff. it is locally of finite type, is quasi-compact and has finite fibers. 
Since $f$ is proper, it already has the first two properties.
Therefore, we are reduced to showing that $f$ has finite fibers. 

To do this, we will find a surjective morphism $\widetilde{Y} \to Y$ such that $\widetilde{Y} \to \Spec(R)$ has finite fibres. Let $\widetilde{Y}$ be the integral closure of $\Spec(R)$ in the points $\{y_1, \dots, y_n\}$ of $Y$ in the generic fibre of $f$. Note that by the going down theorem for flat morphisms, $\{y_1, \dots, y_n\}$ are exactly the generic points of $Y$. %
That is, their closure is all of $Y$, i.e.,  $\overline{\{y_1, \dots, y_n\}} = Y$ (topologically). %
Since the local rings at closed points of $\widetilde{Y}$ are in bijection with the extensions of valuation of $R$ to the $k(y_i)$, 
by the valuative criterion for properness, there exists a (unique) factorisation $\widetilde{Y} \to Y \to \Spec(R)$. %
Since $Y \to \Spec(R)$ is generically finite, all field extensions $k(y_i) / \Frac(R)$ are finite, so there are finitely many extensions of $R$ to each $k(y_i)$. Moreover, an extension of valuation rings induces a bijection on primes, 
so the morphism $\widetilde{Y} \to \Spec(R)$ has finite fibres. 

\begin{center} \label{grap:YYR}
\includegraphics[width=7cm]{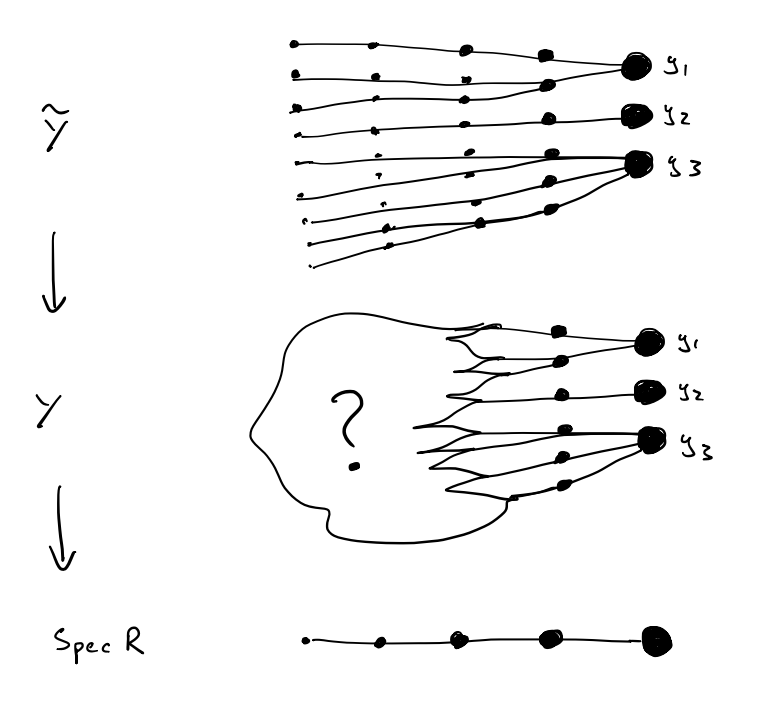}
\end{center}

It remains to show that $\widetilde{Y} \to Y$ is surjective. Since the image contains all generic points of $Y$, it suffices to show that it is universally closed. 
Now the composition $\widetilde{Y} \to \Spec(R)$ is integral, so it is universally closed, \cite[Tag 01WM]{stacks-project}, %
so it satisfies the existence part of the valuative criterion, \cite[Tag 01KF]{stacks-project}. %
But $Y \to \Spec(R)$ is proper, so it also satisfies the uniqueness part of the valuative criterion, so $\widetilde{Y} \to Y$ satisfies the existence part, and is therefore universally closed, \cite[Tag 01KF]{stacks-project}. %
\end{proof}

In the following lemma, $\sP$ is a modulus pair of the form $(\Spec(R \times_{\kappa} \OO), (f_0, f_1))$ where 
\begin{enumerate}
 \item $\OO$ is a local ring with maximal ideal $\m$ and residue field $\kappa := \OO/\m$,
 \item $R$ is a valuation ring of $\kappa$,
 \item $(R \times_{\kappa} \OO)[(f_0, f_1)^{-1}] = \OO$, where $(f_0,f_1) \in R \times_{\kappa} \OO$ is a non-zero divisor. 
\end{enumerate}

\begin{rema}
Rings of the form $R \times_{\kappa} \OO$ are called \emph{semi-valuation rings},%
\marginpar{\fbox{semi-valuation rings}}%
\index{semi-valuation rings} %
 cf.~\cite[Above Rem.3.1.2]{Tem11}. Later on, in \S \ref{chap:locProMod} we will call such modulus pairs $\ulMZar$-local modulus pairs, cf.~Prop.~\ref{prop:bigZarPoint}.
\end{rema}

Note that we deduce immediately from the above that:
\begin{enumerate} \setcounter{enumi}{3}
 \item The preimage $\p \subseteq R \times_\kappa \OO$ of $\m \subseteq \OO$ is $\p := \{0\} \times \m$.%
 \footnote{If $(a_0, a_1) \in R \times_\kappa \OO$ is an element such that $a_1 \in \m$, then the image of $a_1$ under $R \to \kappa = \OO/\m$ is zero, so the image of $a_0$ under $R \to \kappa = Frac(R)$ is zero, so $a_0$ is zero.}

 \item $R[f_0^{-1}] = \kappa$.\footnote{$R = (R \times_\kappa \OO) / (\{0\} \times \m)$ so $R[f_0^{-1}] = (R \times_\kappa \OO)[(f_0,f_1)^{-1}] / (\{0\} \times \m) = \OO / \m = \kappa$.}

 \item $\p \subsetneq ((f_0, f_1))$.\footnote{$(R \times_{\kappa} \OO)[(f_0, f_1)^{-1}] = \OO$ implies that $f_1 \in \OO^* = \OO \setminus \m$ so $(f_0, f_1) \notin \p = \{0\} \times \m$. On the other hand, for any $g \in \m$ we have $(0,g) = (f_0, f_1)(0, f_1^{-1}g)$ so $\p = \{0\} \times \m \subseteq ((f_0, f_1))$.}
 
 \item $\ol{\{p\}} = \Spec R$.
\end{enumerate}

\begin{lemm} \label{lemm:ulMZarambtimesFin}
Suppose that $\sP = (\Spec(R \times_{\kappa} \OO), (f_0, f_1))$ in the above notation, 
$\sQ \to \sP$ an ambient minimal morphism with $\iQ \to \iP$ finite and $\hQ \to \hP$ proper. Then $\hQ \to \hP$ is finite. 
\end{lemm}

\begin{proof}
It suffices to show that $\hQ \to \hP$ is quasi-finite \cite[02LS]{stacks-project}. Over $\iP$ it is quasi-finite by assumption, so it suffices to show that $\hQ \to \hP$ is finite over $\ol{\{p\}} \subseteq \hP$, where $p \in \hP$ is the closed point of $\iP$. For this, we want to apply Lem.~\ref{lemm:valRingProperFlatGenFin}. We already know that $\ol{\{p\}} \times_{\hP} \hQ \to \ol{\{p\}}$ is proper and generically finite so it remains to show that it is flat.

Being flat over a valuation ring is equivalent to being torsion-free. Since $R[f_0^{-1}] = \kappa$, being flat over $R$ is equivalent to being $f_0$-torsion free.\footnote{An $R$-module is flat if and only if it is $R$-torsion free if and only if $M \to M \otimes_R \Frac(R)$ is injective.} So it remains to show that the structure sheaf of $\ol{\{p\}} \times_{\hP} \hQ$ is $f_0$-torsion-free, or equivalently, $f$-torsion-free where $f = (f_0, f_1)$. Note that we already know the structure sheaf of $\hQ$ is $f$-torsion-free because $\sQ \to \sP$ is minimal, so the image of $f$ generates an effective Cartier divisor.

Write $A = R \times_{\kappa} \OO$ 
and choose an open affine $\Spec(B)$ of $\hQ$. So $B$ is an $f$-torsion free $A$-algebra, and our goal is to show that $R \otimes_A B$ is also $f$-torsion-free. Indeed, $A \to R$ is surjective so $B \to R \otimes_A B$ is surjective. So $R \otimes_A B$ has an $f$-torsion element if and only if there is some $b \in B \setminus \p B$ such that $fb \in \p B$. That is, $fb = gb'$ for some $g \in \p, b' \in B$. Now, $\p \subsetneq (f)$ so $g = fg'$ for some $g' \in A$. Since $\p$ is prime and $f \not\in \p$, we have $g' \in \p$. By assumption $B$ is $f$-torsion-free, so $fb = gb' = fg'b'$ implies $b = g'b' \in \p B$. But this contradicts the initial assumption that $b \not\in \p B$. So we deduce that $B \otimes_A R$ is $f$-torsion-free. 

Hence, $\ol{\{p\}} \times_{\hP} \hQ$ has $f$-torsion-free structure sheaf, so $\ol{\{p\}} \times_{\hP} \hQ \to \ol{\{p\}}$ is flat.
\end{proof}

In fact, abstract admisible blowups are also ``lighter'' than finite surjective morphisms. We do not use the following anywhere, but include it for interest.

\begin{prop} \label{prop:fsadmbulighter}
Let $\sX$ be a qcqs modulus pair, $f: \sV \to \sX$ an admissible blowup, and $g: \sW \to \sV$ a finite surjective ambient morphism. Then there exists a finite surjective ambient morphism $a: \sY \to \sX$ and an admissible blowup $b: \sW \to \sY$ such that $ab = fg$.
\[ \xymatrix{
\sW \ar[r]_g^{fin.surj.} \ar[d]_{adm.b.u.}^b & \sV \ar[d]_f^{adm.b.u.} \\
\sY \ar[r]_{fin.surj.}^a & \sX
} \]
\end{prop}

\begin{proof}
Consider the Stein factorisation $\hY := \ul{\Spec}((fg)_*\OO_{\hW})$, \cite[03H2]{stacks-project}, so that we obtain a commutative square
\[ \xymatrix{
\hW \ar[r]_g^{fin.surj.} \ar[d]_{proper}^b & \hV \ar[d]_f^{adm.b.u.} \\
\hY \ar[r]_{integral, surj.}^a & \hX
} \]
where $b$ is proper, $a$ is integral, and $\hY$ is the normalisation of $\hX$ in $\hW$. Note $a$ is surjective because $fg$ is surjective. Since $f$ is an isomorphism over $\iX$, the normalisation claim implies that $b$ is also an isomorphism over $\iX$. Moreover, the generators of $\mX$ are non-zero divisors in $\OO_{\hY} = (fg)_*\OO_{\hW}$, since their restrictions to all affine opens of $\hW$ are nonzero divisors. So if $a$ is finite type we are finished. If not,\footnote{I would like to see an example of a proper morphism towards an affine whose global sections is not finite!} take a finite sub-$\OO_{\hX}$-algebra $\OO_{\hT} \subseteq \OO_{\hY}$ which contains the generators of $\mX$, and such that $\OO_{\hT}|_{\iX} = \OO_{\hY}|_{\iX}$, \cite[Prop.6.9.14(i)]{EGAI}, and replace $\hY$ with $\hT = \ul{\Spec}(\OO_{\hT})$. Since $\OO_{\hT} \subseteq \OO_{\hY}$, the restriction $\mX|_{\hT}$ continues to be an effective Cartier divisor, and since $\hW \to \hX$ is proper and $\hT \to \hX$ separated (since it's affine), $\hW \to \hT$ is proper.
\end{proof}


\chapter{Local (pro)modulus pairs} \label{chap:locProMod}

In Section~\ref{sec:sites} we will define various sites of modulus pairs. An obvious question which arises is: Can we describe their fibre functors? By \cite[IV.6.8.7]{SGA41}, fibre functors correspond to pro-objects satisfying the following locality condition.

\begin{defi} \label{defi:tauLocal}
Suppose that $C$ is a category equipped with a class $τ$ of families of morphisms called ``coverings''. We say a pro-object $(P_λ)_{λ \in \Lambda}$ of $C$ is 
\emph{$τ$-local}%
\marginpar{\fbox{$τ$-local-pro-object}}%
\index{local@$τ$-local-pro-object} %
if for every covering family $\{U_i \to X\}_{i \in I}$ the morphism of sets
\[ 
\coprod_{i \in I} \varinjlim_{λ \in \Lambda} \hom_C(P_λ, U_i) \to 
\varinjlim_{λ \in \Lambda} \hom_C(P_λ, X)
\]
is surjective. 
\end{defi}

We can of course use this definition for constant pro-objects:

\begin{defi} \label{defi:tauLocalObject}
Suppose that $C$ is a category equipped with a class $τ$ of families of morphisms called ``coverings''. We say an object $P$ of $C$ is 
\emph{$τ$-local}%
\marginpar{\fbox{$τ$-local object}}%
\index{local@$τ$-local object} %
if for every covering family $\{U_i \to X\}_{i \in I}$ the morphism of sets
\[ 
\coprod_{i \in I}\hom_C(P, U_i) \to 
\hom_C(P, X)
\]
is surjective. 
\end{defi}

First we describe local modulus pairs in Section~\ref{sec:locModPair}. Then we will upgrade this to descriptions of pro-modulus pairs in Section~\ref{sec:locProMod}. 

\section{Local modulus pairs} \label{sec:locModPair}

\subsection{The Zariski case}

We begin with Zariski topology. 

\begin{prop} \label{prop:bigZarPoint}
A modulus pair $\sP \in \ulPSCH$ is $\ulPZar$-local and $\uS$-local%
, cf.~Def.\ref{defi:tauLocalObject}, Def.\ref{defi:coveringsMSchamb}, if and only if
\begin{enumerate}
 \item $\hP$ is the spectrum of a local ring,
 \item $\iP$ is the spectrum of a local ring,
 \item $\ol{\{p\}} \subseteq \hP$ is the spectrum of a valuation ring, where $p$ is the closed point of the local scheme $\iP$.
 \item $\ol{P} = \ol{\{p\}} \sqcup_{\{p\}} P^\circ$ in the category of schemes.
\end{enumerate}
\end{prop}

\begin{proof}\ 
 {($\Rightarrow 1$)}
 Any open affine covering $\{U_i \to \hP\}_{i \in I}$ induces a $\ulPZar$-covering. Since $\sP$ is $\ulPZar$-local, there is a section $\hP \to U_i \to \hP$ for some $i$. But since $U_i \to \hP$ is an open immersion, this implies $U_i = \hP$. That is, $\hP$ is affine. Let $A = \Gamma(\hP, \OO_{\hP})$. We want to show that for all $g \in A$, we have $g \in A^*$ or $1{-}g \in A^*$. For any $g \in A$, the morphisms $A \to A[g^{-1}]$ and $A \to A[(1{-}g)^{-1}]$ induce a $\ulPZar$-covering.\footnote{Since for any ideal $\p$ if $g \in \p$ and $1{-}g \in \p$ then $1 \in \p$ so $\p$ is not a prime ideal.} As in the first part, this implies that $A = A[g^{-1}]$ or $A = A[(1{-}g)^{-1}]$. Hence, $g \in A^*$ or $1{-}g \in A^*$.

 {($\Rightarrow 2$)}
  Set $A = \Gamma(\hP, \OO_{\hP})$ and $\mP = (f)$ so $A[f^{-1}] = \Gamma(\iP, \OO_{\iP})$. For a scheme $X$ and global section $g \in \Gamma(X, \OO_X)$ we will write $X[g^{-1}]$ for $\underline{\Spec} \left (\OO_{X}[g^{-1}]\right )$.
 
We will show by contradiction that there do not exists two distinct closed points $p_1, p_2$ of $\iP$. Indeed, choose subsets $\{g_{1i}\}_{i \in I}$, $\{g_{2i}\}_{i \in I}$ of $A$ such that for $\epsilon = 1, 2$, the $\{\iP[g_{\epsilon i}^{-1}]\}_{i \in I}$ form an open cover of $\iP \setminus \{p_\epsilon\}$. As $\iP$ is quasi-compact, there are finite subfamilies $\{g_{\epsilon 1}, \dots, g_{\epsilon n}\} \subseteq \{g_{\epsilon i}\}_{i \in I}$ which still have the property that the $\{\iP[g_{\epsilon 1}^{-1}], \dots, \iP[g_{\epsilon n}^{-1}]\}$ form a cover of $\iP \setminus \{p_\epsilon\}$.
 
Now consider the blowup $Bl_{A} I \to \Spec(A)$ of $\Spec(A)$ at the ideal $I = \langle g_{\epsilon i} \rangle_{\epsilon \in  \{1, 2\}, i \in \{1, \dots, n\}}$ equipped with its canonical open covering $\hU_{\epsilon i} = \Spec (A[\tfrac{I}{g_{\epsilon i}}])$.\footnote{Here $A[\tfrac{I}{g_{\epsilon i}}]$ is the stacks project's notation for $\varinjlim (A \stackrel{g_{\epsilon i}}{\to} I \stackrel{g_{\epsilon i}}{\to} I^2 \dots )$ where transition maps are multiplication by $g_{\epsilon i}$.} Since $I$ becomes the unit ideal on $\iP$, the blowup is an isomorphism over $\iP$. So when equipped with the minimal modulus structures, the family $\{\hU_{\epsilon, i} \to \hP\}$ is the composition of a $\ulPZar$-covering and an \abb. Since $\sP$ is local for these two classes, there exists a factorisation $\hP \to \hU_{\epsilon i} \to \hP$ for some $\epsilon$, $i$, and pulling back, a factorisation $\iP \to \iP[g_{\epsilon i}^{-1}] \to \iP$. Since $p_{3-\epsilon}$ is not in $\iP[g_{\epsilon i}^{-1}]$ we have a contradiction.

 {($\Rightarrow 3$)}
  Let $\p \subseteq \Gamma(\ol{P}, \OO_{\ol{P}})$ be the prime ideal of the closed point $p$ of $P^\circ$, and choose $f, g \in \Gamma(\ol{P}, \OO_{\ol{P}}) \setminus \p$. The goal is to find $h$ such that $fh = g$ or $f = gh$ (mod $\p$). Since $f, g \notin \p$, the closed subscheme of $\hP$ associated to $\langle f, g \rangle$ does not intersect $P^\circ$. Setting $I:=\langle f,g \rangle$ and $A:=\Gamma(\hP,\OO_{\hP})$, we have the standard open covering %
 $\{%
 \Spec A[\tfrac{I}{f}], %
 \Spec A[\tfrac{I}{g}] %
 \}$ %
 of the blowup $Bl_{A}I$. Equipping these with the minimal moduli, we obtain a composition of a $\ulPZar$-covering and an \aab. Since $\sP$ is local for these two classes, we get a factorisation $\hP \to \hU \to \hP$ where $\hU$ is $\Spec A[\tfrac{I}{f}]$ or $\Spec A[\tfrac{I}{g}]$. Taking $h$ to be the image of $\tfrac{g}{f}$ (resp. $\tfrac{f}{g}$), we obtain $fh = g$ (resp. $f = gh$).

 {($\Rightarrow 4$)}
  We want to show that the canonical morphism $A \to A/\p \times_{k(\p)} A_\p$ is an isomorphism where $A = \Gamma(\ol{P}, \OO_{\ol{P}})$, and $\p$ is the prime corresponding to the closed point of $P^\circ$, \cite[Thm.5.1]{Fer03}. For this, it suffices to show that the induced map $\p \to \p A_\p$ is an isomorphism, cf.~\cite[Lem.3.12]{HK18}. Since $\mP$ is an effective Cartier divisor, $A \to A_\p$ is injective, so $\p \to \p A_\p$ is injective. To show surjectivity, choose some $\tfrac{a}{s} \in \p A_\p$, that is, $a \in \p, s \in A \setminus \p$. Since $s$ is invertible on $\iP$, the ideal $\langle s \rangle$, and therefore $\langle a, s \rangle$, produces an admissible blowup of $\sP$. It has standard open covering 
  $\{%
 \Spec A[\tfrac{\langle a, s \rangle}{a}], %
 \Spec A[\tfrac{\langle a, s \rangle}{s}] %
 \}$ %
and since $\sP$ is $\ulPZar$-local and $\uS$-local, we obtain a factorisation and we deduce that there is $b \in A$ with either $ab = s$ or $a = bs$. But in the former case, $ab = s \in \p$, contradicting the initial assumption, so $a = bs$. That is, $\tfrac{a}{s} = b \in A$. We deduce that $b \in \p$ from injectivity of $A/\p \to k(\p)$. So $\p \to \p A_\p$ is surjective.

 {($\Leftarrow)$}
  Suppose that $\sP$ is a modulus pair satisfying (1)-(4). Any $\ulPZar$-covering admits a section by (1) so consider an {\abb} $\sQ \to \sP$. As $\iQ = \iP$ and $p \in \iP$, we obtain a commutative square
 \[ \xymatrix{
\{p\} \ar[r] \ar[d] & \hQ \ar[d] \\
\ol{\{p\}} \ar[r] \ar@{-->}[ur] & \hP
 } \]
and by the valuative criterion for properness, there is a unique diagonal morphism making the diagram commutative. Then by (4) we can glue this with the factorisation $\iP \to \hQ \to \hP$ to obtain a lifting $\hP \to \hQ$. Finally, $\hP \to \hQ$ is admissible because $\id: \hP \to \hP$ and $\hQ \to \hP$ are admissible. 
\end{proof}

\subsection{The Nisnevich case}

\begin{prop} \label{prop:bigNisPoint}
A modulus pair $\sP \in \ulPSCH$ is $\ulPNis$-local and $\uS$-local%
, cf.~Def.\ref{defi:tauLocalObject}, Def.\ref{defi:coveringsMSchamb}, if and only if
\begin{enumerate}
 \item $\hP$ is the spectrum of a henselian local ring,
 \item $\iP$ is the spectrum of a henselian local ring,
 \item $\ol{\{p\}} \subseteq \hP$ is the spectrum of a henselian valuation ring, where $p$ is the closed point of the local scheme $\iP$.
 \item $\ol{P} = \ol{\{p\}} \sqcup_{\{p\}} P^\circ$ in the category of schemes.
\end{enumerate}
\end{prop}

\begin{proof}
($\Rightarrow$) By Prop.~\ref{prop:bigZarPoint} it suffices to show that $\ol{P}$, $P^\circ$, and $\ol{\{p\}}$ are henselian. Closed subschemes of henselian schemes are henselian, \cite[Tag 04GG(1)$\Leftrightarrow$(9)]{stacks-project}, so $\ol{\{p\}}$ follows from the $\hP$ case.

To show that $\hP$ is henselian it suffices to show that every Nisnevich covering admits a section, cf.~\cite[Tag 04GG(1)$\Leftrightarrow$(7)]{stacks-project}. Since $\ulPNis$-coverings of $\sP$ correspond to $\Nis$-coverings of $\hP$, this follows directly from the assumption that $\sP$ is $\ulPNis$-local.

Showing $P^\circ$ is henselian is a little more involved. We will show: if $Q^\circ \to P^\circ$ is a finite morphism with $Q^\circ$ connected then $Q^\circ$ is local. Or rather that, what amounts the same thing by the going up property, there is a unique point of $Q^\circ$ lying over the closed point of $P^\circ$, cf.~\cite[Tag 04GG(1)$\Leftrightarrow$(9)]{stacks-project}. We will prove this by extending $Q^\circ \to P^\circ$ to a finite morphism $\hQ \to \hP$ with $\hQ$ connected, and deducing $\iQ$ is local from $\hQ$ being local.

The proof is algebraic so let $A = \Gamma(\ol{P}, \OO_{\ol{P}})$, let $\p$ be the prime corresponding to the closed point of $P^\circ$ so $P^\circ = \Spec(A_\p)$ and let $A_\p \to B_0$ be a finite morphism with $Q^\circ = \Spec(B_0)$ connected. Our goal is to show that $B_0$ has a unique prime lying over $\p$.

Let $\q_1, \dots, \q_n$ be the primes of $B_0$ lying over $\p$ (which are incidentally, precisely the maximal ideals of $B_0$ by the going up theorem for finite morphisms) and let $B := B_0 \times_{\prod k(\q_i)} \prod_{i = 1}^n A/\p \subseteq B_0$, equipped with its canonical morphism $A \to B$. One calculates directly that $B / \p B = \prod_{i = 1}^n A / \p$ and $A_\p \subseteq B \otimes_A A_\p \subseteq B_0$, so $A \to B$ is a finite morphism, \cite[Thm.2.2]{Fer03}. On the other hand, $\Spec(B)$ has the amalgamated sum topological space $\Spec(B_0) \amalg_{\{q_1, \dots, q_n\}} (\sqcup_{i = 1}^n \Spec(A/\p))$ \cite[Scl.4.3, Thm.5.1]{Fer03}, so it is connected since $\iQ$ is. It follows that $B$ has a unique maximal ideal, as $A$ is a hensel local ring. However, studying the amalgamated sum we see that the closed points of $\Spec(B)$ are in canonical bijection with the closed points $\{q_1, \dots, q_n\}$ of $\Spec(B_0)$, since the primes of the valuation ring $A/\p$ are totally ordered. Hence, $n = 1$.

($\Leftarrow$) Noting that Conditions (1)-(4) are stronger than the assumptions in Proposition \ref{prop:bigZarPoint}, $\uS$-locality follows from that proposition. 
For $\ulPNis$-locality, suppose that $\sP \to \sX$ is a morphism, and $\{\sU_i \to \sX\}_{i \in I}$ is a $\ulPNis$-covering. Since $\ambtimes$ preserves $\ulPNis$-coverings, Prop.~\ref{prop:coveringAmbtimes}, in order to show that $\sP \to \sX$ lifts through some $\sU_i$, it suffices to consider the case $\sX = \sP$. But $\ulPNis$-coverings of $\sP$ correspond to $\Nis$-coverings of $\hP$, so since all morphisms in play are minimal, the result follows from the classical version, cf.~\cite[Tag 04GG(1)$\Leftrightarrow$(7)]{stacks-project}.
\end{proof}

\subsection{The {\'e}tale case}

\begin{coro} \label{coro:bigEtPoint}
A modulus pair $\sP \in \ulPSCH$ is $\ulPet$-local and $\uS$-local%
, cf.~Def.\ref{defi:tauLocalObject}, Def.\ref{defi:coveringsMSchamb}, if and only if
\begin{enumerate}
 \item $\hP$ is the spectrum of a strictly henselian local ring,
 \item $\iP$ is the spectrum of a henselian local ring,
 \item $\ol{\{p\}} \subseteq \hP$ is the spectrum of a strictly henselian valuation ring, where $p$ is the closed point of the local scheme $\iP$.
 \item $\ol{P} = \ol{\{p\}} \sqcup_{\{p\}} P^\circ$ in the category of schemes.
\end{enumerate}
\end{coro}

\begin{proof}
($\Rightarrow$) By Prop.~\ref{prop:bigNisPoint}, it suffices to show that the closed point of $\hP$ is separably closed. For this, it suffices to show that every finite étale morphism $q_0 \to p_0$ towards the closed point $p_0$ admits a section. But $\hP$ is a hensel local ring, so the category of finite étale morphisms over $p_0$ is equivalent to the category of finite étale morphisms over $\hP$, \cite[Tag 04GK]{stacks-project}, so it suffices to show any finite étale morphism $\hQ \to \hP$ has a section. Suppose $\hQ \to \hP$ is a finite étale morphism with $\hQ$ connected. Then $(\hQ, \mP|_{\hQ}) \to \sP$ is an $\ulPet$-covering, and so it has a section.

($\Leftarrow$) The proof of Prop.~\ref{prop:bigNisPoint}($\Leftarrow$) works verbatim.
\end{proof}

\begin{exam}
In the above corollary, one might expect that $P^\circ$ be strictly local, but this is not necessarily the case: $(\Spec \C[[t]], t)$ is $\ulMet$-local, for example.
\end{exam}

\subsection{The $\protect \fppf$ case}

Recall that an \emph{$\fppf$-covering} %
\marginpar{\fbox{$\fppf$-covering}} %
\index{covering!$\fppf$-covering} %
 is a family of morphisms of schemes $\{f_i: U_i \to X\}_{i \in I}$ such that each $f_i$ is flat and locally of finite presentation and $\cup_i f_i(U_i) = X$. 

\begin{lemm} \label{lemm:fppfRefine}
Let $X$ be a scheme. Every $\fppf$-covering $\{W_j \to X\}$ is refinable by one of the form
\[ \{ V_i \to U_i \to X\}_{i \in I} \]
where $\{U_i \to X\}_{i \in I}$ is a Nisnevich covering, and $V_i \to U_i$ are finite flat and surjective.
\end{lemm}

\begin{proof}
If $X$ is the spectrum of a hensel local ring, then by \cite[Cor. 17.16.2]{EGAIV4} the morphism $\amalg W_j \to X$ is refinable by a morphism $V \to X$ which is finite presentation, flat, and quasi-finite. Then since $X$ is the spectrum of a hensel local ring, $V = V_0 \amalg V_1$ where $V_0 \to X$ is finite surjective, and $V_1 \to X$ does not hit the closed point, \cite[Thm. 18.5.11(c)]{EGAIV4}. Since $V_0$ is a disjoint union of local rings, \cite[Thm. 18.5.11(a)]{EGAIV4}, and satisfies the going down theorem by virtue of being flat, we find a finite flat surjective morphism $V' \to X$ from a local scheme $V'$ which factors through one of the $W_j$ (because $V'$ is connected).

If $X$ is affine (or rather qcqs) then for each $x \in X$ consider the filtered system of {\'e}tale morphisms $(U_{x,λ})_{λ \in \Lambda}$ defining $\Spec(\OO_{X, x}^h)$. Pulling $\{W_j \to X\}$ back to $\Spec(\OO_{X, x}^h)$ we reduce to the previous case, and find a finite flat surjective morphism $V' \to \Spec(\OO_{X, x}^h)$ factoring through one of the $W_j$. This refinement lifts to some $V_{x,λ} \to U_{x, λ}$ by \cite[Thm.8.8.2]{EGAIV3}. The morphism $V_{x,λ} \to U_{x, λ}$ can be assumed to be surjective, finite \cite[Thm.8.10.5(vi,x)]{EGAIV3}, and flat \cite[04AI]{stacks-project}. Combining these we obtain a covering $\{V_{λ, x} \to U_{λ, x} \to X\}_{x \in X}$ as in the statement.

In general, pull $\{W_j \to X\}$ back along an open affine covering of $X$, apply the previous step, and then combine the Zariski and Nisnevich coverings.
\end{proof}

\begin{lemm} \label{lemm:fppfAlgClRes}
If $A$ is an $\fppf$-local ring, then every residue field is algebraically closed.
\end{lemm}

\begin{proof}
Let $\p$ be a prime of $A$, and $f(T) = \sum_{i = 0}^n a_iT^i \in k(\p)[T]$ a monic. We can write it as $\sum_{i = 0}^n \tfrac{b_i}{s^{n-i}} T^i$ for some $b_i, s \in A/\p$ with $s \neq 0$. Note $f(T)$ has a solution in $k(\p)$ if and only if $s^nf(T) = \sum b_i (sT)^i$ has a solution in $k(\p)$, if and only if the monic $\sum b_i X^i \in (A/\p)[X]$ has a solution in $k(\p)$. The monic $\sum b_i X^i$ lifts to a monic in $A[T]$ which has a solution in $A$ by virtue of $A$ being $\fppf$-local. 
\end{proof}

\begin{prop} \label{prop:bigFppfPoint}
A modulus pair $\sP \in \ulPSCH$ is $\ulPfppf$-local and $\uS$-local%
, cf.~Def.\ref{defi:tauLocalObject}, Def.\ref{defi:coveringsMSchamb}, if and only if
\begin{enumerate}
 \item $\hP$ is an $\fppf$-local scheme,
 \item $\iP$ is an $\fppf$-local scheme,
 \item $\ol{\{p\}} \subseteq \hP$ is the spectrum of a valuation ring with algebraically closed function field (and therefore is an $\fppf$-local scheme).
 \item $\ol{P} = \ol{\{p\}} \sqcup_{\{p\}} P^\circ$ in the category of schemes.
\end{enumerate}
\end{prop}

\begin{proof}
($\Rightarrow$) By the Zariski version, Prop.~\ref{prop:bigZarPoint}, it suffices to show the $\fppf$-local claims, (1), (2), and (3). %
For (1), if $\{\hU_i \to \hX\}_{i \in I}$ is an $\fppf$-covering and $\hP \to \hX$ a morphism, then there is a factorisation $\sP \to (\hU_i, \varnothing) \to (\hX, \varnothing)$ for some $i$ by $\ulPfppf$-locality, hence, $\hP$ is $\fppf$-local.

For (3), by the Zariski version, Prop.~\ref{prop:bigZarPoint}, we know that it is the spectrum of a valuation ring, so it suffices to show that $k(p)$ is algebraically closed, but this follows from (1), since residue fields of $\fppf$-local schemes are algebraically closed, Lem.~\ref{lemm:fppfAlgClRes}.

For (2), by the Nisnevich version, Prop.~\ref{prop:bigNisPoint}, we know that $P^\circ$ is a local hensel scheme. So by the refinement of Lemma~\ref{lemm:fppfRefine} it suffices to show that every finite flat surjective morphism $Q^\circ \to P^\circ$ admits a section. Let $Q^{ic}$ be the integral closure of $\ol{\{p\}}$ in $q := p \times_{P^\circ} Q^\circ$.\footnote{$q$ is not necessarily a single point, but we find the notation convenient and readable.} %
Since $q \to p$ is finite, there is some finite sub extension $Q^{ic} \to \breve{Q} \to \ol{\{p\}}$ such that $q \to \breve{Q}$ is an isomorphism over $p$ (i.e., we have taken a ``pseudo-integral closure'' in the sense of \cite{Kel19}). 

In fact, $\breve{Q} \to \overline{\{p\}}$ is finite, flat and surjective: Since $\Gamma(\breve{Q}, \OO_{\breve{Q}}) \to \Gamma(q, \OO_{q})$ is injective (by definition), and $\Gamma(\ol{\{p\}}, \OO_{\ol{\{p\}}})$ is a valuation ring, $\breve{Q} \to \ol{\{p\}}$ is flat, since being flat is equivalent to being torsion free for valuation rings. It's also surjective since it is finite (so satisfies going up) and its generic fibre is non-empty. 

Now by (4) and \cite[Thm.2.2(iv)]{Fer03} we can glue our two finite flat surjective morphisms $\breve{Q} \to \ol{\{p\}}$ and $\iQ \to \iP$ along $q$ to obtain a globally finite flat surjective morphism $\ol{Q} \to \ol{P}$ which is $Q^\circ$ over $P^\circ$. Since $\ol{P}$ is $\fppf$-local, this global morphism has a section, and therefore $Q^\circ \to P^\circ$ has a section.

($\Leftarrow$) By Prop.~\ref{prop:bigZarPoint} we know that any $\sP$ satisfying (1)-(4) is $\uS$-local, so it suffice to show $\ulPfppf$-locality. Let $\sP \to \sX$ be a morphism and $\{\sU_i \to \sX\}_{i \in I}$ a $\ulPfppf$-covering. We want to show that $\sP \to \sX$ factors through some $\sU_i$. Since $\ambtimes$ preserves $\ulPfppf$-coverings, Prop.~\ref{prop:coveringAmbtimes}, we can assume $\sX = \sP$. But $\ulPfppf$-coverings of $\sP$ correspond to $\fppf$-coverings of $\hP$, so since all morphisms in play are minimal, the result follows from assumption (1).
\end{proof}

\subsection{The ${\protect \rqfh}$ case}

Recall that a ring $A$ is an \emph{absolutely integrally closed domain} if it is integral and every monic $f(T) \in A[T]$ has a solution in $A$. Equivalently, $A$ is an absolutely integrally closed domain if it is integral, normal, and the fraction field is algebraically closed.

\begin{lemm} \label{lemm:LaicdLocalSchemes}
Let $P$ be a scheme. Then $P$ is $\rqfh$-local if and only if it is the spectrum of a local absolutely integrally closed domain.
\end{lemm}

\begin{proof}
Suppose $P$ is $\rqfh$-local.

\textit{$P$ is local.} $P$ is affine because any open affine covering of $P$ admits a section. $P$ has a unique closed point because given two closed points $p_0, p_1$, if $p_0 \neq p_1$, the family $\{P \setminus \{p_0\}, P \setminus \{p_1\}\}$ is an open covering of $P$. But this covering has a section which is impossible.

Let $P = \Spec(A)$.

\textit{$A$ is integral.} Given $a, b\in A$ such that $ab = 0$ we get an $\rqfh$-covering $\{\Spec(A/a) \to \Spec(A), \Spec(A/b) \to \Spec(A)\}$. Since $P$ is $\rqfh$-local, this covering has a section, so $a = 0$ or $b = 0$.

\textit{$A$ is a.i.c.} Suppose that $f(T) \in A[T]$ is a monic. The morphism $\Spec(A[T]/f) \to \Spec(A)$ is a $\rqfh$-covering, and therefore has a section, that is, $f(T)$ has a solution.

Conversely, let $A$ be a local absolutely integrally closed domain, and $P = \Spec(A)$. To show that $P$ is $\rqfh$-local, it suffices to show that every Zariski covering admits a section, and every finite covering admits a section. The Zariski case follows from locality, so let $\{Y_i \to P\}_{i \in I}$ be a jointly surjective family of finite morphisms.
Since the faction field of $A$ is algebraically closed, there is a section $\eta := \Spec(\Frac(A)) \to Y_i \to P$ for some $i$. Let $Z \subseteq Y_i$ be the scheme theoretic image of $\eta$. In particular, $Z$ is integral, and has the same fraction field as $P$. But $P$ is normal and $Z \to P$ is finite, so $Z = P$ by \cite[Tag~0AB1]{stacks-project}. So we have a factorisation $P = Z \to Y_i \to P$.
\end{proof}

\begin{prop} \label{prop:bigLaicdPoint}
A modulus pair $\sP \in \ulPSCH$ is $\ulPrqfh$-local and $\uS$-local%
, cf.~Def.\ref{defi:tauLocalObject}, Def.\ref{defi:coveringsMSchamb}, if and only if
\begin{enumerate}
 \item $\hP$ is the spectrum of a local absolutely integrally closed domain,
 \item $\iP$ is the spectrum of a local absolutely integrally closed domain,
 \item $\ol{\{p\}} \subseteq \hP$ is the spectrum of a valuation ring with algebraically closed function field (and therefore is an $\fppf$-local scheme).
 \item $\ol{P} = \ol{\{p\}} \sqcup_{\{p\}} P^\circ$ in the category of schemes.
\end{enumerate}
\end{prop}

\begin{proof}
($\Rightarrow$) By the Zariski version, Prop.~\ref{prop:bigZarPoint}, it suffices to show the claims about local absolutely integrally closed domains. Since this class of rings is closed under quotient by primes, and localisation at a prime, it suffices to consider $\ol{P}$. By the corresponding result for the $\rqfh$-topology on the category of schemes, 
Lem.~\ref{lemm:LaicdLocalSchemes}, it suffices to show that $\hP$ is $\rqfh$-local. For this, it suffices to show every $\rqfh$-covering $\{\hU_i \to \hP\}_{i \in I}$ admits a section. Up to refinement, by Lem.~\ref{lemm:schRqfhGivesmpRqfh} we can assume that $\{(\hU_i, \mP|_{\hU_i}) \to \sP\}_{i \in I}$ is is a $\ulPrqfh$-covering. Then the section comes from $\ulPrqfh$-locality of $\sP$.

($\Leftarrow$) By Prop.~\ref{prop:bigZarPoint} we know that any $\sP$ satisfying (1)-(4) is $\uS$-local, so it suffice to show $\ulPrqfh$-locality. Let $\sP \to \sX$ be a morphism and $\{\sU_i \to \sX\}_{i \in I}$ a $\ulPrqfh$-covering. We want to show that $\sP \to \sX$ factors through some $\sU_i$. Since $\ambtimes$ preserves $\ulPrqfh$-coverings, Prop.~\ref{prop:coveringAmbtimes}, we can assume $\sX = \sP$. But then $\{\hU_i \to \hP\}_{i \in I}$ a $\rqfh$-covering, so has a section by the assumption (1). Since the $\sU_i \to \sP$ are minimal, this section is admissible, and we get a section of the original family.
\end{proof}

The following lemma is an analogue of the decomposition of finite algebras over henselian local rings. 

\begin{lemm} \label{lemm:finiteulMNislocal}
Let $\tau$ be $\Nis$ or $\et$.
Suppose $\sP$ is a $\ulPtau$-local and $\uS$-local modulus pair, and $\sQ \to \sP$ is an ambient minimal morphism such that $Q^\circ \to P^\circ$ is finite and $\ol{Q} \to \ol{P}$ is proper. Then in fact $\ol{Q} \to \ol{P}$ is finite and there is an abstract admissible blowup 
\[ \sqcup_{i = 1}^n \sQ_{i} \stackrel{}{\to} \sQ \]
such that each $\sQ_{i}$ represents an $\ulMtau$-local modulus pair in $\ulMSCH$. In other words, up to isomorphism in $\ulMSCH$, the pair $\sQ$ is a finite sum of $\ulMtau$-local pairs.
\end{lemm}

\begin{proof}
First note that replacing $\sQ$ by abstract admissible blowup $\sQ' \to \sQ$ does not change any of the hypotheses. 
This observation shows that we can assume $\iQ$ is local: Since $\iQ \to \iP$ is finite over a hensel local ring, it is a sum of hensel local schemes. Then by Prop.~\ref{prop:finiteCoprod}, up to admissible blowup, we can assume that $\hQ$ is a sum whose components are in bijection with the components of $\iQ$. In this way, we reduce to the case $\iQ$ is local.
So, we are reduced to showing that $\sQ$ is a $\ulMtau$-local modulus pair. 
For this, it suffices to show that every $\ulMtau$-covering $\sV \to \sQ$ admits a section in $\ulMSCH$ by Prop.~\ref{prop:descendFlatCoverings}. 
Up to refinement and abstract admissible blowup $\sQ' \to \sQ$, any $\ulMtau$-covering is represented by $\ulPtau$-coverings by Prop.~\ref{prop:heavyRefine} and Thm.~\ref{thm:etaleHeavy}\eqref{thm:etaleHeavy:Nis}.
So, we may assume that $\sV \to \sQ$ is a $\ulPtau$-covering, and it suffices to show that $\sV \to \sQ$ admits a section in $\ulPSCH$.
By Lem.~\ref{lemm:ulMZarambtimesFin}, the morphism $\hQ \to \hP$ is finite.
If $\tau = \Nis$ (resp. $\et$), then $\hP$ is henselian local (resp. strictly henselian local) by Prop.\ref{prop:bigNisPoint} (resp. Cor.\ref{coro:bigEtPoint}).
Therefore, $\hQ$ is a finite disjoint union of henselian local (resp. strictly henselian local) schemes \cite[Tag 04GE(1)$\Leftrightarrow$(10)]{stacks-project}.
Hence, any $\Nis$-covering $\hV \to \hQ$ has a section, so any $\ulPNis$-covering $\sV \to \sQ$ has a section.
\end{proof}

\section{Local pro-modulus pairs} \label{sec:locProMod}

Now we upgrade the results from Section~\ref{sec:locProMod} to ``nice'' pro-objects. 

\begin{prop} \label{prop:prouSLocLim}
Let $\sS$ be a qcqs modulus pair. 
Let $(\sP_λ)_λ \in \Pro(\ulPSCH^\qcqs)$ be a pro-object with affine minimal transition morphisms. %
Then $(\sP_λ)_λ$ is $\uS$-local if and only if ${\amblim} \sP_λ$ is $\uS$-local.
\end{prop}

\begin{proof}
Set $\sP = \amblim \sP_λ$.

\textit{$(\Rightarrow)$.} %
First we will show that $(\sP_λ)_λ$ is $\uS$-local implies $\sP$ is $\uS$-local. %
Since $\ambtimes$ preserves $\uS$, Lem.\ref{lemm:ambtimesAab} it suffices to show that any \abb{} $\sY \to \sP$ admits a section. %
By Prop.~\ref{prop:descendAdmBlowup}, for some $λ$ there exists an \abb{} $\sY_λ \to \sP_λ$ such that $\sY = \sP \ambtimes[\sP_λ] \sY_λ$. Then since $(\sP_λ)_λ$ is $\uS$-local, we get a factorisation $\sP_μ \to \sY \to \sP_μ$ for some $μ \geq λ$, %
and hence, the desired section.

\textit{$(\Leftarrow)$.} Now assume $\sP$ is $\uS$-local. We will show that $(\sP_λ)_λ$ is $\uS$-local. %
As above, since $\ambtimes$ preserves \abb{}s, it suffices to show that given an \abb{} $\sY \to \sP_λ$ for some $λ$, there exists a factorisation $\sP_μ \to \sY \to \sP_λ$ for some $μ \geq λ$. %
Since $\sP$ is $\uS$-local, we get a factorisation $\sP \to \sY \to \sP_λ$. %
Choose a factorisation $\hY \to \hN \to \hP_λ$ where $\hY \to \hN$ is a closed immersion which is an isomorphism over $\iP_λ$ and $\hN \to \hP_λ$ is a proper morphism of finite presentation, Prop.~\ref{prop:cannotLimitBlowups}. Notice that we get a factorisation
\[ \iP_λ \to \hN \to \hP_λ, \]
and since $\mY$ is an effective Cartier divisor, $\OO_{\hY} \to \OO_{\iY} = \OO_{\iP_λ}$ is injective, so the ideal defining the closed immersion $\hY \to \hN$ is 
\[ \sI := ker(\OO_{\hN} \to \OO_{\iP_λ}). \]
Now since $\hN \to \hP_\lambda$ is of finite presentation, we can use \cite[Prop.8.13.1]{EGAIV3} to descend $\lim \hP_λ {\to} \hN {\to} \hP_λ$ to some $\hP_μ {\to} \hN {\to} \hP_λ$ for some $μ \geq λ$. Since all morphisms in question are minimal, to finish it suffices to show that this factors through the closed subscheme $\hP_μ \to \hY \subseteq \hN$, or equivalently, that $\sI$ is sent to zero in $\OO_{\hP_μ}$. Since $\OO_{\hP_μ} \to \OO_{\iP_μ}$ is injective, it suffices to show that $\sI$ is sent to zero in $\OO_{\iP_μ}$, but this follows from the factorisation $\iP_μ \to \iP_λ \to \hN$.
\end{proof}

\begin{prop} \label{prop:proFlatLocLim}
Let $\sS$ be a qcqs modulus pair. 
Let $(\sQ_λ)_λ \in \Pro(\ulPSCH^\qcqs)$ be a pro-object with affine minimal transition morphisms. %
Suppose $\tau$ is one of $\Zar$, $\Nis$, $\et$, $\fppf$. 

Then $(\sQ_λ)_λ$ is $\ulPtau$-local if and only if ${\amblim} \sQ_λ$ is $\ulPtau$-local.
\end{prop}

\begin{proof}
Set $\sP = \amblim \sQ_λ$.

\textit{$(\Leftarrow)$} %
Suppose that $\sP$ is $\ulPtau$-local and %
let $\{\hU_{i} \to \hQ_λ\}_{i \in I}$ be a $\ulPtau$-covering for some $λ$. %
Since $\sP$ is $\ulPtau$-local, there is a factorisation $\hP \to \hU_i \to \hQ_λ$ for some $i$. Then by \cite[Prop.8.13.1]{EGAIV3}, since $\hU_i \to \hQ_\lambda$ is of finite presentation, this descends to a factorisation $\hQ_μ \to \hU_i \to \hQ_λ$ for some $μ \geq λ$.

\textit{$(\Rightarrow)$} %
Suppose that $\sQ = (Q_\lambda)_\lambda$ is $\ulPtau$-local and %
let $\{\hU_i \to \hP\}_{i \in I}$ be a $\ulPtau$-covering. %
By Prop.~\ref{prop:descendFlatCoverings}, up to taking a subcovering, we can assume our covering is pulled back from a covering of some $\hQ_λ$. Since $\sQ$ is $\ulPtau$-local, up to changing $λ$, this descended covering has a section. Therefore the original covering also has a section.
\end{proof}

\begin{prop} \label{prop:prorqfhLocLim}
Let $\sS$ be a qcqs modulus pair. 
Let $(\sQ_λ)_λ \in \Pro(\ulPSCH^\qcqs)$ be a pro-object with affine minimal schematically dominant transition morphisms. 

Then $(\sQ_λ)_λ$ is $\ulPrqfh$-local if and only if ${\amblim} \sQ_λ$ is $\ulPrqfh$-local.
\end{prop}

\begin{proof}
Set $\sP = \amblim \sQ_λ$.

\textit{$(\Rightarrow)$} %
Suppose $(\sQ_λ)_λ$ is $\ulPrqfh$-local. By Lem.~\ref{lemm:LaicdLocalSchemes} it suffices to show that $\hP$ is the spectrum of a local absolutely integrally closed domain. By Prop.~\ref{prop:proFlatLocLim} we know that $\sP$ is $\ulPZar$-local and $\ulPfppf$-local.\footnote{Since every $\fppf$-covering is a $\rqfh$-covering, cf.~Rem.\ref{rema:fppfQfh}} 
In particular, it suffices to show that $\hP$ is integral ($\fppf$-locality implies that every monic in $\Gamma(\hP, \OO_{\hP})$ has a solution). Suppose we have $a, b \in \Gamma(\hP, \OO_{\hP})$ such that $ab = 0$. Descend these to some $a_λ, b_λ \in \Gamma(\hQ_λ, \OO_{\hQ_λ})$ such that $a_λb_λ = 0$. Since $a_λb_λ = 0$, the family $\{V(a_λ) \to \hQ_λ, V(b_λ) \to \hQ_λ\}$ of closed immersions is jointly surjective. Up to replacing $V(a_λ)$ and $V(b_\lambda)$ with the scheme theoretic closure of the preimage of $\iQ_λ$, this induces a $\ulPrqfh$-covering of $\sQ_λ$. Therefore since $\sQ$ is $\ulPrqfh$-local, $\sQ_μ \to \sQ_λ$ factors through one of its terms for some $μ \geq λ$. Therefore $a_λ|_{\hP_μ} = 0$ or $b_λ|_{\hP_μ} = 0$, and consequently $a = 0$ or $b = 0$.

\textit{$(\Leftarrow)$} %
Suppose that $\sP$ is $\ulPrqfh$-local and %
let $\{\hU_{i} \to \hQ_λ\}_{i \in I}$ be a $\ulPrqfh$-covering for some $λ$. %
Since $\sP$ is $\ulPrqfh$-local, there is a factorisation $\hP \to \hU_i \to \hQ_λ$ for some $i$. Then, since $\hQ_\mu \to \hQ_\lambda$ are scheme theoretically dominant by assumption, this descends by 
Prop.~\ref{prop:qcqs-limcolim} to a factorisation $\hQ_μ \to \hU_i \to \hQ_λ$ for some $μ \geq λ$.
\end{proof}

Now we show that the hypotheses above 
that the pro-objects have affine (resp. schematically dominant) transition morphisms are very often satisfied.

\begin{lemm} \label{lemm:ZarLocIsAff}
Let $\sS$ be a modulus pair and let $\sC$ be a full subcategory of $\ulPSCH_\sS$ satisfying:
\begin{enumerate}
 \item[(Op)] For every minimal morphism $\sU \to \sX$ such that $\sX \in \sC$ and $\hU \to \hX$ is an open immersion, we have $\sU \in \sC$.
\end{enumerate}
Then every $\ulPZar$-local object of $\Pro(\sC)$ is in the image of 
\[ \Pro(\sA) \to \Pro(\sC) \]
where $\sA \subseteq \sC$ is the full subcategory of objects $\sX$ such that $\hX$ is affine. In particular, every $\ulPZar$-local object of $\Pro(\sC)$ is isomorphic to one with affine transition morphisms.
\end{lemm}

\begin{proof}
Let $(\sQ_\lambda)_\lambda$ be an $\ulPZar$-local object of $\sC$. It suffices to show that $(\sQ_\lambda)_\lambda / \sA$ is initial in $(\sQ_\lambda)_\lambda / \sC$. Given any $(\sQ_\lambda)_\lambda \to \sX \to \sS$, choose an open affine covering $\{\hU_i \to \hX\}$ of $\hX$. Then this induces a $\ulPZar$-covering of $\hX$, and since $(\sQ_\lambda)_\lambda$ is $\ulPZar$-local, there is a factorisation $(\sQ_\lambda)_\lambda \to \sU_i \to \sX$ for some $i$.

Alternatively, modulus keeping track of size issues, we have the following conceptual argument: Under the hypothesis (Op), the canonical induced morphism of topoi $\Shv_{\ulPZar}(\sA) \to \Shv_{\ulPZar}(\sC)$ is an equivalence. Hence, there is an equivalence of fibre functors, and therefore, and equivalence of $\ulPZar$-local pro-objects, cf.Prop.~\ref{prop:fibArePro}.
\end{proof}

\begin{lemm} \label{lemm:NoethQfDom}
Let $\sS$ be a qcqs modulus pair with Noetherian interior and let $\sC$ be a full subcategory of $\ulPSch_\sS^\min$ satisfying:
\begin{enumerate}
 \item[(Cl)] For every minimal morphism $\sZ \to \sX$ such that $\sX \in \sC$ and $\hZ \to \hX$ is a closed immersion, we have $\sZ \in \sC$.
\end{enumerate}
Then every pro-object of $\sC$ is isomorphic to one with schematically dominant transition morphisms. 

If $\sC$ also satisfies condition (Op) of Lem.~\ref{lemm:ZarLocIsAff} then every $\ulPZar$-local pro-object of $\sC$ is isomorphic to one with affine schematically dominant transition morphisms.
\end{lemm}

\begin{rema}
This is actually if and only if. That is, if every pro-object can be represented by one with schematically dominant transition morphisms, then $\iS$ is Noetherian. Indeed, $\iS$ is Noetherian if and only if every decreasing sequence $\iZ_0 \supseteq \iZ_1 \supseteq \dots$ of closed subschemes stabilises. If $\iS$ is not Noetherian, then choosing a non-stabilising sequence, and letting $\hZ_i$ be the closure of $\iZ_i$ in $\hS$ produces a pro-object $((\hZ_i, \mS|_{\hZ_i}))_{i \in \NN}$ which is not isomorphic to one with schematically dominant transition morphisms.
\end{rema}

\begin{proof}
Let $(\sQ_\lambda)_{\lambda \in \Lambda}$ be a pro-object of $\sC$. %
We will define a new pro-object $(\sR_\lambda)_{\lambda \in \Lambda}$ (same indexing category) and a natural transformation $\sR \to \sQ$ such that the transition morphisms of $\sR$ are schematically dominant, and for every $λ$, there is some ${λ'} \geq λ$ such that the square below admits a dashed morphism making the diagram commutative.\footnote{This diagram implies that the natural transformation of corpresentable functors $L(\sR) \to L(\sQ)$ is an isomorphism, and therefore $\sR \to \sQ$ is an isomorphism of pro-objects, cf.
\S\ref{sec:Proreminders}.} 
\begin{equation} \label{eq:RRQQ}
\xymatrix{
\sR_{λ'} \ar[d] \ar[r] & \sQ_{λ'} \ar[d] \ar@{-->}[dl] \\
\sR_λ \ar[r] & \sQ_λ
}
\end{equation}
In the case (Op) is satisfied we can assume all objects of $\sC$ are affine, cf.Lem.\ref{lemm:ZarLocIsAff}, so we automatically get affine transition morphisms.

Given $u: μ \to λ$ in $\Lambda$, define 
\[ \sA_u = im(\OO_{\hQ_λ} {\to} u_*\OO_{\hQ_μ}) \qquad (\textrm{resp. } \sB = \colim_{u: μ \to λ} \sA_{u}). \]
By construction, the pullback $\sI_{\mQ_λ} \cdot \sA_u$ (resp. $\sI_{\mQ_λ} \cdot \sB$) is locally principal, and locally generated by nonzero divisors.%
\footnote{The claim for $\sA_u$ follows from the assumption that the transition morphisms are minimal. Let $f$ be a local generator for $\mQ_λ$. Since $f|_{\hQ_μ}$ is nonzero for all $u: μ \to λ$, the image of $f$ in $\sA_u$ is nonzero for all $u$ so the image in $\sB$ is nonzero. If there is $g \in \sB$ such that $fg = 0$ in $\sB$, then lifting $g$ to $\OO_{\hQ_λ}$ we have $fg \in \ker(\OO_{\hQ_λ} \to u_*\OO_{\hQ_μ})$ for some $u$ which means $g$ is zero in $\sA_u$, and therefore in $\sB$.} %
Equipping 
\[ \hZ_{u} := \underline{\Spec} (\sA_μ) \qquad (\textrm{resp. } \hR_λ := \underline{\Spec} (\sB)) \]
with the divisor induced by $\mQ_λ$, we obtain a system of morphisms of modulus pairs 
\[ \sR_λ \underset{imm.}{\overset{cl.}{\to}}  \sZ_{u} \underset{imm.}{\overset{cl.}{\to}} \sQ_λ \]
indexed by $u \in \Lambda / \lambda$, which are all closed immersions on the total spaces.   
Now since $\iS$ is noetherian, $\iQ_λ$ is also noetherian since $\iQ_\lambda \to \iS$ is of finite type by assumption, and so the sequence $(\dots \to \sA_μ \to \dots)|_{\iQ_λ}$ stabilises. That is, there is some $u: λ' \to λ$ such that $\sA_μ|_{\iQ_λ} = \sB|_{\iQ_λ}$ for all $μ \geq λ'$. We now have the following diagram on the left with its ``mirror'' diagram of structure sheaves on the right.
\[ \xymatrix{
\iQ_{λ'} \ar[d] \ar[r] & \iR_λ \ar@{=}[r] \ar[d] & \iZ_{u} \ar[r]^{cl.imm.} \ar[d] & \iQ_λ \ar[d] &
 f_*\OO_{\iQ_{λ'}} 
& \sB|_{\iQ_λ} \ar[l] 
& \sA_{λ'}|_{\iQ_λ} \ar@{=}[l] 
& \OO_{\iQ_λ} \ar[l]_{epi.} \\
\hQ_{λ'} \ar@{-->}[r] \ar@/_6pt/[rr] & \hR_λ \ar[r]^{cl.imm.} & \hZ_{u} \ar[r]^{cl.imm.} & \hQ_λ &
f_*\OO_{\hQ_{λ'}} \ar[u]^{monic} &
\sB  \ar@{-->}[l] \ar[u]^{monic} & 
\sA_{λ'} \ar[l]_{epi.} \ar[u]^{monic} \ar@/^6pt/[ll] & 
\OO_{\hQ_λ} \ar[u]^{monic} \ar[l]_{epi.} 
} \]
To obtain the dashed morphism, it suffices to show that $ker(\OO_{\hQ_λ} \to \sB)$ is sent to zero in $f_*\OO_{\hQ_{λ'}}$. This follows from the diagram on the right. So we obtain the lower triangle from the square \eqref{eq:RRQQ} above.

Now noticing that the construction of $\sR_λ$ is functorial in $λ$, we have constructed a pro-object $(\sR_λ)_λ$ equipped with a natural transformation $\sR \to \sQ$. Notice as well that by construction, the transition morphisms of $\sR$ are schematically dominant. Finally, to see that we also have commutativity of the upper triangle in the square \eqref{eq:RRQQ} above, it suffices to notice that $\hR_λ \to \hQ_λ$ is a closed immersion, and therefore a categorical monomorphism. So commutivity of the upper triangle follows from commutivity of the lower triangle and the outside square.
\end{proof}

Finally, we get to the main result of this section.

\begin{theo} \label{theo:locProIsRep}
Let $\sS$ be a qcqs modulus pair,
let $\sC$ be one of the symbols in the left column, 
let $τ$ be one of the symbols in the top row, 
and suppose that the $(\sC, τ)$-entry is non-empty.
\[
\begin{array}{c|ccccc}
{_\sC} \diagdown{^τ}    & \Zar & \Nis & \et & \fppf & \rqfh \\ \hline
\Sch^\min_\sS        & (N)  & (N)  & (N) & (N)   & (N)   \\ 
\Sm^\min_\sS         & (F)  & (F)  & (F) &       &       \\ 
\end{array}
\]
In case (N) assume $\iS$ is Noetherian.

Then there is a canonical adjunction
\[ 
L
:
\left \{ 
\begin{array}{c}
\ulPtau\textrm{-local } \uS\textrm{-local } \\
\textrm{ objects of } \\
\ulPSCH_\sS^\min
\end{array}
\right \}
\rightleftarrows
\left \{ 
\begin{array}{c}
\ulMtau\textrm{-local } \\
\textrm{ objects of } \\
\Pro(\ulM\sC)
\end{array}
\right \}
:
\amblim
\]
\begin{enumerate}
 \item The left adjoint $L$ sends a modulus pair $\sP$ to the pro-object indexed by the undercategory $\sP/\ulM\sC$. Moreover, $L$ is a localization. 
 \item The right adjoint is fully faithful and satisfies 
 \begin{equation} \label{eq:limIntlim}
 \left ( \amblim(\sQ_λ)_λ\right )^\circ = \lim (\iQ_λ).
 \end{equation}
\end{enumerate}
\end{theo}

\begin{proof}
First we construct the right adjoint. It is the composition of the equivalence 
\begin{equation} \label{eq:PtauMtauProPro}
\left \{ 
\begin{array}{c}
\ulPtau\textrm{-local } \uS\textrm{-local }\\
\textrm{ objects of } \\
\Pro(\ulP\sC)
\end{array}
\right \}
\cong
\left \{ 
\begin{array}{c}
\ulMtau\textrm{-local } \\
\textrm{ objects of } \\
\Pro(\ulM\sC)
\end{array}
\right \}
\end{equation}
of Prop.~\ref{prop:proAdjoint}, and the categorical limit 
$\amblim$ of $\ulPSCH_\sS$, cf.Lem.~\ref{lemm:limitExists}, Lem.~\ref{lemm:limitExistsExists}. 
Since all morphisms in $\ulP\sC$ are minimal, to know that the limit exists it suffices to know that the transition morphisms can be assumed to be affine.
This is what Lem.~\ref{lemm:ZarLocIsAff} says. Note that we already have Eq.\eqref{eq:limIntlim} at this stage.

Next we check that the image lands in the correct full subcategory of $\ulPSCH_\sS^\min$. Given a $\ulPtau$-local $\uS$-local pro-object $\sQ = (\sQ_θ)_θ$, set $\sP := \amblim \sQ_θ$.
Then Prop.~\ref{prop:prouSLocLim} says 
$\sP$ is $\uS$-local if and only if $\sQ$ is $\uS$-local. 
In the case $\tau = \Zar, \Nis, \et, \fppf$, Proposition~\ref{prop:proFlatLocLim} says that 
$\sP$ is $\ulPtau$-local if and only if $\sQ$ is $\ulPtau$-local. 
In the case $\tau = \rqfh$,  
Proposition~\ref{prop:prorqfhLocLim} says that 
$\sP$ is $\ulPtau$-local if and only if $\sQ$ is $\ulPtau$-local 
(we can assume the transition morphisms are schematically dominant by Lem.~\ref{lemm:NoethQfDom}). 

The claim that $L$ is a localization follows from the fully faithfulness of the right adjoint $\amblim$.
To show fully faithfulness of $\amblim$, we claim that for $\sQ$ and $\sR$ on the right, we have 
\[ 
\hom_{\ulPSCH_\sS}(\amblim \sQ_θ, \amblim \sR_ρ)
\cong
\lim_ρ \colim_θ 
\hom_{\ulPSch_\sS}(\sQ_θ, \sR_ρ),
\]
cf.Eq.\eqref{eq:PtauMtauProPro}. Since $\amblim$ is the categorical limit in $\ulPSCH_\sS$, it suffices to show that for every $ρ$ we have 
\[ 
\hom_{\ulPSCH_\sS}(\amblim \sQ_θ, \sR_ρ)
\cong
\colim_θ 
\hom_{\ulPSCH_\sS}(\sQ_θ, \sR_ρ).
\]
In the case (F), the morphism $\iR_ρ \to \iS$ is flat and of finite presentation, so since both $\amblim \sQ_θ$ and $(\sQ_θ)_θ$ are $\uS$-local by Prop.~\ref{prop:allIsoToAmbEtale} below, we can assume that $\hR_ρ \to \hS$ is of finite presentation. Then the 
desired isomorphism follows from Prop.~\ref{prop:qcqs-limcolim}. 

In the case (N), Lemma \ref{lemm:NoethQfDom} allows us to assume the transition morphisms of $(\sQ_θ)_θ$ are schematically dominant, and, again, the 
desired isomorphism follows from Prop.~\ref{prop:qcqs-limcolim}. 

Next we construct the left adjoint. Here we use the description of pro-objects in terms of corpresentable functors, Thm.\ref{theo:proBasics}. The left adjoint will be
\[ \sP \mapsto L\sP := \hom_{\ulPSCH_\sS^\min}(\sP, -) \in \Fun(\ulM\sC, \Set). \]
As a functor on $\ulP\sC$, the functor $L\sP$ is indeed a pro-object, since $\ulP\sC$ has finite limits and the inclusion $\ulP\sC \to \ulPSCH_\sS^\min$ commutes with them, Prop.~\ref{prop:minFibPro}, Lem.\ref{lemm:proComDiag}. When $\sP$ is $\uS$-local, the functor $L\sP$ factors uniquely through the localisation $\ulP\sC {\to} \ulM\sC$. Clearly, $L\sP$ is $\ulMtau$-local whenever $\sP$ is $\ulPtau$-local. So we do indeed get a functor between the correct categories. To show that $L$ is the right adjoint to $\amblim$, it suffices to show that for $\sP$ on the left and $\sQ = (\sQ_θ)_θ$ on the right, we have
\[ 
\hom_{\ulPSCH_\sS^\min}(\sP, \amblim \sQ_θ) \cong \hom_{\Fun(\ulM\sC, \Set)}(L\sQ, L\sP),
\]
where $L\sQ := ``\colim" \hom_{\ulPSch_\sS^\min}(\sQ_θ, -)$. But this is obvious since $``\colim"$ means colimit in the category of functors. 
\end{proof}

\begin{prop} \label{prop:allIsoToAmbEtale}
Let $\sS$ be a qcqs modulus pair, %
and $\sX \to \sS$ a minimal ambient morphism of finite type %
such that $\iX \to \iS$ is flat and of finite presentation. %
Then there exists a minimal ambient morphism $\sY \to \sT$ %
such that $\hY \to \hT$ is flat of finite presentation, and {\aab}s fitting into a commtuative square
\[ \xymatrix{
\sY \ar@{-->}[r]^{adm.bu} \ar@{-->}[d]_{tot. ppf} & \sX \ar[d]^{int. ppf} \\
\sT \ar@{-->}[r]^{adm.bu} & \sS.
} \]
\end{prop}

\begin{proof}
By Prop.~\ref{prop:cannotLimitBlowups} there is a factorisation $\hX \to \hW \to \hS$ with $\hX \to \hW$ a closed immersion, which is an isomorphism over $\iS$, and $\hW \to \hX$ a morphism of finite presentation. Then by Raynuad-Gruson flatification, \cite[Thm.5.2.2]{RG71} or \cite[Tag 0815]{stacks-project}, since $\hW \to \hS$ is flat and of finite presentation over $\iS$, there exists an {\abb} $\sT \to \sS$ such that the strict transform $\hW^{st} \to \hT$ is flat and of finite presentation. So we have the following commutative diagram of strict transforms on the left.
\[ \xymatrix{
\hX^{st} \ar[r] \ar[d] & \hX \ar[d] &&
\sX^{st} \ar[r] \ar[d] & \sX \ar[dd] \\
\hW^{st} \ar[r] \ar[d] & \hW \ar[d] &&
\sW^{st} \ar[d] &   \\
\hT \ar[r] & \hS &&
\sT \ar[r] & \sS
} \]
Since $\hW^{st} \to \hT$ is flat, $\hS|_{\hW^{st}} = \hT|_{\hW^{st}}$ is an effective Cartier divisor, so we actually have the diagram of ambient minimal morphisms of modulus pairs on the right. To finish, it suffices to prove that $\hX^{st} \to \hW^{st}$ is an isomorphism. As a strict transform of a closed immersion, it is a closed immersion, and moreover it induces an isomorphism $(X^{st})^\circ \cong (W^{st})^\circ$. So $\OO_{\hW^{st}} \to \OO_{\hX^{st}}$ is a surjective morphism which embeds into an isomorphism $\OO_{(X^{st})^\circ} \cong \OO_{(W^{st})^\circ}$. Therefore $\OO_{\hW^{st}} \to \OO_{\hX^{st}}$ is an injective surjection, or in other words, $\hX^{st} \cong \hW^{st}$.
\end{proof}

\section{Further remarks}

\begin{rema}
Notice that in Thm.~\ref{theo:locProIsRep} we restricted to categories of minimal morphisms in such statements. This is sufficient to get a conservative family of fibre functors on topoi such as $\Shv_{\ulMNis}(\ulMSm_{\sS})$, but not to get all fibre functors. This is because there are fibre functors of $\Shv_{\ulMNis}(\ulMSm_{\sS})$ whose corresponding pro-object has modulus which does not stabilise. We have avoided writing about this for the ease of the reader, but its entirely possible to obtain statements such as: there is an equivalence of categories between fibre functors of $\Shv_{\ulMNis}(\ulMSm_{\sS})$ and pro-objects $(\sP_λ)_λ$ of $\ulMSm_{\sS}$ such that 
\begin{enumerate}
 \item all $\hP_λ$ are affine,
 \item $\lim \hP_λ$ is the spectrum of a hensel local ring,
 \item $\lim \iP_λ$ is the spectrum of a hensel local ring (but not necessarily open in $\lim \hP_λ$! Only pro-open.),
 \item $\overline{\{p\}}$ is the spectrum of a hensel valuation ring where $p$ is the closed point of $\lim \hP_λ$.
 \item $\lim \hP_λ = \ol{\{p\}} \sqcup_{\{p\}} \lim \iP_λ$ in the category of schemes.
\end{enumerate}
The problem comes from the existence of valuation rings whose totally ordered set of prime ideals does not admit arbitrary intersections.
\end{rema}

\begin{rema}
In all the above statements we have always insisted on working in $\ulPSCH$ instead of $\ulMSCH$. This is because the equivalence between $\ulPtau$-local $\uS$-local pro-objects of $\ulPSch_\sS$ and $\ulM\tau$-local pro-objects of $\ulMSch_\sS$ (for example) does not respect the subcategories $\ulPSch_\sS \subseteq \Pro(\ulPSch_\sS)$ and $\ulMSch_\sS \subseteq \Pro(\ulMSch_\sS)$.

For example, consider any complete discrete valuation ring $R$ with uniformiser $\pi$ which admits a finite field extension $L / \Frac(R)$ such that the integral closure $A \subseteq L$ (necessarily another complete discrete valuation ring) of $R$ in $L$ is \emph{not} of finite type over $R$. Then consider any finite $R$-algebra $\breve{A}$ such that $\Frac(\breve{A}) = L$. In this case, $(\Spec(\breve{A}), (\pi))$ will be $\ulMNis$-local as an object of $\ulMSch_{R}$ but not $\uS$-local as an object of $\ulPSch_{R}$ (cf. also Lem.~\ref{lemm:finiteulMNislocal} which this kind of situation actually does arise naturally in the course of working towards Prop.~\ref{prop:CechCycles}---a necessary piece of the motivic foundations). The $\uS$-local pro-object of $\ulPSCH$ corresponding to $(\Spec(\breve{A}), (\pi))$ comes from the collection of sub-$R$-algebras of $A$ of finite type containing $\pi$.
\end{rema}


\chapter{Modulus sites} \label{sec:sites}

\section{Ambient sites}\label{sec:ambsites} 

In this subsection, we discuss various (big and small) sites on the category of modulus pairs with ambient morphisms. 
As a preliminary, we recall various sites of schemes. 
For a scheme $S$, 
\begin{enumerate}
 \item $\SCH_S$ is the category of all $S$-schemes, 
 \item $\SCH_S^{\qcqs}$ is the category of qcqs $S$-schemes, 
 \item $\Sch_S$ is the category of finite type separated $S$-schemes, 
 \item $\Fppf_S$ is the category of $S$-schemes which are flat of finite presentation.
 \item $\Sm_S$ is the category of smooth separated $S$-schemes.
 \item $\Et_S$ is the category of \'etale $S$-schemes.
 \item $\Op_S$ is the category of open immersions to $S$.
 \item $\Qf_S$ is the category of $S$-schemes which can be written as a finite composition of open immersions and finite morphisms.%
 \footnote{By Zariski's Main Theorem, \cite
 [05K0]{stacks-project}, when $S$ is qcqs this is equivalent to the category of quasi-finite $S$-schemes.}
\end{enumerate}

\begin{defi} \label{defi:schSites}
Let $S$ be a scheme, let $\sC$ be one of the categories in the left column, let $τ$ be one of the classes in the top row, cf.~Def.~\ref{defi:coveringsSch}, and suppose that the $(\sC, τ)$-entry has a star.
\[
\begin{array}{cc|cc|c|c}
{_\sC} \diagdown{^τ} & \Zar & \Nis & \et & \fppf & \rqfh \\
\SCH_S                 & \ast & \ast & \ast& \ast  & \ast  \\
\SCH_S^\qcqs           & \ast & \ast & \ast& \ast  & \ast  \\
\Sch_S                 & \ast & \ast & \ast& \ast  & \ast  \\ \hline 
\Fppf_S                & \ast & \ast & \ast& \ast  &       \\
\Sm_S                  & \ast & \ast & \ast&       &       \\
\Et_S                  & \ast & \ast & \ast&       &       \\
\Op_S                  & \ast &      &     &       &       \\ \hline
\Qf_S                  &      &      &     &       & \ast  \\
\end{array}
\]
Then the $τ$-topology on $\sC$ is the topology generated by $τ$-coverings. 
\end{defi}

Similarly, we consider various categories over a modulus pair $\sS$. Recall that for any property $(P)$ of a morphism of schemes, we say an ambient morphism $\sX \to \sS$ has property $(P)$ if $\hX \to \hS$ has property $(P)$, cf. Def.~\ref{def:ambientmorph}.

\begin{defi} \label{defi:ambSitesCat} Let $\sS$ be a modulus pair.
\begin{enumerate}
 \item[Big sites.]
 \item $\ulPSCH_\sS$ is the category of all ambient $\sS$-modulus pairs, with commutative triangles of ambient morphisms as morphisms.

 \item $\ulPSCH_\sS^{\qcqs}$ is the full subcategory of $\ulPSCH_\sS$ of $\sS$-modulus pairs whose structure morphism is qcqs.

 \item[Medium sites.]
 \item $\ulPSch_\sS$ is the full subcategory of $\ulPSCH_\sS$ of $\sS$-modulus pairs whose structure morphism is finite type and separated.
 
 \item $\ulPSm_\sS$ is the full subcategory of $\ulPSch_\sS$ of those ambient $\sS$-modulus pairs $\sX \to \sS$ such that $\iX \to \iS$ is smooth.
 
 \item $\ulPSch_\sS^{\min}$ is the full subcategory of $\ulPSch_\sS$ consisting of those morphisms $\sX \to \sS$ which is minimal. 
    
 \item $\ulPSm_\sS^{\min}$ is the full subcategory of $\ulPSm_\sS$ consisting of those morphisms $\sX \to \sS$ which is minimal. 
 
 \item[Small sites.] 
 \item \label{item:SmallSites} Let $\sC$ be one of the symbols $\Qf$, $\Fppf$, $\Et$, $\Op$. Then $\ulP\sC_\sS$ is the full subcategory of $\ulPSch_\sS$ of those ambient $\sS$-modulus pairs $\sX \to \sS$ whose structural morphism is minimal, and $\iX \to \iS$ is in $\sC_{\iS}$.
\end{enumerate}
\end{defi}

\begin{defi} \label{defi:ambSites}
Let $\sS$ be a modulus pair, %
let $\sC$ be one of the categories in the left column, let $τ$ be one of the classes in the top row, cf.~Def.~\ref{defi:coveringsSch}, and suppose that the $(\sC, τ)$-entry has a star.
\[
\begin{array}{cc|cc|c|c}
{_\sC} \diagdown{^τ} & \Zar & \Nis & \et & \fppf & \rqfh \\
\ulPSCH_\sS                 & \ast & \ast & \ast& \ast  & \ast  \\
\ulPSCH_\sS^{\qcqs}           & \ast & \ast & \ast& \ast  & \ast  \\
\ulPSch_\sS^{\min}                & \ast & \ast & \ast& \ast  & \ast  \\ 
\ulPSm_\sS^{\min}                  & \ast & \ast & \ast&       &       \\ \hline 
\ulPQf_\sS                  &      &      &     &       & \ast  \\
\ulPFppf_\sS                & \ast & \ast & \ast& \ast  &       \\
\ulPEt_\sS                  & \ast & \ast & \ast&       &       \\
\ulPOp_\sS                  & \ast &      &     &       &       \\ 
\end{array}
\]
Then the $\ulPtau$-topology on $\sC$ is the topology generated by $\ulPtau$-coverings.
\end{defi}

\begin{rema}
The most prominent combinations in the motivic framework are the $\ulPNis$- and $\ulPet$-topologies on $\ulPSm_\sS$.
\end{rema}

\section{Non-ambient sites}

Next, we will equip non-ambient categories with various topologies. 
Recall that the category $\ulMSCH$ is defined by 
\[
\ulMSCH := \ulPSCH[\ul{\Sigma}^{-1}],
\]
where $\ul{\Sigma}$ is the class of abstract admissible blow-ups (see Definition \ref{defn:Sigma}). 

The following definition is the $\ulM$ version of Def.~\ref{defi:ambSitesCat}.

\begin{defi} \label{defi:modSitesCat} Let $\sS$ be a modulus pair.
\begin{enumerate}
 \item[Big sites.]
 \item $\ulMSCH_\sS$ is the category of all $\sS$-modulus pairs, with commutative triangles as morphisms.

 \item $\ulMSCH_\sS^{\qcqs}$ is the full subcategory of $\ulMSCH_\sS{}$ of $\sS$-modulus pairs whose structure morphism is %
isomorphic to an ambient %
qcqs %
morphism.

 \item[Medium sites.]
 \item $\ulMSch_\sS{}$ is the full subcategory of $\ulMSCH_\sS{}$ whose of $\sS$-modulus pairs whose structure morphism is %
isomorphic to an ambient %
finite type and separated %
morphism.
 
 \item $\ulMSm_\sS$ is the full subcategory of $\ulMSch_\sS$ of those $\sS$-modulus pairs $\sX \to \sS$ such that $\iX \to \iS$ is smooth.

 \item $\ulMSch^\min_\sS{}$ is the full subcategory of $\ulMSch_\sS{}$ consisting of those objects $\sX \to \sS$ which are isomorphic to a minimal ambient morphism.
   
 \item $\ulMSm^\min_\sS$ is the full subcategory of $\ulMSm_\sS{}$ consisting of those objects $\sX \to \sS$ which are isomorphic to a minimal ambient morphism.
 
 \item[Small sites.] 
 \item Let $\sC$ be one of the symbols $\Qf$, $\Fppf$, $\Et$, $\Op$. Then $\ulM\sC_\sS$ is the full subcategory of $\ulMSch_\sS$ of those $\sS$-modulus pairs whose structure morphism is 
isomorphic (in the category of morphisms of $\ulMSch$) to an ambient minimal, of finite type, separated morphism $\sX' \to \sS'$ such that $\hX' \to \hS'$ is in $\sC_{\hS'}$.
\end{enumerate}
\end{defi}

\begin{lemma}\label{lemm:SmallSitesObj} 
Let $\sS$ be a modulus pair, and $\sC = \Fppf, \Et, \Op, \Qf$. Then, for any object $\sX \to \sS$ of $\ulM\sC_{\sS}$, there exists a commutative diagram 
\[\xymatrix{
\sY \ar[r]^{adm.b.u.} \ar[d]_{\sC_{\hT}\ \ni} & \sX \ar[d]^{\in\ \ulM\sC_\sS} \\
\sT \ar[r]_{adm.b.u.} & \sS
}\]
where $s$ and $t$ are {\aab}s, and $g$ is an ambient minimal morphism such that $\ol{g} : \hY \to \hT$ is an object of $\sC_{\hT}$.
\end{lemma}

\begin{proof}
When $\sC = \Op, \Et, \Fppf$, this is just a rephrasing of Theorem \ref{thm:etaleHeavy} (4), (6), (7).
The case $\sC = \Qf$ follows from Theorem \ref{thm:finHeavy}.
\end{proof}

\begin{defi} \label{defi:mtop}
Let $\sS$ be a modulus pair, let $\sC$ be one of the categories in the left column, cf.~Def.\ref{defi:modSitesCat}, let $τ$ be one of the classes in the top row, cf.~Def.~\ref{defi:coveringsSch}, and suppose that the $(\sC, τ)$-entry has a star.
\begin{equation}\label{table*}\begin{gathered}
\begin{array}{cc|cc|c|c}
{_\sC} \diagdown{^τ} & \Zar & \Nis & \et & \fppf & \rqfh \\
\ulMSCH_\sS                 & \ast & \ast & \ast& \ast  & \ast  \\
\ulMSCH_\sS^{\qcqs}           & \ast & \ast & \ast& \ast  & \ast  \\
\ulMSch_\sS                 & \ast & \ast & \ast& \ast  & \ast  \\ 
\ulMSm_\sS                  & \ast & \ast & \ast&       &       \\
\ulMSch^\min_\sS                 & \ast & \ast & \ast& \ast  & \ast  \\ 
\ulMSm^\min_\sS                  & \ast & \ast & \ast&       &       \\ \hline 
\ulMQf_\sS                  &      &      &     &       & \ast  \\
\ulMFppf_\sS                & \ast & \ast & \ast& \ast  &       \\
\ulMEt_\sS                  & \ast & \ast & \ast&       &       \\
\ulMOp_\sS                  & \ast &      &     &       &       \\ 
\end{array}
\end{gathered}\end{equation}
Then the $\ulMτ$-topology on $\sC$ is the topology generated by $\ulPtau$-coverings. That is, families $\{\sU_i \to \sX\}_{i \in I}$ of ambient minimal morphisms such that $\{\hU_i \to \hX\}_{i \in I}$ is a $τ$-covering.
\end{defi}

\begin{rema} \label{rema:induced}
In other words, the $\ulMtau$-topology is the finest topology such that $\ulPSCH_\sS \to \ulMSCH_\sS$ is continuous where $\ulMSCH_{\hS}$ is equipped with the $\ulPtau$-topology, cf.Rem.~\ref{rem:woheaviness}. 
\end{rema}

\begin{rema}
The most prominent combination in the motivic framework is the $\ulMNis$-topology on $\ulMSm_\sS$.
\end{rema}

\begin{rema}
We will see in Cor.~\ref{coro:normalForm} below that in all cases of Def.~\ref{defi:ambSites}, up to refinement and isomorphism, every covering is a composition of a single $\ulPtau$ covering and an \abb.
\end{rema}

\begin{prop} \label{prop:subLocalisation}
Let $\sS$ be a modulus pair, and let $\sC$ denote one of the following symbols: $\SCH$, $\SCH^{\qcqs}$, $\Sch$, $\Sm$, $\Sch^\min$, $\Sm^\min$, $\Qf$, $\Fppf$, $\Et$, $\Op$.
 Then there exists a canonical identification
 \[
 \ulM\sC_{\sS} \cong \ulP\sC_{\sS} [\uS^{-1}]. 
 \]
where $\ul{\Sigma}$ means $\ulP\sC_\sS \cap \ul{\Sigma}$ by abuse of notation.
\end{prop}

\begin{proof}
Let $\sC$ denote one of the following symbols: $\SCH$, $\SCH^{\qcqs}$, $\Sch$, $\Sm$, $\Sch^\min$, $\Sm^\min$, $\Qf$, $\Fppf$, $\Et$, $\Op$. 
First note that the natural functor $\ulP\sC_{\sS} \to \ulM\sC_{\sS}$ induces a functor $\ulP\sC_{\sS}[\ul{\Sigma}^{-1}] \to \ulM\sC_{\sS}$. 
It suffices to show that this is fully faithful and essentially surjective. 

First we prove the fully faithfulness. For any two objects $\sX \to \sS$ and $\sY \to \sS$ in $\ulP\sC_{\sS}$, the map 
\[
\hom_{\ulP\sC_{\sS}[\ul{\Sigma}^{-1}]} (\sX,\sY) \to \hom_{\ulM\sC_{\sS}} (\sX,\sY)
\]
is injective since the composite \[\hom_{\ulP\sC_{\sS}[\ul{\Sigma}^{-1}]} (\sX,\sY) \to \hom_{\ulM\sC_{\sS}} (\sX,\sY) \to \hom_{\SCH} (\iX,\iY)\] is injective by Lemma \ref{lem:coincidence2}. To show the surjectivity, note that any morphism $f \in \hom_{\ulMSCH_{\sS}} (\sX,\sY)$ is represented by a diagram of ambient $\sS$-morphisms $\sX \xleftarrow{g} \sX' \xrightarrow{} \sY$ with $g \in \ul{\Sigma}$.
It suffices to show that $\sX' \to \sS$ is an object of $\ulP\sC_{\sS}$.
In other words, we need the following claim.

\begin{claim}
If $\sX \to \sS$ is an object in $\ulP\sC_{\sS}$ and if $\sX' \to \sX$ belongs to $\ul{\Sigma}$, then $\sX' \to \sS$ is an object of $\ulP\sC_{\sS}$. 
\end{claim}

\begin{proof}
If $\sC = \SCH$, there is nothing to prove. 
If $\sC = \SCH^{\qcqs}$, this follows from the fact that any morphism in $\ul{\Sigma}$ is quasi-compact and quasi-separated (as a proper morphism). 
The case $\sC = \Sch$ follows from that any proper morphism is of finite type and separated.
The case $\sC = \Sm$ follows from that any morphism in $\ul{\Sigma}$ induces an isomorphism on the interior.
The same argument works for $\sC = \Sch^\min, \Sm^\min$.
When $\sC = \Qf$, $\Fppf$, $\Et$, $\Op$, then the assertion follows from that any morphism in $\ul{\Sigma}$ is minimal and induces an isomorphism on the interiors. 
\end{proof}
Thus we have proven the fully faithfulness.
Finally, the essential surjectivity follows by definition of $\ulM\sC$. This finishes the proof.
\end{proof}

\begin{prop}\label{prop:SmallSitesApprox}
Let $\sS$ be a qcqs modulus pair, and $\sC$ one of the symbols $\Op$, $\Et$, $\Fppf$, $\Qf$. 
Then there is a canonical equivalence of categories
\[ \ulM\sC_\sS{} \cong \varinjlim_{\sS' \to \sS} \sC_{\hS'}. \]
where the colimit is over all {\aab}s.

Moreover, if $\tau$ is a class of coverings such that the $(\sC, \tau)$-entry in Table~\ref{table*} has a star, then the above is an equivalence of sites when the left hand side is equipped with the $\ulM\tau$-topology, and each $\sC_{\hS'}$ on the right is equipped with the $\tau$-topology.
\end{prop}

\begin{proof}
For any {\aab} $\sS' \to \sS$, we consider a functor 
\[
\sC_{\hS'} \to \ulM\sC_{\sS}, \quad (\hX \to \hS') \mapsto (\sX \to \sS' \to \sS),
\]
where $\sX:= (\hX, (S')^\infty \times_{\hS'} \hX)$, which sends any $\sS'$-morphism $\sX \to \sY$ to itself (regarded as an $\sS$-morphism).
Letting $\sS' \to \sS$ run over {\aab}s, we obtain the induced functor 
\[
\varinjlim_{\sS' \to \sS} \sC_{\hS'} \to \ulM\sC_{\sS}.
\]
We claim that this is an equivalence of categories. 
The essential surjectivity follows immediately from Definition \ref{defi:modSitesCat}).
To show that it is fully faithful, take any {\aab} $\sS_0 \to \sS$ and arbitrary objects $\hX_0 \to \hS_0$ and $\hY_0 \to \hS_0$ in $\sC_{\hS_0}$, and consider the map
\begin{equation}\label{eq:colimsCtoC}
\colim_{\sS_1 \to \sS_0} \hom_{\sC_{\hS_1}} (\hX_1 , \hY_1)  \to \hom_{\ulM\sC_{\sS}} (\sX_0,\sY_0),
\end{equation}
where $\sX_0 := (\hX_0 , \mS \times_{\hS_0} \hX_0)$ and $\sY_0 := (\hY_0 , \mS \times_{\hS_0} \hY_0)$, and $\hX_1$, $\hY_1$ denote the ambient spaces of the ambient products $\sX_1 := \sX_0 \ambtimes[\sS_0] \sS_1$, $\sY_1 := \sY_0 \ambtimes[{\sS_0}] \sS_1$ respectively (note that, in general, $\hX_1$ etc. can be different from $\hX_0 \times_{\hS_0} \hS_1$ etc.).
By definition, the map \eqref{eq:colimsCtoC} sends a morphism $\ol{f} : \hX_0 \to \hY_0$ (which is the same as an ambient minimal morphism $f : \sX_0 \to \sY_0$) to a diagram $f \ambtimes[{\sS_0}] \sS_1 : \sX_1 \to \sY_1$.
The map \eqref{eq:colimsCtoC} is injective since the source and the target can be regarded as subsets of $ \hom_{\SCH} (\iX_1,\iX_1) \cong \hom_{\SCH} (\iX_0,\iX_0)$.
The surjectivity follows from Lemma \ref{lemm:SmallSitesObj}.
Thus, we have proven the equivalence of categories. 

Finally, we prove that $F$ is an equivalence of sites. It suffices to show that for any $\ulM\tau$-covering $\{\sV \to \sU\}$ in $\ulM\sC_{\sS}$, there exist an abstract admissible blow up $p : \sS' \to \sS$, an object $\hU' \to \hS'$ in $\sC_{\hS'}$ and a covering $\{\hV' \to \hU'\}$ such that the covering $\{\sV' \to \sU' \to \sU\}$ refines $\{\sV \to \sU\}$, where $\sU'$ and $\sV'$ are the modulus pairs with ambient spaces $\hU'$, $\hV'$ whose modulus is induced from that of $\sS'$.
Since $\ulM\tau$-topology is generated by $\ulP\tau$-covering, we may assume that $\sV \to \sU$ is a $\ulP\tau$-covering. 
By applying Lemma \ref{lemm:SmallSitesObj} to $\sU \to \sS$, there exists a commutative diagram 
\[\xymatrix{
\sU' \ar[r] \ar[d] & \sU \ar[d] \\
\sS' \ar[r] & \sS
}\]
where the horizontal arrows are {\aab}s, and $\sU' \to \sS'$ is a minimal ambient morphism such that $\hU' \to \hS'$ belongs to $\sC_{\hS'}$.
Then we can take $\hV' := \hV \times_{\hU} \hU'$. This finishes the proof.
\end{proof}

\begin{coro}
Let $\sS$ be a qcqs modulus pair. Then $\ulMOp_\sS$ equipped with the $\ulMZar$-topology is equivalent as a site to the relative Riemann-Zariski space $RZ_{\iS}(\hS)$.
\end{coro}

\begin{proof}
This follows from the definition, since by definition $RZ_{\iS}(\hS) = \varprojlim_{\sS' \to \sS} \hS'$ where the limit is in the category of topological spaces and indexed by admissible blowups.
\end{proof}

\section{Induced topologies on $\sC[S^{-1}]$}

As we mentioned in Rem.~\ref{rema:induced}, the $\ulMτ$-topologies are induced by the $\ulPtau$-toplogies. We show here that every covering can be refined by the image of a composition $\to \to \to \dots \to$ of abstract admissible blowups and $\ulPtau$-coverings, Cf.Rem.~\ref{rem:woheaviness}. In fact, due to the results in Chap.~\ref{chap:coverings}, \S \ref{sec:heaviness} it suffices to use a single $\ulPtau$-covering, and a single abstract admissible blowup.

First we recall the precise definition of induced topology.

\begin{defn} \label{defn:inducedTopology}
Let $u: \sC \to \sC'$ be a functor between categories. If $\sC$ is equipped with a topology, then the \emph{induced topology} $τ'$ on $\sC'$ is the coarsest topology (i.e., fewest covering families) making the functor $u: \sC \to \sC'$ continuous. In other words, the coarsest topology such that for every $τ'$-sheaf $F$ the $\sC$ presheaf $F \circ u$ is a $τ$-sheaf.
\end{defn}

\begin{rema}
Another way of stating this is via the left Kan extension/restriction adjunction
\[ u^*: \PSh(C) \rightleftarrows \PSh(C'): u_*. \]
The induced topology on $C'$ corresponds (uniquely) to the maximal subtopos $\Shv (C')$ of $\PSh(C')$ which fits into a commutative square
\[ \xymatrix{
\PSh(C) & \ar[l]_-{u_*} \PSh(C') \\
\Shv(C) \ar[u] & \ar@{-->}[l] \Shv(C') \ar[u] 
} \]
\end{rema}

\begin{prop} \label{prop:heavyRefine}
Suppose that $\sC$ is a category, $S$ is a class of morphisms equipped with a right calculus of fractions, and $τ$ is a topology on $\sC$. Equip $\sC[S^{-1}]$ with the induced topology, Def.~\ref{defn:inducedTopology}.

If $S$ is ``heavier'' than $τ$ in the sense that for every $τ$-covering $\{V_i \to X\}_{i \in I}$ and family $\{s_i: V_i' \to V_i\} \subseteq S$, there exists some $t: X' \to X$ in $S$, a $τ$-covering $\{W_j' \to X'\}_{j \in J}$, and factorisations
\[ \xymatrix@!=0pt{
W_j' \ar[dr]_-{``\in" τ} \ar[rrrr] &&&& V_{i_j}' \ar[dl]^-{\in S} \\
& X' \ar[dr]_-{\in S} && V_{i_j} \ar[dl]^-{``\in" τ} \\
&& X,
} \]
then families in $\sC[S^{-1}]$ of the form
\begin{equation} \label{equa:normalForm}
\{ W_i' \stackrel{f_i}{\to} X' \stackrel{s}{\to} X \}_{i \in I} 
\end{equation}
where $\{W_i' \stackrel{f_i}{\to} X'\}_{i \in I}$ is a $τ$-covering from $\sC$ and $s \in S$, generate the induced topology on $\sC[S^{-1}]$. In other words, the families \eqref{equa:normalForm} are coverings in $\sC[S^{-1}]$, and \emph{every} covering family of $\sC[S^{-1}]$ is refinable by one of the form \eqref{equa:normalForm}.
\end{prop}

\begin{rema}\label{rem:woheaviness}
The below proof also shows that in the absence of the ``heaviness'' hypothesis, the topology on $\sC[S^{-1}]$ is generated by compositions $
\stackrel{f_0}{\to}
\stackrel{s_0}{\to}
\stackrel{f_1}{\to}
\stackrel{s_1}{\to}
\dots
\stackrel{f_n}{\to}
\stackrel{s_n}{\to}
$
of families in $\sC$ where each $s_i$ is a subfamily of $S$ and $f_i$ is a $τ$-covering. 
\end{rema}

\begin{rema}
If we add the hypothesis that morphisms in $S$ are categorical monomorphisms, then there is an alternative more heuristic proof, using the fact that $\PSh(\sC[S^{-1}])$ is the category of sheaves on $\sC$ for the topology whose coverings are singletons $\{s\}$ with $s \in S^{-1}$. Cf.~Prop.\ref{prop:locShv}, Prop.~\ref{prop:locPSh}. 
\end{rema}

\begin{proof}
Let $τ_{ind}$ denote the induced topology on $\sC[S^{-1}]$, and $τ_{ref}$ the collection of those sieves containing families of the form \eqref{equa:normalForm}. We will show that $τ_{ind} = τ_{ref}$.

First recall that the left Kan extension $loc^*: \PSh(\sC) \to \PSh(\sC[S^{-1}])$ can be identified with the functor $a_S: \PSh(\sC) \to \PSh_{S\textrm{-loc}}(\sC)$; $a_SF(X) = \varinjlim_{X' \to X \in S} F(S')$, where $\PSh_{S\textrm{-loc}}(\sC) \subseteq \PSh(\sC)$ is the full subcategory of those presheaves which send elements of $S$ to isomorphisms, Prop.~\ref{prop:locPSh}. It follows from this description that $loc^*$ preserves all finite limits and colimits. In particular, it preserves monomorphisms, so the image of a sieve is a sieve, and the image of sieve generated by a family of morphisms is the sieve generated by the image of those morphisms, that is, $loc^*(im(\sqcup V_i \to X)) = im(\sqcup loc^*V_i \to loc^*X)$.

Now since the functor $\sC \to \sC[S^{-1}]$ is continuous, the left Kan extension of any covering sieve is a covering sieve, \cite[Expos{\'e} III, Prop.1.2(i$\Rightarrow$ii)]{SGA41}.\footnote{Note that a sieve is bicovering as a morphism of presheaves if and only if it is a covering sieve, \cite[Expos{\'e} II, Prop.5.3(i$\Rightarrow$iibis)]{SGA41}.} So the left Kan extension of the sieve generated by a $\tau$-covering $\{W_i' \stackrel{f_i}{\to} X'\}_{i \in I}$ is a covering sieve of $τ_{ind}$. Furthermore, any sieve isomorphic to a covering sieve is a covering sieve. So sieves containing families of the form \eqref{equa:normalForm} are certainly covering sieves of $\sC[S^{-1}]$. That is, $τ_{ref} \subseteq τ_{ind}$.

By definition, $τ_{ind}$ is the coarsest topology making $\sC \to \sC[S^{-1}]$ continuous. So to show $τ_{ref} = τ_{ind}$, it now suffices to show that $τ_{ref}$ is a topology, and that $τ_{ref}$ makes $\sC \to \sC[S^{-1}]$ continuous.

To show that $τ_{ref}$ is a topology we check the two conditions of Prop.~\ref{prop:coverage} below. The condition (Loc) follows immediately from the hypothesis in the statement and the condition (Stab). For the condition (Stab) it suffices to check the following. 
For any family $\{W'_i \to X'\}_{i \in I}$ in $τ$, morphism $X' \stackrel{s}{\to} X$ in $S$, and any morphism $Z \stackrel{t^{-1}}{\to} Y \stackrel{g}{\to} X$ of $\sC[S^{-1}]$ with $t \in S$, $g \in \sC$, there exists a $τ$-covering family $\{V'_j \to Y'\}_{j \in J}$, some function $J \to I$, a morphism $\stackrel{s'}{\to}$ in $S$, and commutative diagrams as below.
\[ \xymatrix@!=6pt{
V'_j \ar@{=}[d] \ar[r]^{} & Y' \ar[r]^{ts' \in S} \ar@{=}[d] & Z \ar[d]^-{t^{-1}} \\
V'_j \ar[d] \ar[r]^{} 
& Y' \ar@{..>}[r]^{s' \in S} \ar@{..>}[d] & Y \ar[d]^g  \\
W_{i_j}' \ar[r]_{} & X' \ar[r]_{s \in S} & X .
} \]
The lower right square comes from the calculus of fractions, and the family $\{V'_j \to Y'\}_{j \in J}$, function $J \to I$ and lower left squares come from the pullback property of the covering sieves of $τ$.

Finally, we want to show that $τ_{ref}$ makes $\sC \to \sC[S^{-1}]$ continuous. But this follows from \cite[Expos{\'e} III, Prop.1.2(ii$\Rightarrow$i)]{SGA41}.
\end{proof}

\begin{prop} \label{prop:coverage}
Let $C$ be a category equipped with a class $\sP$ of families of morphisms $\{U_i \to X\}_{i \in I}$. Let $\sT$ denote the family of sieves $im(\sqcup U_i \to X)$ associated to families in $\sP$.
\begin{enumerate}
 \item For $\sT$ to satisfy the locality axiom of a Grothendieck topology it suffices that $\sP$ satisfy: 
 \item[(Loc)] given $\sP$-families $\{V_{ij} \to U_i \}_{j \in J_i}$ and $\{U_i \to X\}_{i \in I}$ there exists an $\sP$-family $\{W_k \to X\}_{k \in K}$ a function $K \to \sqcup_{i \in I} J_i$ and factorisations
 \[ W_k \to V_{i_k j_k} \to X \]

 \item For $\sT$ to satisfy the base change axiom of a Grothendieck topology it suffices that $\sP$ satisfy: 
 \item[(Stab)]for every $\sP$-family $\{U_i \to X\}_{i \in I}$ and morphism $Y \to X$ there exists a $\sP$-family $\{V_j \to Y\}_{j \in J}$, a function $J \to I$ and commutative squares
 \[ \xymatrix{
V_j \ar[d] \ar[r] & Y \ar[d] \\
U_{j_i} \ar[r] & X
 } \]
\end{enumerate}
In particular, if $\sP$ satisfies (Loc) and (Stab) then $\sT$ together with all maximal sieves $\id_X$ forms a Grothendieck topology. 
\end{prop}

\begin{proof}
Everything follows directly from the definitions.
\end{proof}

\begin{coro} \label{coro:normalForm}
Let $\sS$ be a qcqs modulus pair, and let $\tau$ be one of $\Zar$, $\Nis$, $\et$, $\fppf$, or $\rqfh$. Then any $\ul{\textrm{M}}\tau$-covering in $\ulMSch_\sS$ is refinable by one of the form
\[ \{ \sU_i \to \sY \to \sX\}_{i \in I} \]
where $\sY \to \sX$ is an \abb{} and $\{\sU_i \to \sX\}_{i \in I}$ is an $\ulPtau$-covering.
\end{coro}

\begin{proof}
It suffices to check that the hypotheses of Prop.~\ref{prop:heavyRefine} are satisfied. For $\tau = $  $\Zar$, $\Nis$, $\et$, $\fppf$ this was done in Thm.~\ref{thm:etaleHeavy} and for $\rqfh$ this was done in Thm.~\ref{thm:finHeavy}.
\end{proof}

\section{Subcanonicity of the {$\protect \ulMfppf$}-topology}

It is well-known that, on the category of schemes, the fpqc topology is subcanonical.
In other words, any representable presheaf is a sheaf for fpqc topology.
We have the analogous result for the $\ulMfppf$-topology.

\begin{cor}[Subcanonicity]\label{cor:fppf-subcanonical}
The $\ulMfppf$-topology on $\ulMSCH^\qcqs$ is subcanonical.
In particular, any topology on $\ulMSCH^\qcqs$ coarser than $\ulMfppf$ is subcanonical. 
\end{cor}

\begin{proof}
Let $\sX = (\hX,\mX) \in \ulMSCH^\qcqs$ be a modulus pair, and $h_{\sX}$ the presheaf on $\ulMSCH^\qcqs$ represented by $\sX$.
Moreover, fix a covering $\{u_\lambda : \sU_\lambda \to \sU\}_{\lambda \in \Lambda}$ for the $\ulMfppf$-topology.
It suffices to show that
\[
h_{\sX} (\sU) \xrightarrow{i} \prod_{\lambda \in \Lambda} h_{\sX} (\sU_{\lambda}) \overset{a_1}{\underset{a_2}{\rightrightarrows}} \prod_{\lambda, \mu \in \Lambda} h_{\sX} (\sU_{\lambda} \times_{\sU} \sU_{\mu})
\]
is an equaliser diagram of sets, where $a_i$ is induced by the $i$-th projection for $i=1,2$.
Note that, via the faithful functor $\ulMSCH \to \SCH; \sY \to \iY$, the above diagram is embedded into the following diagram 
\[
h_{\iX} (\iU) \xrightarrow{i^\o} \prod_{\lambda \in \Lambda} h_{\iX} (\iU_{\lambda}) \overset{a_1^\o}{\underset{a_2^\o}{\rightrightarrows}} \prod_{\lambda, \mu \in \Lambda} h_{\iX} (\iU_{\lambda} \times_{\iU} \iU_{\mu}),
\]
which is an equaliser since the family of morphisms $\{\iU_\lambda \to \iU\}_{\lambda}$ is a $\fppf$-covering (see Remark \ref{rem:woheaviness}). In particular, $i$ is injective since $i^\o$ is.

Take an element $(f_\lambda)_\lambda \in  \prod_{\lambda \in \Lambda} h_{\sX} (\sU_{\lambda})$ satisfying $a_1 (f_\lambda)_\lambda = a_2 (f_\lambda)_\lambda$. 
By calculus of fractions, we may assume that $f_\lambda : \sU_\lambda \to \sX$ is ambient for each $\lambda$.
Moreover, by Proposition \ref{prop:heavyRefine} and Theorem \ref{thm:etaleHeavy}, there exists a refinement of $\{\sU_\lambda \to \sU\}_\lambda$ of the form $\{\sU'_\lambda \to \sU' \to \sU\}_\lambda$, where $\{\sU'_\lambda \to \sU'\}_\lambda$ is a $\ulMfppf^\amb$-covering and $\sU' \to \sU$ belongs to $\ul{\Sigma}$.
Therefore, replacing $\{\sU_\lambda \to \sU\}_\lambda$ by $\{\sU'_\lambda \to \sU'\}_\lambda$, we may assume that $\{\sU_\lambda \to \sU\}_\lambda$ is a $\ulMfppf^\amb$-covering. 
Then, the first diagram above induces 
\[
h_{\hX} (\hU) \xrightarrow{\ol{i}} \prod_{\lambda \in \Lambda} h_{\hX} (\hU_{\lambda}) \overset{}{\underset{}{\rightrightarrows}} \prod_{\lambda, \mu \in \Lambda} h_{\hX} (\hU_{\lambda} \times_{\hU} \hU_{\mu}),
\]
which is an equaliser since $\{\hU_\lambda \to \hU\}_\lambda$ is a covering for the $\fppf$-topology known to be subcanonical. 
Moreover, the element $(f_\lambda)_\lambda$ induces an element $(\ol{f}_\lambda)_\lambda \in \prod_{\lambda \in \Lambda} h_{\hX} (\hU_{\lambda})$, and therefore a unique element $\ol{f} \in h_{\hX} (\hU)$.

We are reduced to showing that the morphism $\ol{f} : \hU \to \hX$ satisfies the modulus condition, i.e., that $\mX \times_{\hX} \hU \subset \mU$. 
But by Lemma \ref{lem:moduluscond-fpqc}, this follows from the modulus condition of $f_\lambda$. 
This finishes the proof.
\end{proof}

\begin{remark}
We do not know whether $\ulMfpqc$-topology is subcanonical. The proof strategy of neither Thm.~\ref{thm:etaleHeavy} nor Thm.~\ref{thm:finHeavy} seems to work for fpqc morphisms.
\end{remark}

\begin{coro}
Let $\ulMAGP^{\amb} \subset \ulMSCH^\amb$ denote the full subcategory whose objects are of the form $(\Spec(A), (f))$ for some ring $A$ and nonzero divisor $f$. Let $τ = \ulMZar, \ulMNis, \ulMet$ or $\ulMfppf$ and let $\alpha$ denote the topology on $\ulMAGP^{\amb}$ induced by the functor $\ulMAGP^{\amb} \to \ulMSCH$.\footnote{So coverings can be detected using points, since a family of morphisms of representables is refinable by a covering if and only if its is jointly epimorphic, if and only if its image under a conservative family of fibre functors is jointly surjective.} Then Yoneda induces a fully faithful embedding
\[ \ulMSCH^\qcqs \subseteq \Shv_{\alpha}(\ulMAGP^{\amb}). \]
\end{coro}


\section{Fibre functors}

In this section our goal is to describe fibre functors of various categories of sheaves using pro-objects in general, and modulus pairs in nice cases.
In particular, the description of a conservative family of points with respect to the Nisnevich topology on $\ulMSm_{\sS}$ for $\sS$ qcqs is given in Proposition~\ref{prop:conservativeLocalPairs}.

\begin{prop} \label{prop:conservativeLocalPairs}
Let $\sS$ be a qcqs modulus pair,
let $\sC$ be one of the categories in the left column, 
let $τ$ be one of the classes in the top row, 
and suppose that the $(\sC, τ)$-entry is non-empty.
\[
\begin{array}{cc|cc|c|c}
{_\sC} \diagdown{^τ}    & \Zar & \Nis & \et & \fppf & \rqfh \\
\Sch             & (N)  & (N)  & (N) & (N)   & (N)   \\ 
\Sm              & (F)  & (F)  & (F) &       &       \\ \hline 
\Sch^\min        & (N)  & (N)  & (N) & (N)   & (N)   \\ 
\Qf              &      &      &     &       & (N)   \\ 
\Sm^\min         & (F)  & (F)  & (F) &       &       \\ 
\Fppf            & (F)  & (F)  & (F) & (F)   &       \\
\Et              & (F)  & (F)  & (F) &       &       \\
\Op              & (F)  &      &     &       &       \\ 
\end{array}
\]
In case (N) assume $\iS$ is Noetherian.

Then a morphism of presheaves 
\[ F \to G \]
on $\ulM\sC_\sS$ becomes an isomorphism of $\ulMtau$-sheaves if and only if 
\[ F(\sP) \to G(\sP) \]
is an isomorphism for every 
$\ulP\tau$-local $\uS$-local modulus pair $\sP \in \ulPSCH_\sS$.
Here, we implicitly use the left Kan extension of $F$ to $\ulMSCH_\sS$. 
\end{prop}

\begin{proof}
The $(\Rightarrow)$ direction is an easy exercise.\footnote{Use the definition of left Kan extension, and the facts that (i) an element $s \in F(\sX)$ is mapped to zero in the sheafification if and only if $s|_\sY = 0$ for some covering $\sY \to \sX$, (ii) $F \to G$ becomes a surjective morphism of sheaves if and only if for every $s \in G(\sX)$ there is a covering $\sY \to \sX$ such that $s|_{\sY}$ is in the image of $F(\sY) \to G(\sY)$, and (iii) $\sP$ factors through coverings.}

For $(\Leftarrow)$, first consider the the cases in the lower portion of the table, in which case the morphism in $\ulP \sC$ are minimal. Since $\sS$ is qcqs, coverings families are always refineable by covering families consisting of a finite number of morphisms. Moreover, $\ulM\sC_\sS$ admits fibre products. Hence, by Deligne's theorem, Thm.~\ref{theo:deligne} the collection of fibre functors on $\Shv_{\ulMtau}(\ulM\sC_\sS)$ is conservative. By Prop.~\ref{prop:fibArePro}, the category of fibre functors is equivalent to the category of $\ulMtau$-local pro-objects in $\ulM\sC$. Then by Thm.~\ref{theo:locProIsRep}, $\ulMtau$-local pro-objects are forgetful functors from under categories $\sP / \ulM\sC$ for $\ulPtau$-local $\uS$-local $\sP$. 
Notice that in these cases, it suffices to consider $\sP \to \sS$ which are minimal (we will use this observation soon). 

Now we use the cases in the lower portion of the table to prove the cases in the upper portion. We give the proof for $\sC = \ulMSm_\sS$, $\tau = \Nis$ since this case is of most interest motivically, but the other cases work verbatim.

For each $\sX \in \sC$ the functor $\ulMSm_\sX^\min \to \ulMSm_\sS$ %
is cocontinuous for the $\ulMNis$-topology. That is, restriction 
$\PSh(\ulMSm_\sS) \to \PSh(\ulMSm_\sX^\min)$; $F \mapsto F|_\sX$ 
commutes with sheafification. So the following are equivalent
\begin{enumerate}
 \item $F \to G$ becomes an isomorphism of $\ulMNis$-sheaves.
 \item $F|_{\sX} \to G|_{\sX}$ becomes an isomorphism of $\ulMNis$-sheaves for all $\sX \in \ulMSm_\sS$.
\end{enumerate}
By the previous case, we know that (2) is equivalent to:
\begin{enumerate} \setcounter{enumi}{2}
 \item $F|_{\sX}(\sP) \to G|_{\sX}(\sP)$ is an isomorphism for every \textit{minimal} $\ulPNis$-local $\uS$-local modulus pair $\sP \in \ulPSCH_\sX^\min$.
\end{enumerate}
So it suffices to show that for a given presheaf $F$ on $\ulMSm_\sS$, modulus pair $\sX \in \ulMSm_\sS$, and 
$\ulPNis$-local $\uS$-local 
$\sP \in \ulPSCH^\min_\sX$ we have $F(\sP) = F|_X(\sP)$. For this it suffices to show that the functor $\sP / \ulMSm^\min_\sX \to \sP / \ulMSm_\sS$ between indexing categories is initial cf.\cite[Tag 04E6, 04E7, 09WP]{stacks-project}. 

The functor $\sP / \ulMSm_\sX \to \sP / \ulMSm_\sS$ is initial due to the existence of a right adjoint, and the fact that $\sP / \ulMSm_\sX$ is filtered. The functor $\sP / \ulMSm_\sX^\min \to \sP / \ulMSm_\sX$ is initial by fully faithfulness and because $\sP \to \sX$ is minimal, cf.~Lem.~\ref{lemm:minMinMin}.
\end{proof}

\begin{lemm} \label{lemm:minMinMin}
Suppose that $\sX \to \sY \to \sZ$ are ambient morphisms of modulus pairs such that $\sX \to \sZ$ is minimal. Then there exists a quasi-compact open immersion $\hU \to \hY$ and a factorisation $\sX \to \sU \to \sY \to \sZ$ such that both $\sX \to \sU$ and $\sU \to \sZ$ are minimal, where $\sU = (\hU, \mY|_{\hU})$.
\end{lemm}

\begin{proof}
By definition of admissible, there exists an effective Cartier divisor $D$ on $\hY$ such that $\mZ|_{\hY} = \mY + D$. Set $\hU = \hY \setminus D$. Then $\sU \to \sZ$ is certainly minimal. Moreover, by minimality of $\sX \to \sZ$ we have 
\[ \mX \geq \mY|_{\hX} = \mZ|_{\hX} + D|_{\hX} \geq \mZ|_{\hX} = \mX, \]
so we conclude that $D|_{\hX} = 0$, and therefore $\hX \to \hY$ factors through $\hU$, and $\sX \to \sU$ is also minimal.
\end{proof}

\begin{rema} \label{rema:canAssumeNoethInt}
With the same proof, and only slightly more work, one can show that in the $(\Sm, \Nis)$-case of Prop.~\ref{prop:conservativeLocalPairs}, it suffices to consider those 
$\ulPNis$-local $\uS$-local 
$\sP$ such that $\iP = \Spec(\OO_{X, x}^h)$ for some $x \in \iX$, $\sX \in \ulMSm_\sS$. In particular, if $\iS$ is noetherian (resp. regular), then it suffices to consider $\sP$ with $\iP$ noetherian (resp. regular). 

A similar observation applies to all cases in the top of the table.
\end{rema}


\section{Cocontinuity of ${\protect\ul{\omega}}$.} \label{sec:extending}

In this section we show for various topologies that coverings of schemes can (up to refinement) be extended to coverings of modulus pairs. This implies that $\ul{\omega}: \ulMSch_{\sS} \to \Sch_{\iS}$ and its variations is \emph{cocontinuous},%
\marginpar{\fbox{cocontinuous}}%
\index{cocontinuous} %
\cite[{\S}III.2]{SGA41}, for such topologies and consequently, $\ul{\omega}_*: \Shv(\Sch_{\iS}) \to \Shv(\ulMSch_{\sS})$ is exact, and right Kan extension preserves sheaves, (so $\ul{\omega}_*$ admits a right adjoint), \cite[{\S}III Prop.2.3]{SGA41}.

\begin{prop}[cf. \cite{kami}, \cite{kmsy1}] \label{prop:extendZar}
Let $\sS$ be a qcqs modulus pair and 
$\{\iV_i \to \iS\}_{i \in I}$ 
a Zariski covering. Then there exists an admissible blowup 
$\sY \to \sS$ 
and a 
Zariski 
covering 
$\{\hW_j \to \hY\}_{j \in J}$ 
such that 
$\{\iW_j \to \iS\}$ refines $\{\iV_i \to \iS\}$.
\end{prop}

\begin{rema}
The proof strategy used below for the other topologies is equally valid in this case, but we instead give a proof along the lines of \cite{kami} and \cite{kmsy1} which doesn't use points.
\end{rema}

\begin{proof}
First notice that we can replace the original cover with a refinement. In particular, lets assume all $\iV_i$ are affine with quasi-compact inclusions, so the inclusions $\iV_i \cap \iV_{i'} \to \iV_i$ are also all quasi-compact. Since $\sS$ is quasi-compact, we can also assume there are only finitely many $\iV_i$.

Now note that it's enough to prove: If 
$\{\iZ_i \to \iS\}_{i = 1}^n$ %
is finite collection of coherent closed subschemes satisfying 
$\cap_{i \in I} \iZ_i = \varnothing$, %
then there exists an admissible blowup $\sY \to \sS$ and a collection of closed subspaces
$\{\hA_i \to \hY\}_{i = 1}^n$ %
satisfying 
$\cap \hA_i = \varnothing$ and %
$\hA_{i} \cap \iY = \iZ_i$. 

Since $\iS$ is quasi-compact, and $\hS$ qcqs, we can find closed subschemes $\ol{{\iZ_i}}$ in $\hS$ whose defining sheaf of ideas is coherent, such that $\ol{{\iZ_i}} \cap \iS = \iZ_i$, \cite[Prop.6.9.14(i)]{EGAI}. Set $F = \cap_{i \in I} \ol{{\iZ_i}}$, let $\hY := \Bl_{\hS} F \to \hS$ be the blowup of $\hS$ in $F$ (which will be proper since all $\ol{{\iZ_i}}$ and therefore $F$ is defined by a coherent sheaf of ideals), and $\hA_i = \Bl_{\ol{{\iZ_i}}} F \to \ol{{\iZ_i}}$ be the strict transforms of the $\ol{{\iZ_i}}$. We claim that this is a solution.

Indeed, it suffices to consider the affine case, say 
$\hS = \Spec(R)$, $\iS = \Spec(R[f^{-1}])$, and 
$\iZ_i = \Spec(R[f^{-1}] / \langle g_{i1}, \dots, g_{im} \rangle)$. We can assume the generators of the ideal are in $R \subseteq  R[f^{-1}]$, so the ideal of 
$\ol{{\iZ_i}}$ is $\langle g_{i1}, \dots, g_{im} \rangle \subseteq R$, the ideal of 
$F = \cap \ol{{\iZ_i}}$ is $\langle g_{ij} \rangle_{i,j} \subseteq R$ and the blowup is covered by affines of the form 
$R[\tfrac{g_{11}}{g_{ij}}, \dots, \tfrac{g_{nm}}{g_{ij}}] \subseteq R[\tfrac{1}{g_{ij}}]$. %
The strict transform $\hA_k$ of $\ol{{\iZ_k}}$ in this affine is empty if $k = i$ (because then $g_{ij}$ is zero on $\ol{{\iZ_k}}$) and otherwise is cut out by $\langle \tfrac{g_{k1}}{g_{i1}}, \dots, \tfrac{g_{km}}{g_{im}}\rangle$. However, just the former observation is enough to conclude that the intersection $\cap \hA_i$ of the strict transforms is empty in every one of these covering affines, and therefore is globally empty. Next, since $\cap \iZ_i = \varnothing$, the locus of the blowup is contained in the modulus $F = \cap \ol{{\iZ_i}} \subseteq V(f) = \mS$, so $\hA_{i} \cap \iY = \iZ_i$. So as long as $\mS|_{\Bl_{\hS} F}$ is an effective Cartier divisor, we have indeed defined an admissible abstract blowup. To see this, recall that 
$R[\tfrac{g_{11}}{g_{ij}}, \dots, \tfrac{g_{nm}}{g_{ij}}] \subseteq R[\tfrac{1}{g_{ij}}]$, 
and so $f$ being a nonzero divisor in $R$, passes to $f$ being a nonzero divisor in these latter two rings.\end{proof}

\begin{coro} \label{coro:ZarCocont}
For any qcqs modulus pair $\sS$ the restriction functors
\begin{align*}
\ul{\omega}_*: \Shv_{\Zar}(\SCH^{\qcqs}_{\iS}) &\to \Shv_{\ulMZar}(\ulMSCH^{\qcqs}_{\sS}) \\
\ul{\omega}_*: \Shv_{\Zar}(\Sch_{\iS}) &\to \Shv_{\ulMZar}(\ulMSch_{\sS}) \\
\ul{\omega}_*: \Shv_{\Zar}(\Sm_{\iS}) &\to \Shv_{\ulMZar}(\ulMSm_{\sS}) \\
\ul{\omega}_*: \Shv(\iS_{\Zar}) &\to \Shv_{\ulMZar}(\ulMOp_{\sS}) \\
\ul{\omega}_*(F)(\sX) &= F(\iX) 
\end{align*}
are exact.
\end{coro}

\begin{prop} [cf. \cite{kami}] \label{prop:extendNis}
Let $\sS$ be a modulus pair with $\hS$ qcqs and $\{V^\circ_i \to S^\circ\}$ a Nisnevich covering. Then there exists an admissible blowup $\sY \to \sS$ and a Nisnevich covering $\{\ol{W}_j \to \ol{Y}\}$ such that $\{W^\circ_j \to S^\circ\}$ refines $\{V^\circ_i \to S^\circ\}$.
\end{prop}

\begin{rema} \label{rema:extendqfh}
The same proof works for $\Zar$ (resp. $\fppf$). Replace Prop.~\ref{prop:bigNisPoint} with 
Prop.~\ref{prop:bigZarPoint} (resp. Prop.\ref{prop:bigFppfPoint}) and replace $\ulMEt_\sS$ with $\ulMOp_\sS$ (resp. $\ulMFppf_\sS$). 
\end{rema}

\begin{rema} \label{rema:etNoGood}
The {\'e}tale version of this result is not true. There are no nontrivial $\ulMet$-coverings of $(\Spec(\C[[t]]), (t))$.
\end{rema}

\begin{proof}
Consider $\ulMEt_\sS$ equipped with the $\ulMNis$-topology. Recall that Deligne's theorem gives a (small) conservative set $\Phi$ of fibre functors of $\Shv_{\ulMNis}(\ulMEt_\sS)$, and that formal arguments show that every fibre functor comes from a $\ulMNis$-local pro-object of $\ulMEt_\sS$, Prop.\ref{prop:fibArePro}. 

For each $(\sP_\lambda) \in \Phi$, the scheme $\lim \iP_λ$ is the spectrum of a hensel local ring, cf.Thm.~\ref{theo:locProIsRep}, Prop.~\ref{prop:bigNisPoint}. Therefore the morphism $\lim \iP_λ \to S^\circ$ factors through some $V^\circ_i$. Since $\iV_i \to \iS$ is of finite presentation, we can descend this morphism to some $\iP_λ \to \iV_i$, \cite[Prop.8.13.1]{EGAIV3}.

Make such a choice of $λ$ for all elements of $\Phi$ and consider the induced family of morphisms of modulus pairs 
$\{\sP_\lambda \to \sS\}_{\sP \in \Phi}$. 
This is a $\ulMNis$-covering, since each pro-object in $\Phi$ factors through it, and the family $\{P_\lambda^\circ \to \sS^\circ\}_{P \in \Phi}$ refines $\{V^\circ_i \to S^\circ\}$ by construction. Since $\{P_\lambda \to \sS\}_{P \in \Phi}$ is a $\ulMNis$-covering, by Prop.~\ref{prop:heavyRefine} and Thm.~\ref{thm:etaleHeavy} there is a refinement of $\{P_\lambda \to \sS\}_{P \in \Phi}$ of the form $\{\ol{W}_j \to \ol{Y} \to \ol{S}\}$ in the statement.
\end{proof}

\begin{coro} \label{coro:NisCocont}
For any qcqs modulus pair $\sS$ the restriction functors
\begin{align*}
\ul{\omega}_*: \Shv_{\fppf}(\SCH^{\qcqs}_{\iS}) &\to \Shv_{\ulMfppf}(\ulMSCH^{\qcqs}_{\sS}), \\
\ul{\omega}_*: \Shv_{\Nis}(\SCH^{\qcqs}_{\iS}) &\to \Shv_{\ulMNis}(\ulMSCH^{\qcqs}_{\sS}), \\
\ul{\omega}_*: \Shv_{\fppf}(\Sch_{\iS}) &\to \Shv_{\ulMfppf}(\ulMSch_{\sS}), \\
\ul{\omega}_*: \Shv_{\Nis}(\Sch_{\iS}) &\to \Shv_{\ulMNis}(\ulMSch_{\sS}), \\
\ul{\omega}_*: \Shv_{\Nis}(\Sm_{\iS}) &\to \Shv_{\ulMNis}(\ulMSm_{\sS}), \\
\ul{\omega}_*: \Shv_{\fppf}(\Fppf_{\iS}) &\to \Shv_{\ulMfppf}(\ulMFppf_{\sS}), \\
\ul{\omega}_*: \Shv_{\Nis}(\Et_{\iS}) &\to \Shv_{\ulMNis}(\ulMEt_{\sS}), \\
\ul{\omega}_*(F)(\sX) &= F(\iX) 
\end{align*}
are exact.
\end{coro}

Next, we will prove an analogous result for $\rqfh$-topology.

\begin{lemm} \label{lemm:extendCl}
Let $\sS$ be a modulus pair and $\{\iV_i \to \iS\}_{i \in I}$ a jointly surjective family of closed immersions. Then there exists a family $\{\sV_i \to \sS\}$ of ambient minimal morphisms such that $\{\hV_i \to \hS\}_{i \in I}$ is a jointly surjective family of closed immersions and the interior of $\sV_i$ is $\iV_i$.
\end{lemm}

\begin{proof}
Let $\hV_i := \ul{\Spec}(im(\OO_{\hS} {\to} \OO_{\iV_i}))$ be the scheme theoretic closure of $\iV_i$ in $\hS$. It is enough to show that $\{ \hV_i \to \hS\}$ is jointly surjective and $\mS|_{\hV_i}$ are effective Cartier divisors. Both of these conditions can be checked locally on $\hS$ so let $\hS = \Spec(A)$, $\mS = \Spec(A/f)$ and $\iV_i = \Spec(A[f^{-1}]/I_i)$. Clearly, $f$ is a nonzero divisor in $im(A \to A[f^{-1}]/I_i) \subseteq A[f^{-1}]/I_i$, since it's invertible in $A[f^{-1}]/I_i$. For surjectivity, we want to see that $\cap ker(A \to A[f^{-1}]/I_i)$ contains only nilpotents of $A$. If $a$ is in this intersection, then for all $i$ we have $a \in I_i$, but $\{\iV_i \to \iS\}$ is surjective, so $a$ is nilpotent in $A[f^{-1}]$, so there is some $n, m$ with $f^na^m = 0$ in $A$, but $f$ is a nonzero divisor of $A$, so $a$ is nilpotent in $A$.
\end{proof}

\begin{lemm} \label{lemm:extendFS}
Let $\sS$ be a qcqs modulus pair and $\iT \to \iS$ a finite surjective morphism. Then there exists an ambient morphism $\sT \to \sS$ such that $\hT \to \hS$ is finite surjective, and $\iT \to \iS$ is its interior.
\end{lemm}

\begin{proof}
Let $\sA$ be the integral closure of $\OO_{\hS}$ in $\OO_{\iT}$ (as sheaves of $\OO_{\ol{S}}$-algebras). Since $\hS$ is qcqs, there exists a coherent sub-$\OO_{\hS}$-algebra $\sB$ containing generators of the ideal of $\mS$, such that $\sB|_{\iS} = \OO_{\iT}$, \cite[Prop.6.9.14(i)]{EGAI}. Set $\hT = \ul{\Spec}(\sB)$. Since $\sB$ is a sub-algebra of $\OO_{\iT}$, any local generator of $\mS$ remains a nonzero divisor in $\sB$, so $\mS|_{\hT}$ is an effective Cartier divisor. To check surjectivity, it suffices to note that $\iS$ is dense in $\hS$ by Prop.~\ref{prop:always-dense},
since finite morphisms satisfy the going up property.
\end{proof}

\begin{prop} \label{prop:rqfhExtend}
Let $\sS$ be a qcqs modulus pair with $\hS$ qcqs and $\{\iX_i \to \iS\}_{i \in I}$ a $\rqfh$-covering. Then there exists a finite composition of $\ulPrqfh$-coverings and \abb{}s $\{\sV_i \to \sS\}_{j \in J}$ such that $\{\iV_i \to \iS\}_{j \in J}$ refines $\{\iX_i \to \iS\}_{i \in I}$.
\end{prop}

\begin{rema}
We could also have used the strategy of the Nisnevich case above, but then we would have to assume $\iS$ is Noetherian, since this is needed in Thm.~\ref{theo:locProIsRep} in the $\rqfh$-case. 
\end{rema}

\begin{proof}
By definition $\rqfh$-coverings are compositions of finitely many Zariski and finite coverings. So it suffices to treat these two cases. The Zariski case was treated above, Prop.~\ref{prop:extendZar}, so suppose that $\{\iX_i \to \iS\}_{i \in I}$ is a jointly surjective family of finite morphisms. Setting $\iZ_i := im(\iX_i)$ we obtain a jointly surjective family of closed immersions $\{\iZ_i \to \iS\}_{i \in I}$. By Lem.~\ref{lemm:extendCl} this family is the interior of a family $\{\sZ_i \to \sS\}_{i \in I}$ of minimal ambient morphisms such that $\{\hZ_i \to \hS\}_{i \in I}$ is a jointly surjective family of closed immersions. Then by Lem.~\ref{lemm:extendFS} for each $i$, there exists an ambient minimal morphism $\sX_i \to \sZ_i$ whose interior is $\iX_i \to \iZ_i$ and such that $\hX_i \to \hZ_i$ is finite and surjective. Hence, $\{\sX_i \to \sS\}_{i \in I}$ is a finite covering whose interior is $\{\iX_i \to \iS\}_{i \in I}$.
\end{proof}

\begin{coro} \label{coro:qfhCocont}
For any qcqs modulus pair $\sS$ the restriction functors
\begin{align*}
\ul{\omega}_*: \Shv_{\rqfh}(\SCH^{\qcqs}_{\iS}) &\to \Shv_{\ulMrqfh}(\ulMSCH^{\qcqs}_{\sS}) \\
\ul{\omega}_*: \Shv_{\rqfh}(\Sch_{\iS}) &\to \Shv_{\ulMrqfh}(\ulMSch_{\sS}) \\
\ul{\omega}_*: \Shv_{\rqfh}(\Qf_{\iS}) &\to \Shv_{\ulMrqfh}(\ulMQf_{\sS}) \\
\ul{\omega}_*(F)(\sX) &= F(\iX) 
\end{align*}
are exact.
\end{coro}


\section{Finite homotopy dimension and hypercompleteness} \label{sec:hypercompleteness}

In Voevodsky's theory of motives, Zariski and Nisnevich topologies played a fundamental role. 
In this section, we will see that the topologies $\ulMZar$ and $\ulMNis$ satisfy some properties which are analogous to those of Zariski and Nisnevich topologies. 
Moreover, we will prove that the non-ambient small sites $\sX_{\ulMZar}$ and $\sX_{\ulMNis}$ can be ``approximated'' by the ambient sites, and we will deduce some finiteness properties (e.g., finite homotopy dimension and hypercompleteness) of $\sX_{\ulMZar}$ and $\sX_{\ulMNis}$ from those of ambient sites. 

In the following, we need to employ the language of $\infty$-categories.
In particular, we make a small remark about Grothendieck topologies on $\infty$-categories. 

\begin{remark}\label{rem:Gtop-infty-ordinary}
For any $\infty$-category $\sC$, there exists a natural bijection between the collection of Grothendieck topologies on $\sC$ and Grothendieck topologies on the homotopy category $h\sC$.
This is basically because choosing covering sieves amounts to choosing full ($\infty$)-subcategories, i.e., collection of objects. 
In particular, if $\sC = N(\sD)$ for an ordinary category $\sD$, then choosing a Grothendieck topology on $\sC$ is equivalent to choosing one on $\sD$ since $hN(\sD) = \sD$.
See \cite[Rem.~6.2.2.3]{HTT} for more details.
\end{remark}

\subsection{Ambient approximation of non-ambient small sites}

In the following, we restrict our attention to finitary sites, cf.Def.~\ref{def:finitarysites}.
By Rem.~\ref{rem:Gtop-infty-ordinary}, defining a Grothendieck topology on an ordinary site $\sC$ is equivalent to defining a Grothendieck topology on the nerve of $\sC$ (More generally, a Grothendieck topology on an $\infty$-category is just a Grothendieck topology on the homotopy category $h\sC$).
So, we can always regard a site in ordinary sense as a site in $\infty$-sense. 

\begin{defi}\label{def:finitarysites}
A \emph{finitary site} is a small $\infty$-category $\sC$ with all finite limits with a Grothendieck topology such that every covering sieve admits a refinement which is generated by a finite number of elements. 
\end{defi}

In the following, for a modulus pair $\sS$ and $\tau \in \{\Zar,\Nis,\et\}$, we write $\sS_{\ulM\tau} := (\ulM\sC_{\sS},\tau)$ for the small site introduced in Def.~\ref{defi:modSitesCat} and Def.~\ref{defi:mtop}, where we set $\sC = \Op$ when $\tau = \Zar$ and $\sC = \Et$ otherwise.

\begin{lemma}
If $\sS$ is a qcqs modulus pair, then the following assertions hold:
\begin{enumerate}
\item The underlying category of $\sX_{\ulM\sigma}$ admits all finite limits. 
\item $\sX_{\ulM\sigma}$ is a finitary site.
\end{enumerate}
\end{lemma}

\begin{proof}
(1) follows from Thm.~\ref{thm:pullback}. 
(2) follows from the quasi-compactness of $\hX$ and the refinability results (Proposition \ref{prop:heavyRefine} and Theorem \ref{thm:etaleHeavy}).
\end{proof}

The goal of this subsection is the following reult.

\begin{prop}\label{prop:amb-approx}
Let $\tau \in \{\Zar, \Nis, \et\}$, and $\sS$ a qcqs modulus pair. 
Then there exists an equivalence in the $\infty$-category of finitary sites
\[
F : \varinjlim_{p:\sS' \to \sX} N(\hS'_{\tau}) \xrightarrow{\sim} N(\sS_{\ulM\tau}),
\]
where $\sS' \to \sS$ runs over {\aab}s of $\sS$ and $N$ denotes the nerve functor. 
\end{prop}

\begin{proof}
In Prop.~\ref{prop:SmallSitesApprox}, we have already constructed an equivalence of (1-)sites
\[
\varinjlim_{p:\sS' \to \sX} \hS'_{\tau} \xrightarrow{\sim} \sS_{\ulM\tau}.
\] 
Therefore, by Rem.~\ref{rem:Gtop-infty-ordinary}, it suffices to show an equivalence of $\infty$-categories
\[
N\left( \varinjlim_{p:\sS' \to \sX} \hS'_{\tau} \right) \simeq \varinjlim_{p:\sS' \to \sX} N(\hS'_{\tau}).
\]
Recalling that the transition morphism in the colimit on the right hand side is induced by functors of $1$-categories, the $\infty$-categorical colimit $\varinjlim_p N(\sC_p)$ can be computed by the homotopy colimit in $\mathbf{sSet}$ with respect to Joyal model structure. 
Since small filtered colimits of simplicial sets preserve categorical weak equivalences (=weak equivalence in Joyal's model structure) by \cite[Corollary 3.9.8]{cisinski_2019}, small filtered colimits (i.e., level-wise colimit of sets) already exhibits homotopy colimits in Joyal's model structure. 
Therefore, we have \[(\varinjlim_p N(\sC_p))_n = \varinjlim_p (N(\sC_p)_n) = \varinjlim_p \Hom_{\mathbf{Cat}} (\Delta[n], \sC_p) =^
\dag \Hom_{\mathbf{Cat}} (\Delta[n], \sC) = N(\sC)_n,\]
where $=^\dag$ follows from the compactness of $\Delta[n]$ and $\varinjlim_p \sC_p = \sC$ in $\mathbf{Cat}$.
This finishes the proof.
\end{proof}

\subsection{Finite homotopy dimension and hypercompleteness}

\begin{defi}
A finitary site $\sC$ is called \emph{excisive} if the full subcategory $\Shv(\sC) \subset \PSh(\sC)$ is closed under filtered colimits. 
\end{defi}

\begin{exam} \label{exam:ZarNisFinExc}
For any qcqs modulus pair $\sX$, the finitary sites $\hX_{\Zar}$ and $\hX_{\Nis}$ are excisive.
This follows from the fact that Zariski and Nisnevich topologies are generated by nice cd-structures (cf. \cite[Thm.~3.2.5]{AHW15}
). 
\end{exam}

\begin{prop}\label{prop:aprox-homdim}
Suppose $\sC_i$ is a filtered system of finitary excisive sites with colimit $\sC$ (in the infinity category of small infinity categories).
If each $\infty$-topos $\Shv(\sC_i)$ has homotopy dimension $\leq d$, then so does $\Shv(\sC)$.
\end{prop}

\begin{proof}
This is proven in \cite[Corollary 3.11]{2019arXiv190506611C}. 
\end{proof}

\begin{thm}
For any qcqs modulus pair $\sX$ such that the ambient space $\hX$ has Krull dimension $\leq d$, the $\infty$-topoi $\Shv (\hX_{\Zar})$, $\Shv (\hX_{\Nis})$, $\Shv (\sX_{\ulMZar})$ and $\Shv (\sX_{\ulMNis})$ have homotopy dimension $\leq d$.
In particular, they are hypercomplete, and have cohomological dimension $\leq d$.
\end{thm}

\begin{proof}
The $\infty$-topoi $ \Shv (\hX_{\Zar})$ and $\Shv (\hX_{\Nis})$ have homotopy dimension $\leq d$ by \cite[Theorem 3.12, 3.17]{2019arXiv190506611C}. 
Then, Proposition \ref{prop:amb-approx}, \ref{exam:ZarNisFinExc} and \ref{prop:aprox-homdim} imply that $\Shv (\sX_{\ulMZar})$ and $\Shv (\sX_{\ulMNis})$ have homotopy dimension $\leq d$.
The second assertion follows from the fact that any $\infty$-topos locally of homotopy dimension $\leq n$ for some integer $n$ is hypercomplete \cite[Cor.7.2.1.12]{HTT}, and the last assertion is by \cite[Cor. 7.2.2.30]{HTT}.
\end{proof}


\chapter{Relative cycles with modulus} \label{chap:relCyc}

\section{Relative cycles without modulus}

First let's recall some of the general Suslin--Voevodsky theory of relative cycles. 

In broad strokes, according to Suslin--Voevodsky, a relative cycle (of dimension $r$) of a separated morphism of finite presentation $X \to S$ should be a formal sum 
\[ \sum_{i = 1}^\nu n_iZ_i \]
 of closed integral subschemes $Z \subseteq X$ (such that generically $Z \to S$ has relative dimension $r$) which has a ``well-defined'' specialisation to the fibre $X_s = s \times_S X$ at every point $s \in S$. Here ``well-defined'' means invariant of the ``direction'' that you approach $s$ from. This ``direction'' is formalised using discrete valuation rings in Suslin--Voevodsky's notion of ``fat points'', \cite[Def.3.1.1]{SVRelCyc}. 

An alternative, equivalent approach, is to ask that pullbacks  
\[ f^\circledast: z(X/S, r) \to z(S' \times_S X/S', r) \]
along arbitrary morphisms $S' \to S$ are well-defined, not just inclusions of points. 
One simple way of formalising this alternative approach to the theory is to begin with three simple demands on the $\sum n_i Z_i \in z(X/S, r)$ and the $f^\circledast$.
\begin{axio} \label{axio:relativeCycles} \ 
\begin{enumerate}
 \item[(Gen)] The generic point of each $Z_i$ must lie over some generic point of $S$.
 \item[(Red)] If $i: S_{\red} \to S$ is the canonical inclusion, then $i^\circledast$ is the canonical identification. 
 \item[(Pla)] If $k$ is a field and $f: \Spec(k) \to S$ is morphism whose image lies in the flat locus of $\sqcup Z_i \to S$, then 
pull-back is performed by scheme-theoretic fibre product, then counting multiplicity. More concretely,
 \[ f^\circledast\sum n_iZ_i = \sum n_im_{ij} \overline{\{w_{ij}\}} \]
where $w_{ij}$ are the generic points of $k \times_S Z$ and $m_{ij} = \length (\OO_{k \times_S Z, w_{ij}})$. 
\end{enumerate}
\end{axio}

\begin{property} \label{prop:relativeCycles}
Two easy consequences of the axioms (Gen), (Red), (Pla) are:
\begin{enumerate}
 \item[(Bir)] If $f: S' \to S$ is a proper birational morphism between reduced schemes, then $f^\circledast$ is the canonical identification. (Pullback along the generic points).
 \item[(EdC)] If $f: \Spec(L) \to \Spec(k)$ is a finite field extension then $f^\circledast$ is injective.
\end{enumerate}
\end{property}

In fact, the above three axioms (Gen), (Red), (Pla) uniquely determine pull-backs for a general morphism $f: S' \to S$ as follows:
\begin{enumerate}
 \item By (Red) we can assume $S'$ and $S$ are reduced. 
 \item Each component of $f^\circledast\sum n_iZ_i$ will be flat over any generic point of $S'$ so we can assume $S'$ is $Spec(k)$ for some field $k$. 
 \item By Raynaud--Gruson flatification and (Bir), up to replacing $k$ with some extension $L$ and using (EdC), we can assume each $Z_i \to S$ is flat. In particular,  the image of $f$ lies in the flat locus of each $Z_i \to S$.
 \item Use (Pla).
\end{enumerate}

So if we make the reasonable demands (Gen), (Red), (Pla), our hands are tied for all $f^{\circledast}$. 

\begin{defi}
The presheaf $z(X/S, r)= z(- \times_S X/-, r)$%
\marginpar{\fbox{$z(X/S, r)(-)$}}%
\index[not]{$z(X/S, r)(-)$} %
on the category of finite presentation $S$-schemes as the largest presheaf satisfying (Gen), (Red), and (Pla), (and such that for fields, the $Z_i$ have dimension $r$). 
\end{defi}

More generally, there are four main flavours 
\marginpar{\fbox{$z_{equi}(X/S, r)$}}%
\index[not]{$z_{equi}(X/S, r)$} %
\marginpar{\fbox{$c(X/S, r)$}}%
\index[not]{$c(X/S, r)$} %
\marginpar{\fbox{$c_{equi}(X/S, r)$}}%
\index[not]{$c_{equi}(X/S, r)$} %
\[ z(X/S, r), \qquad z_{equi}(X/S, r), \qquad c(X/S, r), \qquad c_{equi}(X/S, r) \]%
of groups of relative cycles according to various requirements on the components $Z_i$. To wit, $c$ vs $z$ depends whether $Z \to S$ is required to be proper or not, and $(-)_{equi}$ demands $Z \to S$ be equidimensional, cf.\cite[Def.3.1.3, and the paragraph after Lem.3.3.9]{SVRelCyc}.

When the base is noetherian and regular, we get the full free abelian group.

\begin{thm}[{\cite[Cor.3.4.5, 3.4.6]{SVRelCyc}}]
If $S$ is regular and noetherian (e.g., $S$ is a dvr), the group $z_{equi}(X/S, r)$ is the free abelian group generated by points in the generic fibres of $X \to S$ whose closures in $X$ are equidimensional of relative dimension $r$ over $S$.

The group $c_{equi}(X/S, r)$ is the free abelian group generated by points in the generic fibres of $X \to S$ whose closures in $X$ are proper and equidimensional of relative dimension $r$ over $S$.
\end{thm}

For $S$ not regular, a priori the groups of relative cycles will be subgroups of such free groups.

As with usual cycle groups these groups of relative cycles are functorial in $X$ (for fixed $S$) contravariantly for flat morphisms and covariantly for proper morphisms, \cite[Cor.3.6.3, Lem.3.6.4, Prop.3.6.5]{SVRelCyc}. Specifically, proper push-forward along a morphism $f$ sends a formal sum $\sum n_i z_i$ (of points $z_i$ in generic fibres) to the formal sum
\[ f_*(\sum n_iz_i) = \sum n_i [k(z_i):k(f(z_i))] f(z_i) \] 
where $[k(z_i):k(f(z_i))]$ is set to zero if $k(z_i) / k(f(z_i))$ is not finite. 

Using this additional functoriality, given three $S$-schemes $W, X, Y$, one can define a ``composition'' of cycles
\[ 
- \circ- : 
c_{equi}(X{\times_S}Y / X, 0)
\times 
c_{equi}(W{\times_S}X / W, 0) 
\to 
c_{equi}(W{\times_S}Y / W, 0).
\]
Explicitly, given $\alpha = \sum n_i A_i \in c_{equi}(W{\times_S}X / W, 0)$ and $\beta \in c_{equi}(X{\times_S}Y / X, 0)$, let $\phi_i: A_i \to W{\times_S}X \to X$ be the composition of the canonical morphisms. 
Then the cycle $\beta \in c_{equi}(X{\times_S}Y / X, 0)$ gives rise to cycles 
\[ (\phi_i)^\circledast \beta \in c_{equi}(A_i {\times_S}Y / A_i, 0). \]
It turns out these cycles actually lives in $c_{equi}(A_i {\times_S}Y / W, 0)$ and using the morphisms $\theta_i: A_i{\times_S}Y \to W{\times_S}Y$ we get a cycle
\[  \beta \circ \alpha := \sum_i n_i (\theta_i)_*(\pi\iota_i)^\circledast \beta, \]
or more explicitly, 
\begin{equation} \label{eq:compOne}
 \beta \circ \alpha := \sum_i n_i (A_i{\times_S}Y \to W{\times_S}Y)_*(A_i{\to}X)^\circledast \beta.
\end{equation}%
\marginpar{\fbox{composition of cycles}}%
\index{composition of cycles} %
In the notation of \cite{SVRelCyc} this would rather be written as 
\begin{equation} \label{eq:compTwo}
 \beta \circ \alpha := \pi'_*Cor_{W{\times}X{\times}Y/W{\times}X}(cycl(\pi)(\beta), \alpha) 
\end{equation}
where $\pi: W\times X \to X$ and $\pi': W \times X \times Y \to W \times Y$ are the canonical projections, and we have written $\times$ instead of ${\times_S}$. The operation $Cor(-,-)$ is discussed further around Equation~\ref{equa:Cor}.

\begin{defi}
Let $S$ be a separated Noetherian scheme. We will write $\Cor_S$%
\marginpar{\fbox{$\Cor_S$}}%
\index[not]{$\Cor_S$} %
for the category whose objects are separated $S$-schemes of finite type, and 
\[ \hom_{\Cor_S}(X, Y) = c_{equi}(X{\times_S}Y / X, 0), \]
with composition given as in \eqref{eq:compOne} and \eqref{eq:compTwo}.
\end{defi}

The category $\Cor_S$ is equipped with a monoidal structure which is the usual fibre product of schemes on objects, and on morphisms can be calculated as follows, \cite[After Lem.1.4.15, Before Rem.2.1.3]{Ivo05}. Here we start omitting $\times_S$ for readability. Suppose $\alpha \in c_{equi}(XY/X, 0)$ and $\beta \in c_{equi}(X'Y'/X', 0)$ are two morphisms in $\Cor_S$. Consider the cycle theoretic pullback of $\alpha$ to some $\sum n_iA_i \in c_{equi}(XYX'/XX', 0)$ and consider the canonical morphisms $\phi_i: A_i \to XYX' \stackrel{proj.}{\to} X'$. Pulling back $\beta$ along this morphism we get $\phi_i^\circledast \beta \in c_{equi}(A_iY'/A_i, 0)$ which we can pushforward along the closed immersion $\theta_i: A_iY' \to XYX'Y'$ to get a cycle which, one can check, actually lies in $c_{equi}(XYX'Y' / XY, 0)$. Or in other words, $(\theta_i)_*\phi_i^\circledast \beta$ is a morphism $X \times_S Y \to X' \times_S Y'$ of $\Cor_S$. All in all,
\[ \alpha \otimes \beta := \sum n_i (\theta_i)_*\phi_i^\circledast \beta \in c_{equi}(XYX'Y' / XY, 0) \]
or more explicitly,
\begin{equation}
\alpha \otimes \beta = \sum n_i (A_iY' \to XYX'Y')_* (A_i \to X')^\circledast \beta. 
\end{equation}
In the notation of \cite{SVRelCyc} this would rather be written as
\begin{equation}
\alpha \otimes \beta = Cor(cycl(\pi')(\beta), cycl(\pi)(\alpha)). 
\end{equation}
where $\pi: XX' \to X$, and $\pi': XYX' \to Y$ are the canonical projections. The operation $Cor(-,-)$ is discussed further around Equation~\ref{equa:Cor}.

Of course there are a number of axioms to check to ensure that the above definitions of composition and tensor product actually define a monoidal category. These are checked in \cite[Section 2.1]{Ivo05} in notation similar to \cite{SVRelCyc}, and also in \cite{CD19} in a re-interpretted form.

\section{Relative cycles with modulus}

Now we will develop the theory of relative cycles for modulus pairs. 

\begin{defn} \label{defn:presheavesOfRelCyc}
Let $\sX \to \sS \in \ulPSCH$ be an ambient morphism of finite type of 
qc separated modulus pairs with Noetherian interior. 
\begin{enumerate}
 \item A \emph{relative cycle}%
\marginpar{\fbox{relative cycle}}%
\index{relative cycle} %
 is an element of $c_{equi}(X^\circ / S^\circ, 0)$.

\item A relative cycle $\sum n_i Z_i \in c_{equi}(X^\circ / S^\circ, 0)$ is \emph{left proper}%
\marginpar{\fbox{left proper}}%
\index{relative cycle!left proper} %
 if each $\ol{Z}_i \to \ol{S}$ is proper. Here $\ol{Z}_i$ is the closure of $Z_i$ in $\ol{X}$. 

 \item Suppose $\ol{S}$ is the spectrum of a valuation ring, $S^\circ = \ol{S} \setminus S^\infty$ is the generic point, let $\sum n_iz_i \in c_{equi}(X^\circ / S^\circ, 0)$ be a relative cycle, and let $\widetilde{z}_{ij}$ be the spectra of the valuation rings of the extensions of the valuation of $S^\circ$ to $z_i$. We say $\alpha$ is \emph{admissible}%
\marginpar{\fbox{admissible}}%
\index{relative cycle!admissible} %
if for every\footnote{Of course, if $\alpha$ is not left proper, there may be some $i,j$ which do not admit such a diagram.} $i, j$, and every diagram of the form
\[ \xymatrix{
z_i \ar[r] \ar[d] & \widetilde{z}_{ij} \ar[r] \ar[d] & \ol{X} \ar[dl]  \\
S^\circ \ar[r] & \ol{S} & 
} \]
we have $X^\infty|_{\widetilde{z}_{ij}} = S^\infty|_{\widetilde{z}_{ij}}$.

 \item In general a relative cycle is \emph{admissible} if for every minimal (ambient) morphism $f: \sP \to \sS$ with $\hP$ the spectrum of a valuation ring and $\iP$ is the open point, the cycle $f^\circledast \sum n_iZ_i$ is admissible.
\end{enumerate}
We will write 
\[ \mc(\sX/\sS) \]%
\marginpar{\fbox{$\mc(\sX/\sS)$}}%
\index[not]{$\mc(\sX/\sS)$} %
for the group of left proper admissible relative cycles.
\end{defn}

\begin{rema}
Many variations on the above (4) are possible: we relax it by requiring that $\iP$ is actually a point of $\iS$, or alternatively, requiring $k(\iP)$ to be algebraically closed, Lem.~\ref{lemm:admAlgClo}, or requiring that the valuations be henselian, or strictly henselisan, \dots
\end{rema}

\begin{rema}
The fact that there is an equality $\mS|_{\widetilde{z}_{ij}} = \mX|_{\widetilde{z}_{ij}}$ rather than the inequality $\geq$ is explained heuristically by Cor.~\ref{coro:cartSquaQfhRelCyc} which describes $\ZZtr(\sX/\sS)$ as the intersection of $(\ZZ \hom_{\ulMSch_{\sS}}(-, \sX))_{\ulMrqfh}(\sS)$ %
and $\ZZtr(\iX/\iS)$ in $\ZZ_{\qfh}(\iX/\iS)(\iS)$.
See also \S \ref{sec:KMSY-comparison} below.
\end{rema}

\begin{prop}[Functoriality] \label{prop:cycFun}
Suppose that $f: \sT \to \sS, \sX \to \sS \in \ulPSCH$ are ambient morphisms of modulus pairs and $\sZ \in \mc(\sX/\sS) \subseteq c_{equi}(\iX / \iS, 0)$. Then the relative cycle theoretic pullback $f^\circledast \sZ \in c_{equi}(\iT \times_{\iS} \iX / \iT, 0)$ is both left proper and admissible with respect to $\sT \ambtimes[\sS] \sX \to \sT$. In other words, the cycle theoretic pullback
\[ f^\circledast: 
c_{equi}(\sX^\circ / \sS^\circ, 0) \to 
c_{equi}(\iT \times_{\sS^\circ} \iX / \iT, 0) \]
descends to a morphism of subgroups
\[ f^\circledast: \mc(\sX/\sS)  \to \mc(\sT \ambtimes[\sS] \sX / \sT) \]  
\end{prop}

\begin{proof}
Let $\sB = \sT \ambtimes[\sS] \sX$. In other words, $\ol{B}$ is the blowup of $\ol{T} \times_{\ol{S}} \ol{X}$ at the intersection of the pullbacks of $T^\infty$ and $X^\infty$, and $B^\infty = T^\infty + X^\infty - E$. Notice that $B^\circ = T^\circ \times_{S^\circ} X^\circ$. 

Let $\sW = f^\circledast \sZ \in c_{equi}(B^\circ / T^\circ, 0)$. Note that the support of $\sW$ is contained in the pullback of the support of $\sZ$, \cite[Lem.3.3.6]{SVRelCyc}. 
In other words, if $\sZ = \sum n_iz_i$ and $\sW = \sum m_jw_j$ then for each $j$, there is some $i$ such that $z_i$ specialises to $g(w_j)$ where $g$ is the morphism $g: B^\circ \to X^\circ$. It follows that if $Z, W$ are the closures of $z_i, w_j$ in $\ol{X}, \ol{B}$ respectively, then we obtain a commutative diagram
\[ \xymatrix@R=12pt{
W \ar[d] \ar[r] & Z \ar[d] \\
\ol{B} \ar[r] \ar[d] & \ol{X} \ar[d] \\
\ol{T} \ar[r] & \ol{S}.
} \]
If $\sZ$ is left proper, then $Z \to \ol{S}$ is proper, so $\ol{T} \times_{\ol{S}} Z \to \ol{T}$ is proper. Since $W \to \ol{B}$ is a closed immersion, $W \to \ol{T} \times_{\ol{S}} Z$ is proper, and it follows that $W \to \ol{T}$ is proper. Hence, $\sW$ is also left proper.

Admissibility is straight-forward since for every $\sP$ with $\ol{P}$ the spectrum of a valuation ring, each morphism $g: \sP \to \sT$ gives rise to a morphism $fg: \sP \to \sS$. So since $(fg)^\circledast \sZ = g^\circledast f^\circledast \sZ$ it follows that $f^\circledast \sZ$ is admissible if $\sZ$ is.
\end{proof}

\begin{lemm}[Admissible blowup invariance in the source] \label{lemm:cycSourceAdmBu}
Let $\sX \to \sS$ be a (separated, finite type) ambient morphism, and $\sX' \to \sX$ an admissible blowup and $\sZ \in c_{equi}(X^\circ / S^\circ, 0) = c_{equi}((X')^\circ / (S')^\circ, 0)$ a relative cycle. 
\begin{enumerate}
 \item $\sZ$ is left proper for $\sX$ if and only if it is left proper for $\sX'$.
 \item $\sZ$ is admissible for $\sX$ if and only if it is admissible for $\sX'$.
\end{enumerate}
Consequently, 
\[ \mc(\sX'/S) = \mc(\sX/S). \]
\end{lemm}

\begin{proof}
 (1): Let $Z', Z$ be the closures of an integral component $Z_i$ of $\sZ$ in $\ol{X}'$, $\ol{X}$ respectively. 
Since $\ol{X}', \ol{X} \to \ol{S}$ are both separated and finite type, the same is true of $Z', Z \to \ol{S}$. Since $\sX' \to \sX$ is an admissible blowup, $\ol{X}' \to \ol{X}$ is a proper surjective morphism, and therefore the induced morphism $Z' \to Z$ is also proper and surjective, and in particular, is surjective and universally closed. Hence, $Z' \to \ol{S}$ is universally closed if and only if $Z \to \ol{S}$ is universally closed.

(2): It suffices to consider the case that $\ol{S}$ is a valuation ring and $S^\circ$ is the generic point. Let $z$ be an integral component of $\sZ$.
The question is about diagrams
\[ \xymatrix{
z \ar[r] \ar[d] & Z \ar[r] \ar[d] & \ol{X} \ar[dl] && z \ar[r] \ar[d] & Z' \ar[r] \ar[d] & \ol{X}' \ar[dl] \\
S^\circ \ar[r] & \ol{S} &  && S^\circ \ar[r] & \ol{S}
} \]
such that $Z, Z'$ are spectrums of an extension of the valuation ring $\OO(\ol{S}) \subseteq k(\ol{S})$ to $k(z)$. 
We are trying to show 
$X^\infty|_{Z} = S^\infty|_{Z}$ for all such diagrams on the left if and only if 
$X'^\infty|_{Z'} = S^\infty|_{Z'}$ for all such diagrams on the right. Here $X'^\infty = X^\infty|_{\ol{X}'}$.

Notice that any given diagram on the right induces one on the left by composition with $\ol{X}' \to \ol{X}$ and any diagram on the left induces one on the right by the valuative criterion for properness. Moreover, this correspondence is bijective, so we are actually only dealing with diagrams of the form
\[ \xymatrix{
z \ar[r] \ar[d] & Z \ar[r] \ar[d] & \ol{X}' \ar[r] & \ol{X} \ar[dll]  \\
S^\circ \ar[r] & \ol{S} & &&
} \]
Hence, its clear that $X^\infty|_{Z} = S^\infty|_{Z}$ if and only if $X'^\infty|_{Z} = S^\infty|_{Z}$ since $X'^\infty = X^\infty|_{\ol{X}'}$.
\end{proof}

\begin{prop}[Admissible blowup invariance in the target] \label{prop:cycTarAdmBu}
Suppose that $f: \sT \to \sS, \sX \to \sS$ are ambient morphisms of modulus pairs with $\sX \to \sS$ separated and finite type, and $\sT \to \sS$ an admissible blowup. Then the induced morphism
\[ \mc(\sX / \sS) \to \mc(\sT \ambtimes[\sS] \sX / \sT) \]
is an isomorphism.
\end{prop}

\begin{proof}
Certainly, since $\sT^\circ = \sS^\circ$ we have $(\sT \times_\sS \sX)^\circ = \sX^\circ$, so $c_{equi}(\sX^\circ / \sS^\circ, 0) = c_{equi}((\sT \times_\sS \sX)^\circ / \sT^\circ, 0)$.

For left properness, let $Z \subseteq \sX^\circ$ be the support of a cycle in this group. Let $Z_1$ be its closure in $\ol{\sX}$ and $Z_2$ the closure in $\ol{B}$, the total space of the modulus pullback $\ol{T} \times_{\ol{S}} \ol{X}$. This produces the following diagram.
\[ \xymatrix{
Z_2 \ar[d]  \ar[rr] && Z_1 \ar[d] \\
\ol{B} \ar[r] & \ol{\sT} \times_{\ol{S}} \ol{X} \ar[r] \ar[d]_{\circ} & \ol{X} \ar[d]^\circ \\
& \ol{T} \ar[r] & \ol{S} 
} \] 
Here, all morphisms except those labelled $\circ$ are proper. If $Z_1 \to \ol{S}$ is proper, then it follows that $Z_2 \to \ol{S}$ and therefore $Z_2 \to \ol{T}$ is proper. On the other hand, suppose $Z_2 \to \ol{T}$ is proper. Then $Z_2 \to \ol{S}$ is proper. But $Z_2 \to Z_1$ is surjective, so it follows that $Z_1 \to \ol{S}$ is proper. So our cycle $Z$ is left proper with respect to $\sX / \sS$ if and only if its left proper with respect to $\sT \ambtimes[\sS] \sX / \sT$. 

For admissibility, first note, as above, that by the valuative criterion for properness, morphisms $\sP \to \sT$ with $\ol{P}$ a valuation ring are in canonical bijection with morphisms $\sP \to \sS$. So we are reduced to the question of whether admissibility is independent of the total space; the two spaces in question being $\ol{P} \times_{\ol{S}} \ol{X}$ and $\ol{P} \times_{\ol{T}} \ol{B}$. However, $\sP \ambtimes[\sT] \sB \to \sP \ambtimes[\sS] \sX$ is an admissible blowup, so the result follows from Lem.~\ref{lemm:cycSourceAdmBu}. 
\end{proof}

\begin{coro}
The presheaves $\mc(\sX/\sS)$ on $\ulPSCH_\sS$ are well-defined for morphisms $\sX/\sS$ in $\ulMSCH$, and descend to presheaves on $\ulMSCH_\sS$.
\end{coro}

\begin{proof}
Invariance under admissible blowup of the source is Lem.~\ref{lemm:cycSourceAdmBu} and invariance under admissible blowup of the target is Prop.~\ref{prop:cycTarAdmBu}.
\end{proof}

Recall that given two separated morphisms $Y \to X \to S$, there are morphisms of groups (in fact presheaves) \cite[Cor.3.7.5]{SVRelCyc}
\[ Cor_{Y/X}: c_{equi}(Y/X, 0) \otimes c_{equi}(X/S, 0) \to c_{equi}(Y/S, 0). \]%
\marginpar{\fbox{$Cor_{Y/X}$}}%
\index[not]{$Cor_{Y/X}$} %
This is defined as 
\begin{equation} \label{equa:Cor}
 Cor_{Y/X}(\sW, \sZ) = \sum n_i (Z_i \times_X Y {\to} Y)_*(Z_i {\to} X)^\circledast \sW 
\end{equation}
where $\sZ = \sum n_iZ_i \in c_{equi}(X/S, 0)$ and $\sW \in c_{equi}(Y/X, 0)$.
Furthermore, given $g: T \to S$ we have \cite[Thm.3.7.3]{SVRelCyc}
\begin{equation} \label{eq:gstarCor}
g^\circledast Cor_{Y/X}(\sW, \sZ) = Cor_{Y/X}((g {\times_S} X)^\circledast \sW, g^\circledast \sZ).
\end{equation}

\begin{prop}[Correspondence homomorphisms]
Suppose that $\sY \to \sX \to \sS$ are two ambient separated finite type morphisms of modulus pairs, and suppose $\sW \in c_{equi}(Y^\circ/X^\circ, 0), \sZ \in \otimes c_{equi}(X^\circ/S^\circ, 0)$ are two relative cycles.
\begin{enumerate}
 \item If $\sW$ and $\sZ$ are left proper, then so is $Cor_{Y/X}(\sW, \sZ)$.
 \item If $\sW$ and $\sZ$ are admissible, then so is $Cor_{Y/X}(\sW, \sZ)$.
\end{enumerate}
\end{prop}

\begin{proof}
Let $\sZ = \sum n_i Z_i$, write $\ol{Z}_i$ for the closure of $Z_i$ in $\ol{X}$, set $Z^\infty_i = X^\infty|_{\ol{Z}_i}$ and $\sZ_i = (\ol{Z}_i, Z^\infty_i)$. 

(1): By Prop.~\ref{prop:cycFun}, since $\sW$ is left proper, the cycles $(Z_i {\to} X)^\circledast \sW$ are left proper considered as relative cycles of $\sZ_i \times_\sX \sW \to \sZ_i$. In other words, if $w_{ij} \subseteq Z_i \times_{X^\circ} Y^\circ$ are the integral components with closure $\ol{W}_{ij} \subseteq \overline{\sZ_i \times_\sX \sW}$, then the morphisms $\ol{W}_{ij} \to \ol{Z}_i$ are proper. This implies that the morphisms $\ol{W}_{ij} \to \ol{S}$ are proper. Now $Cor_{Y/X}(\sW, \sZ)$ has integral components consisting of (some subset of) the $w_{ij}$ considered as closed subschemes of $Y^\circ$. If $\overline{W}_{ij}'$ are the closures in $\ol{Y}$, since the two morphisms $\overline{\sZ_i \times_\sX \sW} \to \ol{Z}_i \times_{\ol{X}} \ol{W} \to \ol{W}$ are proper, we obtain proper surjective morphisms $\ol{W}_{ij} \to \ol{W}_{ij}'$. It now follows that the $\ol{W}_{ij}' \to \ol{S}$ are also proper.

(2): By \eqref{eq:gstarCor} it suffices to consider the case when $\ol{S}$ is the spectrum of a valuation ring and $S^\circ$ is the generic point. Let $w \in Y$ be any integral component of $Cor_{Y/X}(\sW, \sZ)$. So we are considering diagrams of the form
\[ \xymatrix{
w \ar[r] \ar[d] & W \ar[r] \ar[d] & \ol{Y} \ar[dl]  \\
S^\circ \ar[r] & \ol{S} & 
} \]
such that $W$ is the spectrum of an extension of the valuation ring $\OO(\ol{S}) \subseteq k(\ol{S})$ to $k(w)$, we wish to have 
\begin{equation} \label{eq:YVWcond}
Y^\infty|_{W} = S^\infty|_{W}. 
\end{equation} 
As each $Z_i \times_X Y {\to} Y$ is a closed immersion, this $w$ is an integral component of some $(Z_i {\to} X^\circ)^\circledast \sW$. As such, it lies over some $Z_i$ and taking the spectrum $V$ of the induced extension of valuation rings we get a diagram of the form
\[ \xymatrix{
w \ar[r] \ar[d] & W \ar[r] \ar[d] & \ol{Y} \ar[d]  \\
Z_i \ar[r] \ar[d] & V \ar[r] \ar[d] & \ol{X} \ar[dl] \\
S^\circ \ar[r] & \ol{S} & 
} \]
Since $\overline{Z}_i \to \ol{S}$ is finite, and $\OO(V)$ is integrally closed in $\OO(Z_i)$, we get a morphism $V \to Z_i$ compatible with the two morphisms to $\ol{X}, S$, enlarging our diagram to
\[ \xymatrix{
w \ar[r] \ar[d] \ar@{}[dr]|{(B)}& W \ar[r] \ar[d] & V \times_{\ol{X}} \ol{Y} \ar[dl] \ar[r] & \overline{Z}_i \times_{\ol{X}} \ol{Y} \ar[dl] \ar[r] & \ol{Y} \ar[ld]  \\
Z_i \ar[r] \ar[d] \ar@{}[dr]|{(A)} & V \ar[r]^{\textrm{bir.}} \ar[d] & \ol{Z}_i \ar[r] \ar[dl] & \ol{X} \ar[dll] \\
S^\circ \ar[r] & \ol{S} & &&
} \]
Now lets start considering moduli. Since $\sZ$ is admissible, using (A) we have
\[ S^\infty|_{V} = X^\infty|_{V}. \]
So lets equip $V$ with this modulus; that is, 
\[ V^\infty := S^\infty|_{V}. \]
In particular, $V {\to} \ol{X}$ induces a minimal morphism of modulus pairs $\sV \to \sX$. 
Now since $\sW$ is admissible, using (B) we have $Y^\infty|_{W} = V^\infty|_{W}$. But then we obtain
\[ Y^\infty|_{W} = V^\infty|_{W} = S^\infty|_{V}|_{W} = S^\infty|_{W} \]
which is precisely the condition \eqref{eq:YVWcond} we were searching for.
\end{proof}

\begin{coro} \label{coro:corOk}
Suppose that $\sY \to \sX \to \sS$ are composable morphisms in $\ulMSCH$ representable by separated finite type ambient morphisms. Then the correspondence morphisms
\[ Cor_{Y^\circ/X^\circ}: c_{equi}(Y^\circ/X^\circ, 0) \otimes c_{equi}(X^\circ/S^\circ, 0) \to c_{equi}(Y^\circ/S^\circ, 0) \]
preserve the subgroups of left proper admissible elements. In other words, they induce morphisms%
\marginpar{\fbox{$Cor_{\sY/\sX}$}}%
\index[not]{$Cor_{\sY/\sX}$}%
\[ Cor_{\sY/\sX}: \mc(\sY/\sX) \otimes \mc(\sX/\sS) \to \mc(\sY/\sS). \]
\end{coro}

\begin{lemm} \label{lemm:cycPush}
Suppose that $\sY \stackrel{f}{\to} \sX \to \sS$ are separated finite type ambient morphisms and consider the proper push-forward, \cite[Cor.3.6.3]{SVRelCyc}
\[ f_*: c_{equi}(Y^\circ/S^\circ, 0) \to c_{equi}(X^\circ/S^\circ, 0). \]
If $\sZ \in c_{equi}(Y^\circ/S^\circ, 0)$ is left proper, then $f_* \sZ$ is left proper. Furthermore, in this case if $\sZ$ is admissible then so is $f_*\sZ$.
\end{lemm}

\begin{proof}
If $W \subseteq X^\circ$ is an integral component of $f_*\sZ$, then there is some integral component $Z$ of $\sZ$ with induced morphism $Z \to W$ which is finite. If $\ol{Z}$ and $\ol{W}$ are the respective closures in $\ol{Y}$ and $\ol{X}$, then we get an induce morphism $\ol{Z} \to \ol{W}$ of integral schemes which is dominant. The morphism $\ol{Z} \to \ol{S}$ is proper by assumption, so $\ol{Z} \to \ol{W}$ is also proper. Since it is dominant, it is surjective, and therefore it follows that $\ol{W} \to \ol{S}$ is also proper.

For admissibility it suffices to consider the case that $\ol{S}$ is the spectrum of a valuation ring and $S^\circ$ is the generic point since $(-)^\circledast (-)_* = (-)_* (-)^\circledast$, \cite[Prop.3.6.2]{SVRelCyc}. Suppose this is the case and consider a diagram of the form
\[ \xymatrix{
w \ar[r] \ar[d] & W \ar[r] \ar[d] & \ol{X} \ar[dl]  \\
S^\circ \ar[r] & \ol{S} & 
} \]
such that $w$ is an integral component of $f_*\sZ$, the scheme $W$ is the spectrum of an extension of the valuation ring $\OO(\ol{S}) \subseteq k(\ol{S})$ to $k(w)$. 
We want to have $X^\infty|_{W} = S^\infty|_{W}$. By definition of the pushforward $f_*$, there exists some integral component $z in Y$ of $\sZ$ and a finite morphism $z \to w$. Extending the valuation ring of $W$ to $k(z)$ and taking the spectrum produces an extended diagram of the form
\[ \xymatrix{
z \ar[r] \ar[d] & V \ar@{-->}[r] \ar[d] & \ol{Y} \ar[d] \\
w \ar[r] \ar[d] & W \ar[r] \ar[d] & \ol{X} \ar[dl]  \\
S^\circ \ar[r] & \ol{S} . & 
} \]
The dashed morphism exists by the valuative criterion for properness since we have seen above that $\ol{z} \to \ol{w}$ is proper, where the closures are taken in $\ol{Y}$ and $\ol{X}$.%
\footnote{The use of left properness here is the entire reason it is considered at all as a property.} %
Now since $\OO(V) / \OO(W)$ is an extension of valuation rings, it induces an \textit{inclusion} of value groups. Hence, $S^\infty|_W = X^\infty|_W$ if and only if $S^\infty|_V = X^\infty|_V$. But $\sZ$ is admissible so 
\[ S^\infty|_V = Y^\infty|_V \geq X^\infty|_V \geq S^\infty|_V. \]
Hence, $S^\infty|_V = X^\infty|_V$.
\end{proof}

Recall that if $X, Y, Z$ where separated finite type $S$-schemes, the composition of cycles
\[ c_{equi}(XY/X, 0) \otimes c_{equi}(YZ/Y, 0) \to c_{equi}(XZ/X, 0) \]
is defined as 
\begin{equation} \label{eq:corrCompDef}
(\alpha, \beta) \mapsto (XYZ{\to}XZ)_*Cor_{XYZ/XY}\biggl ( (XY{\to}Y)^\circledast\beta, \alpha \biggr )
\end{equation}
(i.e., $Cor(\beta, \alpha)$ but decorated with the obvious morphisms necessary to put all cycles in the appropriate groups).

\begin{thm}
Suppose that $\sX, \sY, \sZ$ are $\sS$-modulus pairs representable by ambient morphisms which are separated and of finite type. Then the composition of cycles
\[ c_{equi}(\iX \iY/\iX, 0) \otimes c_{equi}(\iY \iZ/\iY, 0) \to c_{equi}(\iX \iZ/ \iX, 0) \]
preserves the subgroups of left proper admissible cycles. In other words, there are induced morphisms
\[ \mc(\sX\sY/\sX) \otimes \mc(\sY\sZ/\sY) \to \mc(\sX\sZ/\sX). \]
\end{thm}

\begin{proof}
This follows from the definition \eqref{eq:corrCompDef} since $(XY{\to}Y)^\circledast$, $Cor_{XYZ/XY}$, and $(XYZ{\to}XZ)_*$ preserve the subgroups by Prop.~\ref{prop:cycFun}, Cor.~\ref{coro:corOk}, and Lem.~\ref{lemm:cycPush} respectively.
\end{proof}

Recall that if $X, Y, X', Y'$ are separated $S$-schemes of finite type, the tensor product of cycles 
\[ c_{equi}(XY/X, 0) \otimes c_{equi}(X'Y'/X', 0) \to c_{equi}(XYX'Y'/XX', 0) \]
is defined as
\begin{equation} \label{equa:CorTens}
(\alpha, \beta) \mapsto Cor_{XYX'Y'/XYX'} \biggl ( (XYX' {\to} X')^\circledast(\beta), (XX' {\to} X)^\circledast(\alpha) \biggr ). 
\end{equation}

\begin{thm}
Suppose that $\sX, \sY, \sX', \sY'$ are $\sS$-modulus pairs representable by ambient morphisms which are separated and of finite type. Then the tensor product of cycles
\[ c_{equi}(\iX \iY/\iX, 0) \otimes c_{equi}(X^{\prime \o}Y^{\prime \o}/X^{\prime \o}, 0) \to c_{equi}(\iX \iY X^{\prime \o} Y^{\prime \o}/\iX X^{\prime \o}, 0) \]
preserves the subgroups of left proper admissible cycles. In other words, there are induced morphisms
\[ \mc(\sX \sY/\sX) \otimes \mc(\sX' \sY'/\sX') \to \mc(\sX \sY \sX' \sY'/\sX \sX'). \]
\end{thm}

\begin{proof}
This follows from the definition \eqref{equa:CorTens} since $Cor_{XYX'Y'/XYX'}$, $(XYX' {\to} Y)^\circledast$, and $(XX' {\to} X)^\circledast$ preserve the subgroups by Prop.~\ref{prop:cycFun} and Cor.~\ref{coro:corOk}.
\end{proof}

\begin{coro}
Let $\sS$ be a separated modulus pair. Then the composition \eqref{eq:corrCompDef} and tensor product \eqref{equa:CorTens} define a structure of monoidal category whose objects are separated finite type $\sS$-modulus pairs. Moreover, writing%
\marginpar{\fbox{$\ulMCor_\sS$}}%
\index[not]{$\ulMCor_\sS$}%
\[ \ulMCor_\sS \]
for this category, the interior $\sX \mapsto X^\circ$ functor induces functor 
\[ (-)^\circ: \ulMCor_\sS \to \Cor_{S^\circ} \]
which is faithful, essentially surjective, and monoidal. 
\end{coro}

\section{The graph functor}

For any finite separated morphism of Noetherian schemes $X \to S$, let $\ZZ(X/S)$ denote the free abelian group generated by $\hom_S(S, X)$. This is functorial in $S$ (i.e., associated to $T \to S$ we have $\ZZ(X/S) \to \ZZ(T \times_S X/T)$, and there is a canonical morphism 
\[ \ZZ(X/S) \to \ZZtr(X/S) \]
defined by sending $f: S \to X$ to $\sum f(s_i)$ where the sum is over the generic points of $S$. This morphism is an inclusion when $S$ is reduced.\footnote{Indeed, if $S$ is reduced, then it's isomorphic to the closure of the images of its generic points in $X$.}
Moreover, for a topology $\tau$ on the category of schemes over $S$, we denote by $\ZZ_\tau (X/S)$ the $\tau$-sheafification of $\ZZ (X/S)$.

\begin{defi} \label{defi:ZZsXsS}
Let $\sX \to \sS$ be a morphism of modulus pairs with Noetherian interior. We write%
\marginpar{\fbox{$\ZZ(\sX/\sS)$}}%
\index[not]{$\ZZ(\sX/\sS)$}%
\[ \ZZ(\sX/\sS) = \bigoplus_i \ZZ \hom_{\ulMSCH_\sS}(\sS_i, \sX) \]
where the sum is indexed by the connected components of $\sS$.

Moreover, for any topology $\tau$ on $\ulMSCH_{\sS}$, let $\ZZ_\tau (\sX/\sS)$ %
\marginpar{\fbox{$\ZZ_\tau (\sX/\sS)$}}%
\index[not]{$\ZZ_\tau (\sX/\sS)$}%
 denote the $\tau$-sheafification of $\ZZ (\sX/\sS)$.
\end{defi}

\begin{rema}
Since the interior is Noetherian, there are finitely many connected components of $\iS$. There may be fewer connected components of $\hS$ for some specific model of $\sS$, but up to admissible blowup, the connected components of $\hS$ and $\iS$ both agree, and are in bijection with the connected components of $\sS$, cf. the proof of Prop.~\ref{prop:extendZar}.
\end{rema}

\begin{lemm}\label{lem:graphfuncmod}
Let $\sX \to \sS$ be a separated finite type ambient morphism of modulus pairs with Noetherian interior. %
The above inclusions induce a morphism 
\[ \ZZ(\sX/\sS) \subseteq \ZZtr(\sX/\sS). \]
which is an inclusion if $\hS$ is reduced.
\end{lemm}

\begin{rema}
Note that $\hS$ is reduced iff. $\iS$ is reduced, since $\OO_{\hS} \to \OO_{\iS}$ is injective, cf. Prop.~\ref{prop:always-dense}.
\end{rema}

\begin{proof}
Let $f: \sS \to \sX$ be an $\sS$-morphism, and $\sum f(s_i) \in \ZZtr(\iX/\iS)$ the corresponding cycle. We want to show this cycle is admissible and left proper. We can assume that $f$ and $p$ are ambient. Then since $\cup \overline{\{f(s_i)\}} = f(\hS) \subseteq \hX$ (since all morphisms in question are separated), we have left properness. Since everything is functorial (cf. Axiom~\ref{axio:relativeCycles}(Pla)), to prove admissibility, it suffices to consider the case where $\hS$ is the spectrum of a valuation ring, and $\iS$ is the generic point. But then $\hS$ is the unique extension of $\ol{S}$ through the extension of fields $k(\iS) / k(\iS)$, and we have $\mX|_{\hS} = \mS$ by virtue of the factorisation $\sS \stackrel{f}{\to} \sX \stackrel{p}{\to} \sS$, since $\mS \geq f^*\mX \geq f^*p^*\mS = \mS$.
\end{proof}

\begin{cor}
Let $\sS$ be a noetherian separated modulus pair. 
Then there exists a natural functor $\ulMSch_{\sS} \to \ulMCor_{\sS}$ which sends any modulus pair to itself and fits into the following commutative diagram:
\[\xymatrix{
\ulMSch_{\sS} \ar[r] \ar[d] &  \ulMCor_{\sS} \ar[d] \\
\Sch_{\iS} \ar[r] & \Cor_{\iS},
}\]
where the vertical arrows are the faithful functors given by sending a modulus pair $\sX$ to its interior $\iX$, and the bottom horizontal arrow is the usual graph functor. 
\end{cor}

\begin{proof}
The existence is a direct consequence of Lemma \ref{lem:graphfuncmod}.
The faithfulness follows from that the other three functors in the above commutative diagram are faithful.
\end{proof}

\section{{$\protect \ulMrqfh$}-morphisms}

In this section we consider the possibility of representing cycles by morphisms, at least after $\ulMrqfh$-sheafification.

\begin{rema}
In this section, all presheaves are considered as living in $\PSh(\ulMSch_{\sS})$ (for the appropriate $\sS$), however we will see later that this is not strictly necessary, cf.Thm.\ref{theo:sheafificationHasTransfers}.
\end{rema}

Recall the following result of Suslin--Voevodsky from \cite{SVRelCyc}.
For a topology $\tau$ on $\Sch_S$ and an object $X \to S$ in $\Sch_S$, define $\Z_\tau (X/S)$ to be the $\tau$-sheafification of the presheaf which sends $U$ to the free abelian group $\Z \hom_S (U,X)$.

\begin{theo}[{\cite[Thm.4.2.12]{SVRelCyc}}] \label{theo:SVqfh}
Let $X \to S$ be a separated finite type morphism of separated Noetherian schemes. %
Then the canonical morphism
\[ \ZZ_\qfh(X/S) \to \ZZtr(X/S, 0)_{\qfh} \]
is an isomorphism in $\PSh(\Sch_S)$.
\end{theo}

Since $\ZZtr(X/S)$ is $\qfh$-separated\footnote{This follows easily from Axiom~\ref{axio:relativeCycles}.} the canonical morphism $\ZZtr(X/S) \to \ZZtr(X/S, 0)_{\qfh}$ is an inclusion. Combining this with the above isomorphism we obtain an inclusion
\[ \ZZtr(X/S, 0) \subseteq \ZZ_\qfh(X/S). \]

The above state of affairs also holds in the modulus world.
For a topology $\tau$ on $\Sch_S$ and an object $\sX \to \sS$ in $\ulMSch_S$, define $\Z_\tau (\sX/\sS)$ to be the $\tau$-sheafification of the presheaf which sends $\sU$ to the free abelian group $\Z \hom_{\sS} (\sU,\sX)$.

\begin{prop} \label{prop:Zqfhtrmodulus}
Let $\sX \to \sS$ be a separated finite type ambient morphism of qc separated modulus pairs with Noetherian interior. %
Then the canonical morphism 
\[ \ZZ_{\ulMrqfh}(\sX/\sS) \subseteq \ZZtr(\sX/\sS)_{\ulMrqfh} \]
is an isomorphism.
\end{prop}

\begin{proof}
Since $\ulMrqfh$-locally every modulus pair is locally integral, we consider only the case $\hS$ is integral.

Since %
$\ZZ(\sX/\sS) \to \ZZtr(\sX/\sS)$ %
is injective for reduced $\hS$, and every $\ulMrqfh$-covering is refinable by one with reduced total space, 
$\ZZ_{\ulMrqfh}(\sX/\sS) \to \ZZ_{\tr}(\sX/\sS)_{\ulMrqfh}$ is injective.

Now consider the following commutative square of abelian groups
\[ \xymatrix{
\ZZ(\sX/\sS) \ar[r] \ar[d]_{|\cap} & \ZZtr(\sX/\sS) \ar[d]^{|\cap} \\
\ZZ(\iX/\iS) \ar[r]        & \ZZtr(\iX/\iS) 
} \]
where the two horizontal morphisms are inclusions since we are assuming $\hS$ is reduced.
Suppose that $a \in \ZZ_{\tr}(\sX/\sS)_{\ulMrqfh}(\sS)$ is a section. Choose a $\ulMrqfh$-covering $\sT \to \sS$ such that $a|_{\sT}$ is the image of some $\alpha \in \ZZ_{\tr}(\sX/\sS)(\sT) \subseteq \ZZ_{\tr}(\iX/\iS)(\iT)$. Now due to the isomorphism $\ZZ_\qfh(X/S) \cong \ZZtr(X/S, 0)_{\qfh}$ of Theorem~\ref{theo:SVqfh}, there is a $\qfh$-covering $\iU \to \iT$ such that $\alpha|_{U}$ is the image of some $\sum n_i f_i \in \ZZ(\iX/\iS)(\iU)$. By Prop.~\ref{prop:rqfhExtend}, possibly refining $\iU \to \iT$ we can assume that this latter covering is the interior of a $\ulMrqfh$-covering $\sU \to \sT$, and we can assume that $\hU$ is locally integral. Pulling back $\alpha$, we now have $\sum n_i z_i \in \ZZ_{\tr}(\sX/\sS)(\sU)$, such that each $\ol{\{z_i\}} \to \hU$ is proper and an isomorphism over a connected component of $\iU$. Replacing $\sU$ with an admissible blowup of $\sU$, by Raynaud-Gruson flatification we can assume that the $\ol{\{z_i\}} \to \hU$ are also flat. This implies\footnote{Flat implies constant relative dimension, which is zero (or $-1$) in this case, so quasi-finite, then quasi-finite + proper implies finite, so there are finite flat of degree one, i.e., locally an isomorphism, or empty.} that each $\ol{\{z_i\}}$ is isomorphic to a connected component $\hU_i$ of $\hU$. Finally, we want to check that each of the isomorphisms $\ol{\{z_i\}} \cong \hU_i$ defines a morphism of modulus pairs $\sU_i \to \sU \times_{\sS} \sX =: \sY$. That is, we want to know that $\mU \geq \mY|_{\hU}$. We know that the relative cycle $\sum n_iz_i \in \ZZtr(\sY/\sU)$ is admissible (since it's the cycle theoretic pullback of an admissible relative cycle), so for every minimal morphism $\sP \to \sU$ such that $\hP$ is the spectrum of a valuation ring and $\iP$ is the generic point, the pullback to $\sP$ is admissible. Since each $\ol{\{z_i\}} \cap \iY \to \iU_i$ is an isomorphism, this pullback is just the scheme-theoretic pullback, cf. Axiom~\ref{axio:relativeCycles}(Pla), inside $\hP \times_{\hU} \hY$, cf. Lemm.~\ref{lemm:minProduct}. So admissibility says precisely that $\mU|_{\hP} = \mY|_{\hP}$. But since this holds for all such $\sP \to \sU$, we deduce that $\mU \geq \mY|_{\hU}$, Prop.~\ref{prop:detectEqualOnValRing}.
\end{proof}

\begin{lemm} \label{lemm:ZardetectPropAmb}
Let $\sS$ be a qc separated modulus pair with Noetherian interior. Suppose that $\sT \to \sS$ is a $\ulPZar$-covering. Let $\sX / \sS$ be separated of finite type, and $\alpha \in \mc(\iX/\iS)$ be such that $f^\circledast\alpha$ is left proper (resp. admissible). Then $\alpha$ is left proper (resp. admissible). 

Consequently, the square of abelian groups
\[ \xymatrix{
\mc(\sX/\sS) \ar@{}[r]|\subseteq \ar[d] & \mc(\iX/\iS) \ar[d] \\
\mc(\sT \times_\sS\sX/\sT) \ar@{}[r]|\subseteq & \mc(\iT \times_{\iS} \iX/\iS)
} \]
is cartesian.
\end{lemm}

\begin{proof}
Left properness follows from properness of morphisms being checked Zariski locally. Admissibility follows from the fact that any minimal morphism $\sP \to \sS$ with $\hP$ the spectrum of a valuation ring and $\iP$ the generic point will factor through $\sT$. 
\end{proof}

\begin{lemm} \label{lemm:admAlgClo}
Let $\sX \to \sS$ be a separated finite type morphism of modulus pairs with Noetherian interior. %
A relative cycle $\alpha \in c_{equi}(\iX/\iS, 0)$ is admissible if and only if $f^\circledast\alpha$ is admissible for every morphism $f: \sP \to \sS$ such that $\hP$ is the spectrum of a valuation ring with \emph{algebraically closed fraction field}, and $\iP$ the generic point.
\end{lemm}

\begin{proof}
Suppose $f: \sP \to \sS$ is a morphism as in the statement, but without requiring the fraction field to be algebraically closed (i.e., a morphism as in Def.~\ref{defn:presheavesOfRelCyc}(4)). We want to show $f^\circledast\alpha$ is admissible. 

Choose an algebraic closure of the fraction field $k^a / k(\iP)$, set $\iQ := \Spec(k^a)$, and let $g: \iQ \to \iP$ be the induced morphism. 
Let $f^\circledast \alpha = \sum n_i z_i$, and $g^\circledast f^\circledast \alpha = \sum n_i m_{ij} w_{ij}$, cf. Axiom~\ref{axio:relativeCycles}(Pla), so there are commutative squares 
\[ \xymatrix{
w_{ij} \ar[d] \ar[r] & z_{ij} \ar[d] \\
\iQ \ar[r] & \iP.
} \]
Since $k(\iQ) = k^a$ is algebraically closed and $k(w_{ij})/k(\iQ)$ is a finite extension, we have $w_{ij} = \iQ$ for all $i,j$. Now given an extension $\widetilde{z}_{ij}$ of the valuation of $\iP$ to $z_{ij}$, extend it further to $Q^\circ$, and let $\sQ \to \sP$ be the induced minimal morphism of modulus pairs. As such, if there is a factorisation $\widetilde{z}_{ij} \to \hP \times_{\hS} \hX \to \hS$, we obtain a diagram of the form
\[ \xymatrix@R=6pt{
\iQ \ar[rr] \ar[dd] \ar[dr] && \hQ \ar[r] \ar[dd]|\hole \ar[dr] & \hQ \times_{\hS} \hX \ar[ddl]|(0.32)\hole|(0.52)\hole  \ar[dr] \\
&z_i \ar[rr] \ar[dd] && \widetilde{z}_{ij} \ar[r] \ar[dd] & \hP \times_{\hS} \hX \ar[ddl]  \\
\iQ \ar[rr]|\hole \ar[dr] && \ol{Q} \ar[dr] & \\
& \iP \ar[rr] && \ol{P} & 
} \]
(Note that by Lemma~\ref{lemm:minProduct} we have $\sP \times_{\sS} \sX = (\hP \times_{\hS} \hX, \mX|_{\hP \times_{\hS} \hX})$ and similar for $Q$). %
Now since $\hQ$ is the spectrum of a valuation ring with algebraically closed fraction field, by hypothesis $g^\circledast f^\circledast \alpha$ is admissible. Hence, %
$\mS|_{\hQ} = \mX|_{\hQ}$. %
Now $\hQ \to \widetilde{z}_{ij} \to \hP$ is a tower of extensions of valuation rings, and so the induced morphisms of Cartier divisors (i.e., value groups) are injective. So this latter equality implies $\mS|_{\widetilde{z}_{ij}} = \mX|_{\widetilde{z}_{ij}}$.
\end{proof}

\begin{prop} \label{prop:FindetectPropAmb}
Suppose that $\sT \to \sS$ is a minimal ambient morphism with $f: \hT \to \hS$ a finite and surjective morphism between integral Noetherian schemes. Let $\sX / \sS$ be separated of finite type, and $\alpha \in \mc(\iX/\iS)$ be such that $f^\circledast\alpha$ is left proper (resp. admissible). Then $\alpha$ is left proper (resp. admissible). 

Consequently, the square of abelian groups
\[ \xymatrix{
\mc(\sX/\sS) \ar@{}[r]|\subseteq \ar[d] & \mc(\iX/\iS) \ar[d] \\
\mc(\sT \times_\sS\sX/\sT) \ar@{}[r]|\subseteq & \mc(\iT \times_{\iS} \iX/\iT)
} \]
is cartesian.
\end{prop}

\begin{proof}
\textit{Left properness.} Write $\alpha = \sum n_iz_i$. By the axiom (Pla) in the above discussion, we see that $f^\circledast\alpha = \sum n_im_{ij}w_{ij}$ where the $w_{ij}$ are the points of $t \times_s z_i$, and $m_{ij}$ various multiplicities, where $t, s$ are the generic points of $\ol{T}, \ol{S}$ respectively. Considering the closures in the respective ambient spaces, for each $ij$ we obtain a commutative square
\[ \xymatrix{ 
\ol{\{w_{ij}\}} \ar[r] \ar[d] & \ol{\{z_i\}} \ar[d] \\
\hT \ar[r] & \hS .
} \]
Here, $\ol{\{w_{ij}\}} \to \hT \to \hS$ are proper, and $\ol{\{w_{ij}\}} \to \ol{\{z_i\}}$ surjective, and $\ol{\{z_i\}} \to \hS$ is finite type and separated, so it follows that $\ol{\{z_i\}} \to \hS$ is proper.

\textit{Admissibility.} We leverage Lem.~\ref{lemm:admAlgClo}. Let $\sP \to \sS$ be a morphism from a modulus pair such that $\hP$ is the spectrum of a valuation ring with algebraically closed fraction field, and $\iP$ is the open point. We want to show that $\alpha|_{\sP}$ is admissible. Since $\hT \to \hS$ is proper, there is a factorisation
\[ \sP \stackrel{g}{\to} \sT \stackrel{f}{\to} \sS. \]
Then since $f^\circledast\alpha$ is admissible, $g^\circledast f^\circledast\alpha$ is admissible.
\end{proof}

\begin{coro} \label{coro:qfhpullbackDetectsMCycles}
Let $\sS$ be a qc separated modulus pair with Noetherian interior, and $\sX \to \sS$ a finite type separated morphism. %
A relative cycle $\alpha \in \mc(\iX/\iS)$ is in $\mc(\sX/\sS)$ if and only if $f^\circledast\alpha \in \mc(\sX/\sS)(\sT) \subseteq \ZZtr(\iX/\iS)(\iT)$ for some $\ulMrqfh$-covering $f: \sT \to \sS$.
\end{coro}

\begin{proof}
``Only if'' follows from functoriality which was proved in the previous section. For ``If'' suppose that $f^\circledast \alpha$ is left proper and admissible for some $\ulMrqfh$-covering. Since the $\ulMrqfh$-topology is generated by admissible blowups, $\ulPZar$-coverings, and $\ulPfin$-coverings, we can assume that $f$ is a finite composition of morphisms of this form. But left properness and admissibility is detected by $\ulPZar$-coverings, Lem.~\ref{lemm:ZardetectPropAmb}, finite surjective morphisms, Prop.~\ref{prop:FindetectPropAmb}, and admissible blowups, Prop.~\ref{prop:cycTarAdmBu}.
\end{proof}

\begin{coro} \label{coro:cartSquaQfhRelCyc}
Let $\sS$ be a qc separated modulus pair with Noetherian interior, and $\sX \to \sS$ a finite type separated morphism. %
Then we have a Cartersian square of inclusions in $\PSh(\ulMSch_{\sS})$
\[ \xymatrix{
\ZZtr(\sX/\sS) \ar@{}[r]|\subseteq \ar@{}[d]|{|\cap} & 
\ZZtr(\iX/\iS) \ar@{}[d]|{|\cap} \\
\ZZ_{\ulMrqfh}(\sX/\sS) \ar@{}[r]|\subseteq & 
\ZZ_{\qfh}(\iX/\iS)
} \]
\end{coro}

\begin{rema}
The left vertical morphism sends a cycle $\alpha \in \ZZtr(\sX/\sS)$ to $f^\circledast \alpha \in ker\left (\ZZ(\sX/\sS)(\sT) \to \ZZ(\sX/\sS)(\sT \times_{\sS} \sT)\right )$ for an appropriate $\ulMrqfh$-covering $f:\sT \to \sS$, which exists by Theorem \ref{theo:SVqfh}. The non-modulus version is analogous.
\end{rema}

\begin{proof}
The inclusion $\ZZtr(\sX/\sS) \to \ZZtr(\iX/\iS)$ is by definition. The vertical morphisms are from the isomorphisms 
\[ \ZZ_\qfh(X/S) \cong \ZZ_\tr(X/S)_{\qfh}, \qquad \qquad  \qquad \ZZ_{\ulMrqfh}(\sX/\sS) \cong \ZZtr(\sX/\sS)_{\ulMrqfh} \]
of Theorem \ref{theo:SVqfh} and Proposition \ref{prop:Zqfhtrmodulus} respectively, and they are inclusions because $\ZZtr(\iX/\iS)$ is $\qfh$-separated. Finally, the morphism $\ZZ(\sX/\sS) \to \ZZ(\iX / \iS)$ is a monomorphism whenever $\hS$ is locally integral, so it becomes an inclusion after sheafification.

For Cartesianness, consider an element $s \in \ZZ_{\qfh}(\iX/\iS)$ which is in the left and top subgroups. Since its in the top subgroup we can represent it as some $\alpha \in \ZZtr(\iX/\iS)$, and since it's in the left subgroup, there is an $\ulMrqfh$-covering $f: \sT \to \sS$ such that $f^\circledast\alpha$ is in $\Ztr(\sX / \sS)(\sT)$. But then the original cycle is in $\ZZtr(\sX/\sS)$ by Coro.~\ref{coro:qfhpullbackDetectsMCycles}.
\end{proof}

If we use rational coefficients, the presheaves studied above collapse to a single presheaf.

\begin{prop}[{\cite[Prop.3.3.14,Prop.4.2.7]{SVRelCyc}}] \label{prop:QQqfhTr}
Let $X \to S$ be a finite type separated morphism of Noetherian separated schemes. Then $\QQtr(X/S)$ is a $\qfh$-sheaf, isomorphic to $\QQ_{\qfh}(X/S)$.
\end{prop}

The same is true in the modulus world.

\begin{coro} \label{coro:mqfhSheafRep}
Let $\sX \to \sS$ be a finite type separated morphism of Noetherian separated modulus pairs. Then $\QQtr(\sX/\sS)$ is a $\ulMrqfh$-sheaf, isomorphic to $\QQ_{\ulMrqfh}(\sX/\sS)$.
\end{coro}

\begin{proof}
Tensoring the Cartesian square of Cor.~\ref{coro:cartSquaQfhRelCyc} with $\QQ$, the right hand vertical inclusion becomes the isomorphism of Prop.~\ref{prop:QQqfhTr}, and so therefore the pullback $\QQtr(\sX/\sS) \to \QQ_{\ulMrqfh}(\sX/\sS)$ is also an isomorphims of presheaves.
\end{proof}

\section{Comparison with [KMSY] correspondences}\label{sec:KMSY-comparison}

The goal of this section is to prove the following theorem, which relates the modulus condition on cycles in the sense of Kahn-Miyazaki-Saito-Yamazaki and the modulus condition in the sense of Definition \ref{defn:presheavesOfRelCyc} above. 

\begin{thm}\label{thm:rel-cor}
Suppose that $k$ is a field and $\sX, \sY$ are modulus pairs over $k$ with integral total space and $k$-smooth interior. Then
\begin{equation}\label{eq:KMSY-KM} \hom_{\ulMCor_k}(\sX, \sY) \cong \mc(\sX \times_k \sY / \sX), \end{equation}
where $\ulMCor_k$%
\marginpar{\fbox{$\ulMCor_k$}}%
\index[not]{$\ulMCor_k$}%
 denotes the category of modulus correspondences in \cite{kmsy1} and \cite{kmsy2}.
Moreover, these isomorphisms are compatible with composition.
\end{thm}

Before proving the theorem, we prepare some lemmas. 

\begin{lemma} \label{prop:KMSYadmissibility}
Let $\sX \to \sS \in \ulPSCH$ be an ambient morphism of finite type of 
qc separated modulus pairs with Noetherian interior. 
Let $\alpha = \sum n_i Z_i \in c_{equi}(\iX / \iS, 0)$ be a relative cycle, and 
let $\widetilde{Z_i} \to \hX$ be the normalisation of the closure $\ol{Z_i}$ of $Z_i$ in $\hX$. The following are equivalent.
\begin{enumerate}
 \item $\alpha$ is admissible (as a relative cycle).
 \item $\mS|_{\widetilde{Z_i}} = \mX|_{\widetilde{Z_i}}$ for each $i$.
\end{enumerate}
\end{lemma}

\begin{proof}
(1) To check that $\mS|_{\widetilde{Z}_i} = \mX|_{\widetilde{Z}_i}$ it suffices to check Zariski locally on $\widetilde{Z}_i$ that $f/g \in \OO^*$ where $f, g$ are the local generators for $\mS|_{\widetilde{Z}_i}, \mX|_{\widetilde{Z}_i}$ respectively, and we are thinking of $f/g$ as an element of the fraction field $k(\widetilde{Z}_i)$. Any normal ring $A$ is the intersection $A = \cap_{R \subseteq \Frac(A)} R$ of the valuation rings of its fraction field that contain it, and similarly, $A^* = \cap_{R \subseteq \Frac(A)} R^*$. So to check $f/g \in \OO^*$ it suffices to check it on each valuation ring $R$. By Axiom~\ref{axio:relativeCycles}(Gen), the generic point $z_i$ of $\widetilde{Z}_i$ is over a generic point $s$ of $\iS$. Let $R' = R \cap k(s) \subseteq k(z_i)$ be the valuation ring of $k(s)$ induced by $R$. by Axiom~\ref{axio:relativeCycles}(Bir), the pullback of $\alpha$ to $s$ is the same formal sum $n_i z_i$ (we replace the $Z_i$ with their generic points), so we obtain the following diagram
 \[ \xymatrix{
z_i \ar[r] \ar[d] & \Spec(R) \ar[d] \ar[r] & \widetilde{Z}_i \ar[d] \ar[r] & \hX  \ar[dl] \\
s \ar[r] & \Spec(R') \ar[r] & \hS & 
 } \]
Since $\alpha$ is admissible, we have $\mS|_{\Spec(R)} = \mX|_{\Spec(R)}$.

(2) Let $f: \sP \to \sS$ be a minimal (ambient) morphism with $\hP$ the spectrum of a valuation ring and $\iP$ the open point, let $\sum m_{ij} w_{ij} = f^\circledast\alpha \in \ZZtr(\iX/\iS)(\iP)$ and let $\widetilde{w}_{ijk}$ be a valuation ring of $w_{ij}$ extending $\hP$ and consider a diagram
 \[ \xymatrix{
w_{ij} \ar[d] \ar[r] & \widetilde{w}_{ijk} \ar[d] \ar[r] & \overline{(\sP \times_{\sS} \sX)} \ar[r] \ar[dl] & \hX \ar[dl] \\
\iP \ar[r] & \hP \ar[r] & \hS
 } \]
The relative cycle $\alpha$ is admissible if and only if for every such diagram we have $\mP|_{\widetilde{w}_{ijk}} = (\sP \times_{\sS} \sX)^\infty|_{\widetilde{w}_{ijk}}$. Since $\sP \to \sS$ is minimal, the left hand side becomes $\mS|_{\widetilde{w}_{ijk}}$, and the right hand side is $\mX|_{\widetilde{w}_{ijk}}$, cf.Lem.~\ref{lemm:minProduct}.

Now we observe that $w_{ij} \to \hX$ factors through $Z_i$, \cite[Lem.3.3.6]{SVRelCyc}, so $\widetilde{w}_{ijk} \to \hX$ factors through $\ol{Z_i}$. Let $w_{ij}^a$ be the spectrum of the algebraic closure of $k(w_{ij})$, and $\widetilde{w}^a_{ijk}$ any extension of the valuation ring of $\widetilde{w}_{ijk}$ to $w_{ij}^a$. Since $\widetilde{Z}_i \to \ol{Z_i}$ is integral, we can extend our diagram:
 \[ \xymatrix{
&&& \widetilde{Z_i} \ar[d] \\
w_{ij}^a \ar[d] \ar[r] \ar@{-->}[rrru] & \widetilde{w}^a_{ijk} \ar[d] 
&& \ol{Z_i} \ar[d] \\
w_{ij} \ar[d] \ar[r] & \widetilde{w}_{ijk} \ar[d] \ar[r] \ar@{-->}[rru] & \overline{(\sP \times_{\sS} \sX)} \ar[r] \ar[dl] & \hX \ar[dl] \\
\iP \ar[r] & \hP \ar[r] & \hS
 } \]
Now by the valuative criterion for separatedness and universally closedness, \cite[01KF, 01WM]{stacks-project}, there exists a canonical factorisation
\[ \widetilde{w}^a_{ijk} \to \widetilde{Z}_i \]
making the above diagram commute. Consequently, 
\[ 
\mS|_{\widetilde{Z}_i} = \mX|_{\widetilde{Z}_i} 
\implies 
\mS|_{\widetilde{w}^a_{ijk}} = \mX|_{\widetilde{w}^a_{ijk}} 
\]
and to finish the proof, it suffices to show 
\[ 
\mS|_{\widetilde{w}^a_{ijk}} = \mX|_{\widetilde{w}^a_{ijk}} 
\implies 
\mS|_{\widetilde{w}_{ijk}} = \mX|_{\widetilde{w}_{ijk}} 
\]
But $\widetilde{w}_{ijk}^a \to \widetilde{w}_{ijk}$ is an extension valuation rings, so the induced morphism on Cartier divisors is nothing more than the corresponding inclusion of value groups $\Gamma_{\widetilde{w}_{ijk}} \subseteq \Gamma_{\widetilde{w}_{ijk}^a}$.
\end{proof}

We also need the following elementary

\begin{lemma}\label{lem:rephrasing}
Let $X$ be a scheme, $D_1,D_2$ effective Cartier divisors on $X$, and $E=D_1 \times_X D_2$. 
Assume that $E$ is an effective Cartier divisor on $X$.
Then the following conditions are equivalent:
\begin{enumerate}
\item $E=D_2$. 
\item The Cartier divisor $D_1 - D_2$ is effective.
\end{enumerate}
\end{lemma}

\begin{proof}
Since the problem is Zariski local on $X$, we may assume that $X= \Spec (R)$ is affine, and $D_1$, $D_2$, and $E$ are represented by nonzerodivisors $d_1$, $d_2$ and $e$ in $R$, respectively.
The equality $D_1 \times_X D_2 = E$ is equivalent to $d_1 R+ d_2 R=eR$.
Therefore, 
\begin{align*}
(1)
&\Longleftrightarrow eR = d_2 R \\
&\Longleftrightarrow d_1 R+ d_2 R = d_2 R \\
&\Longleftrightarrow d_1d_2^{-1} R+  R = R \\
&\Longleftrightarrow d_1d_2^{-1} \in R 
\Longleftrightarrow (2),
\end{align*}
where the sums are taken in the total quotient ring $\{\text{nonzerodivisors}\}^{-1}R$.
\end{proof}

Now we are ready to prove the main result of this subsection. 

\begin{proof}[Proof of Theorem \ref{thm:rel-cor}]
Recall that the fiber product $\sX \times_k \sY$ is represented by a modulus pair $P$ given by 
\begin{align*}
\hP &= \Bl_{X^\infty \times_k Y^\infty} (\hX \times_k \hY), \\
\mP &= \pi^* (\mX \times_k \hY) + \pi^* (\hX \times_k \mY) - E
\end{align*}
where $\pi : \hP \to \hX \times_k \hY$ denotes the blow up along $X^\infty \times_k Y^\infty$, and $E = \pi^{-1} (\mX \times_k \mY)$ the exceptional divisor. 

Both sides of \eqref{eq:KMSY-KM} are subgroups of 
\[
\hom_{\ulMCor}(\iX, \iY) = \mc(\iX \times_k \iY / \iX). 
\]
Take an integral cycle $Z$ in this group, and let $\ol{Z}$ (resp. $\ol{Z}'$) be the closure of $Z$ in $\ol{X} \times_k \ol{Y}$ (resp. $\hP$), and $\tilde{Z} \to \ol{Z}$ (resp. $\tilde{Z}' \to \ol{Z}'$) its normalisation. 
By the universal property of normalisation, the natural map $\ol{Z}' \to \ol{Z}$ which is induced by $\pi$ lifts to 
a proper surjective morphism $\tilde{Z}' \to \tilde{Z}$.

Note that $\ol{Z}$ is proper over $\ol{X}$ if and only if $\ol{Z}'$ is proper over $\ol{X}$.
Therefore, it suffices to show that $Z$ is admissible as a relative cycle if and only if it is admissible as a finite correspondence.
By Proposition \ref{prop:KMSYadmissibility}, the former condition is equivalent to $\mP |_{\tilde{Z}'} = \mX |_{\tilde{Z}'} = \pi^* (\mX \times_k \hY) |_{\tilde{Z}'}$.
By the definition of $\mP$ above, this is equivalent to $\pi^* (\hX \times_k \mY) |_{\tilde{Z}'} = E |_{\tilde{Z}'}$.
By Lemma \ref{lem:rephrasing} applied to $X = \tilde{Z}'$, $D_1 = \pi^* (\mX \times_k \hY) |_{\tilde{Z}'}$ and $D_2  = \pi^* (\hX \times_k \mY) |_{\tilde{Z}'}$, and by the equality $E |_{\tilde{Z}'}  = \pi^* (\mX \times_k \hY) |_{\tilde{Z}'} \times_{\tilde{Z}'} \pi^* (\hX \times_k \mY) |_{\tilde{Z}'}$, 
the last condition is equivalent to $\pi^* (\mX \times_k \hY) |_{\tilde{Z}'} \geq \pi^* (\hX \times_k \mY) |_{\tilde{Z}'}$.
Since $\tilde{Z}' \to \tilde{Z}$ is surjective, \cite[Lem.~2.2]{KP} implies that this is equivalent to $(\mX \times_k \hY) |_{\tilde{Z}} \geq (\hX \times_k \mY) |_{\tilde{Z}}$, which means $Z \in \ulMCor (\sX,\sY)$.
This finishes the proof.
\end{proof}

\section{Filtered limits of cycles}

The following result leads to one possible extension of the theory from Noetherian interior to arbitrary qc separated schemes.

\begin{theo}[{\cite[Prop.9.3.9]{CD19}}] \label{theo:limitCycles}
Let $S$ be a noetherian separated scheme, $Y \to S$ a separated morphism of finite type, and $(T_\lambda)$ a filtered system of separated finite type $S$-schemes with affine dominant transition morphisms. Then
\[ 
\ZZtr((\varprojlim T_\lambda) \times_S Y / (\varprojlim T_\lambda)) \cong \varinjlim \ZZtr(T_\lambda \times_S Y / T_\lambda). 
\]
\end{theo}

It also helps us study fibres of cycle presheaves.

\begin{coro} \label{coro:fibCorOk}
Let $\sS$ be a separated modulus pair with Noetherian interior, and $(\sP_λ)_λ$ a pro-object of $\ulPSch_\sS$ with affine minimal transition morphisms, such that $\lim \iP_λ$ is again Noetherian. Then for any $\sX \in \ulMSch_\sS$ we have 
\[ 
\ZZtr(\sP \times_{\sS} \sX / \sP) 
\cong 
\colim 
\ZZtr(\sP_λ \times_{\sS} \sX / \sP_λ)
\] 
where $\sP = \amblim \sP_λ$.
\end{coro}

\begin{rema}
We would like to use Coro.~\ref{coro:cartSquaQfhRelCyc} but we are obstructed by the fact that the interior of the limit will not in general be Noetherian. 
\end{rema}

\begin{proof}
By Cor.\ref{coro:cartSquaQfhRelCyc} it suffices to know that $\rqfh$-sheafification commutes with the colimit. This latter is a fact about general presheaves that has nothing to do with cycles. If follows from Prop.\ref{prop:descendPQfhCoverings}.

There is probably also a more direct proof using the commutative square
\[ 
\xymatrix{
\colim \ZZtr(\sP_λ \times_{\sS} \sX / \sP_λ) \ar[d] \ar[r]^-\subseteq & 
\colim \ZZtr(\iP_λ \times_{\iS} \iX / \iP_λ) \ar[d]^\cong \\
\ZZtr(\sP \times_{\sS} \sX / \sP) \ar[r]^-\subseteq & \ZZtr(\iP \times_{\iS} \iX / \iP).
}
\]
This finishes the proof.
\end{proof}


\chapter{Sheaves with transfers} \label{chap:shTr}

\section{Sheafification of presheaves with transfers}

\begin{defi}[Modulus sheaves with transfers]
Let $\sS$ be a qc separated modulus pair with Noetherian interior. 
Let $\ulMtau$ be a topology on $\ulMSch_\sS$. Then 
\[ \Shv_{\ulMtau}(\ulMCor_\sS) \]%
\marginpar{\fbox{$\Shv_{\ulMtau}(\ulMCor_\sS)$}}%
\index[not]{$\Shv_{\ulMtau}(\ulMCor_\sS)$}%
is the full subcategory of $\PSh(\ulMCor_\sS)$ of those presheaves whose restriction to $\ulMSch_\sS$ is a ${\ulMtau}$-sheaf. For a topology on $\ulMSm_\sS$, we make the analogous definition for%
\marginpar{\fbox{$\Shv_{\ulMtau}(\ulMSmCor_\sS)$}}%
\index[not]{$\Shv_{\ulMtau}(\ulMSmCor_\sS)$}%
\[ \Shv_{\ulMtau}(\ulMSmCor_\sS). \]
\end{defi}

\begin{prop} \label{prop:CechCycles}
Let $\sS$ be a qc separated modulus pair with Noetherian interior, 
let ${\ulMtau} = \ulMNis$ or $\ulMet$, 
and $\sV \to \sX$ is a ${\ulMtau}$-covering in $\ulMSch_\sS$. Then the sequence
\[ \dots 
\to \mc(\sV\times_\sX \sV /\sS) 
\to \mc(\sV/\sS) 
\to \mc(\sX/\sS) 
\to 0 \]
in $\PSh(\ulMSch_\sS)$ becomes exact after ${\ulMtau}$-sheafification.
\end{prop}

\begin{rema} \label{rema:noFppf}
The $\ulMrqfh$ 
version of this proposition follows from the fact that locally every correspondence is a morphism, cf.~Lem.~\ref{lemm:CechCyclesQfh} below. 

The proof also works in the $\ulMfppf$-case modulo the fact that we may have to deal with $\ulMfppf$-local modulus pairs which do not have Noetherian interior. So we cannot just blindly assume the base is $\ulMfppf$-local, we should instead work with the pro-system.
\end{rema}

\begin{proof}
We follow the proof of \cite[Prop.3.1.3]{VoeTri}. %
By Prop.\ref{prop:conservativeLocalPairs}, Rem.~\ref{rema:canAssumeNoethInt}, Cor.~\ref{coro:fibCorOk} it suffices to consider the case where $\sS$ is a $\ulPNis$-local $\uS$-local modulus pair. %
Up to isomorphism in $\ulMSch$, we can assume $\sX \to \sS$ is ambient. Now consider the filtered system of closed subschemes $\ol{Z} \subseteq \ol{X}$ such that $Z^\circ \to S^\circ$ is finite and $\ol{Z} \to \ol{S}$ is proper. The sequence in the statement is the filtered colimit of the corresponding subsequences ranging over such $\sZ$ (with $\sV$ replaced with $\sZ \times_\sX \sV$, clearly). So we can assume that $X^\circ \to S^\circ$ is finite and $\ol{X} \to \ol{S}$ is proper. By Lem.~\ref{lemm:finiteulMNislocal}, $\sX$ is a disjoint union of $\ulMtau$-local modulus pairs. As such, $\sV \to \sX$ admits a section $\sX \to \sV$. It follows that the sequence 
\[ \dots 
\to \mc(\sV \times_{\sX} \sV /\sS) 
\to \mc(\sV/\sS) 
\to \mc(\sX/\sS) 
\to 0 \]
is exact as a sequence of $\ulMtau$-sheaves on $\Sch_\sS$; an explicit chain homotopy is induced by $s \times_{\sX} \id: \sV^{\times_{\sX} k} \to \sV^{\times_{\sX} k+1}$ where $s$ is the section of $\sV \to \sX$.
\end{proof}

\begin{lemm} \label{lemm:CechCyclesQfh}
Let $\sS$ be a qc separated modulus pair with Noetherian interior, 
and $\sV \to \sX$ is a $\ulMrqfh$-covering in $\ulMSch_\sS$. Then the sequence
\[ \dots 
\to \ZZtr(\sV\times_\sX \sV /\sS) 
\to \ZZtr(\sV/\sS) 
\to \ZZtr(\sX/\sS) 
\to 0 \]
is exact as a sequence of $\ulMrqfh$-sheaves on $\ulMSch_\sS$.
\end{lemm}

\begin{proof}
In Prop.~\ref{prop:Zqfhtrmodulus} we saw that there are isomorphisms of sheaves 
\[ \ZZ_{\ulMrqfh}(\sY/\sS) \cong \ZZtr(\sY/\sS)_{\ulMrqfh} \]
for any $\sY \to \sS \in \Sch_\sS$, so the sequence in the statement is the complex associated to the image of \v{C}ech-hypercover of $\sV \to \sX$, which is well-known to be exact. For example, there exists a section $\ZZ_{\ulMrqfh}(\sX/\sS) \to \ZZ_{\ulMrqfh}(\sV/\sS)$ from which a contracting homotopy can be built as in the proof of Prop.~\ref{prop:CechCycles}.
\end{proof}

\begin{lemm} \label{lemm:canSquareCyc}
Let $\sS$ be a qc separated modulus pair with Noetherian interior, 
and suppose that $\sY \to \sX$ is a correspondence of finite type $\sS$-modulus pairs, and $\sV \to \sX$ is a $\ulMtau$ covering where $\ulMtau = \ulMNis$, $\ulMet$ or $\ulMrqfh$. Then there exists a covering $\sW \to \sY$ and a correspondence $\sW \xrightarrow{\beta} \sV$ making a commutative square
\[ \xymatrix{
\sW \ar[r] \ar[d] & \sV \ar[d] \\
\sY \ar[r] & \sX .
} \]
\end{lemm}

\begin{proof}
This follows from Prop.~\ref{prop:CechCycles} and Lem.~\ref{lemm:CechCyclesQfh} or it could be proven directly in the same ways as those propositions are proven.
\end{proof}

\begin{thm} \label{theo:sheafificationHasTransfers}
Let $\sS$ be a qc separated modulus pair with Noetherian interior. 
Choose a $(\ast)$ in the table below and let $\sS, \sC$ be the first entry in the row and $\tau$ the first entry in that column. %
\[
\begin{array}{rcccccc}
{_\mathscr{S}}, &{_\mathscr{C}} \diagdown{^{τ}}    & \Zar & \Nis & \et & \fppf & \rqfh \\
\Sch_\sS, &\Cor_\sS             & (\ast)  & (\ast)  & (\ast) &    & (\ast)   \\ 
\Sm_\sS, &\SmCor_\sS              & (\ast)  & (\ast)  & (\ast) &       &       \\  
\end{array}
\]
The forgetful functor $\Shv_{\ulMtau}(\ulM\mathscr{C}) \to \PSh(\ulM\mathscr{C})$ admits a left adjoint $a^{tr}$ fitting into a commutative square
\[ \xymatrix{
\PSh(\ulM \mathscr{C}) \ar[r]^-{a^{tr}} \ar[d]_{forget} & \Shv_{\ulMtau}(\ulM\mathscr{C}) \ar[d]^{forget} \\
\PSh(\ulM \mathscr{S}) \ar[r]_-a & \Shv_{\ulMtau}(\ulM\mathscr{S})  \\
} \]
\end{thm}

\begin{rema}
The result is true for $τ = \fppf$ as well, but we do not include it for the reason explained in Remark~\ref{rema:noFppf}.
\end{rema}

\begin{proof}
With Lem.~\ref{lemm:canSquareCyc} in hand, the same proof as \cite[Lem.3.1.6]{VoeTri} works. Alternatively, the proof of \cite[Lem.10.3.7]{CD19} also works.
\end{proof}

\begin{coro}[{Cf.\cite[Prop.10.3.9]{CD19}}]
In the notation and assumptions of Thm.\ref{theo:sheafificationHasTransfers}, the category $\Shv_{\ulMtau}(\ulM\mathscr{C})$ is a Grothendieck abelian category, and the functor $\Shv_{\ulMtau}(\ulM\mathscr{C}) \to \Shv_{\ulMtau}(\ulM\mathscr{S})$ commutes with all limits and colimits.
\end{coro}

\section{qfh-sheaves with transfers}

Recall that the qfh topology on Noetherian schemes is generated by the Nisnevich topology and finite surjective morphisms.

\begin{thm} \label{thm:theo:qfhSchCorRat}
Let $\sS$ be a qc separated modulus pair with Noetherian interior. 
The canonical forgetful functor
\[ \Shv_{\ulMrqfh}(\ulMCor_\sS) \to \Shv_{\ulMrqfh}(\ulMSch_\sS) \]
is an equivalence of categories.
\end{thm}

\begin{proof}
Since $\gamma: \ulMSch_\sS \to \ulMCor_\sS$ is essentially surjective, the forgetful functor 
\[ \gamma_* :\PSh(\ulMCor_\sS) \to \PSh(\ulMSch_\sS) \]
is faithful on presheaves, and therefore faithful on sheaves as well. So we just need to prove that the forgetful functor is essentially surjective and full. To prove this it suffices to show that the unit of the canonical adjunction $\id \to \gamma_* \gamma^*$ becomes an isomorphism after $\ulMrqfh$-sheafification.

Since both $\gamma^*$ and $\gamma_*$ preserve colimits and every sheaf is a colimit of representables, it suffices to show that for every $\sX$ the canonical morphism $(\ZZ \sX)_{\ulMrqfh} \to (\mc \sX)_{\ulMrqfh}$ is an isomorphism.

It is a monomorphism because $\ZZ \sX \to \mc \sX$ is a monomorphism, and sheafification is exact. To show it's an epimorphisms we show that for every $\alpha: \sY \to \sX$ in $\ulMCor_\sS$, there is a $\ulMrqfh$-covering $f: \sV \to \sY$ such that $\alpha \circ f$ is in $\ulMSch_\sS$. The non-modulus version is \cite[Thm.4.2.12]{SVRelCyc} which says precisely that $\ZZ \iX \to \ZZ_{\tr} \iX$ produces an isomorphism of qfh-sheaves on the category of finite type separated $\iS$-schemes. So for our $\alpha^\circ \in \hom_{\Cor_{\iS}}(\iY, \iX)$ there exists a qfh-covering $f^\circ: \iV \to \iY$ such that $\alpha^\circ \circ f^\circ$ is a morphism in $\Sch_{\iS}$. By Prop.~\ref{prop:rqfhExtend} we can extend $\iV \to \iY$ to a $\ulMrqfh$ covering $\sV \to \sY$. Now inserting the definitions, one sees that the condition that a morphism of $\ulMCor_\sS$ is in $\ulMSch_\sS$ can be checked on the interiors, so we confirm that $\alpha \circ f$ is in $\ulMSch_\sS$, since $\alpha^\circ \circ f^\circ$ is in $\Sch_{\iS}$.
\end{proof}

\begin{thm} \label{theo:qfhetNisCorRat}
The canonical functor 
\[ forget: \Shv_{\ulMrqfh}(\ulMCor_\sS, \QQ) \to \Shv_{{\ulMtau}}(\ulMCor_\sS, \QQ) \]
is an equivalence of categories for ${\ulMtau} = \ulMNis$ and $\ulMet$.
\end{thm}

\begin{proof}
We want to show that the unit $\id \to forget \circ a$ is an isomorphism where
\[ a: \Shv_{{\ulMtau}}(\ulMCor_\sS, \QQ) \to \Shv_{\ulMrqfh}(\ulMCor_\sS, \QQ) \]
is the sheafification functor. For representable sheaves, this is precisely Cor.~\ref{coro:mqfhSheafRep}. Explicitly, for any separated finite type $\sS$-scheme $\sX$ the presheaf $\QQ_{\tr}(\sX)$ is automatically a $\ulMrqfh$-sheaf on $\ulMSch_\sS$.

Now, it follows that any $\ulMrqfh$-covering $\sY \to \sX$ produces a surjection of ${\ulMtau}$-sheaves $\Q_\tr (\sY) \to \Q_\tr (\sX)$.
Indeed, it has a section since we have
\[ \xymatrix{
\hom_{\Shv_{\ulMtau}(\ulMCor_{\sS}, \QQ)}(\QQ_{\tr}\sX, \QQ_{\tr}\sY) \ar@{}[r]|= \ar[d] & \hom_{\Shv_{\ulMrqfh}(\ulMCor_{\sS}, \QQ)}(\QQ_{\tr}\sX, \QQ_{\tr}\sY) \ar[d] \\
\hom_{\Shv_{\ulMtau}(\ulMCor_{\sS}, \QQ)}(\QQ_{\tr}\sX, \QQ_{\tr}\sX) \ar@{}[r]|= & \hom_{\Shv_{\ulMrqfh}(\ulMCor_{\sS}, \QQ)}(\QQ_{\tr}\sX, \QQ_{\tr}\sX)
} \]
and the right hand side is surjective because $\mcq(\sY) \to \mcq(\sX)$ is the image of the surjection $\sY \to \sX$ in $\Shv_{\ulMrqfh}(\ulMSch_\sS)$ by Prop.~\ref{prop:Zqfhtrmodulus}. {\color{blue} From this it follows that $forget$ preserves epimorphisms. } 
Since it also preserves sums, we discover that in fact it preserves all colimits. Since $a$ also preserves colimits, and every $F \in \Shv_{\ulMrqfh}(\ulMCor_\sS, \QQ)$ can be written as a colimit of representables, $F \to forget \circ a (F)$ being an isomorphism follows from the case $F$ is representable.
\end{proof}

Define $\ulMSmCor_\sS$ to be the full subcategory of $\ulMCor_\sS$ consisting of those objects $\sX \to \sS$ such that $\iX \to \iS$ is smooth. 
Notice that for $\tau = \Zar, \Nis, \et$, the topology $\ulM\tau$ on $\ulMCor_\sS$ induces one on $\ulMSmCor_\sS$.

\begin{coro}
The canonical functor 
\[ forget: \Shv_{\ulMNis}(\ulMSmCor_\sS, \QQ) \to \Shv_{\ulMet}(\ulMSmCor_\sS, \QQ) \]
is an equivalence of categories.
\end{coro}

\begin{rema}
This could also have been proved using the method from \cite[Thm.3.3.23]{CD19}.
\end{rema}

\begin{proof}
The functor $\ulMSmCor \to \ulMCor$ is full so the induced left Kan extension
\[ \PSh(\ulMSmCor) \to \PSh(\ulMCor) \]
is also full, and identifies the former with the smallest subcategory of the latter closed under colimits and containing the image of $\ulMSmCor$. Moreover, this property passes to the categories of sheaves, since $\ulMNis$ and $\ulMet$ coverings of objects of $\ulMSmCor$ are refinable by coverings contained in this subcategory. So it suffices to show 
\[ \Shv_{\ulMNis}(\ulMCor_\sS, \QQ) \cong \Shv_{\ulMet}(\ulMCor_\sS, \QQ), \]
but this follows from Thm.~\ref{theo:qfhetNisCorRat}.
\end{proof}


\chapter{Box product and comparison with [KMSY]}  \label{chap:box}

\section{Comparison with [KMSY]'s smooth modulus pairs}

In Kahn-Miyazaki-Saito-Yamazaki's paper \cite{kmsy1}, the category $\ulMSm_k$ of ``smooth modulus pairs over $k$'' is introduced. Let's write $\ulMSm_k^{\KMSY}$\marginpar{\fbox{$\ulMSm_k^{\KMSY}$}}\index[not]{$\ulMSm_k^{\KMSY}$} %
 for this category. We recall the definition briefly (see also \cite{nistopmod}).

\begin{defn}\label{def:ulMSm-kmsy}
The objects of $\ulMSm_k^{\KMSY}$ are those modulus pairs $\sX=({\hX},{\mX})$ such that
\begin{enumerate}
\item ${\hX}$ is separated and of finite type over $k$, and 
\item ${\iX} = {\hX} - {\mX}$ is smooth over $k$.
\end{enumerate}
In other words, objects of $\ulMSm_k^{\KMSY}$ are the same as the objects of $\ulPSm_k$.

A morphism $f : \sX \to \sY$ in [KMSY]'s category is a morphism between the \emph{interiors} $f^\o : {\iX} \to {\iY}$ which satisfies the following \emph{left properness condition}\marginpar{\fbox{left properness}}\index{left properness} %
 and \emph{modulus condition}\marginpar{\fbox{modulus condition}}\index{modulus condition} :
\begin{itemize}
\item (left properness) Let $\Gamma$ be the graph of the rational map ${\hX} \dashrightarrow {\hY}$ defined by $f^\o$, i.e., $\Gamma$ is the scheme theoretic closure of the graph of $f^\o$ in ${\hX} \times_k {\hY}$. Then the composite map $\Gamma \to {\hX} \times_k {\hY} \to {\hX}$ is proper.
\item (modulus) Let $\nu : \Gamma^N \to \Gamma$ be the normalisation of $\Gamma$. Then we have an inequality $\nu^{-1} ({\mX} \times_k {\hY}) \geq \nu^{-1} ({\hX} \times_k {\mY} )$, where we regard both sides as Weil divisors on $\Gamma^N$. 
Since $\Gamma^N$ is Noetherian and normal, the algebraic Hartogs lemma shows that this is equivalent to that $\nu^{-1} ({\mX} \times_k {\hY}) \supset \nu^{-1} ({\hX} \times_k {\mY} )$.
\end{itemize}
We can show that the usual composition of morphisms of schemes preserves these two conditions.
\end{defn}

\begin{prop}
There exists a natural fully faithful functor 
\[
\ulMSm_k^{\KMSY} \to \ulMSCH_k 
\]
which sends $\sX$ to $\sX$. This induces an equivalence 
\[ 
\ulMSm_k^{\KMSY} \cong \ulMSm_k
\]
with out category. 
\end{prop}

\begin{proof}
Let $\sX$ and $\sY$ be two objects in $\ulMSm_k$.
We define a map of sets
\begin{equation}\label{eq1.8}
\alpha : \ulMSm_k (\sX,\sY) \to \ulMSCH_k (\sX,\sY)
\end{equation}
as follows. Let $f \in \ulMSm_k (\sX,\sY)$.
It is by definition a morphism of schemes $f^\o : {\iX} \to {\iY}$ which satisfies the left properness condition and the modulus condition. 
Let $\Gamma$ be the graph of the rational map $\ol{f} : {\hX} \to {\hY}$ defined by $f^\o$, and let $\nu : \Gamma^N \to \Gamma$ be the normalisation. 
Set 
\[
\sX' := (\Gamma^N, {\mX} \times_{{\hX}} \Gamma^N ).
\]
Then the composite $\ol{p}:\Gamma^N \xrightarrow{\nu} \Gamma \xrightarrow{\pr_1} {\hX}$ defines a minimal morphism $(p:\sX' \to \sX) \in \ulPSCH_k (\sX',\sX)$.
Since $\nu$ is proper (since it is integral and is \textit{of finite type}), and since $\pr_1$ is proper by the left properness condition, so is $\ol{p}$.
Moreover, $\ol{p}$ is an isomorphism over ${\iX} \subset {\hX}$.
Therefore, $p \in \ul{\Sigma}$, and hence it is an isomorphism in $\ulMSCH_k$.

On the other hand, the composite $\ol{q} : \Gamma^N \xrightarrow{\nu} \Gamma \xrightarrow{\pr_2} {\hY}$ satisfies $\ol{q}^{-1} ({\mY}) \subset \ol{p}^{-1} ({\mX})$ by the modulus condition for $f$.
Therefore, $\ol{q}$ defines a morphism $q : \sX' \to \sY$ in $\ulPSCH_k$.
We define
\[
\alpha (f) := q \circ p^{-1} \in \ulMSCH_k (\sX,\sY) .
\]
Then we have $\alpha (f)^\o =f^\o$ by construction.

We check that $\alpha$ respects the composition.
Let $\sX \xrightarrow{f} \sY \xrightarrow{g} \sZ$ be a string of morphisms in $\ulMSm_k$.
Then we have $\alpha (gf)^\o = (gf)^\o =g^\o f^\o = \alpha (g)^\o \alpha (f)^\o = (\alpha (g) \alpha (f))^\o$.
Therefore, Lemma \ref{lem:coincidence2} (2) implies $\alpha (gf) = \alpha (g) \alpha (f)$, as desired. 
Thus, we obtain a functor 
\[
\alpha : \ulMSm_k \to \ulMSCH_k .
\]

Next, we prove that $\alpha$ is fully faithful, i.e., that for any $\sX,\sY \in \ulMSm_k$ the map \eqref{eq1.8} is bijective.
We define the inverse map $\beta$ as follows.

Let $f \in \ulMSCH_k (\sX,\sY)$.
We claim that $f^\o : {\iX} \to {\iY}$ defines a morphism $\beta (f) : \sX \to \sY$ in $\ulMSm_k$.
Indeed, by calculus of fractions, there are $(s:\sX' \to \sX) \in \ul{\Sigma}$ and $g\in \ulPSCH_k (\sX',\sY)$ such that $f=gs^{-1}$.
Note that ${\hX}'$ is separated and of finite type over $k$ since $\ol{s} : {\hX}' \to {\hX}$ is proper by definition of $\ul{\Sigma}$.
Let $\Gamma$, $\Gamma^N$, $\ol{p} : \Gamma^N \to {\hX}$ and $\ol{q} : \Gamma^N \to {\hY}$ be as above.
Then there exists a natural rational map ${\hX}' \dashrightarrow \Gamma^N$.
Let ${\hX}'' := \Gamma_1^N$ be the normalisation of the graph $\Gamma_1$ of this rational map.
Thus, we obtain a commutative diagram
\[\xymatrix{
{\hX}'' \ar[r]^{\ol{u}} \ar[d]_{\ol{t}} & \Gamma^N \ar[d]^{\ol{p}} \ar@/^10pt/[rdd]^{\ol{q}} \\
{\hX}' \ar[r]^{\ol{s}} \ar@/_10pt/[drr]_{\ol{g}} & {\hX} \\
&& {\hY}
}\]
where ${\hX}''$ is normal by definition.
Moreover, note that $\ol{u}$ is equal to the composite map
\[
{\hX}'' \xrightarrow{\text{normalisation}} \Gamma_1 \xrightarrow{\text{closed immersion}} {\hX}' \times_{{\hX}} \Gamma^N \xrightarrow{\text{projection}} \Gamma^N,
\]
where the first and second maps are obviously proper, and the third map is also proper by the properness of $\ol{s} : {\hX}' \to {\hX}$.
Therefore, $\ol{u}$ is proper, and hence $f^\o$ satisfies the left properness condition in Definition \ref{def:ulMSm-kmsy}.

Since $\ol{s}^{-1} {\mX} \supset \ol{g}^{-1} {\mY}$ by the modulus condition for $g$, the commutativity of the above diagram shows 
$\ol{u}^{-1} \ol{p}^{-1} {\mX} \supset \ol{u}^{-1} \ol{q}^{-1} {\mY}$.
Then \cite[Lemma 2.2]{KP} implies (by normality of ${\hX}''$ and $\Gamma^N$) that $\ol{p}^{-1} {\mX} \supset \ol{q}^{-1} {\mY}$, i.e. that $f^\o$ satisfies the modulus condition in Definition \ref{def:ulMSm-kmsy}.
This finishes the proof of the claim. 

Thus, we obtain a map $\beta : \ulMSCH_k (\sX,\sY) \to \ulMSm_k (\sX,\sY)$.
It suffices to show that $\alpha$ and $\beta$ are inverse to each other.
For any $f \in \ulMSm_k (\sX,\sY)$, we have $(\beta \alpha (f))^\o = f^\o$ by construction. Therefore, we have $\beta \alpha (f) = f$.
Conversely, for any $g \in \ulMSCH_k (\sX,\sY)$, we also have $(\alpha \beta (g))^\o = g^\o$ by construction. 
Therefore, Lemma \ref{lem:coincidence2} (2) shows $\alpha \beta (f) = f$. 
This finishes the proof.
\end{proof}

\section{The box product of modulus pairs}

We have proven in Theorem \ref{thm:pullback} that the category $\ulMSCH_{\sS}$ of modulus pairs admits all finite limits. In particular, it admits categorical product. Therefore, we already obtained a tensor structure of modulus pairs. 

In this section, we introduce another tensor structure on $\ulMSCH_{\sS}$, which we call \emph{box product}, and study its relation with categorical product. The box product will play an important role in the construction of the category of motives with modulus.

\section{The box product}\label{sec:boxtensor}

In this subsection, we introduce another tensor structure on $\ulMSCH$, which we call \emph{the box product}. 

\begin{defn}\label{def:tensor-S}
Let $\sS$ be a modulus pair, and let $f : \sX \to \sS$, $g : \sY \to \sS \in \ulPSm_{\sS}$ be ambient morphisms.  
Let $\hP_0$ be the scheme theoretic closure of the open immersion $\iX \times_{\iS} \iY \to \hX \times_{\hS} \hY$.
Then the same argument as in Construction \ref{cons:admBlowup} shows that the pullbacks of effective Cartier divisors $\mS |_{\hP}$, $\mX |_{\hP}$ and $\mY |_{\hP}$ are well-defined.
\footnote{For example, tso see that $\mS |_{\hP}$ is well-defined, it suffices to check that $f : \hP \to \hS$ satisfies the assumption in Lemma \ref{lem:pb-dense}. This is OK since $f^{-1}(\iS)$ contains $\iX \times_{\iS} \iY$ which is scheme theoretically dense in $\hP$ by Lemma \ref{lem:pb-sch-dense}.}

We define a modulus pair $\sX \boxtimes_{\sS} \sY$ as follows:\marginpar{\fbox{$\boxtimes$}}\index[not]{$\boxtimes$}
\[
\sX \boxtimes_{\sS} \sY := (\hP , \mX |_{\hP} + \mY |_{\hP} - \mS |_{\hP}),
\]
which we call the \emph{box product of $\sX$ and $\sY$ over $\sS$}\marginpar{\fbox{box product}}\index{product of modulus pairs!box product}.

Note that $(\sX \boxtimes_{\sS} \sY)^\o = \iX \times_{\iS} \iY$ since $\mS|_{\hP} \subset \mX|_{\hP}$ and $\mS|_{\hP} \subset \mY|_{\hP}$.
\end{defn}

\begin{rema}\ 
\begin{enumerate}
\item In general, the box product $\boxtimes_{\sS}$ is different from the categorical product $\times_{\sS}$. However, we will see in Lemma \ref{lem:boxtimes-to-times} that there is a canonical map from the former to the latter. 

\item If $\mS = \varnothing$ and if either of $\hX$ or $\hY$ is flat over $\hS$, then the box product is simply computed as
\[
\sX \boxtimes_{\sS} \sY := (\hX \times_{\hS} \hY , \mX \times_{\hS} \hY + \hX \times_{\hS} \mY).
\]
When $\hS$ is a spectrum of a field, then this definition coincides with the tensor product of modulus pairs introduced in \cite{kmsy1}. In {\it{loc. cit.}}, the tensor product is denoted by $\otimes$, but we will use $\boxtimes$ in order to make it clear that this product is different from the categorical product in $\ulMSCH$.

\item  In the definition of box product, we need to subtract $\mS |_{\hP}$ for various reasons. 
One important reason is that $\sS$ becomes a unit with this definition. 
Another reason is that for a modulus pair $\sX$ over $\sS$, the $0$ or $1$ embedding $\sX \hookrightarrow \sX \boxtimes_{\sS} \bcube_{\sS}$ becomes admissible (in fact, minimal).
\end{enumerate}
\end{rema}

The box product satisfies the obvious functorial property:

\begin{prop}
For any modulus pair $\sS$, the box tensor product defines a symmetric monoidal structure 
\[
-\boxtimes_{\sS} - : \ulPSCH_{\sS} \times \ulPSCH_{\sS} \to \ulPSCH_{\sS},
\]
with $\sS$ the unit object. 
Moreover, this descends to $\ulMSCH_{\sS}$: 
\[
-\boxtimes_{\sS} - : \ulMSCH_{\sS} \times \ulMSCH_{\sS} \to \ulMSCH_{\sS}.
\]
\end{prop}

\begin{proof}
With an ambient morphism $\sX \to \sS$ fixed, we obtain an  endofunctor $\sX \boxtimes_{\sS} -$ on $\ulPSCH_{\sS}$.
Indeed, given a morphism $\sZ \to \sY$ in $\ulPSCH_{\sS}$, the natural morphism $\hX \times_{\hS} \hZ \to \hX \times_{\hS} \hY$ induces $p:\hP_{Z} \to \hP_{Y}$, where $\hP_{Z}$ (resp. $\hP_{Y}$) is the scheme theoretic closure of $\iX \times_{\iS} \iZ \to \hX \times_{\hS} \hZ$ (resp. $\iX \times_{\iS} \iY \to \hX \times_{\hS} \hY$).
Since we obviously have $p^\ast (\sX \boxtimes_{\sS} \sY)^\infty \subset (\sX \boxtimes_{\sS} \sZ)^\infty$ by the admissibility of $\sZ \to \sY$, we obtain an ambient morphism $\sX \boxtimes_{\sS} \sZ \to \sX \boxtimes_{\sS} \sY$.
We obviously have $\sS \boxtimes_{\sS} \sX = \sX = \sX \boxtimes_{\sS} \sS$. 
The associativity and symmetry of $\boxtimes_{\sS}$ are also obvious.

To prove the descent, it suffices to check that for any $\sX \to \sS$, the endofunctor $\sX \boxtimes_{\sS} -$ sends the class of admissible morphisms $\ul{\Sigma}$ to $\ul{\Sigma}$. 
Take $(\pi: \sZ \to \sY) \in \ul{\Sigma}$. 
Then, similarly as above, the proper birational morphism $\ol{\pi} : \hZ \to \hY$ induces a proper birational morphism $p:\hP_{Z} \to \hP_{Y}$, and we have $p^\ast (\sX \boxtimes_{\sS} \sY)^\infty = (\sX \boxtimes_{\sS} \sZ)^\infty$ by the minimality of $\pi$.
Therefore, the induced morphism $\sX \boxtimes_{\sS} \sZ \to \sX \boxtimes_{\sS} \sY$ belongs to $\ul{\Sigma}$, as desired.
\end{proof}

Moreover, we have:

\begin{lemma}\label{lem:boxtimes-to-times}
For any $\sX, \sY \in \ulMSCH_{\sS}$, there exists a unique morphism in $\ulMSCH_{\sS}$ of the form
\[
\sX \boxtimes_{\sS} \sY \to \sX \times_{\sS} \sY
\]
which induces $\id_{\iX \times_{\iS} \iY}$ on the interior. 
\end{lemma}

\begin{proof}
Since both sides of the map is functorial on $\ulMSCH_{\sS}$, we may assume that $\sX \to \sS$ and $\sY \to \sS$ are ambient by calculus of fractions. 
By definition, we have
\[
\hP := \ol{\sX \boxtimes_{\sS} \sY} = \text{the scheme theoretic closure of $\iX \times_{\iS} \iY \to \hX \times_{\hS} \hY$},
\]
\[
\hQ := \ol{\sX \times_{\sS} \sY} = \text{the blow up of $\hP = \ol{\sX \boxtimes_{\sS} \sY}$ along $\mX |_{\hP} \times_{\hP} \mY |_{\hP}$}.
\]
Therefore, we have a canonical blow-up $q:\hQ \to \hP$ which is an isomorphism over $\iX \times_{\iS} \iY$.
Set $\mP := (\sX \boxtimes_{\sS} \sY)^\infty$ and $\mQ := (\sX \times_{\sS} \sY)^\infty$.
To conclude the proof, it suffices to show that $q^* \mP \supset \mQ$ (Note the direction of the inclusion!).
This follows from the inclusion $\mS |_{\hQ} \subset E$, where $E=q^{-1} (\mX |_{\hP} \times_{\hP} \mY |_{\hP})$ denotes the exceptional divisor, which holds by the admissibility of $\sX \to \sS$ and $\sY \to \sS$.
\end{proof}

In some special cases, the box tensor product $\boxtimes_{\sS}$ is computed easily.

\begin{lemma}\label{lem:minimal-boxtensor}
Let $f : \sX \to \sS$ and $g : \sT \to \sS$ be ambient morphisms.

\begin{enumerate}
\item If $\mX \times_{\hS} \hY$ and $\hX \times_{\hS} \mY$ are effective Cartier divisors on $\hX \times_{\hS} \hY$, then $\mS \times_{\hS} (\hX \times_{\hS} \hY)$ is also an effective Cartier divisor, and we have
\[
\sX \boxtimes_{\sS} \sT = (\hX \times_{\hS} \hY, \mX \times_{\hS} \hY + \hX \times_{\hS} \mY - \mS \times_{\hS} (\hX \times_{\hS} \hY) ).
\]

\item Assume that $g$ is minimal and $\ol{g}$ is flat.
Then we have
\[
\sX \boxtimes_{\sS} \sT = (\hX \times_{\hS} \hT , \mX \times_{\hS} \hT) = \sX \times_{\sS} \sT.
\]
\end{enumerate}
\end{lemma}

\begin{proof}
(1): By Proposition \ref{prop:always-dense} and the assumptions, the open immersion $\iX \times_{\iS} \iY \to \hX \times_{\hS} \hY$ is scheme theoretically dense. 
The rest of the assertion follows immediately from this and by construction of $\boxtimes$. 

(2): This is immediate from (1), noting that the pullbackability of $\mX$ implies that of $\mT$ (by minimality of $\sT \to \sS$).
\end{proof}

\section{Local comparison of categorical product and box product}

In general, the morphism in Lemma \ref{lem:boxtimes-to-times} is not an isomorphism. 
However, the next proposition shows that in a Zariski local, $\bcube \otimes (-)$-invariant setting, it becomes an isomorphism. The technique is a generalisation of the one used in \cite[Lem.10]{KS19} to show that $(\PP^2, \PP^1)$ is contractible.

\begin{defi}
Let $\sX$ be a modulus pair and $\sL \in \Pic(\hX)$ a line bundle on $\hX$. We define \marginpar{\fbox{$\bcube_\sX(\sL)$}}\index[not]{$\bcube_\sX(\sL)$}
\[ \bcube_\sX(\sL) = (\PP(\sL \oplus \OO_{\hX}), X^\infty + \{\infty\}) \]
where we use the same symbol for ~$X^\infty$ and its pullback to $\PP(\sL \oplus \OO_{\hX})$, and $\{\infty\} = \PP(\sL) \subseteq \PP(\sL \oplus \OO_{\hX})$. 
\end{defi}

Consider the following three modulus pairs, associated to a pair $\sX = ({\hX}, {\mX})$. Cf. the diagram below.
\begin{enumerate}
 \item 
 \[ {\hX} \ambtimes \bcube = \sA = (Bl_{{\hX} {\times} \PP^1} ({\mX} \cap \{\infty\}), \pi^*{\mX} {+} \pi^*\{\infty\} {-} E). \]
 The total space is the blowup of ${\hX} {\times} \PP^1$ in the intersection of the two pullbacks ${\mX} \times \PP^1$ and $\{\infty\} \times {\hX}$, and the divisor is the sum of these two pullbacks minus the exceptional divisor. This is a model for ${\hX} \times \bcube$ in $\ulMSCH$. 
 \item 
 \[ \bcube_\sX(\sI_{{\mX}}) = \sB = ({\hB}, {\mB}) = (\PP(\sI_{{\mX}} \oplus \sO_{{\hX}}), {\mX} + \{\infty\}). \]
The total space is the $\PP^1$-bundle associated the the locally free $\sO_{{\hX}}$-module $\sI_{{\mX}} \oplus \sO_{{\hX}}$ where $\sI_{{\mX}}$ is the locally free rank one ideal sheaf associated to the effective Cartier divisor ${\mX}$. The divisor is the sum of the pullback of ${\mX}$ from the base ${\hX}$, and the divisor at infinity $\{\infty\} = \PP(\sI_{{\mX}} )$ associated to the canonical quotient $\sI_{{\mX}} \oplus \sO_{{\hX}} \to \sI_{{\mX}}$.
 \item 
 \[ \sC = (Bl_{{\hB}}({\mX} \cap \{0\}), \beta^*{\mB}). \]
In addition to the divisor $\{\infty\}$ on ${\hB}$, there is the divisor at zero $\{0\} = \PP(\sO_{{\hX}})$ associated to the other canonical quotient $\sI_{{\mX}} \oplus \sO_{{\hX}} \to \sO_{{\hX}}$. To obtain ${\hC}$, we blowup the intersection ${\mX} \cap \{0\}$. The divisor is the pullback of ${\mB}$ along the canonical projection $\beta: {\hC} \to {\hB}$.
\end{enumerate}

\begin{figure}
\[ \xymatrix@R=0pt{
& \sX \otimes \bcube & \ar[l]_-\pi {\hX} \ambtimes \bcube & = \sA \\
\sB = & \bcube_\sX(\sI_{{\mX}}) & \ar[l]^-\beta \sC
} \]
Warning: $\pi$ is not admissible!
\begin{center}
\includegraphics[width=10cm]{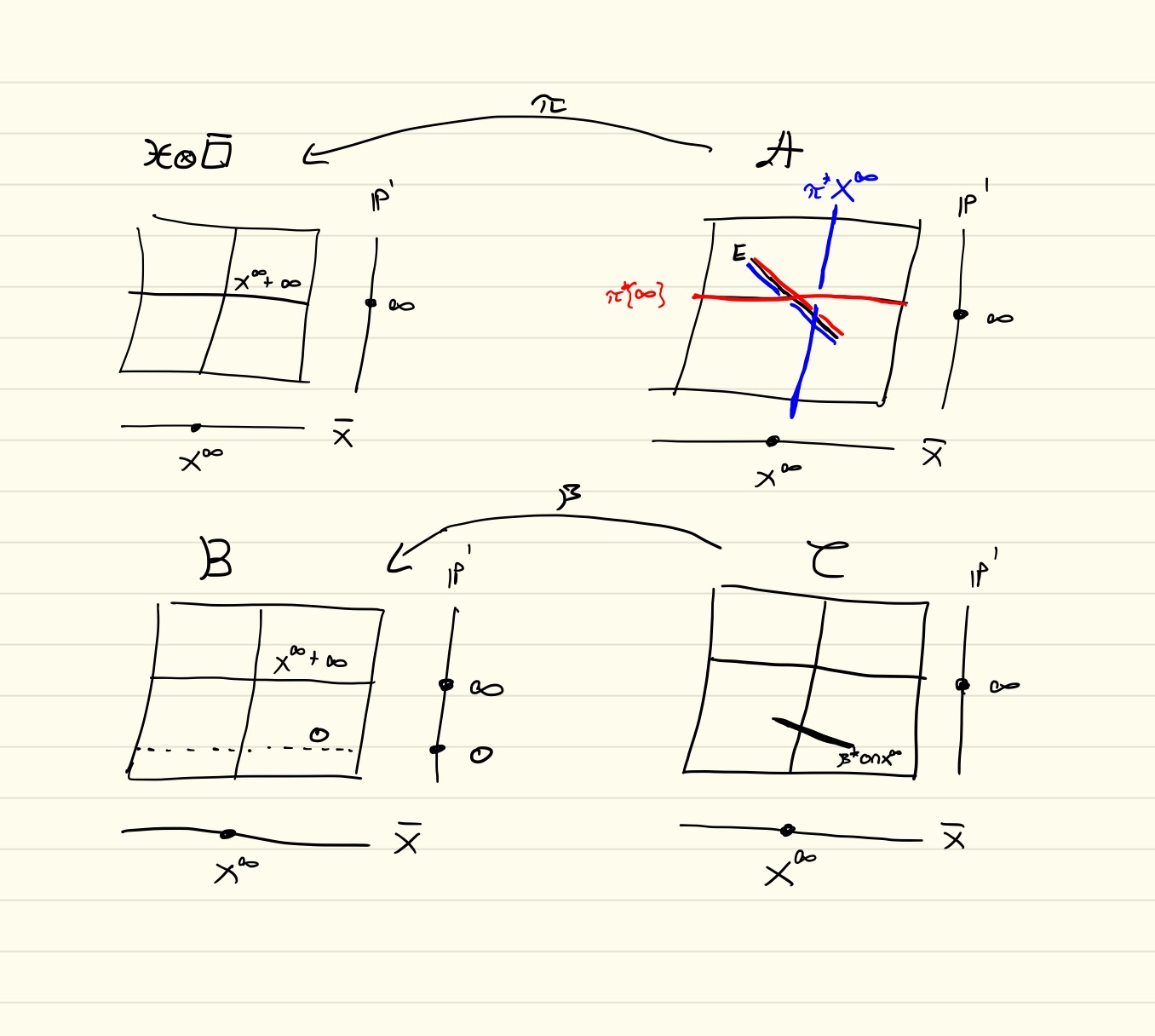}
\end{center}
\end{figure}

\begin{prop} \label{prop:AisoC}
There is an \emph{ambient} isomorphism $\sA \cong \sC$. That is, an isomorphism of schemes ${\hA} \cong {\hC}$ which identifies the divisors ${\mA}$ and ${\mC}$. Consequently, there is an admissible blowup
\[ \bcube\ambtimes {\sX}   \to \bcube_\sX(\sI_{{\mX}}) \]
and therefore, an isomorphism
\[ \bcube \times {\sX} \cong \bcube_\sX(\sI_{{\mX}}) \qquad \in \ulMSCH \]
\end{prop}

\begin{rema}
Note that the isomorphism above is not an equality on the interior. There is a twisting by the parameter for $\bcube$ which happens.
\end{rema}

\begin{proof}
The first observation is that the constructions of the schemes ${\hA}, {\hB}, {\hC}$ and the divisors ${\mA}, {\mB}, {\mC}$ are functorial in ${\hX}$ for open immersions ${\hU} \to {\hX}$. So it suffices to define isomorphisms $\sA \cong \sC$ in the case ${\hX}$ is affine and ${\mX}$ is globally free, which are functorial for open immersions. %
In the case ${\hX} = \Spec(R)$ is affine and ${\mX} = (f)$ is globally free,  all rings that we need can be considered as subrings of $R[\tfrac{1}{f}][t, \tfrac{1}{t}]$. This observation is more than bookkeeping or notation. It ensures that there is a canonical way to glue the various affine schemes together. %
An alternative point of view is to replace $R$ and $\Spec$ everywhere with $\sO_{{\hX}}$ and $\ul{\Spec}$. Then instead of all rings being subrings of $R[t, \tfrac{1}{tf}]$, all $\sO_{{\hX}}$-algebras are subalgebras of $\sO_{X^\circ}[t, \tfrac{1}{t}]$.

The scheme ${\hX} \times \PP^1$ has the standard open covering 
\[ 
 \Spec(R[t]), \qquad 
\Spec(R[\tfrac{1}{t}]). \] 
In this notation, the ideal of the closed subscheme ${\mX} \cap \{\infty\}$ is $(f, \tfrac{1}{t})$ (it is contained inside $\Spec(R[\tfrac{1}{t}])$ and does not intersect $\Spec(R[t])$). The blowup of the closed subscheme ${\mX} \cap \{\infty\} = (f, \tfrac{1}{t})$ then has the standard open covering\footnote{\label{footnote:blowupcalculation} Here we use the fact that for some ideal $I = (a, b)$ of a ring $A$, the blowup is open covered by the affine schemes of $A[\tfrac{a}{b}]$ and $A[\tfrac{b}{a}]$. In our case $A = R[\tfrac{1}{t}], a = f, b = \tfrac{1}{t}$.} 
\[ {\hA} \supset \qquad 
\Spec(R[t]), \qquad 
\Spec(R[\tfrac{1}{t}, tf]), \qquad 
\Spec(R[\tfrac{1}{t}, \tfrac{1}{ft}]). \]
So we have described ${\hA}$. Let is describe ${\mA}$. On the above three affines, 
\begin{enumerate}
 \item the divisor $\pi^*{\mX}$ is $(f)$ (on all three), 
 \item the divisor $\pi^*\{\infty\}$ is $(1), (\tfrac{1}{t}), (\tfrac{1}{t})$ respectively, and 
 \item the divisor $E$ is $(1), (\tfrac{1}{t}), (f)$ respectively. Indeed, on the second open we have $(\tfrac{1}{t}) = (tf \cdot \tfrac{1}{t}, \tfrac{1}{t}) = (f, \tfrac{1}{t})$ and on the third open $(f) = (f, f \cdot \tfrac{1}{ft}) = (f, \tfrac{1}{t})$. 
\end{enumerate}
Hence, the divisor ${\mA} = \pi^*{\mX} + \pi^*\{\infty\} - E$, on the above three opens, is respectively
\[ {\mA} = (f), (f), (\tfrac{1}{t}). \]
Now we turn to $\sB$ and $\sC$. We continue to consider all rings as subrings of $R[\tfrac{1}{f}][t, \tfrac{1}{t}]$. The scheme ${\hB}$ has the standard open covering 
\[ 
\Spec(R[ft]), \qquad 
\Spec(R[\tfrac{1}{ft}]). 
\] Of course, since we are assuming ${\mX}$ is gobally free, we could also have used $\Spec(R[t]), \Spec(R[\tfrac{1}{t}])$, but the presence of the coefficient $f$ is what produces the twisted bundle $\PP(\sI_{X^{\infty}} \oplus \sO_{{\hX}})$ instead of $\PP^1_{{\hX}}$ when we glue the locally principal pieces together. The divisor ${\mB} = {\mX} + \{\infty\}$ on these two open affines is $(f)$ and $(\tfrac{1}{t}) = (f \cdot \tfrac{1}{ft})$ respectively. The divisor $\{0\}$ on these two opens is $(ft)$ and $(1)$ respectively. The intersection $\{0\} \cap {\mX}$ therefore has ideal $(ft, f)$ on $\Spec(R[ft])$, and is disjoint from $\Spec(R[\tfrac{1}{ft}])$. The blowup ${\hC}$ has the standard open covering\footnote{Here we have used the same observation as Footnote~\ref{footnote:blowupcalculation} with $A = R[ft]$, $a = ft$, $b = f$ to construct the first two affines.}
\[
{\hC} \supset \qquad 
\Spec(R[ft, t]), \qquad
\Spec(R[ft, \tfrac{1}{t}]), \qquad 
\Spec(R[\tfrac{1}{ft}]). \] 
The divisor ${\mC} = \beta^*{\mB}$ on these three affines is, respectively,
\[ {\mC}: \qquad (f), \qquad (f), \qquad (\tfrac{1}{t}).  \]
At this point, one sees directly that ${\hC}$ (resp. ${\mC}$) is \emph{equal} to ${\hA}$ (resp. ${\mA}$). One also checks that all constructions are independent of the choice of generator of $(f)$, i.e., for any unit $u \in R^*$ replacing $f$ with $uf$ gives all the same rings. Also, everything is functorial in $R$. 
\end{proof}

\section{Compatibility between fiber product and box product}

For later use, we study further the relation between fiber product and box product. The following result shows that the box product preserves fiber product. 

\begin{prop}\label{prop:tensor-fiber}
Let $\sT \to \sS$ be a morphism in $\ulMSCH$, and let $\sY \to \sX$, $\sZ \to \sX$ be morphisms in $\ulMSCH_{\sS}$.
Then there exists a natural isomorphism in $\ulMSCH_{\sS}$:
\[
(\sY \times_{\sX} \sZ) \boxtimes_{\sS} \sT \xrightarrow{\cong} (\sY \boxtimes_{\sS} \sT) \times_{\sX \boxtimes_{\sS} \sT } (\sZ \boxtimes_{\sS} \sT).
\]
\end{prop}

\begin{proof}
It suffices to show that the natural morphism 
\[
f:(\sY \times_{\sX} \sZ) \boxtimes_{\sS} \sT \xrightarrow{} (\sY \boxtimes_{\sS} \sT) \times_{\sX \boxtimes_{\sS} \sT } (\sZ \boxtimes_{\sS} \sT),
\]
which is induced by the universal property of fiber product, is an isomorphism in $\ulMSCH$.
By calculus of fractions, we may assume that $\sT \to \sS$, $\sY \to \sX$ and $\sZ \to \sX$ are ambient, and that $\sY \times_{\sX} \sZ = \sY \ambtimes_{\sX} \sZ$.
Moreover, again by calculus of fractions, there exists a morphism $p:P \to (\sY \times_{\sX} \sZ) \boxtimes_{\sS} \sT$ which belongs to $\ul{\Sigma}$ such that the composite $g:=f \circ p$ is ambient.  

By construction, the above morphism $f$ induces on the interiors the identification
\[
(\iY \times_{\iX} \iZ) \times_{\iS} \iT = (\iY \times_{\iS} \iT) \times_{\iX \times_{\iS} \iT } (\iZ \times_{\iS} \iT),
\]
and that the ambient spaces of the source and the target of $f$ are proper over $(\hY \times_{\hX} \hZ) \times_{\hS} \hT = (\hY \times_{\hS} \hT) \times_{\hX \times_{\hS} \hT } (\hZ \times_{\hS} \hT)$.
Therefore, to prove that $f$ is an isomorphism in $\ulMSCH$, it suffices to show that the above ambient morphism $g$ is minimal.
That is, we need to check
\begin{equation}\label{eq:LR}
\ol{p}^* ((\sY \times_{\sX} \sZ) \boxtimes_{\sS} \sT)^\infty 
= \ol{g}^\ast ((\sY \boxtimes_{\sS} \sT) \times_{\sX \boxtimes_{\sS} \sT } (\sZ \boxtimes_{\sS} \sT))^\infty .
\end{equation}

We compute the left hand side of \eqref{eq:LR} as follows.
\begin{align*}
\text{(LHS)}
&=(\sY \times_{\sX} \sZ)^\infty |_{\hP} + \mT |_{\hP} - \mS |_{\hP} \\
&=(\mY |_{\hP} + \mZ |_{\hP} - E)
+ \mT |_{\hP} - \mS |_{\hP} \\
&= \mY |_{\hP} + \mZ |_{\hP} + \mT |_{\hP} - (E + \mS |_{\hP}),
\end{align*}
where we set $E = (\mX |_{\hP}) \times_{\hP} (\mY |_{\hP})$. 
Note that $E$ is an effective Cartier divisor on $\hP$ by construction. 
We compute the right hand side of \eqref{eq:LR} as follows.
\begin{align*}
\text{(RHS)}
&= (\sY \boxtimes_{\sS} \sT)^\infty |_{\hP} + (\sZ \boxtimes_{\sS} \sT)^\infty |_{\hP} - F \\
&= (\mY |_{\hP} + \mT |_{\hP} - \mS |_{\hP}) + (\mZ |_{\hP} + \mT |_{\hP} - \mS |_{\hP}) - F \\
&= \mY |_{\hP} + \mZ |_{\hP} + \mT |_{\hP} - (F+2\mS |_{\hP} - \mT |_{\hP}),
\end{align*}
where $F = (\sY \boxtimes_{\sS} \sT)^\infty |_{\hP} \times_{\hP} (\sZ \boxtimes_{\sS} \sT)^\infty |_{\hP}$, which is an effective Cartier divisor on $\hP$ by construction.
Therefore, we are reduced to proving the following claim.
\begin{claim}\label{claim:E-F}
$F = E + \mT |_{\hP} - \mS |_{\hP}$.
\end{claim}

For the proof of the claim, we need the following elementary lemma.

\begin{lemma}\label{lem:DDH}
Let $X$ be a scheme, $D_1,D_2,H$ effective Cartier divisors on $X$. 
Assume that $E:=D_1 \times_X D_2$ is also an effective Cartier divisor on $X$.
Then we have
\[
(D_1 + H) \times_{X} (D_2 + H) = E + H, 
\]
where $+$ denotes the sum of effective Cartier divisors. 
In particular, the left hand side is an effective Cartier divisor. 
\end{lemma}

\begin{proof}
This can be checked by an easy local computation. 
\end{proof}

We prove Claim \ref{claim:E-F} as follows:
\begin{align*}
F 
&= (\sY \boxtimes_{\sS} \sT)^\infty |_{\hP} \times_{\hP} (\sZ \boxtimes_{\sS} \sT)^\infty |_{\hP} \\
&= (\mY |_{\hP} + \mT |_{\hP} - \mS |_{\hP}) \times_{\hP} (\mZ |_{\hP} + \mT |_{\hP} - \mS |_{\hP}) \\
&=^\dag (\mY |_{\hP} \times_{\hP} \mZ |_{\hP}) + \mT |_{\hP} - \mS |_{\hP} \\
&= E + \mT |_{\hP} - \mS |_{\hP} ,
\end{align*}
where we obtain $=^\dag$ by Lemma \ref{lem:DDH} noting $\mT |_{\hP} - \mS |_{\hP}$ is an effective Cartier divisor by the admissibility of $\sT \to \sS$.
This finishes the proof.
\end{proof}


%

\appendix

\chapter{Miscellanea}



\section{Localisation as sheafification}

In this section we give two different ways of thinking about presheaves on localisations by identifying the three adjunctions below.

\begin{align*}
loc^*: \PSh(C) &\rightleftarrows \PSh(C[S^{-1}]) : loc_* \\
a_S: \PSh(C) &\rightleftarrows \PSh_{S\textrm{-loc}}(C) \\
a_\sigma: \PSh(C) &\rightleftarrows \Shv_\sigma(C) 
\end{align*}
Here, $loc^*$ is the left Kan extension along $C \to C[S^{-1}]$, 
, $loc_*$ is composition with $C \to C[S^{-1}]$, 
the functor $a_S$ is defined as 
\begin{equation} \label{eq:aS}
a_S: F \mapsto \left ( X \mapsto \varinjlim_{V \in S/X} F(V)\right )
\end{equation}
and $a_\sigma$ is sheafification where $\sigma$ is the topology whose covering sieves are those sieves containing an element of $S$. Finally, $\PSh_{S\textrm{-loc}}(C) \subseteq \PSh(C)$ is the full subcategory of presheaves which are \emph{$S$-local}\marginpar{\fbox{$S$-local}}\index{$S$-local} in the sense that $S$-local presheaves send elements of $S$ to isomorphisms.

\begin{prop} \label{prop:locPSh}
Suppose that $C$ is a category and $S$ is class of morphisms admitting a right calculus of fractions. Then the restriction 
$ loc_*: \PSh(C[S^{-1}]) \to \PSh(C)$
along the canonical functor $loc: C \to C[S^{-1}]$ induces an equivalence of categories 
\[ \PSh(C[S^{-1}]) \cong \PSh_{S\textrm{-loc}}(C). \]
In particular, $loc_*$ is fully faithful. Via this equivalence, the left Kan extension $loc^*: \PSh(C) \to \PSh(C[S^{-1}])$ is canonical identified with the functor $a_S$ defined in Eq.~\eqref{eq:aS}.
\end{prop}

\begin{proof}
The equivalence $\PSh(C[S^{-1}]) \cong \PSh_{S\textrm{-loc}}(C)$ follows from the universal property of $loc: C \to C[S^{-1}]$, so it suffices to show that $a_S$ is left adjoint to the inclusion. It's an easy exercise about right calculus of fractions to check that $a_SF$ sends elements of $S$ to isomorphisms, and it's also straight forward to check that if $G$ is $S$-local, then $\hom(a_SF, G) = \hom(F, G)$.
\end{proof}

\begin{prop} \label{prop:locShv}
Suppose that $C$ is a category and $S$ is class of morphisms admitting a right calculus of fractions. Suppose further that every morphism in $S$ is a categorical monomorphism, and let $\sigma$ be the topology whose coverings are sieves containing an element of $S$.

Then there is a canonical equality of subcategories 
\[ \Shv_\sigma(C) = \PSh_{S\textrm{-loc}}(C) \]
of $\PSh(C)$. %
Consequently, the two left adjoints to the two (equal) inclusions are equal. That is, sheafification is identified with $a_S$ from 

\end{prop}

\begin{rema}
Notice that having a right calculus of fractions implies $S$ forms a ``coverage'' (a more general version of a pretopology applicable to categories without fibre products).
\end{rema}

\begin{proof}
Suppose that $F$ satisfies the sheaf condition, $f: Y \to X$ is in $S$, and consider the covering sieve $R = im(Y \to X) \subseteq X$ of some object $X \in C$. The sheaf condition is that 
\[ \hom_{\PSh}(X, F) = \hom_{\PSh}(R, F). \]
However, $f$ is a categorical monomorphism, so $R = Y$, so the sheaf condition becomes $F(X) = F(Y)$. That is, a presheaf $F$ is satisfies the sheaf condition if and only if it sends elements of $S$ to isomorphisms. 
\end{proof}


\section{Pro-objects} \label{chap:proObjects}

\subsection{Reminders} \label{sec:Proreminders}

\begin{defi}[{\cite[Def.I.2.7]{SGA41}}]
A category $I$ is \emph{cofiltered} if:
\begin{enumerate} \setcounter{enumi}{-1}
 \item $I$ is nonempty.
 \item For any two objects $a, b$ of $I$, there exists a diagram of the form $c {^\nearrow_\searrow}\overset{a}{\underset{b}{\vphantom{l}}}$.
 \item Every diagram of the form $j \overset{u}{\underset{v}{\rightrightarrows}} i$ fits into a diagram $k\stackrel{w}{\to} j \overset{u}{\underset{v}{\rightrightarrows}} i $ such that $uw = vw$.
\end{enumerate}
\end{defi}

\begin{defi}[{\cite[I.8.2.1]{SGA41}}]
A \emph{pro-object} $(P_{λ})_{λ \in \Lambda}$ in a category $\sC$ is a covariant functor $P_\cdot: \Lambda \to \sC$ from a cofiltered category $\Lambda$. 
\end{defi}

If the indexing category $\Lambda$ is essentially small, then we can associate to $(P_{λ})_λ$ the functor 
\[ L(P): \sC \to \Set; \qquad \qquad L(P)(\cdot) := ``\colim" \hom_{\sC}(P_λ, \cdot) \]
where $``\colim"$ means the colimit is taken in the category of functors. We will implicitly assume all pro-objects are indexed by essentially small categories. We will always only consider pro-objects with values in essentially small categories, but this is just for convenience.\footnote{We do have some statements 
using pro-objects in categories which are not essentially small, but those don't use any results from this section.}

\begin{defi}[{\cite[I.8.2.4]{SGA41}}]
A morphism of pro-objects $(P_λ)_λ \to (Q_μ)_μ$ is a natural transformation of the associated functors. That is,
\[ \hom_{\Pro(\sC)}(P, Q) := \hom_{\Fun(\sC, \Set)}(L(Q), L(P)), \]
\cite[Eq.I.8.2.4.3]{SGA41}. 
\end{defi}

As limits are calculated objectwise in the category of functors, a simple calculation shows that 
\[ \hom_{\Pro(\sC)}(P, Q) \cong \lim_μ \colim_λ \hom_\sC(P_λ, Q_μ), \]
\cite[Eq.I.8.2.5.1]{SGA41}. 
%
Hence, by definition we have a fully faithful embedding
\[ \Pro(\sC)^{\op} \hookrightarrow \Fun(\sC, \Set). \]
Notice the $(-)^\op$. This is for compatibility with coYoneda: 
\[ \sC^{\op} {\hookrightarrow} \Pro(\sC)^{\op} {\hookrightarrow} \Fun(\sC, \Set). \]

If $F: \sC \to \Set$ is a functor, we consider the comma category $F / \sC$ whose objects are natural transformations $F \to \hom_{\sC}(X, \cdot)$ and morphisms are commutative triangles $F \to \hom_{\sC}(X, \cdot) \to \hom_{\sC}(Y, \cdot)$, \cite[I.8.3.2]{SGA41}. This category is equipped with a canonical forgetful functor
\[ F / \sC \to \sC, \]
\cite[Eq.I.8.3.2.2]{SGA41}. The functor $F$ is always the limit of this functor,
\[ F(\cdot) \cong \lim_{X \in F / \sC} \hom_{\sC}(X, \cdot), \]
\cite[8.3.2]{SGA41}.

\begin{theo}[{\cite[Thm.8.3.3]{SGA41}}] \label{theo:proBasics}
Let $\sC$ be a small category and consider a functor $F: \sC \to \Set$. The following are equivalent:
\begin{enumerate}
 \item $F$ is in the image of $\Pro(\sC)^{\op} \hookrightarrow \Fun(\sC, \Set)$.
 \item The category $F / \sC$ (defined above) is filtered.
 \item (If $\sC$ admits finite colimits.) The functor $F$ preserves finite colimits.
\end{enumerate}
\end{theo}

\begin{rema} \label{rema:underPro}
By the above theorem, any pro-object $(P_λ)_λ$ admits a canonical indexing category, namely, its undercategory $L(P)/\sC$. 
\end{rema}

\begin{rema}
In the ∞-category setting, Theorem~\ref{theo:proBasics} is a \emph{definition} of pro-objects, cf.\cite[Def.5.3.5.1]{HTT}. The ∞-category version of Thm.~\ref{theo:proBasics} is \cite[Cor.5.3.5.4]{HTT}.
\end{rema}


\subsection{Functoriality of pro-objects} \label{sec:ProObFun}

\begin{lemm} \label{lemm:proComDiag}
Given a functor $f: \sC \to \sD$ between small categories, the embedding in Thm.~\ref{theo:proBasics} fits into a commutative square
\begin{equation} \label{lemm:proComDiag:sq1}
\xymatrix{
\Pro(\sC)^{\op} \ar[r]\ar@{}@<0.5em>[r]^{\textrm{fully faith.}} \ar[d]_{\Pro(f)^{\op}} & \Fun(\sC, \Set) \ar[d]^{\textrm{left Kan extension}} \\
\Pro(\sD)^{\op} \ar[r]\ar@{}@<0.5em>[r]^{\textrm{fully faith.}} & \Fun(\sD, \Set)
} 
\end{equation}
Restriction induces a right adjoint to $\Pro(\sC)^{\op} \to \Pro(\sD)^{\op}$ forming a commutative square
\begin{equation} \label{lemm:proComDiag:sq2}
\xymatrix{
\Pro(\sC)^{\op} \ar[r]\ar@{}@<0.5em>[r]^{\textrm{fully faith.}}  & \Fun(\sC, \Set)  \\
\Pro(\sD)^{\op} \ar[r]\ar@{}@<0.5em>[r]^{\textrm{fully faith.}} \ar@{-->}[u] & \Fun(\sD, \Set) \ar[u]_{\textrm{restriction}}
}
\end{equation}
if and only if for every $X \in \sD$, the comma category $(X \downarrow f) = \hom_\sD(X, f(\cdot))/\sC$ is cofiltered. In particular, if $\sC$ has finite limits and $f$ commutes with them, then we get this commutative square.
\end{lemm}

\begin{rema} \label{rema:proComDiag}
Clearly, $X \in \sD$ is sent in \eqref{lemm:proComDiag:sq2} to the pro-object of $\sC$ indexed by $(X \downarrow f)$.
\end{rema}

\begin{proof}
The first part follows from the definition of left Kan extension as the unique colimit preserving functor compatible with coYoneda. The second part follows from Thm.~\ref{theo:proBasics}. For the ``in particular'' it suffices to show that $(X \downarrow f)$ is filtered. But this follows from $f$ sending finite limits to finite limits.
\end{proof}

\subsection{Pro-objects in localisations} \label{sec:liftProOb}

\begin{lemm} \label{lemm:commFilt}
Suppose that $\sC$ is a category, $S$ a class of morphisms satisfying a right calculus of fractions, and consider the canonical functor $loc: \sC \to \sC[S^{-1}]$. For any $X \in Ob(\sC) = Ob(\sC[S^{-1}])$ the comma category $(X \downarrow loc)$ is filtered.
\end{lemm}

\begin{proof}
The identity $\id_X$ shows the category is nonempty. Given any two objects $X \stackrel{s}{\leftarrow} X' \to Y$ and $X \stackrel{s'}{\leftarrow} X' \to Y'$ of the comma category, since $S$ satisfies a right calculus of fractions, we can assume that $s = s'$. Then $X \stackrel{s}{\leftarrow} X'$ is an object of the comma category mapping to both $Y$ and $Y'$. Similarly, given any two morphisms $u, v: Y \rightrightarrows Y'$ of $\sC$ and a morphism $X \stackrel{s}{\leftarrow} X' \stackrel{w}{\to} Y$ of $\sC[S^{-1}]$ such that $uws^{-1} = vws^{-1}$ in $\sC[S^{-1}]$, up to refining $s$, we can assume that $uw = vw$ in $\sC$. Then $w$ is a morphism in the comma category from $X \stackrel{s}{\leftarrow} X'$ to $X \stackrel{s}{\leftarrow} X' \stackrel{w}{\to} Y$ such that $uw = vw$ in the comma category.
\end{proof}

\begin{prop} \label{prop:proAdjoint}
Suppose that $\sC$ is a category and $S$ a class of morphisms satisfying a right calculus of fractions. Then
\[ \Pro(\sC) \to \Pro(\sC[S^{-1}]) \]
admits a fully faithful 
\emph{left} adjoint. The essential image of the left adjoint consists of those pro-objects $\sP$ such that $\hom_{\Pro(\sC)}(\sP, -)$ sends elements of $\sS$ to isomorphisms.

In particular, if $S$ consists of categorical monomorphisms, and 
$\sC$ is equipped with a class of families of morphisms $\tau$ (possibly empty) then localisation induces an equivalence of categories
\[ 
\left \{
\begin{array}{c}
S\textrm{-local } τ\textrm{-local} \\
\textrm{pro-objects of } \sC
\end{array}
\right \}
\cong
\left \{
\begin{array}{c}
τ\textrm{-local } \\
\textrm{pro-objects of } \sC[S^{-1}]
\end{array}
\right \}
\]
\end{prop}

\begin{proof}
Since the comma categories $(X \downarrow loc)$ are all filtered, cf.~Lem.~\ref{lemm:commFilt}, in addition to the square \eqref{lemm:proComDiag:sq1} of Lem.~\ref{lemm:proComDiag} we also get \eqref{lemm:proComDiag:sq2}. Note that due to the $(-)^{\op}$, restriction 
\[ restriction: \Fun(\sC[S^{-1}], \Set) \to \Fun(\sC, \Set) \]
corresponds to a left adjoint on the $\Pro(-)$'s. It is fully faithful by the universal property of localisation, and by the same property, we identify its image as those functors $F$ which send elements of $S$ to isomorphisms, or more importantly, the image of the induced fully faithful right adjoint functor $\Pro(\sC[S^{-1}]) \to \Pro(\sC)$ has as image those pro-objects $\sP$ such that $\hom_{\Pro(\sC)}(\sP, \cdot)$ sends elements of $S$ to isomorphisms. 

\end{proof}

\subsection{Fibre functors}

\begin{defi}[{\cite[Def.IV.6.2]{SGA41}}]
Let $\sC$ be a small category equipped with a topology $τ$. A \emph{fibre functor}\footnote{We insist on using ``fibre functors'' instead of ``points'' due to the widespread misuse of the term ``point''.} of the topos $\Shv_τ(\sC)$ is a functor $\Shv_τ(\sC) \to \Set$ which commutes with all colimits and finite limits.
\end{defi}

\begin{exam} \label{exam:topFibre}
Let $X$ be a topological space and $x \in X$. Then the functor $F \mapsto \colim_{x \in U} F(U)$ is a fibre functor of $\Shv(X)$. Notice that the pro-object $(U_λ)_{x \in U_λ}$ indexing the colimit is $τ$-local as a pro-object in the category of open subsets of $X$.
\end{exam}

The collection of fibre functors form a full subcategory of the category of all functors $\Shv_τ(\sC)$. 

We are interested in fibre functors because in nice situations they can be used to detect isomorphisms, (and therefore monomorphisms and epimorphisms) of sheaves.
The following result is due to Deligne (cf. \cite[Prop.~VI.9.0, Cor.~VI.9.5]{SGA42}, \cite[Thm.~7.44, Cor.~7.17]{Joh77}).

\begin{theo} \label{theo:deligne} 
Let $\sC$ be a small category admitting fibre products and equipped with a topology $τ$ such that every covering family is refinable by a covering family with finitely many elements. 

Then there exists a (small) conservative set of fibre functors of $\Shv_τ(\sC)$. 
That is, a set $\Phi$ such that a morphism 
\[ f: F \to G \]
of presheaves induces an isomorphism of sheaves if and only if 
\[ \phi(f) \]
is an isomorphism for every fibre functor $\phi \in \Phi$.
\end{theo}

Example~\ref{exam:topFibre} above is typical of fibre functors. 

\begin{prop}[{\cite[IV.6.8.7]{SGA41}}] \label{prop:fibArePro}
Let $\sC$ be a small category equipped with a topology $τ$. There is an equivalence of categories:
\[
\left \{ 
\begin{array}{c}
\textrm{ fibre functors } \\
\textrm{ of } \Shv_τ(\sC)
\end{array}
\right \}^{op}
\cong
\left \{
\begin{array}{c}
τ\textrm{-local } \\
\textrm{ pro-objects in } \sC
\end{array}
\right \}
\]
\end{prop}

Combining the above we obtain:

\begin{coro}
Let $\sC$ be a small category admitting fibre products and equipped with a topology $τ$ such that every covering family is refinable by a covering family with finitely many elements. 

Then a morphism of presheaves
\[ f: F \to G \]
of presheaves induces an isomorphism of sheaves if and only if 
\[ \colim_{λ} f(\sP_λ) \]
is an isomorphism for every $τ$-local pro-object $(\sP_λ)_λ$ of $\sC$.
\end{coro}



\section{Integral closure}

In the study of modulus pairs, it is often useful to take normalisation of ambient spaces (under some finiteness constraints). 
In this section, we study the behavior of such operation. 
In particular, we want to know ``$\sO_{\ol{X}}$ integrally closed in $\sO_{X^\circ}$'' is preserved by {\'e}tale morphisms (Corollary \ref{cor:etIntClo}).

\subsection{A characterisation of integral closures}

Recall that if $A$ is an integral domain with fraction field $B = Frac(A)$, then the normalisation $C$ can be expressed as
\[ C = \bigcap R \subseteq B \]
where the intersection is over all valuation rings $R$ of $B = \Frac(A)$ containing $A$. This is a special case of the following more general expression. Recall that a ring homomorphism $A \to C$ \emph{satisfies the valuative criterion for properness} if for every commutative square
\begin{equation} \label{equa:valCri}
\xymatrix{
K & \ar[l] B \ar[l] \\
R \ar[u]^{\cup|} & A \ar[u] \ar[l]
}
\end{equation}
such that $R$ is a valuation ring with fraction field $K$, there exists a diagonal morphism $B \to R$ making the two resultant triangles commute. Note uniqueness is automatic because $R \subseteq K$ is injective.

\begin{prop} \label{prop:intClos}
Suppose that $\phi:A \to B$ is any ring homomorphism, and let $C \subseteq B$ be the integral closure of $A$ in $B$. Then $A \to C \subseteq B$ is the largest sub-$A$-algebra of $B$ satisfying the valuative criterion for properness. More explicitly, we have 
\[ C = \bigcap (R \times_{K} B)  \]
where the intersection is over commutative squares as in \eqref{equa:valCri} (note $R \times_{K} B \subseteq B$ since $R \subseteq K$).
\end{prop}

\begin{proof}
First recall than a morphism of rings is integral if and only if it satisfies the valuative criterion \cite[01WM, 01KF]{stacks-project}. Certainly, $A \to C$ is integral by definition, so $C \subseteq R \times_{K} B$ for all squares \eqref{equa:valCri}. On the other hand, if $b \in B$ is in $\bigcap (R \times_{K} B)$, then $A \to C[b] \subseteq B$ satisfies the valuative criterion, so $A \to C[b]$ is integral, and therefore $b \in C$.
\end{proof}

\begin{rema} \label{rema:valCriIntSurj}
In the intersection of Prop.~\ref{prop:intClos}, it suffices to consider squares \eqref{equa:valCri} such that $R \otimes_A B \to K$ is surjective. Indeed, given a general square 
\begin{equation}
\xymatrix{
K' & \ar[l] B  \\
R' \ar[u]^{\cup|} & A \ar[u] \ar[l]
}
\end{equation}
The image $im(R \otimes_A B {\to} K')$ is a localisation $R_\p = im(R \otimes_A B {\to} K')$ of the valuation ring $R$ at some prime $\p$, since it's a subring of a field containing a valuation ring. Then considering the bicartesian\footnote{Bicartesianness is well-known by now, but see \cite[Lem.2.14]{HK18} for a proof.} square
\begin{equation}
\xymatrix{
k(\p) & \ar[l] R_\p \ar[l] \\
R/\p \ar[u]^{\cup|} & R \ar[u] \ar[l]
}
\end{equation}
one deduces that the valuative criterion for the square of $R/\p \subseteq k(\p)$ implies it for that of $R \subseteq K$. The morphism $(R/\p) \otimes_A B \to k(\p)$ is surjective because $R \otimes_A B \to R_\p \to k(\p)$ are surjective. 
\end{rema}

\begin{rema} \label{rema:valCriIntIso}
Remark~\ref{rema:valCriIntSurj} can be specialised to the case that $B = A[f^{-1}]$ is a localisation. In this case, it says that it suffices to consider squares \eqref{equa:valCri} such that $R[f^{-1}] \cong K$ (where we have confused $f \in A$ with its image $f \in R$).
\end{rema}

\begin{rema}
The category of squares \eqref{equa:valCri} in the valuative criterion is not filtered, but if we allow $R$ to be a finite product of semilocal Pr{\"u}fer domains, it becomes filtered (and the result stays the same). However, any limit is a filtered limit over finite limits, so this is perhaps not necessary depending on the desire application. Allowing these more general squares also means that the category becomes stable under {\"e}tale base change, which may be useful.
\end{rema}

\subsection{Integral closure is preserved by {\'e}tale morphisms}

\begin{prop}[{\cite[03GE]{stacks-project}}] \label{prop:closureetale}
Suppose that $\phi: A \to B$ is a ring homomorphism, and let $C \subseteq B$ be the integral closure of $A$ in $B$. Then for any {\'e}tale morphism $A \to A'$, the tensor product $A' \otimes_A C$ is the integral closure of $A'$ in $A' \otimes_A B$.
\end{prop}

\begin{cor} \label{cor:etIntClo}
Suppose that $(\ol{X}, X^\infty)$ is an modulus pair such that $\sO_{\ol{X}}$ is integrally closed in $\sO_{X^\circ}$, and $\ol{Y} \to \ol{X}$ is an {\'e}tale morphism. Then $\sO_{\ol{Y}}$ is integrally closed in $\sO_{Y^\circ}$, where $Y^\circ = \ol{Y} \times_{\ol{X}} X^\circ$.
\end{cor}

\subsection{Divisors which die in valuation rings die on an admissible blowup}

\begin{prop} \label{prop:detectZeroOnValRing}
Let $\sX$ be a modulus pair, and suppose that for every minimal morphism $\sP \to \sX$ such that $\hP$ is the spectrum of a valuation ring we have $\mP = 0$. Then $\mX = 0$.
\end{prop}

\begin{proof}
By Prop.~\ref{prop:intClos} applied to $B = A[f^{-1}]$ we see that for any open affine $\Spec(A) \subseteq \hX$ such that $\mX$ is globally principle with generator $f \in A$ on $\Spec(A)$, the element $\tfrac{1}{f} \in A[f^{-1}]$ is integral over $A$. So $\Spec(A[f^{-1}]) \to \Spec(A)$ is an open immersion which satisfies the going up property, and is therefore an isomorphism.
\end{proof}

\begin{prop} \label{prop:detectEqualOnValRing}
Let $\sX$ be a modulus pair, $\mX \leq \mY$ an effective Cartier divisor on $\hX$, and suppose that for every minimal morphism $\sP \to \sX$ such that $\hP$ is the spectrum of a valuation ring and $\iP$ is the generic point, we have $\mX|_{\hP} = \mY|_{\hP}$. Then $\mX = \mY$.
\end{prop}

\begin{proof}
Since $\mX|_{\hP} = \mY|_{\hP}$, a morphism $\hP \to \hX$ induces a minimal morphism for $(\hX, \mX)$ if and only if induces a minimal morphism for $(\hX, \mY)$. By Prop.~\ref{prop:intClos} applied to $B = A[g^{-1}]$ we see that for any open affine $\Spec(A) \subseteq \hX$ such that $\mX, \mY$ are globally principal with respective generators $f, g \in A$ on $\Spec(A)$, the elements $\tfrac{g}{f}, \tfrac{f}{g} \in A[g^{-1}]$ are integral over $A$. Note, $\mX \leq \mY$ implies that in fact, $\tfrac{g}{f} \in A$, so $\Spec(A[\tfrac{f}{g}]) \to \Spec(A)$ is an integral dense open immersion, i.e., an isomorphism, so $\tfrac{f}{g}$ is in $A$ as well. Consequently, $\mX|_{\Spec(A)} = \mY|_{\Spec(A)}$ for every open in a covering of $\hX$, so $\mX = \mY$ globally.
\end{proof}

\bibliographystyle{amsalpha}
\bibliography{bib}

\printindex
\printindex[not]

\end{document}